\documentclass[11pt,leqno]{article}
\usepackage{amsmath,amssymb,mathrsfs,amsthm,bm,ddag}

\usepackage[normalem]{ulem}
\usepackage{eucal}

\usepackage{hyperref}

\usepackage{graphicx}

\usepackage{tikz}
\usepackage{xy}
\input{xy}
\xyoption{all}
\xyoption{pdf}

\usepackage{threeparttable}

\usepackage{xspace}

\usepackage[geometry]{ifsym}

\usepackage{comment}

\makeatletter
\def\@seccntformat#1{\csname the#1\endcsname.\hspace{2ex}}

 \renewcommand{\subsection}%
  {\@startsection{subsection}%
  {2}%
  {\z@}%
  {2ex}
  {0ex}
  {\reset@font\normalsize\bfseries}}%

 \@addtoreset{equation}{subsection}

 \newcommand{\nsection}{\@startsection{section}{1}{\z@}%
     {-5ex}
     {1ex}
     {\reset@font\center\large\sc}}

 \renewenvironment{thebibliography}[1]
 {\nsection*{\refname\@mkboth{\refname}{\refname}}%
   \list{\@biblabel{\@arabic\c@enumiv}}%
   {\settowidth
   \labelwidth{\@biblabel{#1}}%
   \leftmargin
	\labelwidth
        \advance
	 \leftmargin
	 \labelsep
         \@openbib@code
         \usecounter{enumiv}%
         \let\p@enumiv\@empty
	 \parskip=0pt
	 \itemsep=1pt
	 \parsep=1pt
	 \itemindent=\z@
         \renewcommand\theenumiv{\@arabic\c@enumiv}}%
   	 \sloppy
   	 \clubpenalty4000
   	 \@clubpenalty\clubpenalty
   	 \widowpenalty4000%
   	 \footnotesize
   	 \sfcode`\.\@m}
  	 {\def\@noitemerr
    	 {\@latex@warning{Empty `thebibliography' environment}}%
   	 \endlist}

\makeatother

\allowdisplaybreaks[2]

\newtheoremstyle{thm}
 {1em}
 {3pt}
 {\itshape}
 {}
 {\bf}
 {. ---}
 {0.5em}
 {}
\newtheoremstyle{dfn}
 {1em}
 {3pt}
 {}
 {}
 {\bf}
 {. {---}}
 {0.5em}
 {}

\swapnumbers
\theoremstyle{thm}
\newtheorem{thm}[subsection]{Theorem}
\newtheorem{lem}[subsection]{Lemma}
\newtheorem*{lem*}{Lemma}
\newtheorem{cor}[subsection]{Corollary}
\newtheorem*{cor*}{Corollary}
\newtheorem{prop}[subsection]{Proposition}
\newtheorem*{prop*}{Proposition}

\newtheorem*{conj*}{Conjecture}
\newtheorem*{thm*}{Theorem}

\theoremstyle{dfn}
\newtheorem{dfn}[subsection]{Definition}
\newtheorem*{dfn*}{Definition}

\newtheorem*{ex*}{Example}

\newtheorem*{rem*}{Remark}

\usepackage{color}
\newenvironment{meta}{
\noindent \color{red}
\sffamily[}{\upshape]}

\usepackage{framed}
\definecolor{shadecolor}{gray}{0.80}
\specialcomment{test}%
     {\begin{shaded}
       \noindent\textcolor{red}{{\large\bf Personal comment}}\\
       }%
      {\end{shaded}}%
\excludecomment{test} 

\newsavebox{\circlebox}
\savebox{\circlebox}{\fontencoding{OMS}\selectfont\char13}
\newlength{\circleboxwdht}
\newcommand{\ccirc}[1]{
  \setlength{\circleboxwdht}{\wd\circlebox}
  \addtolength{\circleboxwdht}{\dp\circlebox}
  \raisebox{0.2\dp\circlebox}{
    \parbox[][\circleboxwdht][c]{\wd\circlebox}
    {\centering\scriptsize #1}}
  \llap{\usebox{\circlebox}}
}

\renewcommand{\H}{\mc{H}^*}
\newcommand{\rH}{\mr{H}^*}

\newcommand{\cH}{\mc{H}_{\mr{c}}^*}

\newcommand{\ul}[1]{\underline{#1}}

\newcommand{\Hbm}{\mr{H}^{\mr{BM}}}

\newcommand{\Cat}{\mc{C}\mr{at}}
\newcommand{\Shv}{\mc{S}\mr{hv}}
\newcommand{\Spc}{\mc{S}\mr{pc}}
\newcommand{\PShv}{\mc{P}}
\newcommand{\sSet}{\mc{S}\mr{et}_{\Delta}}

\newcommand{\Map}{\mr{Map}}

\newcommand{\Mor}{\mr{Mor}}
\newcommand{\sMor}{\ul{\mr{Mor}}}

\newcommand{\Mdf}{\mr{Mdf}}

\newcommand{\Mod}{\mr{Mod}}

\newcommand{\Hcoe}{\ul{\mr{H}}^{\mr{BM}}}
\newcommand{\LinCat}{\mc{L}\mr{in}\mc{C}\mr{at}}
\newcommand{\PrL}{\mc{P}\mr{r}^{\mr{L}}}
\newcommand{\Corr}{\mc{C}\mr{orr}}
\newcommand{\Comp}{\mr{Fl\acute{E}t}}

\newcommand{\rk}{d}
\newcommand{\Sys}{\mc{S}\mr{ys}}

\newcommand{\RTop}{\mc{R}\mc{T}\mr{op}}
\newcommand{\LTop}{\mc{L}\mc{T}\mr{op}}

\newcommand{\Fun}{\underline{\mr{Fun}}}
\newcommand{\site}{\mc{S}\mr{ite}}
\newcommand{\Eff}{\ms{N}}
\newcommand{\til}{\mr{fat}}
\newcommand{\Tot}{T}

\begin{document}
\title{Ramification theory from homotopical point of view, I}
\author{Tomoyuki Abe}
\date{}
\maketitle

\section*{Introduction}
In the theory of $\ms{D}$-modules for complex varieties, the characteristic cycle plays an important role,
both in establishing the theory of $\ms{D}$-modules itself and applications.
In the context of the $\ell$-adic cohomology theory, similarities with the ramification theory were observed by experts including P. Deligne and G. Laumon.
Since then, establishing the theory of characteristic cycles in the \'{e}tale setting had been an important challenge in the ramification theory of \'{e}tale sheaves.
A breakthrough was brought by A. Beilinson and T. Saito in around 2015.
Beilinson defined the singular support for \'{e}tale sheaves, and based on this work, Saito finally succeeded in defining the characteristic cycle.
Saito also established several functorial properties of the characteristic cycle, but the compatibility with proper pushforward,
which could be seen as a vast generalization of Grothendieck-Ogg-Shafarevich formula, remained as a conjecture \cite[Conjecture 1]{S2}.
One of our main goals of this paper is to solves the conjecture of Saito up to $p$-torsion.
Here is one of the main results of this paper (cf.\ Remark \ref{compwithsaitocor}):
\begin{thm*}
 Let $k$ be a perfect field of characteristic $p>0$,
 and let $\Lambda$ be a finite local ring whose residue field is of characteristic different from that of $k$.

 Now, let $h\colon X\rightarrow Y$ be a morphism between separated\footnote{
 The theory of characteristic cycles works without separatedness assumption, and it is not too hard to lift this assumption by using Zariski descent.
 However, to avoid unnecessary complication, we leave the details to the interested reader.}
 equidimensional smooth $k$-schemes of finite type.
 Then we have the diagram $T^*X\xleftarrow{g}T^*Y\times_Y X\xrightarrow{h'}T^*Y$.
 For a closed subset $C$ of $T^*X$, we put $h_\circ(C):=h'g^{-1}(C)$.
 Let $\mc{F}$ be a constructible finite Tor-dimensional complex of $\Lambda$-sheaves such that the morphism
 $\mr{Supp}(\mc{F})\rightarrow Y$ is proper.
 Then we have the identity
 \begin{equation*}
  \mr{CC}(\mr{R}h_*\mc{F})=h'_*g^!\bigl(\mr{CC}(\mc{F})\bigr)
 \end{equation*}
 in $\mr{CH}_{\dim(Y)}(h_{\circ}\mr{SS}(\mc{F}))[1/p]$.
\end{thm*}
The case where $Y=\mr{Spec}(k)$ and $h$ is projective, the theorem has been proven in the fundamental paper \cite{S}.
In the case where $h$ is projective, and $\dim(h_\circ\mr{SS}(\mc{F}))\leq\dim(Y)$ (and slightly more conditions), the theorem has been proven in \cite{S3}.
Unfortunately, contrary to the characteristic $0$ situation, this inequality is frequently not satisfied ({\em e.g.}\ when $h$ is the Frobenius endomorphism),
which demonstrates the complexity of the ramification theory for positive characteristic situation.
The biggest difference from this theorem of Saito is that, under the dimension assumption, we have
$\mr{CH}_{\dim(Y)}(h_\circ\mr{SS}(\mc{F}))=\mr{Z}_{\dim(Y)}(h_\circ\mr{SS}(\mc{F}))$,
and the equality is an equality of actual cycles, whereas in the general case,
we must {\em construct} a homotopy between two cycles.
Very roughly speaking, we locally construct such a homotopy using the deformation argument of Laumon,
and we glue them together using the $\infty$-categorical machinery.
For simplicity, in this introduction, we assume $k$ to be algebraically closed and $\Lambda$ to be a field.

\subsection{}
Let us briefly recall the theory of Beilinson and Saito.
Let $X$ be a smooth scheme of dimension $d$ over $k$ and $\mc{F}$ be a bounded constructible $\Lambda$-complex.
Then the characteristic cycle $\mr{CC}(\mc{F})$ is the unique cycle $\sum n_i[C_i]$, where $C_i$ is an integral subscheme of $T^*X$ of dimension $d$ such that
for any ``isolated characteristic function'' $f\colon U\rightarrow\mb{A}^1$ for an open subscheme $U\subset X$, we have
\begin{equation}
 \label{intMilnorformu}
 -\mr{totdim}_u\Phi_f(\mc{F}|_U)=\bigl(\sum n_i[C_i],\mr{d}f\bigr)_{T^*U,u}
\end{equation}
for any point $u\in U$.
This is called the {\em Milnor formula}.
Here $\mr{totdim}$ is the total dimension, and roughly speaking, an isolated characteristic function is a function $f$ such that
the formal sum $\sum_{u\in U}(\mr{totdim}_u\Phi_f(\mc{F}))\cdot[u]$ defines a $0$-cycle on $U$.
Beilinson's theorem on the existence of singular support implies that we have ``enough supply'' of isolated characteristic functions
so that we can determine characteristic cycle uniquely.
A difficulty of proving pushforward formula is that if we are given a proper morphism $h\colon X\rightarrow Y$,
even if we are given an isolated characteristic function $f$ on $Y$ with respect to $\mr{R}h_*\mc{F}$,
the pullback $h^*f$ may fail to be isolated charateristic with respect to $\mc{F}$.
We revisit the construction of characteristic cycle, and define it without using Beilinson's theorem and in a wider context.

\subsection{}
Let us explain our method.
The pullback $0^*\mr{CC}(\mc{F})$ in $\mr{CH}_0(X)$, where $0\colon X\rightarrow T^*X$ is the zero-section, is called the
{\em non-microlocalized} characteristic cycle\footnote{
It is more common to call it {\em characteristic class}.
}.
Following Beilinson \cite{B}, we reverse the order of the construction, and we construct the non-microlocalized one first.
We use homotopical argument for the microlocalization, but since the microlocalization process is essentially parallel to \cite{B},
we do not go into the details in this introduction.

To explain the construction of non-microlocalized characteristic cycle,
let us introduce a (pre)sheaf $\Comp_X$ on $\mr{Sch}_{/k}$ of {\em flat \'{e}tale systems}.
The definition of this sheaf is modeled on the definition of Suslin-Voevodsky's relative cycle group in \cite{SV},
and we make use of the theory of nearby cycles over general base.
Let $X$ be a $k$-scheme separated of finite type.
In Definition \ref{neadimboun} (see also Remark \ref{neadimboun}), we will introduce a presheaf $\Comp_X$ having the following properties:
\begin{enumerate}
 \item The global section is $\mr{K}_0\mr{Cons}(X)$,
       the Grothendieck group of the category of constructible $\Lambda$-modules on $X$.

 \item For a $k$-scheme $T$, $\Comp_X(T)$ is isomorphic to the subgroup of $\prod_{u\in\mr{geom}(T)}\mr{K}_0\mr{Cons}(X_{u})$,
       where $\mr{geom}(T)$ is the set of isomorphism classes of geometric points of $T$, consisting of $\{F_u\}_{u\in\mr{geom}(T)}$
       such that $F_u$ are ``good'' and for any specialization map $u\rightsquigarrow t$, we have $\Psi_{t\leftarrow u}(F_u)=F_t$.
\end{enumerate}
On the other hand, consider the six functor formalism of motivic cohomology developed by Voevodsky, Ayoub, Cisinski-D\'{e}glise.
It associates to each $T\in\mr{Sch}_{/k}$ an $R$-linear symmetric monoidal triangulated category $D_T$ and they admit six fundamental operations.
Here $R$ is a commutative ring such that $\mr{char}(k)^{-1}\in R$.
This triangulated category has an $\infty$-enhancement, which we denote by $\mc{D}_T$.
Using this formalism, we have the presheaf $\pi_0\Hcoe_X$ and the $\infty$-presheaf $\Hcoe_X$ sending $T\in\mr{Sch}_{/k}$ to
\begin{equation*}
 \pi_0\Hbm(X_T/T)\simeq\mr{Hom}_{D_T}\bigl(h_{T!}h_T^*\mbf{1}_T,\mbf{1}_T\bigr),
  \qquad
  \Hbm(X_T/T)\simeq\Map_{\mc{D}_T}\bigl(h_{T!}h_T^*\mbf{1}_T,\mbf{1}_T\bigr),
\end{equation*}
where $\mbf{1}_T$ is the unit object of $\mc{D}_T$.
In other words, these are relative Borel-Moore homologies.
Note that $\pi_0\Hcoe_X(k)\cong\mr{CH}_0(X;R)$.
Using these presheaves, our goal is to construct a (sensible) homomorphism $\mr{CC}\colon\Comp_X(k)\rightarrow\pi_0\Hcoe_X(k)$ between the global sections.
We construct not only the morphism between the global sections, but a morphism of presheaves $\Comp_X\rightarrow\pi_0\Hcoe_X$.
Let $\Comp_{X,0}\subset\Comp_X$ be the subpresheaf consisting of objects whose support is of dimension $0$.
Then by taking the rank, we may construct a morphism $\tau\colon\Comp_{X,0}\rightarrow\pi_0\Hcoe_X$ using the trace map.
Since the characteristic cycle of a constructible sheaf with $0$-dimensional support should merely be the zero cycle associated to the rank,
we should have $\mr{CC}|_{\Comp_{X,0}}=\tau$.
Now, we may ask ourselves how many extensions $\mr{CC}$ of $\tau$ there can be.
Surprisingly, if we replace $\pi_0\Hcoe_X$ with the ($\infty$-)presheaf $\Hcoe_X$, we have the following extension theorem:
\begin{thm*}
 The restriction map
 \begin{equation*}
  \mr{Hom}_{\mr{h}\PShv(\mr{Sch}_{/k})}(\Comp_X,\Hcoe_X)
   \rightarrow\mr{Hom}_{\mr{h}\PShv(\mr{Sch}_{/k})}(\Comp_{X,0},\Hcoe_X)
 \end{equation*}
 is an isomorphism, where $\mr{h}\PShv$ denotes the category of {\normalfont(}$\mb{Z}$-valued{\normalfont)} $\infty$-presheaves.
\end{thm*}
Let us explain the rough idea of the proof.
Assume that for each $F\in\Comp_X(T)$, we can find $F'\in\Comp_X(T\times\Box)$, where $\Box:=\mb{P}^1\setminus\{1\}$,
such that $F'|_{T\times\{0\}}=F$ and $F'|_{T\times\{\infty\}}\in\Comp_{X,0}(T)$.
Consider $\mr{CC}(F')\in\Hcoe_X(T\times\Box)$.
Since $\Hcoe_X$ is $\mb{A}^1$-invariant, we have the diagram
\begin{equation*}
 \xymatrix{
  &\Hcoe_X(T)\ar[d]^{q^*}\ar@/^10pt/@{=}[dr]\ar@/_10pt/@{=}[dl]&\\
 \Hcoe_X(T)&
  \Hcoe_X(T\times\Box)
  \ar[r]^-{i_{\infty}^*}_-{\sim}\ar[l]_-{i_0^*}^-{\sim}&
  \Hcoe_X(T),}
\end{equation*}
where $i_x\colon T\times\{x\}\hookrightarrow T\times\Box$ is the closed immersion, and $q\colon T\times\Box\rightarrow T$ is the projection.
The equivalence on the right implies that $\mr{CC}(F')$ is equal to $(i_\infty^*)^{-1}\tau(i_\infty^*F')$.
Thus, the diagram implies that $\mr{CC}(F)=\tau(i_\infty^*F')$,
and the value of $\mr{CC}(F)$ is determined once we fix $F'$.
The element $F'$ can be regarded as a ``deformation'' of $F$, and this deformation ``localizes'' $F$ to an element of $\Comp_{X,0}$ at $\infty$.
We can summarize the above observation by saying that if we have enough supply of deformations, then the extension will be determined uniquely.

Now, how can we construct meaningful deformations?
In the celebrated work of Laumon \cite{Lau}, he considered interesting deformations of $\ell$-adic sheaves inspired by works of E. Witten.
Following this idea, we will take $F'$ as $\mc{E}(F;f):=F\otimes\mc{L}(ft)$ where $f$ is a (local) function on $X$ parametrized by $T$,
$t$ is the coordinate of $\Box$, and $\mc{L}$ is the Artin-Schreier sheaf associated with a non-trivial additive character.
We remark that if we take $f$ to be an isolated characteristic function with respect to $F\in\Comp_X(k)$,
then $\tau i_\infty^*\mc{E}(F;f)$ coincides with the $0$-cycle in (\ref{intMilnorformu}).
This comparison is a (rather straightforward) consequence of Laumon's computation in \cite{Lau}.

\subsection{}
\label{propofsys}
The actual situation is more complicated in various points, but in this introduction, we focus on the following two obstacles:
\begin{itemize}
 \item Unfortunately, the deformation $\mc{E}(F;f)$ is not an element of $\Comp_X(T\times\Box)$.
       What we can say is that there exists a modification $T'\rightarrow T\times\Box$ over which $F\otimes\mc{L}(ft)$ becomes an element of $\Comp_X$.
       This deficit reflects the philosophy of the theory of nearby cycles over general base
       that the situation is ``good'' only after taking a suitable modification on the base.

 \item In order to kill the arbitrariness of the function of deformation, we need to consider iterative deformation,
       and must show that this process terminates in finite steps.
\end{itemize}
To treat the first point, we introduce an $\infty$-presheaf $\partial\mc{F}$ for any ($\infty$-)presheaf $\mc{F}$ on $\mr{Sch}_{/k}$.
Intuitively, $\partial\mc{F}$ is the sheafification of $\mc{F}$ with respect to a variant of cdp-topology.
Roughly speaking, a section of $\partial\mc{F}(T)$ (more or less) corresponds to a section of $\mc{F}(V\times_{\Box}\{\infty\})$ for some modification $V\rightarrow T\times\Box$.
Then, even though our $F'$ does not belong to $\Comp_X(T\times\Box)$, it {\em does} belong to $\Comp_X(V)$ for some $V$ and $i_{\infty}^*F'$ is well-defined in $\partial\Comp_X(T)$.
Using this operation, we may formulate the deformation cleanly; however the above strategy breaks down because the resulting morphism becomes
$\Comp_X\rightarrow\partial\Hcoe_X$ and $\partial\Hcoe_X$ is no longer $\mb{A}^1$-invariant.
If we consider the abelian presheaf $\pi_0\Hcoe_X$ in place of $\Hcoe_X$, we are unable to proceed further.
However, if we consider the $\infty$-presheaf $\Hcoe_X$,
we may find a section to the canonical map $\Hcoe_X\rightarrow\partial\Hcoe_X$, which we call the {\em rewind morphism}.
A crucial difference between $\pi_0\Hcoe_X$ and $\Hcoe_X$ in the construction of rewind morphism lies in the fact that $\Hcoe_X$ admits cdh-descent.
The idea of construction of rewind morphism will be explained shortly.

For the second point, let us first explain why we need multiple deformations.
As we have seen, deformations are not unique, but we consider deformations parameterized by a space of functions.
Assume we are given $F\in\Comp_X(T)$, and two functions $f_1$, $f_2$ on $X$ parameterized by $T$.
Then we may consider deformations $\mc{E}(F;f_1)$, $\mc{E}(F;f_2)$.
Assume we are trying to extend $\tau$ to $\mr{CC}$, and assume we are in a good situation where $\mc{E}(F;f_1)$, $\mc{E}(F;f_2)$ belong to
$\Comp_X(T\times\Box)$ (not just $\partial\Comp_X(T)$) and even nicer $f_i$ localizes $F$ to an element of $\Comp_{X,0}$,
namely $\mc{E}_{\infty}(F;f_i):=i^*_{\infty}\mc{E}(F;f_i)\in\Comp_{X,0}(T\times\{\infty\})$.
Then the observation above implies that $\mr{CC}(F)$ should be equal to $\tau(\mc{E}_{\infty}(F;f_i))$.
However, we have two choices, $i=1,2$.
This implies that to define the map $\mr{CC}$ in this manner, we at least have
$\tau(\mc{E}_{\infty}(F;f_1))=\tau(\mc{E}_\infty(F;f_2))$.
A benefit of considering presheaves on $\mr{Sch}_{/k}$ is that the required information to construct the desired equality is already contained!
Indeed, we may take a function $g$ on $X$ parameterized by $T\times\mb{A}^1$ such that the restriction to $T\times\{i\}$ is equal to $f_i$.
Let $\mr{pr}\colon T\times\mb{A}^1\rightarrow T$ be the projection, and we may consider the deformation $\mc{E}(\mr{pr}^*F;g)$.
Assume for simplicity that we are in a lucky situation where $\mc{E}(\mr{pr}^*F;g)\in\Comp_X(T\times\mb{A}^1)$.
If $\mc{E}_\infty(\mr{pr}^*F;g)$ belongs to $\Comp_{X,0}(T\times\mb{A}^1)$, then
$\tau\mc{E}_\infty(\mr{pr}^*F;g)$ yields the ``homotopy'' which gives the desired equality.
However, $\mc{E}_\infty(\mr{pr}^*F;g)$ often does not belong to $\Comp_{X,0}$.
Nevertheless, we do not need to be desperate.
We may take a secondary function $g_2$, and try to deform $\mc{E}_\infty(\mr{pr}^*F;g)$ using that function.
More precisely, we may consider $\mc{E}_\infty(\mc{E}_\infty(\mr{pr}^*F;g);g_2)$ and see if this belongs to $\Comp_{X,0}$ or not.
If this one does not belong to $\Comp_{X,0}$, we take a tertiary function and we proceed this process.
In order for this strategy to work, one finds immediately that there should exist a sequence of functions of {\em finite length} $g_1,\dots,g_n$
such that the iterated deformation $\mc{E}_\infty(\dots\mc{E}_{\infty}(\mr{pr}^*F;g_1)\dots;g_n)$ belongs to $\Comp_{X,0}$.
One of the hardest fact to show in this work is that, in fact, we can always find such a finite sequence.
For the proof of this fact, we need new ideas, and the fact will be shown in Part II of this paper.
For more details, we refer to the introduction of [Part II].

Looking this construction closely, we notice that the following properties of
$\Comp_X$ and $\Hcoe_X$ are needed for the uniqueness and existence of the extension:
\begin{enumerate}
 \item\label{propofsys-1}
      For any element of $F\in\Comp_X(T)$ and for any function $f$ on $X$ parameterized by $T$, we have $\mc{E}(F;f)\in\partial\Comp_X(T)$
      (functorially with respect to $F$).
      Furthermore, we can always find a sequence $g_1,\dots,g_n$ so that $\mc{E}_\infty(F;g_1,\dots,g_n)$ belongs to $\Comp_{X,0}$.

 \item\label{propofsys-2}
      The sheaf $\Hcoe_X$ is $\mb{A}^1$-invariant and admits a rewind structure.
\end{enumerate}
A pair of presheaf $\ms{S}_0\subset\ms{S}$ which has properties \ref{propofsys-1} is called a {\em localizing pair},
and a ($\infty$-)presheaf having properties \ref{propofsys-2} is called a {\em reversible coefficient} (cf.\ \ref{mainconscc}, \ref{defofcoefsys}).
The above argument implies that given a localizing pair $\ms{S}_0\subset\ms{S}$ and a reversible coefficient $\mc{H}$,
any morphism $\ms{S}_0\rightarrow\mc{H}$ extends (essentially) uniquely to a morphism $\ms{S}\rightarrow\mc{H}$ up to canonical homotopy.

\subsection{}
We have reduced the construction of the morphism $\Comp_X\rightarrow\Hcoe_X$ to
showing that $\Comp_{X,0}\subset\Comp_X$ is a localizing pair and to constructing a rewind structure on $\Hcoe_X$.
As we mentioned, the proof that $\Comp_{X,0}\subset\Comp_X$ is a localizing pair will be given in Part II.
Let us explain how we endow $\Hcoe_X$ with rewind structure.
Note that, up till now, we have only used $1$-categorical argument, but this is one of the points that requires the use of $\infty$-categorical arguments,
since we need ``gluing of cohomological elements''.

Recall that, for an $\infty$-sheaf $\mc{F}$, a section of $\partial\mc{F}(T)$ corresponds (more or less) to a section of
$\mc{F}(V_\infty)$ ($V_\infty:=V\times_{\Box}\{\infty\}$) for some modification $V\rightarrow T\times\Box$.
Thus, we wish to construct an element of $\Hcoe_X(T)$ out of an element $F\in\Hcoe_X(V_\infty)$.
Let $\alpha\colon X_{T\times\Box}\rightarrow T\times\Box$, and put $\cH:=\alpha_!\alpha^*\mbf{1}_{T\times\Box}$.
Let $\eta$ be the generic point of $\Box$.
By using cdh-descent (or cdp-descent), we can assume that $V\rightarrow T\times\Box$ is an isomorphism over $T\times\eta$.
For a closed point $x$ of $\Box$, let $i_{x}\colon T\times\{x\}\hookrightarrow T\times\Box$ and $j_{}\colon T\times\Box'\hookrightarrow T\times\Box$
where $\Box':=\Box\setminus\{\infty\}$ be the immersions.
Put $i:=i_\infty$.
Since we have the morphism $\Hcoe_X(V_{\infty})\rightarrow\Map(\cH,i_{*}i_{}^*j_{*}\mbf{1}_{T\times\Box'})$,
it suffices to construct a retract of
$\Hcoe_X(T)\simeq\Map(\cH,i_{*}i_{}^*\mbf{1}_{T\times\Box})\rightarrow\Map(\cH,i_{*}i_{}^*j_*\mbf{1}_{T\times\Box'})$.
In view of the localization exact triangle
$\mbf{1}_{T\times\Box}\rightarrow j_{*}\mbf{1}_{T\times\Box'}\rightarrow i_{*}i^!_{}\mbf{1}_{T\times\Box}[1]$,
this is equivalent to giving a {\em section} of
$\Map(\cH,i_{*}i_{}^*j_{*}\mbf{1}_{T\times\Box'})\rightarrow\Map(\cH,i_{*}i^!_{}\mbf{1}_{T\times\Box}[1])$.
Using the localization sequence again, constructing the section
is equivalent to constructing a homotopy between the morphism $\Map(\cH,i_{*}i^!_{}\mbf{1}_{T\times\Box}[1])\rightarrow\Map(\cH,\mbf{1}_{T\times\Box}[1])$ and $0$.
Let $y\in|\Box|$.
Applying the functor $i_{y*}i^*_y$ to the morphism $i_{*}i^!_{}\mbf{1}_{T\times\Box}\rightarrow\mbf{1}_{T\times\Box}$,
we have the homotopy commutative diagram
\begin{equation*}
 \xymatrix@C=50pt{
  \Map(\cH,i_{*}i^!_{}\mbf{1}_{T\times\Box}[1])\ar[d]\ar[r]&\Map(\cH,\mbf{1}_{T\times\Box}[1])\ar[d]^{\sim}\\
 \Map(\cH,i_{y*}i_y^*i_{*}i^!_{}\mbf{1}_{T\times\Box}[1])\ar[r]&\Map(\cH,i_{y*}i_y^*\mbf{1}_{T\times\Box}[1]).
  }
\end{equation*}
The right vertical morphism is an equivalence since the theory is $\mb{A}^1$-invariant, and if we take $y\neq\infty$,
the left bottom object is $0$, which yields the desired homotopy.

\subsection{}
After these discussions, we are in the position to prove the non-microlocalized version of the pushforward formula.
Assume we are given a proper morphism $h\colon X\rightarrow Y$.
The commutation of the nearby cycle functor and $\mr{R}h_*$ induces the morphism $\Comp_X\rightarrow\Comp_Y$.
By the functoriality of Borel-Moore homology, we also have the pushforward $\Hcoe_X\rightarrow\Hcoe_Y$.
For the pushforward formula, we need to construct the following (homotopy) commutative diagram of $\infty$-sheaves
\begin{equation*}
 \xymatrix{
  \Comp_X\ar[r]^-{\mr{CC}_X}\ar[d]_{h_*}&\Hcoe_X\ar[d]^{h_*}\\
 \Comp_Y\ar[r]^-{\mr{CC}_Y}&\Hcoe_Y.
  }
\end{equation*}
Since $\Comp_{X,0}\subset\Comp_X$ is a localizing pair and $\Hcoe_Y$ is a reversible coefficient,
if we fix a morphism $\Comp_{X,0}\rightarrow\Hcoe_Y$, extensions are all homotopic.
Thus, in order to construct the commutative diagram above, we only need to construct a commutative diagram for $\Comp_{X,0}$.
This is just a property of the trace map, and we have the pushforward formula.

There are also other interesting consequences of our construction.
First, non-microlocalized version of characteristic cycle can be defined for any $k$-scheme $X$ separated of finite type.
In the case where $X$ can be embedded into a smooth scheme, this can be retrieved from Saito's characteristic cycle, but in general, the construction seems to be new.
Another consequence is that the non-microlocalized characteristic cycle is compatible with nearby cycle functor.
Note that the singular support is {\em not} compatible with nearby cycle in general \cite[2.3.2]{S3}.
This implies that the microlocalized version cannot be expected, unless we change the definition of the (microlocalized) characteristic cycle.
See Remark \ref{summarymain} for some explanations.
\bigskip

To conclude the introduction, let us outline the structure of this paper.
In \S\ref{sect1}, we formulate an abstract formalism of localization.
We introduce the notion of a {\em deformation system}, which describes the deformation of sheaves abstractly.
We say that a deformation system is {\em effaceable} if it satisfies a property similar to \ref{propofsys}.\ref{propofsys-1}.
Using these notions, we define localizing pairs and prove an extension theorem in this abstract setting.
In \S\ref{sect2}, we construct a rewind morphism on $\Hcoe_X$, and show that $\Hcoe_X$ is, in fact, a reversible coefficient.
In \S\ref{sect3}, we introduce an abstract formalism for microlocalization, inspired by Beilinson.
In \S\ref{sect4}, we define the abelian presheaf $\Comp_X$.
For this, we need to use the theory of nearby cycles over general base, and we start by reviewing the theory.
In \S\ref{sect5}, we show that $\Comp_{X,0}\subset\Comp_X$ is a localizing pair.
Here, we assume the main result of Part II.
Summing up these results, we construct the characteristic cycle and establish the pushforward formula for our characteristic cycle.
In the final section, \S\ref{sect6}, we compare our characteristic cycle and the one defined by Saito.
In addition, we have three sections in the appendix, proving some technical results on $\infty$-categories.

We use the theory of $\infty$-categories, more precisely quasi-categories, for the most part of this paper.
In \S\S\ref{sect1},\ref{sect2},\ref{sect3} in particular, we use $\infty$-category theory heavily;
however, a basic knowledge of \cite{HTT} is sufficient for these sections, as more technical results are treated in the appendix.

\subsection*{Acknowledgments}\mbox{}\\
This paper grew out from a project with Deepam Patel on understanding geometric $\varepsilon$-factors in the $\ell$-adic setting,
and this works would have never appeared without it.
The author greatly appreciates Deepam.
He thanks Takeshi Saito for his continuous supports and encouragements: The influence of his mathematics to this work is tremendous and apparent.
He thanks Offer Gabber for allowing the author to use his results ({\it i.e.}\ Theorem \ref{Gabbthm}),
which made the construction much cleaner, and various comments on the manuscript.
A referee suggested a drastic simplification of the proof of the extension theorem and the construction of rewind morphism from the previous version of the paper.
Especially, the argument in \S\ref{sect2} is essentially due to the referee, and the author wishes to thank the referee heartily.
He also thanks H\'{e}l\`{e}ne Esnault, Shuji Saito, Enlin Yang for helps and showing their interest.
Last but not least, the author thanks Alexander Beilinson for feedbacks and inspiring discussions after this paper was written.

This work is supported by JSPS KAKENHI Grant Numbers 16H05993, 18H03667, 20H01790.

\subsection*{Changes from Version 1}\mbox{}\\
As we mentioned in the Acknowledgments, several proofs and constructions have been significantly simplified since Version 1, following the suggestions of the referee.
We would like to express our gratitude to the referee once again for these insights.
Consequently, some terminology and definitions have been updated; we list the most notable changes below for the convenience of the reader.

\begin{itemize}
 \item The term ``pre-localization system'' used in the previous version has been replaced with ``deformation system''.
       While the original definition of a pre-localization system required a grading of the sheaf as a part of the information,
       we do not consider a grading in the definition of deformation system.

 \item The term ``localization system'' is now replaced with ``effaceable deformation system''.
       As with the change above, we no longer consider a grading.

 \item The functor $\partial$ in this version corresponds to $\partial_{\infty}$ in the previous version.
       However, the definitions are somewhat different, and we do know if they are equivalent.

 \item The definition of reversible coefficients is simplified thanks to the simplification of the proof of Theorem \ref{mainconscc} suggested by the referee.

 \item The construction of rewind morphism in \S\ref{sect2} is now significantly simplified, and the resulting statement is stronger.
       The construction is new, which is due to the referee as well.
\end{itemize}

\subsection*{Notations and conventions}

\begin{itemize}
 \item We fix a field $k$ of positive characteristic.
       We put $\Box:=\mb{P}^1_k\setminus\{1\}\,(\cong\mb{A}^1_k)$.

 \item A morphism $f\colon X\rightarrow Y$ is said to be {\em birational} if it is proper and there exists an open dense subscheme $V\subset Y$
       such that $f^{-1}(V)\rightarrow V$ is an isomorphism.
       The morphism is said to be a {\em modification} if it is birational and each generic point of $X$ is sent to a generic point of $Y$.

 \item We denote by $\mr{Sch}_{/S}$ the category of separated $S$-schemes of finite type, and by $\mr{SCH}$ the category of finite Krull dimensional noetherian $k$-schemes.

 \item In principle, we follow \cite{HTT} for the terminologies and notations of $\infty$-categories.
       Especially, by $\infty$-category we always mean quasi-category, and by category we always mean (ordinary) $1$-category.
       Here are some exceptions.
       We often identify (ordinary) categories with their nerves.
       We denote by $\Spc$ the category of spaces.
       We denote isomorphisms by $\cong$ and equivalences in $\infty$-categories by $\simeq$.

 \item Assume we are given topoi $\mc{T}$, $\mc{T}'$.
       By {\em geometric morphism} from $\mc{T}$ to $\mc{T}'$, we mean a pair of adjoint functors $L\colon\mc{T}'\rightarrow\mc{T}\colon R$
       such that $L$ is left exact.
       If we denote by $F\colon\mc{T}\rightarrow\mc{T}'$ the geometric morphism, $(L,R)$ are denoted by $(F^*,F_*)$.
       Let $\mc{C}$ be a stable $\infty$-category equipped with t-structure.
       As in \cite{HA}, we use homological notation.
       For example $X\in\mc{C}$, we denote $\tau_{\geq n}\tau_{\leq n}(X)\in\mc{C}^{\heartsuit}$,
       where $\mc{C}^{\heartsuit}$ denotes the heart of the t-structure, by $\pi_n X$.
       This is the same as $\mc{H}^{-n}X$ in the cohomological notation.

 \item Let $\mc{C}$ be an ($\infty$-)category admitting fiber products, and let $X\rightarrow Y$ be a morphism.
       If we are given another morphism $g\colon Y'\rightarrow Y$,
       we often denote by $X_{Y'}$ or by $g^*X$ the base change $X\times_Y Y'$,
       if no confusion may arise.

 \item For a stable $\infty$-category, recall from \cite[1.1.1.4]{HA} that a cofiber sequence is a pushout diagram,
       equivalently a pullback diagram, of the form
       \begin{equation*}
	\xymatrix{
	 X\ar[r]\ar[d]&Y\ar[d]\\
	0\ar[r]&Z,}
       \end{equation*}
       where $0$ is a zero object of $\mc{C}$.
       We also denote this cofiber sequence simply by $X\rightarrow Y\rightarrow Z$ if no confusion may arise.
\end{itemize}

\section{Abstract formalism}
\label{sect1}

\subsection{}
Recall that a morphism $f\colon X\rightarrow Y$ is said to be a {\em cdp-morphism} if it is proper and for any $y\in Y$ there exists a point $x\in X$
such that $f(x)=y$ and the induced homomorphism of residue fields $k(y)\rightarrow k(x)$ is an isomorphism.
We first define a Grothendieck topology, which plays an essential role in our construction.

\begin{dfn*}
 Let $S$ be a separated noetherian $k$-scheme (not necessary of finite type).
 For $T\in\mr{Sch}_{/S\times\Box}$, we define $\mr{Cov}(T)$ to be the set of families $\{T_i\rightarrow T\}$
 such that $\coprod T_i\rightarrow T$ is proper and $(\coprod T_i)\times_{\Box}U\rightarrow T\times_{\Box}U$
 is a cdp-morphism for some non-empty open subset $U$ of $\Box$.
 These data form a pretopology on $\mr{Sch}_{/S}$, and the induced Grothendieck topology is called the {\em $\eta$-cdp topology}.
\end{dfn*}

\begin{rem*}
 The cdp-topology is coarser than h-topology.
 However, $\eta$-cdp topology is not.
 Indeed, for an $\eta$-cdp covering $\{T_i\rightarrow T\}$, the morphism $\coprod T_i\rightarrow T$ need {\em not} be surjective.
\end{rem*}

\subsection{}
\label{basicsitediag}
We denote by $\mr{pr}\colon S\times\Box\rightarrow S$ the projection and by $i_{\infty}\colon S\cong S\times\{\infty\}\hookrightarrow S\times\Box$ the closed immersion.
We have the following diagram on the left of {\em cocontinuous} functors of ($1$-)sites (cf.\ [SGA 4, Exp.\ III, 2.1] or \ref{cocontmorlem}):
\begin{equation}
 \label{topoidiagsite}
 \xymatrix{
  \mr{Sch}_{/S\times\Box,\eta\mbox{-}\mr{cdp}}\ar[r]^-{\mr{id}}&
  \mr{Sch}_{/S\times\Box}\ar[r]^-{\mr{pr}\circ}&
  \mr{Sch}_{/S}\\
 &\mr{Sch}_{/S},\ar[u]^{i_{\infty}\circ}\ar@{=}[ur]&
  }\quad
  \xymatrix{
  \Shv(\mr{Sch}_{/S\times\Box,\eta\mbox{-}\mr{cdp}})\ar[r]^-{L}&
  \PShv(S\times\Box)\ar[r]^-{q}&
  \PShv(S)\\
 &\PShv(S).\ar[u]^{i_{\infty}}\ar@{=}[ur]&
  }
\end{equation}
We put $\PShv(S):=\PShv(\mr{Sch}_{/S})$, the $\infty$-category of $\infty$-presheaves on $\mr{Sch}_{/S}$.
This induces the diagram of geometric morphisms of $\infty$-topoi on the right above.
We define an endo-functor on $\PShv(S)$ as
\begin{equation*}
 \partial:=q_*\circ L_*\circ L^*\circ i_{\infty,*}.
\end{equation*}
The adjunction morphism $\mr{id}\rightarrow L_*\circ L^*$ induces the morphism of endo-functors $I\colon\mr{id}\rightarrow\partial$.
For $\mc{F}(S)\in\PShv$, the morphism $I(\mc{F})\colon\mc{F}\rightarrow\partial\mc{F}$ is also denoted by $I_{\mc{F}}$.

\begin{ex*}
 Let $V\rightarrow\Box$ be a morphism.
 Let $\mc{I}\subset\mc{O}_V$ be the ideal generated by elements which are supported on non-flat irreducible components of $V$ over $\Box$.
 We put $V_{\mr{fl}}:=\mr{Spec}(\mc{O}_V/\mc{I})$.
 Note that as a topological space $V_{\mr{fl}}=(V\times_{\Box}\eta)^{-}$, where the closure is taken in $V$.
 Now, let $\mc{F}$ be a $1$-presheaf on $\mr{Sch}_{/S}$, and let $\{T_i\rightarrow T\times\Box\}$ be a cdp-covering.
 Put $T_{ij}:=T_i\times_{T\times\Box}T_j$.
 Assume we are given elements $g_i\in\Gamma(T_{i,\mr{fl}},i_{x,*}\mc{F})\cong\Gamma(T_{i,\mr{fl}}\times_{\Box}\{x\},\mc{F})$
 for any $i$ such that $g_i|_{T_{ij,\mr{fl}}\times_{\Box}\{x\}}=g_{j}|_{T_{ij,\mr{fl}}\times_{\Box}\{x\}}$ for any $i,j$.
 Then the family $\{g_i\}$ gives rise to an element of $\pi_0\Gamma(T,\partial\mc{F})$ since $\partial$ is left t-exact.
\end{ex*}

\subsection{}
\label{preramdef}
Any $\infty$-category admitting colimits is naturally tensored over the $\infty$-category of spaces.
For a category $\mc{C}$ with colimits, $X\in\mc{C}$, and a space $K$, we denote this tensor by $X\otimes K$.
When $K$ is a set, we have $X\otimes I\simeq\coprod_{i\in I}X$.

\begin{dfn*}
 A {\em deformation system over $S$} is a collection of data $(\ms{S},\mu,\mc{E})$ where:
 \begin{itemize}
  \item $\ms{S}$ is an abelian presheaf on $\mr{Sch}_{/S}$;

  \item $\mu$ is a {\em pointed set-valued} presheaf on $\mr{Sch}_{/S}$, whose basepoint is denoted by $0$.
	This presheaf $\mu$ admits a {\em geometric contraction} (which is not a part of the data), namely a morphism
	$c_{\mu}\colon\mu\rightarrow p_*p^*\mu$ ($p\colon\mb{A}^1_S\rightarrow S$ is the projection)
	such that the composite $\mu\rightarrow p_*p^*\mu\rightarrow p_*i_{x*}i_x^*p^*\mu\simeq\mu$
	($i_x\colon S\hookrightarrow\mb{A}^1_S$ is the closed immersion for $x\in\{0,1\}$) is the identity for $x=1$
	and the constant morphism values at $0$ for $x=0$;

  \item $\mc{E}$ is a homomorphism $\ms{S}\otimes\mu\rightarrow\partial\ms{S}$ of $\infty$-presheaves on $\mr{Sch}_{/S}$.
	Here, $\otimes$ is taken as the underlying {\em set} rather than pointed set, namely,
	$\ms{S}\otimes\mu\cong\bigoplus_{f\in\mu}\ms{S}$.
 \end{itemize}
 This data is subject to the following compatibility condition:
 \begin{itemize}
  \renewcommand{\labelitemi}{$\blacktriangleright$}
  \item The following diagram commutes:
	\begin{equation}
	 \label{comaxpresys}
	 \xymatrix@C=70pt{
	  \ms{S}
	  \ar[d]_{\mr{id}\otimes\{0\}}\ar[rd]^{I_{\ms{S}}}&\\
	 \ms{S}\otimes\mu
	  \ar[r]_-{\mc{E}}&
	  \partial\ms{S}.
	  }
	\end{equation}
 \end{itemize}
\end{dfn*}

\begin{rem*}
 Note that, since the cohomological operations are used in the $\infty$-categorical setting,
 the functor $\partial$ is ``derived'' by default.
 However, since $\partial$ is left t-exact, $\pi_i\partial\ms{S}=0$ for $i>0$, in other words, the higher homotopies vanish.
 Thus, giving $\mc{E}$ is equivalent to giving a homomorphism $\ms{S}\rightarrow\pi_0\partial\ms{S}$.
 Namely, deformation system is purely a $1$-categorical notion.
 This indicates that deformation systems are not too hard to handle, at least compared to $\infty$-categorical data.
\end{rem*}

Let us introduce some notations.
For $\mbf{f}=(f_1,\dots,f_n)\in\mu^{n}$ and $\mbf{g}=(g_1,\dots,g_{n'})\in\mu^{n'}$,
we define the concatenation $\mbf{f}\vee\mbf{g}:=(f_1,\dots,f_n,g_1,\dots,g_{n'})\in\mu^{n+n'}$.
Using the base point of $\mu$, we have a map $z_m\colon\mu^{m}\xrightarrow{\mr{id}\times\{0\}}\mu^{m}\times\mu\cong\mu^{m+1}$
where the first $m$-components are identities.
The colimit over $m\in\mb{N}$ is denoted by $\mu^{\infty}$.
Note that the operation $\vee$ is {\em not} defined on the level of $\mu^{\infty}$.
For an abelian presheaf $\mc{F}$ and sections $F\in\Gamma(T,\mc{F})$, $\mbf{f}\in\Gamma(T,\mu^\infty)$ for $T\in\mr{Sch}_{/S}$,
we denote by $(F;\mbf{f})$ the element of $\Gamma(T,\mc{F}\otimes\mu^{\infty})\cong\bigoplus_{\Gamma(T,\mu^{\infty})}\Gamma(T,\mc{F})$ whose
$\mbf{f}$-component is $F$ and the other components are $0$.

\subsection{}
\label{iteconst}
Let $\mc{C}$ be an $\infty$-category and assume we are given an endo-functor $F\colon\mc{C}\rightarrow\mc{C}$.
Let $X\in\mc{C}$ and $\alpha\colon X\rightarrow F(X)$ be a morphism.
Let $I^{\mb{N}}:=\Delta^{\{0,1\}}\cup\Delta^{\{1,2\}}\cup\dots\subset\Delta^{\mb{N}}$.
Then the edge $\alpha$ induces the diagram $I^{\mb{N}}\rightarrow\mc{C}$ defined by $\alpha\vee F(\alpha)\vee F^2(\alpha)\vee\dots$.
Since $I^{\mb{N}}\hookrightarrow\mr{N}(\mb{N})$ is an inner anodyne, this diagram extends to a diagram
$D\colon\mr{N}(\mb{N})\rightarrow\mc{C}$.
We put $F^{\infty}_\alpha(X):=\indlim(D)$, if it exists.

\subsection{}
\label{dfnramsyscons}
Let $\mc{F}$ be an $\Mod_{\mb{Z}}$-valued $\infty$-presheaf, and $\mu$ be an $\infty$-presheaf.
In view of Lemma \ref{pullbackmonoidal}, pullback commutes with the tensor product.
Thus, we have the natural morphism $\theta\colon\partial(\mc{F})\otimes\mu\rightarrow\partial(\mc{F}\otimes\mu)$.
We have the homomorphism
\begin{equation*}
 \mc{E}_m\colon 
  \ms{S}\otimes\mu^{m+1}
  \cong
  (\ms{S}\otimes\mu)\otimes\mu^m
  \xrightarrow{\mc{E}\times\mr{id}}
  \partial(\ms{S})\otimes\mu^m
  \xrightarrow{\theta^m}
  \partial(\ms{S}\otimes\mu^m).
\end{equation*}
Informally, this sends $F\vee f\vee\mbf{g}$ to $\mc{E}(F;f)\vee\mbf{g}$.
We iteratively compose this morphism, and get
\begin{equation*}
 \mc{E}^m\colon
 \ms{S}\otimes\mu^m
  \xrightarrow{\mc{E}_{m-1}}
  \partial(\ms{S}\otimes\mu^{m-1})
  \xrightarrow{\partial(\mc{E}_{m-2})}
  \partial^2(\ms{S}\otimes\mu^{m-2})
  \rightarrow
  \dots
  \xrightarrow{\partial^{m-1}(\mc{E}_0)}
  \partial^m\ms{S}.
\end{equation*}
We have the following homotopy commutative diagram:
\begin{equation*}
 \xymatrix@C=35pt{
  \ms{S}\otimes\mu^{m+1}\ar@{-}[r]^-{\cong}&
  (\ms{S}\otimes\mu^m)\otimes\mu\ar[r]^-{\mc{E}^m\otimes\mr{id}}&
  \partial^m\ms{S}\otimes\mu\ar[r]^-{\theta^m}&
  \partial^m(\ms{S}\otimes\mu)\ar[r]^-{\partial^m(\mc{E})}&
  \partial^{m+1}\ms{S}\\
 \ms{S}\otimes\mu^{m}\ar@{=}[r]\ar[u]^{\mr{id}\otimes z_m}
 &\ms{S}\otimes\mu^m\ar[r]^-{\mc{E}^m}\ar[u]^-{\mr{id}\otimes\{0\}}&
  \partial^m\ms{S}\ar@{=}[r]\ar[u]^-{\mr{id}\otimes\{0\}}&
  \partial^m\ms{S}\ar[u]^-{\partial^m(\mr{id}\otimes\{0\})}\ar@{=}[r]&
  \partial^m\ms{S}.\ar[u]_-{\partial^m(I_{\ms{S}})}&
  }
\end{equation*}
Taking the colimit of the diagram above over $m$, we have the morphism $\mc{E}^{\infty}\colon\ms{S}\otimes\mu^{\infty}\rightarrow\partial^{\infty}_I\ms{S}$.
We put $\Eff_{(\ms{S},\mu,\mc{E})}:=\mr{Ker}(\mc{E}^{\infty})$.
Here is one of the main definitions of this paper:

\begin{dfn*}
 \label{dfnramsys}
 Let $\Delta^n_{\mr{geom}}:=S\times\mr{Spec}(\mb{Z}[x_0,\dots,x_n]/(1-\sum x_i))$.
 A deformation system $\mb{S}:=(\ms{S},\mu,\mc{E})$ is said to be {\em effaceable} if it satisfies the following condition:
 \begin{itemize}
  \renewcommand{\labelitemi}{$\blacktriangleright$}
  \item For $F\in\ms{S}(T)$ and $\mbf{f}\in\mu^{l}(T\times\Delta^m_{\mr{geom}})$ for some $T$ and $l,m\geq0$,
	we can find a cdh-covering $\{T_i\}\rightarrow T$ and $\mbf{g}_i\in\mu^{l'}(T_i)$ for some $l'$
	such that $(F;\mbf{f}\vee\mbf{g}_i)\in\Eff_{\mb{S}}(T_i\times\Delta^m_{\mr{geom}})$.
 \end{itemize}
\end{dfn*}

\begin{rem*}
 \begin{enumerate}
  \item Since $\partial$ is left t-exact and filtered colimit is t-exact,
	$\pi_n\bigl(\partial^{\infty}_I(\ms{S}\otimes\mu^{\infty})\bigr)=0$ for $n>0$.
	This implies that for a subsheaf $A\subset\ms{S}\otimes\mu^{\infty}$, the condition
	$\mc{E}^{\infty}|_A=0$ is equivalent to $\pi_0\mc{E}^{\infty}|_A=0$.
	Thus, the condition can be expressed entirely in terms of ($1$-)sheaves.

  \item \label{dfnramsys-rem-2}
	When we are given an effaceable deformation system, the following slightly stronger condition holds:
	For a finite set $\{F_k\}_{1\leq k\leq n}$ of elements in $\ms{S}(T)$
	and $\{\mbf{f}_k\}_{1\leq k\leq n}$ in $\mu^{l}(T\times\Delta^m_{\mr{geom}})$ for some $l,m\geq0$,
	we can find a cdh-covering $\{T_i\}\rightarrow T$ and $\mbf{g}_i\in\mu^{l'}(T_i)$ for some $l'$
	(not depending on $1\leq k\leq n$)
	such that $(F_k;\mbf{f}_k\vee\mbf{g}_i)\in\Eff_{\mb{S}}(T_i\times\Delta^m_{\mr{geom}})$.
	Indeed, we use the induction on $n$.
	Assume the condition holds for $n<N$.
	For $\{F_k\}_{1\leq k\leq N-1}$, we may choose a finite cdh-covering $\{T_i\}_{i\in I}$ and $\mbf{g}_i$.
	By definition of localization system, for $F\in\ms{S}(T_i)$ and $\mbf{f}_N\vee\mbf{g}_i\in\mu^{l+l'}(T_i\times\Delta^m_{\mr{geom}})$,
	we can take a finite cdh-covering $\{T_{i,j}\}_{j\in J_i}$ of $T_i$ and $\mbf{g}_{i,j}$
	such that $(F;\mbf{f}_{N}\vee\mbf{g}_i\vee\mbf{g}_{i,j})\in\Eff_{\mb{S}}(T_{i,j}\times\Delta^m_{\mr{geom}})$.
	Thus the cdh-covering $\{T_{i,j}\}_{i\in I,j\in J_i}$ and $\mbf{g}_i\vee\mbf{g}_{i,j}$ meet our need.
 \end{enumerate}
\end{rem*}

\subsection{}
\label{defofcoefsys}
For an $\mb{E}_{\infty}$-ring $R$, let $\Mod_R$ be the $\infty$-category of $R$-module objects in $\mr{Sp}$.
We only consider the case where $R$ is a discrete $\mb{E}_{\infty}$-ring, namely a commutative ring in the usual sense.
Recall that, in this case, the homotopy category $\mr{h}\Mod_R$ is equivalent to the usual derived category $D(R)$ (cf.\ \cite[7.1.2.13]{HA}).

\begin{dfn*}
 Let $\mc{H}$ be a $\Mod_R$-valued presheaf on $\mr{Sch}_{/S}$, an object of $\PShv_{\Mod_R}(S):=\PShv_{\Mod_R}(\mr{Sch}_{/S})$.
 The presheaf $\mc{H}$ is said to be {\em reversible} if the following conditions hold:
 \begin{itemize}
  \renewcommand{\labelitemi}{$\blacktriangleright$}
  \item It is {\em $\mb{A}^1$-invariant}, namely the canonical morphism $\mc{H}\rightarrow p_*p^*\mc{H}$,
	where $p\colon\mb{A}^1_S\rightarrow S$ is the projection, is an equivalence;

  \item It is a $\Mod_R$-valued $\infty$-sheaf with respect to cdh-topology on $\mr{Sch}_{/S}$;

  \item There exists a {\em rewind morphism}, namely a retract diagram
	$\mr{rw}\in\mr{Fun}\bigl(\Delta^2/\Delta^{\{0,2\}},\PShv_{\Mod_R}(S)\bigr)$ (cf.\ \cite[4.4.5]{HTT})
	such that $\mr{rw}|_{\Delta^{\{0,1\}}}=I\colon\mc{H}\rightarrow\partial\mc{H}$.
	Abusing notations, we often denote the retract $\mr{rw}|_{\Delta^{\{1,2\}}}$ of $I$ simply by $\mr{rw}$.
 \end{itemize}
\end{dfn*}

\subsection{}
Now, we are ready to state the main result of this section.
We say that a morphism $i\colon\ms{S}'\rightarrow\ms{S}$ of an abelian presheaves on $S$ is said to be {\em localizing} if there exists an increasing positive filtration
$\{\ms{S}_d\}_{d\geq0}$ of $\ms{S}$ such that $\mr{gr}_d\ms{S}_{\bullet}$ can be endowed with a structure of effaceable deformation system for any $d>0$ and the inclusion
$\ms{S}_0\subset\ms{S}$ is equivalent to $i$.

\begin{thm*}
 \label{mainconscc}
 Let $\ms{S}_0\rightarrow\ms{S}$ be a localizing morphism of abelian presheaves and $\mc{H}$ be a reversible presheaf.
 Then the restriction morphism $\Map_{\PShv(S)}(\ms{S},\mc{H})\rightarrow\Map_{\PShv(S)}(\ms{S}_0,\mc{H})$ is an equivalence of spaces.
\end{thm*}
A large part of the rest of this section is devoted to the proof of this theorem.
We start the preparation of the proof, and the proof will be given in \ref{proofofextthm}.

\subsection{}
We start with some preparations of general non-sense.
The following lemma is standard, and we leave the verification to the reader.

\begin{lem*}
\label{functnattran}
 Let $\mc{C}$, $\mc{D}$ be $\infty$-categories, and assume we are given a morphism $\phi\colon F\rightarrow G$ in $\mr{Fun}(\mc{C},\mc{D})$.
 For $X\in\mc{C}$, we denote by $\phi_X\colon F(X)\rightarrow G(X)$ the induced morphism in $\mc{D}$.
 Let $X,Y$ be objects of $\mc{C}$.
 Then, we have the following homotopy commutative diagram of spaces
 \begin{equation*}
  \xymatrix@C=50pt{
   \Map_{\mc{C}}(X,Y)\ar[r]^-{F}\ar[d]_{G}&\Map_{\mc{D}}\bigl(F(X),F(Y)\bigr)\ar[d]^{\phi_Y\circ}\\
  \Map_{\mc{C}}\bigl(G(X),G(Y)\bigr)\ar[r]^-{\circ\phi_X}&\Map_{\mc{D}}\bigl(F(X),G(Y)\bigr).
   }
 \end{equation*}
\end{lem*}

\begin{lem}
 \label{retractconslem}
 Recall the notation of {\normalfont\ref{iteconst}}.
 Let $\mc{C}$ be an $\infty$-category, and $F$ be an endo-functor of $\mc{C}$.
 For $X\in\mc{C}$, assume we are given a retract diagram $X\xrightarrow{i}F(X)\xrightarrow{r}X$.
 Then this induces the retract diagram $X\xrightarrow{i}F_i^{\infty}(X)\xrightarrow{r}X$, if $F_i^{\infty}(X)$ exists.
\end{lem}
\begin{proof}
 Following \cite[4.4.5]{HTT}, let $\mr{Ret}:=\Delta^2/\Delta^{\{0,2\}}$.
 We construct diagrams $R_n\colon\mr{Ret}\rightarrow\mc{C}$ and $D_n\colon\mr{Ret}\times\Delta^1\rightarrow\mc{C}$ for $n>0$ inductively.
 We put $R_1$ to be the given retract diagram.
 Assume we are given a retract diagram $X\xrightarrow{i_n} F^n(X)\xrightarrow{r_n} X$, denoted by $R_n$.
 Let us construct $D_n$.
 We have the map $(\Delta^2\times\{0\})\coprod_{\partial\Delta^2\times\{0\}}(\partial\Delta^2\times\Delta^1)\rightarrow\mc{C}$
 sending $\Delta^2\times\{0\}$ to the retract diagram $X\rightarrow F^n(X)\rightarrow X$,
 and $\Delta^{\{0,2\}}\times\Delta^1$ is sent to the point $\{X\}$,
 and $\Delta^{\{0,1\}}\times\Delta^1$, $\Delta^{\{1,2\}}\times\Delta^1$ are any homotopy commutative diagrams of the form below:
 \begin{equation*}
  \xymatrix{
   X\ar[r]^-{i_n}\ar@{=}[d]&F^n(X)\ar[d]^{F^n(i)}\\
  X\ar[r]&F^{n+1}(X),
   }\qquad
   \xymatrix@C=50pt{
   F^n(X)\ar[r]^-{r_n}\ar[d]_{F^n(i)}&X\ar@{=}[d]\\
  F^{n+1}(X)\ar[r]^-{r_n\circ F^n(r)}&X.
   }
 \end{equation*}
 By (the dual version of) \cite[3.2.1.8]{HTT}, the diagram extends to a diagram
 $\Delta^2\times\Delta^1\rightarrow\mc{C}$.
 By construction, this map factors through $\mr{Ret}\times\Delta^1$, and the induced map is
 $D_n\colon\mr{Ret}\times\Delta^1\rightarrow\mc{C}$.
 This diagram can be depicted as
 \begin{equation*}
  \xymatrix{
   X\ar[r]^-{i_n}\ar@{=}[d]&F^n(X)\ar[r]^-{r_n}\ar[d]&X\ar@{=}[d]\\
  X\ar[r]&F^{n+1}(X)\ar[r]&X,
   }
 \end{equation*}
 and the rows are retract diagrams.
 Now, let $R_{n+1}:=D_n|_{\mr{Ret}\times\{1\}}$.
 This inductive construction yields, for $n>0$, a diagram $D_n\colon\mr{Ret}\times\Delta^1\rightarrow\mc{C}$ such that
 $D_n|_{\mr{Ret}\times\{0\}}=R_n$ and $D_n|_{\mr{Ret}\times\{1\}}=R_{n+1}$.
 By concatenating, we have $\mr{Ret}\times I^{\mb{N}}\rightarrow\mc{C}$ such that $\mr{Ret}\times\{i\}=R_i$.
 Since $I\hookrightarrow\mr{N}(\mb{N})$ is an inner anodyne, the diagram extends to $\mr{Ret}\times\mr{N}(\mb{N})\rightarrow\mc{C}$ by
 \cite[2.3.2.4]{HTT}.
 Taking the adjoint, we have the diagram $\mr{N}(\mb{N})\rightarrow\mr{Fun}(\mr{Ret},\mc{C})$, and take a colimit of this functor.
 Since colimits can be computed pointwise by \cite[5.1.2.3]{HTT}, the colimit diagram is the one we are seeking.
\end{proof}

\subsection{}
In this paragraph, we introduce the Suslin complex.
The construction is well-known to experts.
Recall $\Delta^n_{\mr{geom}}:=S\times\mr{Spec}(\mb{Z}[x_0,\dots,x_n]/(1-\sum x_i))$.
For $\phi\colon[m]\rightarrow[n]$ in $\mbf{\Delta}$,
we have the homomorphism $\mc{O}_S[x_0,\dots,x_n]/(1-\sum x_i)\rightarrow\mc{O}_S[y_0,\dots,y_m]/(1-\sum y_j)$ sending $x_i$ to $\sum_{j\in\phi^{-1}(i)}y_j$.
This gives us a cosimplicial scheme $\Delta^{\bullet}_{\mr{geom}}\colon\mbf{\Delta}\rightarrow\mr{Sch}_{/S}$.
Furthermore, we have the functor $\gamma\colon\mbf{\Delta}\times\mr{Sch}_{/S}\rightarrow\mr{Sch}_{/S}$ sending $([n],T)$ to $\Delta^n_{\mr{geom}}\times_S T$.
Using this, we define the {\em Suslin complex functor} by the composite
\begin{equation*}
 \mr{Sus}^{\mc{D}}_\bullet\colon
  \PShv_{\mc{D}}(S)
  \cong
  \mr{Fun}(\mr{Sch}_{/S}^{\mr{op}},\mc{D})
  \xrightarrow{\circ\gamma^{\mr{op}}}
  \mr{Fun}\bigl(\mbf{\Delta}^{\mr{op}}\times\mr{Sch}_{/S}^{\mr{op}},\mc{D}\bigr)
  \cong
  \mr{Fun}(\mbf{\Delta}^{\mr{op}},\PShv_{\mc{D}}(S)).
\end{equation*}
The geometric realization of $\mr{Sus}^{\mc{D}}_{\bullet}$ (namely the colimit over $\mbf{\Delta}^{\mr{op}}$) is denoted by $\mr{Sus}^{\mc{D}}$.
When $\mc{D}=\Mod_{\mb{Z}}$, we simply denote $\mr{Sus}^{\Mod_{\mb{Z}}}_{(\bullet)}$ by $\mr{Sus}_{(\bullet)}$.

\begin{rem*}
 \begin{enumerate}
  \item Let $\mc{F}$ be an {\em abelian} presheaf, namely $\mc{F}\in\Mod_{\mb{Z}}^{\heartsuit}$.
	Since for any morphism $f\colon X\rightarrow Y$, the pushforward $f_*\colon\PShv_{\Mod_{\mb{Z}}}(X)\rightarrow\PShv_{\Mod_{\mb{Z}}}(Y)$
	is t-exact, $\mr{Sus}^{\Mod_{\mb{Z}}}_m(\mc{F})$ remains to be an abelian presheaf,
	and equivalent to $\mr{Sus}^{\mr{Ab}}_m(\mc{F})$ considered as an object of $\PShv_{\Mod_{\mb{Z}}}(S)$.
	Note however that $\mr{Sus}^{\Mod_{\mb{Z}}}(\mc{F})$ and $\mr{Sus}^{\mr{Ab}}(\mc{F})$
	viewed as objects of $\PShv_{\Mod_{\mb{Z}}}(S)$ are not equivalent in general.
	By the convention, $\mr{Sus}(\mc{F})$ means $\mr{Sus}^{\Mod_{\mb{Z}}}(\mc{F})$, not $\mr{Sus}^{\mr{Ab}}(\mc{F})$.

  \item Let $\mc{F}$ be an $\mb{A}^1$-invariant $\mc{D}$-valued presheaf. Then the canonical morphism $\mc{F}\rightarrow\mr{Sus}^{\mc{D}}(\mc{F})$
	is an equivalence.
	Indeed, for any edge $e$ of $\mbf{\Delta}^{\mr{op}}$, the map $\mr{Sus}^{\mc{D}}_{\bullet}(e)(\mc{F})$ is an equivalence.
	Thus, by \cite[4.4.4.10]{HTT} (as well as \cite[5.5.8.4, 5.5.8.7]{HTT} to see that $\mbf{\Delta}^{\mr{op}}$ is weakly contractible),
	$\mc{F}\simeq\mr{Sus}^{\mc{D}}(\mc{F})$.

  \item Even though we do not use this, we may check that $\mr{Sus}^{\mc{D}}(\mc{F})$ is an $\mb{A}^1$-invariant $\mc{D}$-valued presheaf.
	Moreover, a $\mc{D}$-valued presheaf is $\mb{A}^1$-invariant if and only if the canonical morphism $\mc{F}\rightarrow\mr{Sus}(\mc{F})$ is an equivalence.
 \end{enumerate}
\end{rem*}

\subsection{}
\label{pspcontlem}
The following lemma is standard.

\begin{lem*}
 Let $\mc{F}$ be a pointed set-valued presheaf on $\mr{Sch}_{/S}$ which admits a geometric contraction morphism
 {\normalfont(}cf.\ Definition {\normalfont\ref{dfnramsys}}{\normalfont)}.
 The morphism of simplicial presheaves $\mr{Sus}_{\bullet}(\mc{F})\rightarrow\Delta^0$ is simplicially homotopic.
\end{lem*}
\begin{proof}
 The base point of $\mc{F}$ induces the morphism $\Delta^0\rightarrow\mr{Sus}_{\bullet}(\mc{F})$ which we denote by $0$.
 We will construct to simplicial homotopy $\mr{Sus}_{\bullet}(\mc{F})\times\Delta^1\rightarrow\mr{Sus}_{\bullet}(\mc{F})$ between $\mr{id}$ and $0$.
 To construct a simplicial homotopy for a simplicial set $K$, we must construct a family of maps
 $h_m^j\colon K_m\rightarrow K_{m+1}$ for $m\geq0$, $m\geq j\geq0$ satisfying the usual compatibility conditions with degeneracy and face maps.
 The $k$-th vertex of $\Delta^m_{\mr{geom}}$ is denoted by  $v^m_k$.
 We define the map $h_m^j$ as the composite
 \begin{equation*}
  \mr{Sus}_m(\mc{F})(T)
   \cong
   \mc{F}(T\times\Delta^m_{\mr{geom}})
   \xrightarrow{c}
   \mc{F}(T\times\Delta^m_{\mr{geom}}\times\mb{A}^1)
   \xrightarrow[\sim]{(a^j_m)^*}
   \mc{F}(T\times\Delta^{m+1}_{\mr{geom}})
   \cong
   \mr{Sus}_{m+1}(\mc{F})(T),
 \end{equation*}
 where $a^j_m$ is the affine morphism $\Delta^{m+1}_{\mr{geom}}\rightarrow\Delta^m_{\mr{geom}}\times\mb{A}^1$
 sending $v^{m+1}_k$ to $(v^m_k,0)$ if $k\leq j$ and to $(v^m_{k-1},1)$ if $k>j$.
 It is not hard to check that this is the desired homotopy, and the claim follows.
\end{proof}

\subsection{}
Let $\mc{X}^{\Delta}_{\bullet}$ be a simplicial set valued presheaf on $\mr{Sch}_{/S}$ and $\mc{F}$ is an object of $\PShv_{\Mod_{\mb{Z}}}(S)$.
Then the tensor $\mc{F}\otimes\mc{X}^{\Delta}_{\bullet}$ is defined in $\mr{Fun}(\mbf{\Delta}^{\mr{op}},\PShv_{\Mod_{\mb{Z}}}(S))$ in an evident manner.
Informally, this is the functor sending $[n]\in\mbf{\Delta}^{\mr{op}}$ to $\mc{F}\otimes\mc{X}^{\Delta}_n$.

Let $\mb{S}:=(\ms{S},\mu,\mc{E})$ be a deformation system.
Recall that we have a subsheaf $\Eff_{\mb{S}}\subset\ms{S}\otimes\mu^{\infty}$ from \ref{dfnramsys}.
For any $\infty$-presheaf $\mc{F}$, we have the morphism $\mc{F}\rightarrow\mr{Sus}_m(\mc{F})$
induced by the projection $\Delta^m_{\mr{geom}}\rightarrow S$, which admits a retraction defined by any rational point of $\Delta^{m}_{\mr{geom}}$.
In particular, we have $\ms{S}\hookrightarrow\mr{Sus}_m(\ms{S})$.
This enables us to define
\begin{equation*}
 \Eff^{\Delta}_{\mb{S},\bullet}:=\mr{Sus}^{\Delta}_{\bullet}(\Eff_{\mb{S}})\cap(\ms{S}\otimes\mr{Sus}^{\Delta}_{\bullet}(\mu^{\infty})),
\end{equation*}
where the intersection is taken in $\mr{Sus}^{\Delta}_{\bullet}(\ms{S}\otimes\mu^{\infty})$,
and which belongs to the abelian category $\mr{Fun}(\mbf{\Delta}^{\mr{op}},\PShv_{\mr{Ab}}(\mr{Sch}_{/S}))$.
More concretely, for $T\in\mr{Sch}_{/S}$,
\begin{equation*}
 \Eff^{\Delta}_{\mb{S},m}(T)=\Eff_{\mb{S}}(T\times\Delta^m_{\mr{geom}})
  \cap
  \bigl(\ms{S}(T)\otimes\mu^{\infty}(T\times\Delta^m_{\mr{geom}})\bigr).
\end{equation*}
The map $\ms{S}\otimes\mr{Sus}^{\Delta}_{\bullet}(\mu^{\infty})\rightarrow\ms{S}\otimes\Delta^0$, induced by the unique map $\mr{Sus}_{\bullet}(\mu^\infty)\rightarrow\Delta^0$,
induces the morphism $\Eff^{\Delta}_{\mb{S},\bullet}\rightarrow\ms{S}\otimes\Delta^0$.
The geometric realization of $\Eff^{\Delta}_{\mb{S},\bullet}$,
namely $\indlim_{\mbf{\Delta}^{\mr{op}}}\Eff^{\Delta}_{\mb{S},\bullet}$ defined in $\PShv_{\Mod_{\mb{Z}}}(\mr{Sch}_{/S})$,
is denoted by $|\Eff^{\Delta}_{\mb{S},\bullet}|$.
The following contractibility of $|\Eff^{\Delta}_{\mb{S},\bullet}|$ is crucial for us.

\begin{lem*}
 Let $L_{\mr{cdh}}\colon\PShv_{\Mod_{\mb{Z}}}(S)\rightarrow\PShv_{\Mod_{\mb{Z}}}(S)$ be the localization functor given by the cdh-sheafification.
 If the deformation system $\mb{S}$ is effaceable,
 the canonical morphism $L_{\mr{cdh}}|\Eff^{\Delta}_{\mb{S},\bullet}|\rightarrow L_{\mr{cdh}}\ms{S}$ is an equivalence in $\PShv_{\Mod_{\mb{Z}}}(S)$.
\end{lem*}
\begin{proof}
 We denote $\Eff_{\mb{S}}$ by $\Eff$.
 Since $\Shv(S_{\mr{cdh}})$ has enough points, it suffices to show that the morphism of stalks
 $|\Eff^{\Delta}_{\bullet}|_z\rightarrow\ms{S}_z$ is an equivalence in $\Mod_{\mb{Z}}$.
 Since the geometric realization commutes with colimits, this is equivalent to showing that $|\Eff^{\Delta}_{\bullet,z}|\rightarrow\ms{S}_z$ is an equivalence in $\Mod_{\mb{Z}}$.
 Since the forgetful functor $\Mod_{\mb{Z}}\rightarrow\mr{Sp}$ taking the underlying spectra is conservative and commutes with colimits (cf.\ \cite[4.2.3.2, 4.2.3.7]{HA}),
 it suffices to show the equivalence as the underlying spectra.
 It suffices to check the equivalence of $\Omega^{\infty}|\Eff^{\Delta}_{\bullet}|_z\rightarrow\Omega^{\infty}(\ms{S}_z)$ (cf.\ \cite[1.4.3.8]{HA}).
 Since both sides are connective, by \cite[1.4.3.9]{HA},
 this morphism is equivalent to the morphism $|\Omega^{\infty}\Eff^{\Delta}_{\bullet}|_z\rightarrow(\Omega^{\infty}\ms{S})_z$.
 For an abelian group $M$, $\Omega^{\infty}M$ is equivalent to the underlying discrete simplicial set of $M$,
 it suffices to show that the morphism of spaces $|\Eff^{\Delta}_{\bullet}|_z\rightarrow(\ms{S}\otimes\Delta^0)_z$ is a homotopy equivalence.
 Thus, it suffices to show that $(\Eff^{\Delta}_{\bullet})_z\rightarrow(\ms{S}\otimes\Delta^0)_z$ is a trivial fibration of simplicial sets.

 Since $\ms{S}_z$ is a discrete simplicial set, it suffices to show that each fiber in $(\Eff^{\Delta}_{\bullet})_z$ is Kan contractible.
 Since $(\Eff^{\Delta}_{\bullet})_z$ is an abelian simplicial set, it is a Kan complex.
 Let $A_{\bullet}$ be the fiber over a vertex $F\in\ms{S}_z$.
 Since $\ms{S}_z$ is discrete, $A_{\bullet}$ is a retract of the Kan compex $(\Eff^{\Delta}_{\bullet})_z$,
 so $A_{\bullet}$ is a Kan complex as well.
 We must solve the following lifting problem for $n\geq0$:
 \begin{equation}
  \label{contrlemcr-1}\tag{$\star$}
  \xymatrix{
   \partial\Delta^n\ar[r]^-{a}\ar@{^{(}->}[d]&A_{\bullet}\ar[d]\\
  \Delta^n\ar[r]\ar@{-->}[ur]&\Delta^0.
   }
 \end{equation}
 The case $n=0$ is evident from the effaceability, so we assume $n>0$.
 There exists $T\in\mr{Sch}_{/S}$ around $z$ so that the map $a$ corresponds to $n+1$ elements of
 $\ms{S}(T)\otimes\mu^l(T\times\Delta^{n-1}_{\mr{geom}})$ for large enough $l$.
 Since $\mu$ admits a contraction morphism $c_\mu$, the Kan fibration
 $\ms{S}(T)\otimes\mr{Sus}_{\bullet}(\mu^l)(T)\rightarrow\ms{S}(T)$ is a trivial fibration by Lemma \ref{pspcontlem}.
 In particular, we can find a map $\Delta^n\rightarrow\ms{S}(T)\otimes\mr{Sus}_{\bullet}(\mu^l)(T)$ whose boundary is equal to $a$.
 Let $\sum(F_i;\mbf{f}_i)$ be the corresponding element of $\ms{S}(T)\otimes\mu^l(T\times\Delta^n_{\mr{geom}})$.
 Note that $\sum F_i=F$.
 By Remark \ref{dfnramsys}.\ref{dfnramsys-rem-2}, by shrinking $T$ around $z$, we may find $\mbf{g}\in\mu^{l'}(T)$ such that
 $(F_i;\mbf{f}_i\vee\mbf{g})$ belongs to $\Eff(T\times\Delta^n_{\mr{geom}})$ for any $i$.
 This element can be identified with a map $\widetilde{b}\colon\Delta^n\rightarrow A_{\bullet}$.
 Now, let $B_k:=\Delta^{[n]\setminus\{k\}}\hookrightarrow \Delta^n$ be the face,
 and $\phi_k\colon\Delta^{n-1}_{\mr{geom}}\hookrightarrow\Delta^n_{\mr{geom}}$ be the closed immersion corresponding to $B_k$.
 We may write the element corresponding to $a|_{B_k}$ as $\sum(G^{(k)}_j;\mbf{h}^{(k)}_j)$
 so that $(G^{(k)}_j;\mbf{h}^{(k)}_j)\in\Eff^{\Delta}_{n-1}$.
 By the choice of $\sum(F_i;\mbf{f}_i)$, we have $\sum(F_i;\phi_k^*\mbf{f}_i)=\sum(G^{(k)}_j;\mbf{h}^{(k)}_j)$.
 The map $\widetilde{b}|_{B_k}$ corresponds to
 \begin{equation*}
  \label{contrlemcr-2}\tag{$\star\star$}
  \sum \phi_k^*(F_i;\mbf{f}_i\vee\mbf{g})
   =
   \sum(F_i;\phi_k^*(\mbf{f}_i)\vee\mbf{g})
   =
   \sum(G^{(k)}_j;\mbf{h}^{(k)}_j\vee\mbf{g}),
 \end{equation*}
 where the first equality holds since $\mbf{g}$ is in $\mu^{\infty}(T)$ instead of $\mu^{\infty}(T\times\Delta^n_{\mr{geom}})$.
 Now, the contraction morphism induces the map $\Delta^1\times B_k\rightarrow\ms{S}(T)\otimes\mr{Sus}_{\bullet}(\mu^{\infty})(T)$
 defined by $\sum(G^{(k)}_j;\mbf{h}^{(k)}_j\vee c_\mu(\mbf{g}))$ (where $c_\mu$ is the contraction morphism).
 Since $(G^{(k)}_j,\mbf{h}^{(k)}_j)$ already belongs to $\Eff_d$, this map factors through $A_\bullet$.
 By construction, this map yields $c\colon\Delta^1\times\partial\Delta^n\rightarrow A_{\bullet}$.
 Finally, consider the following diagram:
 \begin{equation*}
  \xymatrix@C=50pt{
  \Delta^1\times\partial\Delta^n\coprod_{\{1\}\times\partial\Delta^n}\{1\}\times\Delta^n
  \ar[r]^-{c\coprod\widetilde{b}}\ar@{^{(}->}[d]&
  A_{\bullet}\ar[d]\\
  \Delta^1\times\Delta^n\ar[r]&
   \Delta^0,
  }
 \end{equation*}
 where $c\coprod\widetilde{b}$ is well-defined by (\ref{contrlemcr-2}).
 Since the left vertical inclusion is anodyne by \cite[2.1.2.7]{HTT} and $A_\bullet$ is a Kan complex,
 the lifting problem has a solution.
 The restriction to $\{0\}\times\Delta^n$ yields the solution to the lifting problem (\ref{contrlemcr-1}).
\end{proof}

\subsection{}
\label{zerohomcore}
For an $\infty$-category $\mc{C}$, we put $\mr{End}(\mc{C}):=\mr{Fun}(\mc{C},\mc{C})$.
We have the adjunction morphism $\mr{id}\rightarrow L_{\mr{cdh}}$ in $\mr{End}(\PShv_{\Mod_{\mb{Z}}}(S))$.

\begin{cor*}
 Let $(\ms{S},\mu,\mc{E})$ be an effaceable deformation system on $\mr{Sch}_{/S}$.
 The composite $\ms{S}\xrightarrow{I_{\ms{S}}^{\infty}}\partial^{\infty}_I\ms{S}\rightarrow(L_{\mr{cdh}}\circ\mr{Sus})(\partial^{\infty}_I\ms{S})$,
 where the 2nd morphism is the canonical morphism, is homotopic to $0$.
\end{cor*}

\begin{proof}
 We consider the following homotopy commutative diagram in $\mr{Fun}(\mbf{\Delta}^{\mr{op}},\PShv_{\Mod_{\mb{Z}}}(S))$:
 \begin{equation*}
  \xymatrix@C=50pt{
   &\ms{S}\otimes\Delta^0\ar[d]^{\mr{id}\otimes\{0\}}\ar@/^10pt/[rd]^-{I^{\infty}}&\\
  \Eff^{\Delta}_{\bullet}\ar[r]^-{\ccirc{3}}\ar[rd]_-{\ccirc{2}}&
   \ms{S}\otimes\mr{Sus}^{\Delta}_{\bullet}(\mu^{\otimes\infty})\ar[r]^-{\mc{E}^{\infty}}\ar[d]^{\ccirc{1}}&
   \mr{Sus}^{\Delta}_{\bullet}(\partial_I^{\infty}\ms{S}).\\
  &\ms{S}\otimes\Delta^0&}
 \end{equation*}
 By Lemma \ref{pspcontlem}, $\ccirc{1}$ is simplicially homotopic, and thus equivalence if we take the geometric realization.
 Since the composite of the vertical morphisms is an equivalence, $\mr{id}\otimes\{0\}$ is an equivalence as well.
 Thus, we must show $L_{\mr{cdh}}|\mc{E}^{\infty}|$ is homotopic to $0$.
 Since $L_{\mr{cdh}}(|\ccirc{2}|)$ is an equivalence by the previous lemma, $L_{\mr{cdh}}(|\ccirc{3}|)$ is an equivalence as well.
 Thus, we are reduced to showing that the composite $L_{\mr{cdh}}(|\mc{E}^{\infty}\circ\ccirc{3}|)$ is homotopic to $0$.
 However, the map of simplicial sets $\mc{E}^{\infty}\circ\ccirc{3}$ is already $0$ by construction of $\Eff^{\Delta}_{\bullet}$.
\end{proof}

\subsection{Proof of Theorem \ref{mainconscc}}
\label{proofofextthm}\mbox{}\\
First, assume $\ms{S}$ can be endowed with a structure of effaceable deformation system.
Then we wish to show that $\mr{Map}_{\mc{P}}(\ms{S},\mc{H})$ is contractible.
Let $L:=L_{\mr{cdh}}\circ\mr{Sus}\colon\PShv_{\Mod_{\mb{Z}}}(S)\rightarrow\PShv_{\Mod_{\mb{Z}}}(S)$.
Recall that we have the morphisms $\mr{id}\rightarrow L_{\mr{cdh}}$, $\mr{id}\rightarrow\mr{Sus}$, $\mr{id}\rightarrow\partial^{\infty}_I$
in $\mr{End}(\PShv_{\Mod_{\mb{Z}}}(S))$.
Thus, we have the morphism $J\colon\mr{id}\rightarrow L\rightarrow L\circ\partial_I^{\infty}$ of endo-functors.
By Lemma \ref{retractconslem}, the morphism $I^{\infty}_{\mc{H}}$ admits a retract $\mr{rw}_{\infty}$.
Using Lemma \ref{functnattran} for the natural transform $J$ and the notations therein, we have the following homotopy commutative diagram
\begin{equation*}
 \xymatrix{
  \mr{Map}(\ms{S},\mc{H})\ar[r]^-{J_{\mc{H}}\circ}\ar[rd]_-{L\circ\partial^{\infty}_I}&
  \mr{Map}\bigl(\ms{S},L\partial^{\infty}_I\mc{H}\bigr)\ar[r]^-{\mr{rw}_{\infty}\circ}&
  \mr{Map}(\ms{S},L\mc{H})\\
 &\mr{Map}\bigl(L\partial^{\infty}_I\ms{S},L\partial^{\infty}_I\mc{H}\bigr)\ar[u]_{\circ J_{\ms{S}}}&
  }
\end{equation*}
The composite of the horizontal morphisms is an equivalence since $\mc{H}$ is an $\mb{A}^1$-invariant cdh-sheaf.
By Corollary \ref{zerohomcore}, the vertical morphism is homotopic to $0$.
This implies that the identity morphism of $\mr{Map}(\ms{S},\mc{H})$ is homotopic to $0$.
This readily implies that any morphism $0\rightarrow\mr{Map}(\ms{S},\mc{H})$ is an equivalence in $\Spc$, and thus $\mr{Map}(\ms{S},\mc{H})$ is contractible.

Now, let us consider the situation of the theorem.
For $n>0$, we have the cofiber sequence
\begin{equation*}
 \mr{Map}\bigl(\mr{gr}_n(\ms{S}),\mc{H}\bigr)
  \rightarrow
  \mr{Map}(\ms{S}_n,\mc{H})
  \xrightarrow{r_n}
  \mr{Map}(\ms{S}_{n-1},\mc{H}).
\end{equation*}
By what we have shown above, $\mr{Map}\bigl(\mr{gr}_n(\ms{S}),\mc{H}\bigr)$ is contractible.
This implies that  $r_n$ is an equivalence.
Let $\{\ms{S}_\bullet\}$ denotes the functor $\mr{N}(\mb{N})\rightarrow\mr{Ab}$ sending $i\in\mb{N}$ to $\ms{S}_i$ and $\{\ms{S}_0\}$ the constant functor values at $\ms{S}_0$.
Then we have the morphism $\{\ms{S}_0\}\rightarrow\{\ms{S}_\bullet\}$ in $\mr{Fun}(\mr{N}(\mb{N}),\PShv_{\mr{Ab}}(S))$ induced by the canonical inclusion $\ms{S}_0\subset\ms{S}_i$.
This yields the following diagram
\begin{equation*}
 \xymatrix{
  \mr{Map}(\ms{S},\mc{H})\ar[r]^-{\sim}&
\mr{Map}\bigl(\indlim\{\ms{S}_{\bullet}\},\mc{H}\bigr)\ar[r]\ar[d]_{\sim}&
  \mr{Map}\bigl(\indlim\{\ms{S}_0\},\mc{H}\bigr)\ar[d]^{\sim}\ar[r]^-{\sim}&
  \mr{Map}(\ms{S}_0,\mc{H})\\
 &\invlim\mr{Map}(\{\ms{S}_{\bullet}\},\mc{H})\ar[r]^-{\star}&
  \invlim\mr{Map}(\{\ms{S}_0\},\mc{H})&
  }
\end{equation*}
The morphism $\mr{Map}(\{\ms{S}_{\bullet}\},\mc{H})\rightarrow\mr{Map}(\{\ms{S}_0\},\mc{H})$ in $\mr{Fun}(\mr{N}(\mb{N}),\PShv_{\Mod_{\mb{Z}}}(S))$ is an equivalence
since $r_n$ is an equivalence and the equivalence of morphism of functors can be checked pointwise by \cite[5.1.2.3]{HTT} (applied to $K=\Delta^0$).
Thus the morphism $\star$ is an equivalence, and the theorem follows.
\qed

\subsection{Variant}\hspace{1.7ex}
\label{varpartdef}
There is a variant of the construction of the previous paragraphs, and will be used in the microlocalization in \S\ref{sect3}.
Let $p\colon\mc{C}\rightarrow\mr{Sch}_{/S}$ be a Cartesian fibration of ($1$-)categories.
Then by applying the construction of Lemma \ref{carttopo}, $\mc{C}$ is endowed with the {\em $p$-Cartesian topology}.
Furthermore, by Lemma \ref{commurestpush}, we get a diagram of geometric morphisms of $\infty$-topoi similar to (\ref{topoidiagsite}) on the right by replacing
$\Shv(\mr{Sch}_{/S\times\Box,\eta\mbox{-}\mr{cdp}})$ with $\Shv(\mc{C}\times_{p,\mr{Sch}_{/S}}\mr{Sch}_{/S\times\Box,\eta\mbox{-}\mr{cdp}})$.
By performing similar constructions, we also obtain the endo-functor $\partial^{\mc{C}}$ on $\PShv(\mc{C})$.
If no confusion may arise, we often omit the superscripts $(-)^{\mc{C}}$.
Let $T\in\mr{Sch}_{/S}$, and assume we are given a functor $a\colon\mr{Sch}_{/T}\rightarrow\mc{C}$ which preserves Cartesian edges over $\mr{Sch}_{/S}$.
Evidently, all functors appearing in the diagram (\ref{topoidiagsite}) admit right adjoints.
Thus, by Lemma \ref{commurestpush}, we have an equivalence of functors $a^*\partial^{\mc{C}}\simeq\partial^{T}a^*$,
where $\partial^T:=\partial^{\mr{Sch}_{/T}}$ is the functor defined in \ref{basicsitediag} for $S$ replaced with $T$.
For each object $c\in\mc{C}$, we have a functor $a_c\colon\mr{Sch}_{/p(c)}\rightarrow\mc{C}$
which preserves Cartesian edges and such that $a_c(p(c))=c$ in an essentially unique manner.
We may define a slight generalization of deformation system as follows:

\begin{dfn*}
 A {\em $p$-deformation system} is a triple $(\ms{S},\mu,\mc{E})$ where $\ms{S}$ is an abelian presheaf on $\mc{C}$,
 $\mu$ is a pointed set-valued presheaf on $\mc{C}$ such that $a_c^*\mu$ admits a geometric contraction for any $c\in\mc{C}$,
 and $\mc{E}$ is a morphism $\ms{S}\otimes\mu\rightarrow\partial^{\mc{C}}\ms{S}$,
 which are subject to the condition that the diagram (\ref{comaxpresys}) is commutative if we replace $\partial\ms{S}$ with $\partial^{\mc{C}}\ms{S}$.
\end{dfn*}

\subsection{}
We prepare two results for the later use.

\begin{lem*}
 \label{homcommdiagdefsys}
 Let $(\ms{S},\mu,\mc{E})$ be a $p$-deformation system, and let $\mc{H}\in\PShv_{\Mod_R}(\mc{C})$ such that $a_c^*\mc{H}$ is $\mb{A}^1$-invariant for any $c\in\mc{C}$.
 Assume we are given a retract $\mr{rw}\colon\partial^{\mc{C}}\mc{H}\rightarrow\mc{H}$ of $I_{\mc{H}}$ and a morphism of presheaves $\mr{C}\colon\ms{S}\rightarrow\mc{H}$.
 Then we have a homotopy commutative diagram in $\PShv_{\Mod_R}(\mc{C})$ as follows:
 \begin{equation*}
  \xymatrix@C=40pt{
   \ms{S}\otimes\mu\ar[r]^-{\mc{E}}\ar[d]_{\mr{C}\circ\Sigma}&\partial^{\mc{C}}\ms{S}\ar[d]^{\mr{rw}\circ\partial\mr{C}}\\
  \mc{H}\ar@{=}[r]&\mc{H}.
   }
 \end{equation*}
\end{lem*}
\begin{proof}
 For $\mc{F}\in\PShv_{\Mod_\mb{Z}}(\mc{C})$, we have a $\Mod_{\mb{Z}}$-valued presheaf $\mr{Sus}(\mc{F})$ such that $a^*_c\mr{Sus}(\mc{F})\simeq\mr{Sus}(a_c^*\mc{F})$,
 whose construction is left to the reader.
 We have the following homotopy commutative diagram (of solid arrows):
 \begin{equation*}
  \xymatrix@C=30pt{
   \ms{S}\ar@{=}[r]\ar[d]&\ms{S}\ar@{=}[r]\ar[d]&\ms{S}\ar[r]^-{\mr{C}}\ar[d]&\mc{H}\ar[d]^{I_{\mc{H}}}\ar@/^2ex/[rd]^-{\mr{id}}&\\
  \ms{S}\otimes\mu\ar@{=}[r]\ar@{=}[d]&\ms{S}\otimes\mu\ar[r]^-{\mc{E}}\ar[d]&\partial\ms{S}\ar[r]^-{\partial\mr{C}}\ar[d]&\partial\mc{H}\ar[d]\ar[r]^-{\mr{rw}}&\mc{H}\ar[d]^{\sim}\\
  \ms{S}\otimes\mu\ar[r]^-{\alpha}\ar@/^5ex/@{.>}[uu]^{\Sigma}&
   \ms{S}\otimes\mr{Sus}(\mu)\ar[r]^-{\mr{Sus}(\mc{E})}\ar@/^5ex/@{.>}[uu]^(.7){\mr{Sus}(\Sigma)}|(.5)\hole&
   \mr{Sus}(\partial\ms{S})\ar[r]^-{\mr{Sus}(\partial\mr{C})}&\mr{Sus}(\partial\mc{H})\ar[r]^-{\mr{Sus}(\mr{rw})}&\mr{Sus}(\mc{H}).
   }
 \end{equation*}
 The square diagram consisting with the morphisms $\{\Sigma,\alpha,\mr{Sus}(\Sigma),\mr{id}_{\ms{S}}\}$ is also homotopy commutative.
 Using the fact that $\mr{Sus}(\Sigma)$ is a quasi-inverse to the composite of the second horizontal morphisms,
 we get the desired homotopy commutative diagram by chasing the diagram.
\end{proof}

\begin{cor}[of Theorem \ref{mainconscc}]
 \label{familyextthm}
 Assume we are given a morphism $i\colon\ms{S}_0\rightarrow\ms{S}$ of abelian presheaves on $\mc{C}$ and $\mc{H}\in\PShv_{\Mod_{\mb{Z}}}(\mc{C})$.
 Furthermore, we assume:
 \begin{itemize}
  \item The morphism $a_c^*\ms{S}_0\xrightarrow{i}a_c^*\ms{S}$ is localizing on $p(c)$ for each $c\in\mc{C}$;

  \item The presheaf $a_c^*\mc{H}$ is reversible on $p(c)$ for each $c\in\mc{C}$.
 \end{itemize}
 Then the morphism $\mr{Map}_{\PShv(\mc{C})}(\ms{S},\mc{H})\xrightarrow{\circ i}\mr{Map}_{\PShv(\mc{C})}(\ms{S}_0,\mc{H})$ is an equivalence.
\end{cor}
\begin{proof}
 Recall from \cite[3.2.2.13]{HTT} that, for a Cartesian fibration $\mc{D}\rightarrow\mc{E}$ and a coCartesian fibration $\mc{F}\rightarrow\mc{E}$ of $\infty$-categories,
 we have a coCartesian fibration $\mc{X}\rightarrow\mc{E}$ such that $\mr{Fun}_{\mc{E}}(K,\mc{X})\cong\mr{Fun}_{\mc{E}}(K\times_{\mc{E}}\mc{D},\mc{F})$.
 The $\infty$-category $\mc{X}$ is denoted by $\Fun_{\mc{E}}(\mc{D},\mc{F})$.
 The construction $\Fun$ is compatible with base change of $\mc{E}$.
 In particular, the fiber over $e\in\mc{E}$ is $\mr{Fun}(\mc{D}_e,\mc{F}_e)$.

 We put $\mr{Sch}:=\mr{Sch}_{/S}$.
 We have the Cartesian fibration $\alpha\colon\mr{Ar}(\mr{Sch}^{\mr{op}}):=\mr{Fun}(\Delta^1,\mr{Sch}^{\mr{op}})\rightarrow\mr{Sch}^{\mr{op}}$
 given by taking the value at $0\in\Delta^1$.
 By taking a suitable left Kan extension, we have a functor
 $\beta\colon\mc{C}^{\mr{op}}\times_{\mr{Sch}^{\mr{op}},\alpha}\mr{Fun}(\Delta^1,\mr{Sch}^{\mr{op}})\rightarrow\mc{C}^{\mr{op}}$,
 informally given by sending $(c,p(c)\leftarrow T\colon f)$, where $f$ is viewed as a morphism in $\mr{Sch}$, to the Cartesian pullback $f^*(c)$.
 We have the functor
 \begin{align*}
  \til\colon\mr{Fun}(\mc{C}^{\mr{op}},\Mod_{\mb{Z}})
   \xrightarrow{\circ\beta}
   &\mr{Fun}\bigl(\mc{C}^{\mr{op}}\times_{\mr{Sch}^{\mr{op}},\alpha}\mr{Ar}(\mr{Sch}^{\mr{op}}),\Mod_{\mb{Z}}\bigr)\\
   \cong
   &\mr{Fun}_{\mr{Sch}^{\mr{op}}}\bigl(\mc{C}^{\mr{op}},\Fun_{\mr{Sch}^{\mr{op}}}(\mr{Ar}(\mr{Sch}^{\mr{op}}),(\Mod_{\mb{Z}})_{\mr{Sch}^{\mr{op}}})\bigr).
 \end{align*}
 Unwinding the definition, for $c\in\mc{C}^{\mr{op}}$ and $F\colon\mc{C}^{\mr{op}}\rightarrow\Mod_{\mb{Z}}$,
 the object $\til(F)(c)$ is equivalent to $F\circ a_c$.
 Now, consider the functor $\mr{iden}\colon\mr{Sch}^{\mr{op}}\rightarrow\mr{Ar}(\mr{Sch}^{\mr{op}})$ sending $X$ to the identity morphism $X\xrightarrow{\mr{id}}X$.
 The ``precomposition'' $\circ\mr{iden}$ yields a functor
 \begin{equation*}
  \mc{P}:=\Fun_{\mr{Sch}^{\mr{op}}}(\mr{Ar}(\mr{Sch}^{\mr{op}}),(\Mod_{\mb{Z}})_{\mr{Sch}^{\mr{op}}})
   \rightarrow
   \Fun_{\mr{Sch}^{\mr{op}}}(\mr{Sch}^{\mr{op}},(\Mod_{\mb{Z}})_{\mr{Sch}^{\mr{op}}})
   \cong
   (\Mod_{\mb{Z}})_{\mr{Sch}^{\mr{op}}}.
 \end{equation*}
 This induces a retract of $\til$.
 For $T\in\mr{Sch}_{/S}$, the fiber $\mc{P}_T$ is equivalent to $\mr{Fun}(\mr{Sch}_{/T}^{\mr{op}},\Mod_{\mb{Z}})$.

 We must show that the restriction map $\Map(\ms{S},\mc{H})\rightarrow\Map({\ms{S}}_0,\mc{H})$ is an equivalence.
 Since $\til$ admits a retract, it suffices to show that the induced map
 $\Map(\til(\ms{S}),\til(\mc{H}))\rightarrow\Map(\til(\ms{S}_0),\til(\mc{H}))$ is an equivalence.
 In view of Lemma \ref{easycorofGHN}, it suffices to show that for any morphism $x$, $y$ in $\mc{C}^{\mr{op}}$
 the induced morphism
 \begin{equation*}
  \Map_{\mc{P}}\bigl(\til(\ms{S})(x),\til(\mc{H})(y)\bigr)
   \rightarrow
   \Map_{\mc{P}}\bigl(\til(\ms{S}_0)(x),\til(\mc{H})(y)\bigr)
 \end{equation*}
 is an equivalence.
 Since the induced map $\mr{Fun}(\Delta^1,\mc{P})\rightarrow\mr{Fun}(\Delta^1,\mr{Sch}^{\mr{op}})$ is a coCartesian fibration by \cite[3.1.2.1]{HTT},
 it suffices to show the equivalence for maps over $p(x)\leftarrow p(y)\colon f$ in $\mr{Sch}$.
 The coCartesian pushforward of $\til(\ms{S}_{(0)})(x)$ along $f$ is equivalent to $\ms{S}_{(0)}\circ a_{f^*(x)}$, the claim readily follows from the theorem.
\end{proof}

\section{Construction of reversible coefficient}
\label{sect2}
Let $X\rightarrow S$ be a morphism of finite type between noetherian separated schemes,
and assume we are given a cohomology theory with ``suitable'' six functor formalism.
Then we have the ``Borel-Moore homology functor'' $\Hcoe_{X/S}\colon\mr{Sch}_{/S}^{\mr{op}}\rightarrow\Mod_R$ for some commutative ring $R$;
informally, the functor sending $T\rightarrow S$ to $\Hbm(X_T/T)$ (see Introduction for the notation).
When the six functor formalism satisfies some properties, it is not hard to check that this functor is $\mb{A}^1$-invariant and a $\Mod_R$-valued cdh-sheaf.
Furthermore, it admits a rewind morphism, whose construction is the main theme of this section.
We construct the rewind morphism under the setting of \ref{varpartdef} for use in microlocalization later on.

\subsection{}
\label{recalfixsixthe}
Now, we assume we are in the situation of \ref{introsixfuncform}: we are given a cohomology theory with 6-functor formalism such that each stable
$\infty$-category is compactly generated over a noetherian separated scheme $S$ of finite Krull dimension and the theory is weakly $\tau$-continuous.
For a detailed explanation and references of the terminologies, we refer to \S\ref{EnCoOp}.
The main example we are having in mind is the $\infty$-enhancement of Voevodsky's theory of motives (cf.\ Example \ref{introsixfuncform}),
but we do not need to require trace formalism (cf.\ \ref{tracemapdfn}) for the moment.
Given such a theory we have the coCartesian fibration $\mr{Mot}\colon\widehat{\mc{D}}\rightarrow\mr{SCH}^{\mr{op}}$ whose fiber over $T$ is $\mc{D}(T)$,
and the functor associated to a morphism $f\colon X\rightarrow Y$ in $\mr{SCH}$ is the ``ordinary pullback'' $f^*\colon\mc{D}(Y)\rightarrow\mc{D}(X)$.
This map is Cartesian as well, since $f^*$ admits a right adjoint $f_*$.
Let $S$ be a finite dimensional noetherian $k$-scheme.
Then we have the functor $\mr{Sch}_{/S}\rightarrow\mr{SCH}$ sending $S'\rightarrow S$ to $S'$.
We put $F_S\colon\mc{D}_S:=\widehat{\mc{D}}\times_{\mr{SCH}^{\mr{op}}}\mr{Sch}_{/S}^{\mr{op}}\rightarrow\mr{Sch}_{/S}^{\mr{op}}$.
Note that the projection induces the functor $\mc{D}_S\rightarrow\widehat{\mc{D}}$.
An important feature of the theory is that $\mc{D}$ is left-tensored over $\Mod_R$,
and we may apply the constructions of \ref{morphobjstat} or \S\ref{sec-mor} (cf.\ \ref{introsixfuncform}).

\begin{rem*}
 In fact, we may remove the assumption of $\tau$-continuity in the following argument.
 The $\tau$-continuity is used only to construct the localization sequence in the proof of Theorem \ref{revcons}.
 In order to circumvent the use of continuity, we {\em define} $\mc{D}(T_\eta)$ for $T\in\mr{Sch}_{/S\times\Box}$ as $\indlim_{U\subset\Box}\mc{D}(T\times_{\Box}U)$,
 and $j_*\colon\mc{D}(T_\eta)\rightarrow\mc{D}(T)$ by $\indlim_Uj_{U*}$ where $j_U\colon T\times_{\Box}U\hookrightarrow T$.
 To do this in a homotopy coherent manner, we need tedious general nonsense.
 To make the presentation shorter, we decided to assume $\tau$-continuity.
\end{rem*}

\subsection{}
Assume we are given a morphism $f\colon T\rightarrow S$ of finite dimensional noetherian schemes.
Consider the functor $\pi_f\colon\mr{Sch}_{/S}\rightarrow\mr{Sch}_{/T}$ sending $S'\rightarrow S$ to $S'\times_ST\rightarrow T$.
We have the functor $\mr{Sch}_{/S}^{\mr{op}}\rightarrow\mr{Fun}(\Delta^1,\mr{SCH}^{\mr{op}})$ sending $S'\rightarrow S$ to the edge
$(S'\times_ST\rightarrow S')$ considered as a morphism in $\mr{SCH}^{\mr{op}}$.
The adjoint $\mr{Sch}_{/S}^{\mr{op}}\times\Delta^1\rightarrow\mr{SCH}^{\mr{op}}$ is denoted by $\alpha$.
For example, we have $\alpha(S',0)=S'$, $\alpha(S',1)=S'\times_ST$.
Consider the following diagram:
\begin{equation*}
 \xymatrix{
  \pi_f^*\mc{D}_T\times\{1\}\ar[r]^-{\mr{pr}}\ar@{^{(}->}[d]\ar[r]&\mc{D}_T\ar[r]&\widehat{\mc{D}}\ar[d]^{\mr{Mot}}\\
 \pi_f^*\mc{D}_T\times\Delta^1\ar[r]\ar@{-->}[urr]&\mr{Sch}_{/S}^{\mr{op}}\times\Delta^1\ar[r]^-{\alpha}&\mr{SCH}^{\mr{op}}.
  }
\end{equation*}
By taking the $\mr{Mot}$-right Kan extension, we get a functor $\pi_f^*\mc{D}_T\times\Delta^1\rightarrow\widehat{\mc{D}}$, by restricting to $\pi_f^*\mc{D}_T\times\{0\}$,
we get a functor $f_*\colon\pi_f^*\mc{D}_T\rightarrow\mc{D}_S$ of Cartesian and coCartesian fibrations over $\mr{Sch}_{/S}^{\mr{op}}$.
Informally, for $S'\in\mr{Sch}_{/S}$, the induced functor $(f_*)_{S'}$ on the fibers over $S'$
is the functor sending $\mc{M}\in\mc{D}(T\times_SS')\simeq(\pi_f^*\mc{D}_T)_{S'}$ to $f_{S',*}(\mc{M})\in\mc{D}(S')\simeq(\mc{D}_S)_{S'}$.
This functor preserves Cartesian edges.
By \cite[7.3.2.6]{HA}, $f_*$ admits a left adjoint relative to $\mc{D}_S$.
This is denoted by $f^*\colon\mc{D}_S\rightarrow\pi_f^*\mc{D}_T$.
Informally, for $S'\in\mr{Sch}_{/S}$, $f^*$ is the functor sending $\mc{N}\in\mc{D}(S')\simeq(\mc{D}_S)_{S'}$ to $f^*_{S'}(\mc{N})\in\mc{D}(T\times_SS')\simeq(\pi_f^*\mc{D}_T)_{S'}$.
The functor $f^*$ preserves coCartesian edges.
On the other hand, assume that $f$ is {\em proper}.
In this case, $f_*$ preserves coCartesian edges as well by the proper base change property \cite[6.2]{ABiv}.
By (the dual version of) \cite[7.3.2.6]{HA}, $f_*$ admits a right adjoint, denoted by $f^!$.

\subsection{}
\label{constfuncsys}
Assume we are given a Cartesian fibration of $\infty$-categories $p\colon\mc{C}\rightarrow\mr{Sch}_{/S}$.
Let $f\colon T\rightarrow S$ be a morphism in $\mr{SCH}$.
We define the $\infty$-category of {\em $f$-systems} to be
\begin{equation*}
 \Sys^{\mc{C}}_f=\Sys^{\mc{C}}_{T/S}:=\mr{Fun}_{\mr{Sch}_{/S}^{\mr{op}}}\bigl(\mc{C}^{\mr{op}},\pi_f^*\mc{D}_T\bigr).
\end{equation*}
When $\mc{C}=\mr{Sch}_{/S}$, we denote $\Sys_f^{\mc{C}}$ simply by $\Sys_f$.
By \cite[1.1.4.5]{HA}, $\Sys^{\mc{C}}_{T/S}$ is a stable $\infty$-category.
Informally, $\mc{M}\in\Sys^{\mc{C}}_{T/S}$ is an association $\mc{M}_{c}\in\mc{D}(p(c)\times_ST)$ to $c\in\mc{C}$ and,
if we are given a morphism $g\colon c\rightarrow c'$ in $\mc{C}$,
we have the map $\phi_g\colon p(g)_T^*\mc{M}_{c'}\rightarrow\mc{M}_{c}$, where $p(g)_T:=p(g)\times_ST$.
A system $\mc{M}$ is said to be {\em strict} if it preserves coCartesian edges viewed as a functor between coCartesian fibrations over $\mr{Sch}_{/S}^{\mr{op}}$.
The system $\mc{M}$ is strict if and only if $\phi_g$ is an equivalence for any $g$.
We note that $\mc{M}\in\Sys^{\mc{C}}_{T/S}$ is equivalent to $0$ if and only if $\mc{M}_{c}\simeq0$ for any $c\in\mc{C}$.
This follows by \cite[5.1.2.2]{HTT} applied to any map $0\rightarrow\mc{M}$.

Now, assume we are given a morphism $D\colon f'\rightarrow f$ in $\mr{Fun}(\Delta^1,\mr{SCH})$ depicted as
\begin{equation}
 \label{diagSCH}
 \xymatrix{
  T'\ar[r]^-{g'}\ar[d]_{f}\ar@{}[rd]|{D}&T\ar[d]^{f}\\
 S'\ar[r]^-{g}&S.
  }
\end{equation}
Let $\iota_g\colon\mr{Sch}_{/S'}\rightarrow\mr{Sch}_{/S}$ the functor sending $Z\rightarrow S'$ to $Z\rightarrow S'\xrightarrow{g}S$.
We put $\mc{C}_{S'}:=\iota^*_g\mc{C}$ (in other words $\mc{C}\times_{\mr{Sch}_{/S},\iota_g}\mr{Sch}_{/S'}$).
Let $h\colon T'\rightarrow T_{S'}:=T\times_SS'$ be the induced morphism.
By taking the base change of the functor $h^*\colon\mc{D}_{T_{S'}}\rightarrow \pi_h^*\mc{D}_{T'}$ over $\mr{Sch}_{/T_{S'}}$ by $\pi_{f_{S'}}$,
we get the functor $\pi^*_{f_{S'}}\mc{D}_{T_{S'}}\rightarrow\pi^*_{f'}\mc{D}_{T'}$ also denoted by $h^*$.
We define $D^*\colon\Sys^{\mc{C}}_f\rightarrow\Sys^{\mc{C}_{S'}}_{f'}$ to be the composite
\begin{equation*}
 \mr{Fun}(\mc{C}^{\mr{op}},\pi^*_f\mc{D}_T)
  \xrightarrow{\iota^*_g(-)}
  \mr{Fun}\bigl(\iota_g^*\mc{C}^{\mr{op}},\iota_g^*\pi^*_f\mc{D}_T\bigr)
  \cong
  \mr{Fun}\bigl(\mc{C}_{S'}^{\mr{op}},\pi_{f_{S'}}^*\mc{D}_{T_{S'}}\bigr)
  \xrightarrow{h^*\circ}
  \mr{Fun}\bigl(\mc{C}_{S'}^{\mr{op}},\pi_{f'}^*\mc{D}_{T'}\bigr).
\end{equation*}
If we are given morphisms $f\xrightarrow{D}f'\xrightarrow{D'}f''$ in $\mr{Fun}(\Delta^1,\mr{SCH})$, we may easily construct an equivalence $D'^*\circ D^*\simeq(D\circ D')^*$.
The morphism $D$ is said to be {\em proper} if the induced morphism $T'\rightarrow T_{S'}$ is proper.
In this case, $D_*$ preserves strictness.

The functor $D^*$ admits a right adjoint $D_*$.
Indeed, a right adjoint of the 1st functor $\iota^*(-)$ can be constructed by taking the $F_S$-right Kan extension of functors
$\iota_g^*\mc{C}^{\mr{op}}\rightarrow\pi_f^*\mc{D}_{T}$ along $\mc{C}_{S'}^{\mr{op}}\rightarrow\mc{C}^{\mr{op}}$.
For the 2nd functor, we just compose with $h_*$.
We will need the following properties of the functors:

\begin{lem*}
 \label{propfuncSys}
 Let $p\colon T\rightarrow S$ be a morphism of schemes which admits a section $i\colon S\rightarrow T$.
 Consider the morphisms $P:=(\mr{id},p)\colon(i\colon S\rightarrow T)\rightarrow(\mr{id}\colon S\rightarrow S)$
 and $I:=(\mr{id},i)\colon(\mr{id}\colon S\rightarrow S)\rightarrow(i\colon S\rightarrow T)$.
 Then the functors $P^*\colon\Sys_{S}^{\mc{C}}\leftrightarrows\Sys_{i}^{\mc{C}_T}\colon I^*$ is quasi-inverse to each other and yield an equivalence.
 In particular, we have equivalences $I_*\simeq P^*$ and $P_*\cong I^*$.
\end{lem*}
\begin{proof}
 Since $P\circ I\cong\mr{id}$, we only need to show that $P^*\circ I^*\simeq\mr{id}$, which is equivalent to showing that $(I\circ P)^*\simeq\mr{id}$.
 In our situation, we have an isomorphism of functors $\pi_i\circ\iota_{i\circ p}\cong\pi_i$.
 This implies that the composite
 \begin{equation*}
  \mr{Fun}(\mc{C}^{\mr{op}}_T,\pi^*_i\mc{D}_T)
  \xrightarrow{\iota^*_{i\circ p}(-)}
  \mr{Fun}\bigl(\iota_{i\circ p}^*\mc{C}^{\mr{op}}_T,\iota_{i\circ p}^*\pi^*_i\mc{D}_T\bigr)
  \cong
  \mr{Fun}(\mc{C}^{\mr{op}}_T,\pi^*_i\mc{D}_T)
 \end{equation*}
 is equivalent to the identity, and the claim follows by the construction of $(I\circ P)^*$.
\end{proof}

Consider the case where $g=\mr{id}$.
In this case, the functor $D_*$ (resp.\ $D^*$) is simply the composition functor $g'_*\circ$ (resp.\ $g'^*\circ$).
The morphism $D$ being proper is equivalent to $g'$ being proper.
When $g'$ is proper, recall that $g'_*$ admits a right adjoint $g'^!$.
Thus, {\em when $D$ is proper}, $D_*$ admits a right adjoint $D^!$, which is simply the composition functor $g'^!\circ$.
We denote the functors $D^*, D_*, D^!$ simply by $g'^*, g'_*, g'^!$ respectively.
The notation is slightly abusive, but we believe this will not cause any confusion.

\subsection{}
\label{morphobjstat}
Assume we are given $\infty$-categories $\mc{A}$, $\mc{M}$ and a functor $t\colon\mc{A}\times\mc{M}\rightarrow\mc{M}$.
We wish to consider the morphism object, but unlike \cite[\S4.2.1]{HA}, we do not need to consider sophisticated operad structure in this paper.
For a pair $m,m'\in\mc{M}$, the morphism object is a pair of $a\in\mc{A}$ and a morphism $\eta\colon a\otimes m\rightarrow m'$ having the property that the composite morphism
$\Map_{\mc{A}}(a',a)\rightarrow\Map_{\mc{M}}(a'\otimes m,a\otimes m)\xrightarrow{\eta\circ}\Map_{\mc{M}}(a'\otimes m,m')$ is an equivalence for any $a'\in\mc{A}$.
If $\mc{M}$ and $\mc{A}$ are presentable and $t$ commutes with colimits in each variable,
any pair of objects admit a morphism object by the adjoint functor theorem as in the proof of \cite[4.2.1.33]{HA}.
The morphism object is denoted by $\Mor_{\mc{M}}(-,-)$.
The assignment can be made in a homotopy coherent manner using the construction of \cite[4.2.1.31]{HA}.
In our setting, we may consider the morphism object $\Mor_{\mc{D}(T)}(\mc{F},\mc{G})\in\Mod_R$ for any $T\in\mr{SCH}$ and $\mc{F}$, $\mc{G}$ in $\mc{D}(T)$.
Since we also need to express that $\Mor_{\mc{D}(T)}$ behaves functorially when $T$ varies, we need a family version of this construction.

To do this, let $\mc{B}$ be an $\infty$-category, and let $\mc{A}\rightarrow\mc{B}$, $p\colon\mc{M}\rightarrow\mc{B}$ be coCartesian fibrations.
Assume we are given a functor $t_{\mc{M}}\colon\mc{A}\times_{\mc{B}}\mc{M}\rightarrow\mc{M}$ over $\mc{B}$ which preserves coCartesian edges.
The functor $t_{\mc{M}}$ is called the {\em left-tensor functor relative to $\mc{B}$}.
We further assume that the coCartesian fibrations $\mc{A},\mc{M}\rightarrow\mc{B}$ are presentable fibrations,
and the left-tensor commutes with colimits in each variable fiberwise.
Let $\mc{E}\rightarrow\mc{B}$ be a coCartesian fibration and $\mc{F}\rightarrow\mc{B}$ be a Cartesian fibration.
We have the dual coCartesian fibration $\widetilde{\mc{E}}\rightarrow\mc{B}$ (resp.\ dual Cartesian fibration $\widehat{\mc{F}}\rightarrow\mc{B}$)
whose fiber over $b\in\mc{B}$ is $\mc{E}_b^{\mr{op}}$ (resp.\ $\mc{F}_b^{\mr{op}}$) (cf.\ \ref{dualfibint}).
Assume we are given functor $G\colon\mc{E}\rightarrow\mc{M}$ over $\mc{B}$ and a section $m\colon\mc{B}\rightarrow\mc{M}$.
We will construct $M_G(m)\in\mr{Fun}_{\mc{B}}(\widetilde{\mc{E}},\mc{A})$ in \ref{consadjmorcot}.
This object possesses the following property, whose proof will be given in \ref{proofofprop}.

\begin{prop*}
 \label{concdescM}
 We further assume that $G$ preserves coCartesian edges.
 \begin{enumerate}
  \item\label{concdescM-1}
       For an object $e\in\widetilde{\mc{E}}$ over $b\in\mc{B}$, we have an equivalence $M_G(m)(e)\simeq\Mor_{\mc{M}_{b}}(G(e),m(b))$.

  \item\label{concdescM-2}
       Assume we are given a morphism $e_0\rightarrow e_1$ in $\widetilde{\mc{E}}$ over $f\colon b_0\rightarrow b_1$ in $\mc{B}$.
       Let $f_\mc{M}\colon\mc{M}_{b_0}\rightarrow\mc{M}_{b_1}$ be an associated functor.
       Then the morphism $M_G(m)(e_0)\rightarrow M_G(m)(e_1)$ is equivalent via the equivalence given in \ref{concdescM-1} to the composite morphism
	 \begin{align*}
	  \Mor_{\mc{M}_{b_0}}(G(e_0),m(b_0))\xrightarrow{f_{\mc{M}}}
	  \Mor_{\mc{M}_{b_1}}\bigl(f_{\mc{M}}G(e_0),f_{\mc{M}}(m(b_0))\bigr)
	  &\xleftarrow{\sim}
	  \Mor_{\mc{M}_{b_1}}\bigl(G(e_1),f_{\mc{M}}m(b_0)\bigr)\\
	  &\rightarrow
	  \Mor_{\mc{M}_{b_1}}(G(e_1),m(b_1)).
	 \end{align*}

  \item\label{concdescM-3}
       The functor $M_G(-)$ is compatible with base change of $\mc{B}$.
       More precisely, let $h\colon\mc{B}'\rightarrow\mc{B}$ be a functor, and let $\mc{M}'$, $\mc{A}'$, $\mc{E}'$ be the base changes to $\mc{B}'$.
       We also denote by $G'$, $m'$ the base changes of $G$, $m$.
       Then we have a canonical equivalence $M_G(m)\times_{\mc{B}}\mc{B}'\simeq M_{G'}(m')$.
 \end{enumerate}
\end{prop*}

\subsection{}
\label{Homfuncconstapp}
As we explained in \ref{recalfixsixthe}, we may apply the previous construction to our setting.
Let the notations be as in \ref{constfuncsys}.
Let $\mc{M}\in\Sys^{\mc{C}}_f$ be a {\em strict} system.
Then we have the functor
\begin{equation*}
 M_{\mc{M}}\colon\Sys_f=\mr{Fun}_{\mr{Sch}^{\mr{op}}_{/S}}(\mr{Sch}_{/S}^{\mr{op}},\pi_f^*\mc{D}_T)
  \rightarrow
  \mr{Fun}((\mc{C}^{\mr{op}})^{\sim},\Mod_R)\simeq\PShv_{\Mod_R}(\widehat{\mc{C}})
\end{equation*}
by \ref{morphobjstat}.
Since $M_{\mc{M}}$ is a right adjoint, $M_{\mc{M}}$ is an exact functor of stable $\infty$-categories.
For $\mc{N}\in\Sys_S$, we often denote $M_{\mc{M}}(\mc{N})$ by $\sMor_f(\mc{M},\mc{N})$, which is an object of $\PShv_{\Mod_R}(\widehat{\mc{C}})$.
We often abbreviate $\sMor_{S/S}$ by $\sMor_S$.
The functor $\sMor_f(\mc{M},\mc{N})$ sends $c\in\mc{C}$ to $\Mor_{\mc{D}(p(c)\times_ST)}\bigl(\mc{M}(c),\mc{N}(c)\bigr)$.
The following lemma is the only place where we need to go back to the definition of the functor $M_{\mc{M}}$ in this section.

\begin{lem*}
 \label{Moradjointlem}
 Assume we are given $D$ as in {\normalfont(\ref{diagSCH})}.
 We have the functor $\widehat{g}\colon\widehat{\mc{C}}_T\rightarrow\widehat{\mc{C}}$ induced by $\iota_g$.
 Let $\mc{M}\in\Sys^{\mc{C}}_f$, $\mc{N}\in\Sys_{f}$, $\mc{N}'\in\Sys_{f'}$, and assume $\mc{M}$ to be strict.
 We have the canonical morphism
 \begin{equation*}
  r_{\mc{M},\mc{N}}\colon
   \widehat{g}^*\sMor_f(\mc{M},\mc{N})\rightarrow\sMor_{f'}(D^*\mc{M},D^*\mc{N}).
 \end{equation*}
 Let $c\in\mc{C}_{S'}$ and let $h\colon p_{S'}(c)\times_{S'}T'\rightarrow p_{S'}(c)\times_ST$ be the canonical morphism of schemes.
 Then $r_{\mc{M},\mc{N}}(c)$ is homotopic to the morphism induced by $h^*$
 \begin{equation*}
  \Mor_{\mc{D}(p_{S'}(c)\times_ST)}\bigl(\mc{M}(c),\mc{N}(p_{S'}(c))\bigr)
   \rightarrow
   \Mor_{\mc{D}(p_{S'}(c)\times_{S'}T')}\bigl(h^*\mc{M}(c),h^*\mc{N}(p_{S'}(c))\bigr).
 \end{equation*}
 Furthermore, the composite
 \begin{equation*}
  \sMor_f(\mc{M},D_*\mc{N}')
   \xrightarrow{r_{\mc{M},D^*\mc{N}'}}
  \widehat{g}_*\,\sMor_{f'}(D^*\mc{M},D^*D_*\mc{N}')
  \rightarrow
  \widehat{g}_*\,\sMor_{f'}(D^*\mc{M},\mc{N}'),
 \end{equation*}
 where the 2nd morphism is induced by the adjunction $\mc{N}'\rightarrow D_*D^*\mc{N}'$, is an equivalence.
\end{lem*}
\begin{proof}
 For $\mc{F}\in\PShv_{\Mod_R}(\widehat{\mc{C}})$ and $\mc{M}$ as above, we have $\left<\mc{F},\mc{M}\right>_{\mr{Sch}_{/S}^{\mr{op}}}\in\Sys_f$ as introduced in \S\ref{sec-mor}.
 Recall that the functor $\sMor_f(\mc{M},\mc{N})$ can be characterized by the equivalence
 \begin{equation*}
  \Map\bigl(\mc{F},\sMor_f(\mc{M},\mc{N})\bigr)\simeq
   \Map\bigl(\left<\mc{F},\mc{M}\right>_{\mr{Sch}_{/S}^{\mr{op}}},\mc{N}\bigr)
 \end{equation*}
 contravariant in $\mc{F}$.
 Since the functor $h^*$ commutes with colimits, by Lemma \ref{pairbasicprop}, we have an equivalence
 $\left<\widehat{g}^*\mc{F},D^*\mc{M}\right>_{\mr{Sch}_{/S'}^{\mr{op}}}\simeq D^*\left<\mc{F},\mc{M}\right>_{\mr{Sch}_{/S}^{\mr{op}}}$
 Using this compatibility, we have
 \begin{equation*}
  \bigl<\widehat{g}^*\sMor_f(\mc{M},\mc{N}),D^*\mc{M}\bigr>
   \simeq
   D^*\bigl<\sMor_f(\mc{M},\mc{N}),\mc{M}\bigr>
   \rightarrow D^*\mc{N}.
 \end{equation*}
 We define $r_{\mc{M},\mc{N}}$ by taking the adjoint of this morphism.
 We get the concrete description by construction.
 To check the last claim, it suffices to show the equivalence after applying the functor $\Map(\mc{F},-)$ to the morphism for any $\mc{F}\in\PShv_{\Mod_R}(\widehat{\mc{C}})$.
 The verification is standard using the compatibility of pullback and adjunctions.
\end{proof}

\subsection{}
Recall that $\eta$ is the generic point of $\Box$, and let $j\colon S\times\eta\rightarrow S\times\Box$ be the morphism in $\mr{SCH}$.
Assume we are given a fiberwise exact functor $\Gamma\colon\pi_j^*\mc{D}_{S\times\eta}\rightarrow\mc{D}_{S\times\Box}$ of Cartesian fibrations over $\mr{Sch}_{/S\times\Box}$.
Then the composition with $\Gamma$ induces the functor $\Sys_{S\times\eta/S\times\Box}\rightarrow\Sys_{S\times\Box/S\times\Box}$,
which we denote by $\Gamma_{\Sys}$.

\begin{lem*}
 \label{gcdpshfchck}
 Let $p\colon\mc{C}\rightarrow\mr{Sch}_{/S\times\Box}$ be a Cartesian fibration of categories.
 Assume that $\Gamma$ preserves Cartesian edges {\em over proper morphisms in $\mr{Sch}_{/S\times\Box}$}.
 For $\mc{M}\in\Sys^{\mc{C}}_{S\times\Box/S\times\Box}$, $\mc{N}\in\Sys_{S\times\eta/S\times\Box}$ which are {\em strict},
 the presheaf $\sMor_{S\times\Box}\bigl(\mc{M},\Gamma_{\Sys}(\mc{N})\bigr)$ on $\widehat{\mc{C}}$ is a $\eta$-cdp sheaf, more precisely,
 a $\widehat{p}$-Cartesian sheaf on $\mr{Sch}_{/S\times\Box}$ endowed with $\eta$-cdp topology {\normalfont(}cf.\ {\normalfont\ref{varpartdef}}{\normalfont)}.
\end{lem*}
\begin{proof}
 Let $c\in\mc{C}$, and recall the geometric morphism of topoi $r_c\colon\PShv(\mr{Sch}_{/p(c)})\rightarrow\PShv(\widehat{\mc{C}})$ from \ref{carttopo}.
 By Lemma \ref{carttopo}, we must show that $r_c^*\,\sMor\bigl(\mc{M},\Gamma_{\Sys}(\mc{N})\bigr)$ is an $\eta$-cdp sheaf.
 Recall that the functor $a_c\colon\mr{Sch}_{/p(c)}\rightarrow\mc{C}$ sends $T\rightarrow p(c)$ to the Cartesian pullback of $c$ by $T\rightarrow p(c)$.
 We have
 \begin{align*}
  r_c^*\,\sMor\bigl(\mc{M},\Gamma_{\Sys}(\mc{N})\bigr)\simeq
  \sMor\bigl(\mc{M},\Gamma_{\Sys}(\mc{N})\bigr)\circ a_c
  \simeq
  \sMor\bigl(\mc{M}\circ a_c,\Gamma_{\Sys}(\mc{N}|_{\mr{Sch}^{\mr{op}}_{/p(c)}})\bigr),
 \end{align*}
 where the 1st equivalence follows by \ref{carttopo}, the 2nd follows by Proposition \ref{concdescM}.\ref{concdescM-3}
 applied to the base change by $\mr{Sch}_{/p(c)}\rightarrow\mr{Sch}_{/S\times\Box}$.
 Since $a_c$ preserves Cartesian edges, $\mc{M}\circ a_c$, $\mc{N}|_{\mr{Sch}^{\mr{op}}_{/p(c)}}$ are strict systems.
 Thus, we may assume that $\mc{C}=\mr{Sch}_{/S\times\Box}$.

 For $S'$ in $\mr{Sch}_{/S\times\Box}$, we denote by $M_{S'}$ (resp.\ $N_{S'}$) the value,
 which is an object of $\mc{D}(S')$ (resp.\ $\mc{D}(S'_{\eta})$), of $\mc{M}$ (resp.\ $\mc{N}$) at $S'$.
 Similarly, we put $\Gamma_{S'}\colon\mc{D}(S'_\eta)\rightarrow\mc{D}(S')$ the fiber of $\Gamma$ over $S'$.
 Consider an $\eta$-blowup square (in the sense of \ref{bgpropofetacdp}) of the form (\ref{bgpropQ}).
 By Lemma \ref{bgprop} and \ref{bgpropofetacdp},
 we must show that the induced map $\Delta^1\times\Delta^1\rightarrow\Mod_R$ depicted as the diagram below is Cartesian.
 \begin{equation*}
  \xymatrix@C=10pt{
   \Mor(M_B,\Gamma_B(N_{B}))\ar[r]\ar[d]&\Mor(M_Y,\Gamma_Y(N_{Y}))\ar[d]\\
  \Mor(M_A,\Gamma_A(N_{A}))\ar[r]&\Mor(M_X,\Gamma_X(N_{X}))
   }
 \end{equation*}
 Note that any morphism in (\ref{bgpropQ}) is proper.
 Since $\Gamma$ preserves Cartesian edges, $\Gamma_-$ commutes with proper pushforward.
 Since the morphisms $e$, $p$ are proper and $\mc{M}$ is strict,
 it suffices to show the left below diagram, where $q\colon B\rightarrow X$ is the composite, is Cartesian:
 \begin{equation*}
  \xymatrix@C=10pt{
   \Mor(M_X,\Gamma_X(q_*N_{B}))\ar[r]\ar[d]&\Mor(M_X,\Gamma_X(p_*N_{Y}))\ar[d]\\
  \Mor(M_X,\Gamma_X(e_*N_{A}))\ar[r]&\Mor(M_X,\Gamma_X(N_{X})),
   }
   \qquad
  \xymatrix@C=10pt{
   q_*N_{B}\ar[r]\ar[d]&p_*N_{Y}\ar[d]\\
  e_*N_{A}\ar[r]&N_{X}.
   }
 \end{equation*}
 Since $\Gamma_X$ and $\Mor(M_X,-)$ are exact, we must show that the right above diagram in $\mc{D}(X_{\eta})$ is Cartesian.
 Since (\ref{bgpropQ}) is a an $\eta$-blowup square, the diagram $\mbox{(\ref{bgpropQ})}\times_{(S\times\Box)}(S\times\eta)$ is a blowup square,
 and since $\mc{N}$ is strict, the verification of the right above diagram being Cartesian is standard.
\end{proof}

\begin{thm}
 \label{revcons}
 Let $\mc{M}\in\Sys^{\mc{C}}_{S/S}$, $\mc{N}\in\Sys_{S/S}$ be strict systems, and put $\mc{H}:=\sMor_S(\mc{M},\mc{N})$ in $\PShv_{\Mod_R}(\widehat{\mc{C}})$.
 Then the canonical morphism $\mc{H}\rightarrow\partial^{\widehat{\mc{C}}}\mc{H}$
 {\normalfont(}using the notation of {\normalfont\ref{varpartdef}}{\normalfont)} admits a retract.
\end{thm}
\begin{proof}
 The proof is divided into several steps.
 We put $x:=\infty$.
 We denote by $j\colon S\times\eta\hookrightarrow S\times\Box$ the canonical inclusion.
 For a closed point $y$ of $\Box$, we denote by $i_y\colon S\otimes k(y)\hookrightarrow S\times\Box$ the canonical morphism.
 We consider the following commutative diagram
 \begin{equation*}
  \xymatrix@C=50pt{
   S\ar[d]\ar[r]^-{\sim}\ar@{}[rd]|{I}&
   S\times\{x\}\ar@{^{(}->}[r]^-{i_x}\ar[d]&
   S\times\Box\ar[r]\ar[d]\ar@{}[rd]|{P}&
   S\ar[d]\\
  S\ar@{^{(}->}[r]^-{i_x}&S\times\Box\ar@{=}[r]&S\times\Box\ar[r]^-{\mr{pr}}&S,
   }
 \end{equation*}
 and put $\mc{M}_{\Box}:=P^*\mc{M}$, $\mc{N}_{\Box}:=P^*\mc{N}$.
 \medskip

 \noindent{\bf Step 1)}\,
 We claim that $\sMor_{S\times\Box}\bigl(\mc{M}_{\Box},i_{x*}i^*_{x}j_*j^*\mc{N}_{\Box}\bigr)$ is an $\eta$-cdp sheaf.
 Indeed, for $T\in\mr{Sch}_{/S\times\Box}$, let $j_T\colon T\times_{\Box}\eta\hookrightarrow T$ and $i_{Tx}\colon T\times_{\Box}\{x\}\hookrightarrow T$ be the canonical morphisms.
 By Lemma \ref{gcdpshfchck}, it suffices to check that the functors $j_{T*},\,i_{Tx*}i_{Tx}^*j_{T*}\colon\mc{D}(T_\eta)\rightarrow\mc{D}(T_\Box)$ commute with proper pushforward.
 The case for $j_{T*}$ is evident.
 Let $f\colon T\rightarrow T'$ be a proper morphism in $\mr{Sch}_{/S\times\Box}$, and let $f_\eta:=f\times_{\Box}\eta$.
 Unwinding the definition, we must show that the canonical morphisms
 \begin{equation*}
   (i_{T'x*}i_{T'x}^*j_{T'*})\circ f_{\eta*}\rightarrow f_*\circ (i_{Tx*}i_{Tx}^*j_{T*})
 \end{equation*}
 is equivalent.
 This follows by proper base change property.

 Now, we have the following homotopy commutative diagram
 \begin{equation*}
  \xymatrix@C=10pt{
   \mc{H}\ar@{-}[r]^-{\sim}&
   \sMor(\mc{M},P_*i_{x*}I_*\mc{N})\ar@{-}[r]^-{\sim}_-{\star}&
   \mr{pr}_*\sMor(\mc{M}_{\Box},i_{x*}i_{x}^*\mc{N}_{\Box})\ar[d]_{\ccirc{2}}\ar[r]^-{\ccirc{1}}&
   \mr{pr}_*L_{\eta}^*\sMor(\mc{M}_{\Box},i_{x*}i_{x}^*\mc{N}_{\Box})=:\partial_x\mc{H}\ar[d]\\
  &&
  \mr{pr}_*\sMor(\mc{M}_{\Box},i_{x*}i_{x}^*j_*j^*\mc{N}_{\Box})\ar[r]_-{\sim}&
  \mr{pr}_* L_{\eta}^*\sMor(\mc{M}_{\Box},i_{x*}i_{x}^*j_*j^*\mc{N}_{\Box}).
   }
 \end{equation*}
 Here, we used the notations of \ref{basicsitediag}, and $\star$ is an equivalence by Lemma \ref{propfuncSys} and Lemma \ref{Moradjointlem}.
 We must construct a retract of $\ccirc{1}$.
 For this, it suffices to construct a retract of $\ccirc{2}$, which we will do in the following steps.
 \medskip

 \noindent{\bf Step 2)}\,
 We have the following diagram of cofiber sequences in $\Sys_{S\times\Box}$ induced by the adjunction morphism $\mr{id}\rightarrow i_{y*}i_y^*$.
 \begin{equation}
  \label{loclfibseqpro}\tag{$\star$}
  \xymatrix{
   \mc{N}_{\Box}\ar[r]\ar[d]&
   j_*\mc{N}_\eta\ar[r]\ar[d]&
   \bigoplus_{z\in|\Box|} i_{z*}i_z^!\mc{N}_{\Box}[1]\ar[d]\\
  i_{y*}i_y^*\mc{N}_{\Box}\ar[r]&
   i_{y*}i_y^*j_*\mc{N}_{\eta}\ar[r]&
   i_{y*}i_y^!\mc{N}_{\Box}[1]
   }
 \end{equation}
 Indeed, considering the localization sequence, we only need to check that $\phi\colon\indlim_U k_{U*}k_U^*\mc{N}_{\Box}\rightarrow j_*\mc{N}_{\eta}$ is an equivalence,
 where $U$ runs over open subschemes of $\Box$ and $k_U\colon U\hookrightarrow\Box$ is the open immersion.
 The verification is pointwise, and it suffices to check for each $c\in\mc{C}$.
 Since $\mc{D}(p(c)\times\Box)$ is compactly generated, we only need to show that $\Map(X,\phi(c))$
 is an equivalence for each compact object $X$ of $\mc{D}(p(c)\times\Box)$.
 Since the theory is assumed weakly $\tau$-continuous, it suffices to show that $X$ is $\tau$-constructible.
 Since the full subcategory spanned by $\tau$-compact objects form a stable $\infty$-category by \cite[1.1.3.3]{HA},
 we have an equivalence $a\colon X\xrightarrow{\sim}\indlim_{\alpha\in I} C_\alpha$ where $I$ is filtered and $C_\alpha$ is $\tau$-constructible by
 \cite[1.3.17]{CD} and \cite[1.1.3.6]{HA}.
 Since $X$ is compact, this equivalence factors through $b\colon X\rightarrow C_{\alpha_0}$ for some $\alpha_0$.
 The morphism $C_{\alpha_0}\rightarrow\indlim C_\alpha\xrightarrow{a^{-1}}X$ is a retract of $b$.
 Since $\tau$-constructible objects are stable under retract, $X$ is $\tau$-constructible as required.
 \medskip

 \noindent{\bf Step 3)}\,
 Pick some {\em $k$-rational} point $y\in(\Box\setminus\{x\})(k)$.
 The localization cofiber sequences of the previous step induces the following diagram (of solid arrows) in $\PShv(\widehat{\mc{C}})$:
 \begin{equation*}
  \label{bigdiagMor}\tag{$\star\star$}
  \xymatrix{
   &\mr{pr}_*\sMor(\mc{M}_{\Box},i_{x*}i_x^!\mc{N}_{\Box}[1])
   \ar[d]^{\mr{inc}_x}\ar@{.>}[ld]_{\sigma_x^y}&\\
   \mr{pr}_*\sMor(\mc{M}_{\Box},j_*\mc{N}_{\eta})
    \ar[r]^-{\gamma}&
   \bigoplus_z\mr{pr}_*\sMor(\mc{M}_{\Box},i_{z*}i_z^!\mc{N}_{\Box}[1])
   \ar[r]^-{b}\ar[d]^{c_y}&
   \mr{pr}_*\sMor(\mc{M}_{\Box},\mc{N}_{\Box}[1])
   \ar[d]^{\heartsuit}\\
  &
   \mr{pr}_*\sMor(\mc{M}_{\Box},i_{y*}i_y^!\mc{N}_{\Box}[1])
   \ar[r]^-{d}&
   \mr{pr}_*\sMor(\mc{M}_{\Box},i_{y*}i_y^*\mc{N}_{\Box}[1]).
   }
 \end{equation*}
 Here, $c_x$ is the morphism taking the ``$x$-component'', $\mr{inc}_x$ is the morphism whose $x$-component is the identity and the other components are $0$,
 the middle sequence is a cofiber sequence induced by (\ref{loclfibseqpro}).
 The morphism $\heartsuit$ is an equivalence.
 Let us admit this for a while and construct the dotted arrow first.
 Since $x\neq y$, the morphism $c_y\circ\mr{inc}_x$ is homotopic to $0$.
 Using the equivalence of $\heartsuit$, this homotopy determines a homotopy $h_x^y$ between $b\circ\mr{inc}_x$ and $0$ up to a contractible space of choices.
 This homotopy $h_x^y$ induces a morphism $\sigma_x^y$ such that $\gamma\circ\sigma^y_x$ is homotopic to $\mr{inc}_x$.

 Let us show that $\heartsuit$ is an equivalence.
 Consider homotopy commutative diagram
 \begin{equation*}
  \xymatrix{
   \mr{pr}_*\sMor(\mc{M}_{\Box},\mc{N}_{\Box})\ar[r]^-{s}\ar[rd]_{\heartsuit}&
   \mr{pr}_*\sMor\bigl(i_y^*\mc{M}_{\Box},i_y^*\mc{N}_{\Box}\bigr)\ar[d]^{\sim}\\
  &\mr{pr}_*\sMor\bigl(\mc{M}_{\Box},i_{y*}i_y^*\mc{N}_{\Box}\bigr).
   }
 \end{equation*}
 The vertical morphism is an equivalence by the last claim of Lemma \ref{Moradjointlem}.
 We only need to show that $s$ is an equivalence.
 Since the verification is pointwise, it suffices to show that $s(c)$ is an equivalence for any $c\in\mc{C}$.
 By Lemma \ref{Moradjointlem}, $s(c)$ is homotopic to the pullback morphism
 \begin{equation*}
  \Mor_{\mc{D}(p(c)\times\Box)}\bigl(\mr{pr}'^*\mc{M}(c),\mr{pr}'^*\mc{N}(p(c))\bigr)
   \rightarrow
   \Mor_{\mc{D}(p(c)\times\{y\})}\bigl(i'^*_y\mr{pr}'^*\mc{M}(c),i'^*_y\mr{pr}'^*\mc{N}(p(c))\bigr)
 \end{equation*}
 where $p(c)\xrightarrow{i'_y} p(c)\times\Box\xrightarrow{\mr{pr}'}p(c)$ are the closed immersion defined by $y$ and the projection.
 This is an equivalence by $\mb{A}^1$-invariance, thus $\heartsuit$ is an equivalence.
 \medskip

 \noindent{\bf Step 4)}\,
 Now, let us finish the proof.
 Let us consider, again, the following homotopy commutative diagram (of solid arrows) induced by the diagram of Step 2):
 \begin{equation*}
  \xymatrix{
   &
   \mr{pr}_*\sMor(\mc{M}_{\Box},j_*\mc{N}_{\eta})
    \ar[r]^-{\gamma}\ar[d]_{r_x}&
   \bigoplus_z\mr{pr}_*\sMor(\mc{M}_{\Box},i_{z*}i_z^!\mc{N}_{\Box}[1])
   \ar[d]_{c_x}
   \\
  \mr{pr}_*\sMor(\mc{M}_{\Box},i_{x*}i_x^*\mc{N}_{\Box})\ar[r]_-{\ccirc{2}}&
   \mr{pr}_*\sMor(\mc{M}_{\Box},i_{x*}i_x^*j_*j^*\mc{N}_{\Box})\ar[r]_-{\gamma_x}&
   \mr{pr}_*\sMor(\mc{M}_{\Box},i_{x*}i_x^!\mc{N}_{\Box}[1])
   \ar@{.>}[ul]^-{\sigma^y_x}\ar@/_15pt/@{.>}[u]_-{\mr{inc}_x}.
   }
 \end{equation*}
 The composite $r_x\circ\sigma^y_x$ is a section of $\gamma_x$.
 Indeed, we have the sequence of homotopies
 \begin{equation*}
  \gamma_x\circ r_x\circ\sigma^y_x
   \sim
   c_x\circ\gamma\circ\sigma^y_x
   \sim
   \mr{inc}_x\circ c_x
   \sim
   \mr{id}.
 \end{equation*}
 Since the 2nd row of the diagram is a cofiber sequence, $r_x\circ\sigma^y_x$ yields a retract of $\ccirc{2}$, which is what we have been looking for.
\end{proof}

\begin{dfn}
 Let $f\colon X\rightarrow S$ be a separated morphism of finite type.
 We use the notations of \ref{tracemapdfn}.
 Consider the composite $\Hcoe_{X/S}\colon\mr{Sch}_{/S}^{\mr{op}}\rightarrow(\widehat{\mr{Ar}^{\mr{prop}}_S})^{\mr{op}}\xrightarrow{\Hbm}\Mod_R$,
 where the 1st functor sends $T$ to $X_T\rightarrow T$.
 
\end{dfn}

\begin{cor}
 \label{coefsys}
 The presheaf $\Hcoe_{X/S}$ on $\mr{Sch}_{/S}$ is a reversible coefficient.
\end{cor}
\begin{proof}
 It is a cdh-sheaf by Lemma \ref{cdhdsvalshv}, and $\mb{A}^1$-invariant
 since the functor is an $\infty$-enhancement of the Borel-Moore homology theory.
 Finally, we have just constructed a rewind morphism.
\end{proof}

\section{Microlocalization}
\label{sect3}
In this section, we develop a general theory to ``microlocalize'' a map $\ms{S}\rightarrow\Hcoe_{X/S}$ where $\ms{S}$
is an abelian sheaf on $\mr{Sch}_{/S}$ and $X\rightarrow S$ is a morphism of schemes.
This section is inspired strongly by the work of Beilinson in \cite{B}.
We fix a base scheme $S$, which is noetherian separated and of finite Krull dimension.

\subsection{}
\label{dfnbasedata}
Let $f\colon X\rightarrow S$ be a morphism in $\mr{Sch}_{/S}$.
Let $(X/S)_{\mr{zar}}$ be the full subcategory of $\mr{Fun}(\Delta^1,\mr{Sch}_{/S})_{/f}$
spanned by morphisms $V\rightarrow T$ over $X\rightarrow S$ such that the induced morphism $V\rightarrow X\times_ST$ is an {\em open immersion}.
The object $V\rightarrow T$ is often abbreviated as $V/T$.
We denote by $\mc{O}_S$ the abelian presheaf on $(X/S)_{\mr{zar}}$ which associates $\Gamma(T,\mc{O}_T)$ to $V/T$.
For a vector bundle $E$ over $X$, we denote by $\mc{O}(E)$ the $\mc{O}_S$-module associating $\mr{Hom}_V(V,E\times_XV)$ to $V/T$.
Thus, if we are given $f\in\Gamma(V/T,\mc{O}(E))$, we have the associated morphism $V\rightarrow E_V:=E\times_X V$.
We denote $\mc{O}(X\times\mb{A}^1)$ by $\mc{O}_{X/S}$, in other words, the presheaf associating $\Gamma(V,\mc{O}_V)$ to $V/T$.

\begin{dfn*}
 We say that $(E,\nu,\mr{d}_{\nu})$ is a {\em base data of microlocalization over $X$} if
 \begin{enumerate}
  \item $E$ is a vector bundle over $X$;

  \item $\nu$ is an $\mc{O}_S$-module on $(X/S)_{\mr{zar}}$ equipped with homomorphism $\mr{d}_{\nu}\colon\nu\rightarrow\mc{O}(E)$
	of $\mc{O}_S$-modules.
 \end{enumerate}
 A morphism $(E,\nu,\mr{d}_{\nu})\rightarrow(E',\nu',\mr{d}_{\nu'})$ of base data consists of
 morphisms $f\colon E\rightarrow E'$ and $g\colon\nu\rightarrow\nu'$ such that $\mr{d}_{\nu'}\circ g=f\circ\mr{d}_{\nu}$.
 For $V/T\in(X/S)_{\mr{zar}}$ and $\mbf{f}\in\Gamma(V/T,\nu)$,
 we also denote by $\mr{d}_{\nu}\mbf{f}\colon V\rightarrow E_V$ the morphism associated to $\mr{d}_{\nu}\mbf{f}\in\Gamma(V/T,\mc{O}(E))$.
 We often denote $\mr{d}_{\nu}$ by $\mr{d}$.
\end{dfn*}

\begin{ex*}
 For a smooth scheme $X$ over $S$, $(T^*(X/S),\mc{O}_{X/S},\mr{d}_{X/S})$, where
 \begin{equation*}
  \mr{d}_{X/S}(V/T)\colon\mc{O}_{X/S}(V/T)\cong\Gamma(V,\mc{O}_V)\xrightarrow{\mr{d}}\Gamma(V,\Omega_V)\cong\mc{O}(T^*(X/S))(V/T)
 \end{equation*}
 is the total derivative, is a base data of microlocalization over $X$ relative to $S$.
\end{ex*}

\begin{dfn}
 Let $\mr{Sect}$ be the category whose objects consist of triples $(U,A,F)$ where $U\subset X$ is an open subscheme,
 $A$ is a finite dimensional affine space on $S$ ({\it i.e.}\ there exists $n\geq0$ such that $A\cong\mb{A}^n_S$),
 and a section $F\in \Gamma(U\times_SA/A,\nu)$.
 A morphism $(U,A,F)\rightarrow(U',A',F')$ is an inclusion $U\subset U'$ and an affine morphism $A\rightarrow A'$ such that
 the restriction morphism $\Gamma(U'\times_SA'/A',\nu)\rightarrow\Gamma(U\times_SA/A,\nu)$ sends $F'$ to $F$.
 Let $X_{\mr{zar}}$ be the category of (Zariski) open subsets of $X$.
\end{dfn}

For a scheme $T$, let $\mr{Sm}_T$ be the full subcategory of $\mr{Sch}_{/T}$ spanned by smooth morphisms $Y\rightarrow T$.
Let $C\subset E$ be a closed subscheme.
We have the functor $D_C\colon\mr{Sect}\rightarrow\mr{Fun}(\Delta^1\times\Delta^1,\mr{Sm}_X)$, often denoted simply by $D$,
sending $(U,A,F)$ to the Cartesian diagram in $\mr{Sm}_X$
\begin{equation}
 \label{dfnofsectcat}
 \xymatrix@C=50pt{
  U_{F}\ar[r]\ar@{^{(}->}[d]\ar@{}[rd]|\square&E_U\setminus C_U\ar@{^{(}->}[d]\\
 U\times_S A\ar[r]^-{\mr{pr}\circ\mr{d}F}&E_U,
  }
\end{equation}
where $\mr{pr}$ is the projection $E_U\times_SA\rightarrow E_U$.

\subsection{}
For a finite set $I$, we have the free commutative algebra $\mb{Z}[I]$, which is nothing but the polynomial algebra $\mb{Z}[x_i;i\in I]$.
We denote by $\mb{A}^I_{\mb{Z}}:=\mr{Spec}(\mb{Z}[I])$.
If we are given a map of finite sets $\phi\colon I\rightarrow J$, we have the morphism $\mb{A}^I_{\mb{Z}}\rightarrow\mb{A}^J_{\mb{Z}}$ sending
$(a_i)_{i\in I}$ to $\bigl(\sum_{i\in\phi^{-1}(j)} a_i\bigr)_{j\in J}$.
For a scheme $Z$, we put $\mb{A}^I_Z:=\mb{A}^I\times_{\mr{Spec}(\mb{Z})}Z$.

\begin{dfn*}
 For an open subscheme $U\subset X$, an {\em $U$-frame} is a finite linearly ordered set $\mbf{f}$ consisting of elements of $\Gamma(U/S,\nu)$
 such that the induced morphism $\mr{d}\mbf{f}:=\sum_{f\in\mbf{f}}x_f\cdot\mr{d}f\colon\mb{A}^{\mbf{f}}_U\rightarrow E_U:=E\times_XU$ is a {\em surjective} morphism of vector bundles.
 We denote by $\mr{Frm}_U$ the set of $U$-frames.
 We define a (1-)category $\mbf{\Delta}_{\mr{Frm}}$ as follows.
 The object of $\mr{Frm}$ is a tuple $(U;\mbf{f}_0,\dots,\mbf{f}_n)$ where $n\geq0$ is an integer and $\mbf{f}_i\in\mr{Frm}_U$.
 A morphism $(U;\mbf{f}_0,\dots,\mbf{f}_n)\rightarrow(U';\mbf{f}'_0,\dots,\mbf{f}'_m)$ is a order-preserving function $\phi\colon[n]\rightarrow[m]$ such that $U\subset U'$ and
 $\mbf{f}_i=\mbf{f}'_{\phi(i)}|_U$.
\end{dfn*}

Let us construct a functor $\Tot\colon\mbf{\Delta}_{\mr{Frm}}\rightarrow\mr{Sect}$.
Assume we are given an object $(U;\mbf{f}_0,\dots,\mbf{f}_n)\in\mr{Frm}$.
Put $\mbf{f}_{\mr{tot}}:=\mbf{f}_0\sqcup\dots\sqcup\mbf{f}_n$.
Using the $\mc{O}_S$-module structure of $\nu$,
we may define $\Tot(U;\mbf{f}_0,\dots,\mbf{f}_n):=\bigl(U,\mb{A}_S^{\mbf{f}_{\mr{tot}}},F_{\mbf{f}}:=\sum_{f\in\mbf{f}_{\mr{tot}}} x_f\cdot f\bigr)$.
Assume we are given a morphism $(U;\mbf{f}_0,\dots,\mbf{f}_n)\rightarrow(U';\mbf{f}'_0,\dots,\mbf{f}'_m)$.
We have the map $\mbf{f}_{\mr{tot}}\rightarrow\mbf{f}'_{\mr{tot}}$ sending $\mbf{f}_i$ to $\mbf{f}'_{\phi(i)}$.
This map induces the morphism $\alpha\colon\mb{A}^{\mbf{f}_{\mr{tot}}}_S\rightarrow\mb{A}^{\mbf{f}'_{\mr{tot}}}_S$.
It is straightforward to check that the pair $(U\subset U',\alpha)$ induces a morphism
$\Tot(U;\mbf{f}_0,\dots,\mbf{f}_n)\rightarrow\Tot(U';\mbf{f}'_0,\dots,\mbf{f}'_m)$ in $\mr{Sect}$.
We may check easily that the composite $U_{\mb{A}}\xrightarrow{\mr{d}F_{\mbf{f}}}(E_{U})_{\mb{A}}\xrightarrow{\mr{pr}}E_U$, where $\mb{A}=\mb{A}^{\mbf{f}_{\mr{tot}}}$,
is equal to $\mr{d}\mbf{f}_{\mr{tot}}$.

\subsection{}
Now, consider the composite
$\mbf{\Delta}_{\mr{Frm}}\xrightarrow{\Tot}\mr{Sect}\xrightarrow{D_C}\mr{Fun}(\Delta^1\times\Delta^1,\mr{Sm}_X)$.
We may write this diagram $\Delta^1\times\Delta^1\rightarrow\mr{Fun}(\mbf{\Delta}_{\mr{Frm}},\mr{Sm}_X)$ as:
\begin{equation*}
 \xymatrix{
  \mc{U}_{\bullet}\ar[r]\ar@{^{(}->}[d]\ar@{}[rd]|\square&(E\setminus C)_{\mc{U}}\ar@{^{(}->}[d]\\
 \mb{A}^{\bullet}_{\mc{U}}\ar[r]^-{\mr{d}(\bullet)}&E_{\mc{U}},
  }
\end{equation*}
where the vertices correspond to those of the diagram (\ref{dfnofsectcat}).
For example, the value of the functor $\mc{U}_{\bullet}$ at $(U,\mbf{f}_0,\dots,\mbf{f}_n)\in\mbf{\Delta}_{\mr{Frm}}$ is $U_{F_{\mbf{f}}}$.
The following easy fact is a key to our construction:

\begin{lem*}
 \label{affbudllem}
 The evaluation of the map $\mr{d}(\bullet)\colon\mb{A}^{\bullet}_{\mc{U}}\rightarrow E_{\mc{U}}$
 at any object of $\mbf{\Delta}_{\mr{Frm}}$ is an affine bundle.
\end{lem*}
\begin{proof}
 Let $(U,\mbf{f}_0,\dots,\mbf{f}_n)\in\mbf{\Delta}_{\mr{Frm}}$.
 We must show that the morphism $\mr{d}\mbf{f}\colon\mb{A}^{\mbf{f}_{\mr{tot}}}_U\rightarrow E_U$ is an affine bundle.
 The morphism $\mr{d}\mbf{f}$ is equal to the composite
 \begin{equation*}
  \mb{A}^{\mbf{f}_{\mr{tot}}}_U\cong(\mb{A}_U^{\mbf{f}_0})\times\dots\times(\mb{A}_U^{\mbf{f}_n})
   \xrightarrow{(\mr{d}\mbf{f}_0,\dots,\mr{d}\mbf{f}_n)} E^{\times(n+1)}_U\xrightarrow{\sum}E_U.
 \end{equation*}
 The first is an $\mb{A}^{\#\mbf{f}_{\mr{tot}}-(n+1)\mr{rk}(E)}$-bundle and the second is an $\mb{A}^{n\mr{rk}(E)}$-bundle.
\end{proof}

\subsection{}
\label{framediagH}
Now, assume we are in the situation of \ref{recalfixsixthe}.
Let $f\colon X\rightarrow S$ be the structural morphism.
We put $\omega_{X/S}:=f^!R_S$ in $\mc{D}_X$.
Recall from \ref{compbmcohcoh} that we have the functor $\rH(-,\omega_{X/S})\colon\mr{Sm}_X^{\mr{op}}\rightarrow\Mod_R$.
Informally, this is the functor sending $g\colon Y\rightarrow X$ to $\rH(Y/X,\omega_{X/S})\simeq\Mor_{\mc{D}_Y}(R_Y,g^*\omega_{X/S})$.
We finally consider the following composite
\begin{equation*}
 \mbf{\Delta}_{\mr{Frm}}^{\mr{op}}
  \rightarrow
  \mr{Fun}(\Delta^1\times\Delta^1,\mr{Sm}_X)^{\mr{op}}
  \xrightarrow{\rH(,\omega_{X/S})}
  \mr{Fun}(\Delta^1\times\Delta^1,\Mod_R).
\end{equation*}
This functor can be depicted as
\begin{equation*}
 \xymatrix@C=50pt{
  \rH(\mc{U}_{\bullet}/X,\omega_{X/S})\ar@{}[rd]|{\ccirc{$\star$}}&
  \rH((E\setminus C)_{\mc{U}}/X,\omega_{X/S})\ar[l]^-{\sim}\\
 \rH(\mb{A}^{\bullet}_{\mc{U}}/X,\omega_{X/S})\ar[u]&
  \rH(E_{\mc{U}}/X,\omega_{X/S}).
  \ar[u]\ar[l]_{\mr{d}(\bullet)^*}^-{\sim}
  }
\end{equation*}
Here the equivalences of the horizontal edges follow by Lemma \ref{affbudllem}.
We often omit $\omega_{X/S}$ from $\rH(-,\omega_{X/S})$ in what follows.

\begin{dfn*}
 The base data is said to {\em have enough frames} if the sieve	$\{U\in X_{\mr{zar}}\mid\mr{Frm}_U\neq\emptyset\}$ is a Zariski covering sieve of $X$.
 The full subcategory of $X_{\mr{zar}}$ spanned by the above set is denoted by $X_{\mr{zar}}^{\mr{fr}}$.
 Then we have the functor $\mbf{\Delta}_{\mr{Frm}}\rightarrow X_{\mr{zar}}^{\mr{fr}}$ sending $(U;\mbf{f}_0,\dots,\mbf{f}_n)$ to $U$.
 Passing to the opposite category, this defines a functor $p\colon\mbf{\Delta}_{\mr{Frm}}^{\mr{op}}\rightarrow(X_{\mr{zar}}^{\mr{fr}})^{\mr{op}}$.
\end{dfn*}

\begin{lem*}
 \label{desclemeas}
 Consider the functor
 $\rH(\mc{U}_{\bullet}/X)\colon\mbf{\Delta}_{\mr{Frm}}^{\mr{op}}\rightarrow\Mod_R$, and the left Kan extension functor
 \begin{equation*}
  p_!\colon
   \mr{Fun}\bigl(\mbf{\Delta}_{\mr{Frm}}^{\mr{op}},\Mod_R\bigr)
   \rightarrow
   \mr{Fun}\bigl((X_{\mr{zar}}^{\mr{fr}})^{\mr{op}},\Mod_R\bigr).
 \end{equation*}
 If the base data have enough frames, then we have
 $\invlim_{(X_{\mr{zar}}^{\mr{fr}})^{\mr{op}}}p_!\rH(\mc{U}_{\bullet}/X)\simeq\rH((E\setminus C)/X)$.
\end{lem*}
\begin{proof}
 By definition, the functor $\rH((E\setminus C)_{\mc{U}}/X)\colon\mbf{\Delta}_{\mr{Frm}}^{\mr{op}}\rightarrow\Mod_R$ factors through $p$.
 The functor $(X_{\mr{zar}}^{\mr{fr}})^{\mr{op}}\rightarrow\Mod_R$ sending $U$ to $\rH((E\setminus C)_U/X)$ is denoted by $\rH$.
 We have the morphisms
 \begin{equation*}
  p_!\rH(\mc{U}_{\bullet}/X)
   \xleftarrow{\sim}
   p_!\rH((E\setminus C)_{\mc{U}}/X)
   \simeq
   p_!(\rH\circ p)
   \rightarrow
   \rH.
 \end{equation*}
 Let us show that the third morphism is an equivalence.
 Since $p$ is a coCartesian fibration, $p_!$ can be computed fiberwise.
 Thus, by \cite[4.4.4.10]{HTT}, it suffices to show that the fiber of $p$ over any object of $X_{\mr{zar}}^{\mr{fr}}$
 is weakly contractible.
 For $U\in X^{\mr{fr}}_{\mr{zar}}$, the category $\mbf{\Delta}_{\mr{Frm}}^{\mr{op}}\times_{X^{\mr{op}}_{\mr{zar}}}\{U\}$ is a filtered category.
 Thus, (the nerve of) $\mbf{\Delta}_{\mr{Frm}}\times_{X_{\mr{zar}}}\{U\}$ is weakly contractible by \cite[5.3.1.15, 5.3.1.20]{HTT}.
 Now, the claim follows by Zariski descent of $\rH(,\omega_{X/S})$.
\end{proof}

\subsection{}
\label{suppstrdfn}
With the preparation so far, let us return to the theory of microlocalizations.
Let $(X,E,\nu,\mr{d}_{\nu})$ be a base data of microlocalization,
and assume we are given a morphism $\ms{S}\rightarrow\mc{H}$ of presheaves on $\mr{Sch}_{/S}$.
For $\mc{F}\in\Gamma(S,\ms{S})$, we wish to makes sense of the notion of ``microsupport of $\mc{F}$ on $E$''.
In order to do this, we need extra data.
Let $p\colon(X/S)_{\mr{zar}}\rightarrow\mr{Sch}_{/S}$ be the functor sending $V/T$ to $T$.
Then $p$ is a Cartesian fibration.

\begin{dfn*}
 Let $(X,E,\nu,\mr{d}_\nu)$ be a base data of microlocalization, and let $(\ms{S},\nu,\mc{E}')$ be a $p$-deformation system
 (in the sense of Definition \ref{varpartdef}).
 Let $C$ be a conic ({\it i.e.}\ a closed subset stable under the $\mb{G}_{\mr{m}}$-action) in $E$.
 An element $\mc{F}\in\Gamma(X/S,\ms{S})$ is said to be {\em microsupported on $C$} if for any $V\rightarrow T$ in $(X/S)_{\mr{zar}}$
 such that $T\rightarrow S$ is {\em smooth}
 and any $f\in\Gamma(V/T,\nu)$, the restriction of $\mc{E}'(\mc{F},f)\in\Gamma(V/T,\partial\ms{S})$
 to $V\setminus(\mr{d}_\nu f)^{-1}(C)\rightarrow T$ is $0$.
 We denote by $\Gamma(X/S,\ms{S})_C$, or more precisely $\Gamma(X/S,\ms{S})_{\mc{E}',C}$,
 the subgroup of $\Gamma(X/S,\ms{S})$ consisting of elements which are microsupported on $C$.
 These form a subsheaf of $\ms{S}$ denoted by $\ms{S}_C$.
\end{dfn*}

Recall from \ref{introopenbm} that we have the functor ${}^{\mr{open}}\cH\colon\widehat{(X/S)}\vphantom{X}_{\mr{zar}}^{\mr{op}}\rightarrow\mc{D}_S$
which preserves coCartesian edges over $\mr{Sch}_{/S}^{\mr{op}}$.
Consider the functor $\sMor_S({}^{\mr{open}}\cH,R)\colon(X/S)^{\mr{op}}_{\mr{zar}}\rightarrow\Mod_R$ using the notation of Theorem \ref{revcons},
which coincides with ${}^{\mr{open}}\Hbm$ in \ref{introopenbm}.
Informally, this is the functor sending $V/T$ to $\Mor_{\mc{D}(T)}(f_{V!}f^*_VR_{T},R_T)$ where $f_V\colon V\rightarrow T$.
We denote $\sMor_S({}^{\mr{open}}\cH,R)$ by $\mc{H}$, which we view as a $\Mod_R$-valued presheaf on $(X/S)_{\mr{zar}}$.

\subsection{Microlocalization}\mbox{}\\
\label{consmicccabs}
We fix the following data:
\begin{quote}
 a base data of microlocalization {\em with enough frames} $(X,E,\nu,\mr{d}_\nu)$;
 a $p$-deformation system $(\ms{S},\nu,\mc{E}')$ (cf.\ \ref{varpartdef});
 a morphism $\mr{C}\colon\ms{S}\rightarrow\mc{H}$;
 a retract $\mr{rw}\colon\partial\mc{H}\rightarrow\mc{H}$ of the morphism $\mc{H}\rightarrow\partial\mc{H}$.
\end{quote}
For a conic $C\subset E$, we will construct a homomorphism $\mr{C}^{\mu}\colon\ms{S}_C(X)\rightarrow\rH_{C}(E,\omega_{X/S})$ called the {\em microlocalization of $\mr{C}$}
such that the following diagrams are homotopy commutative:
\begin{equation*}
 \xymatrix@C=30pt{
  \ms{S}(X/S)_C\ar[r]^-{\mr{C}^{\mu}}\ar@{^{(}->}[d]&\rH_{C}(E,\omega_{X/S})\ar[d]\\
 \ms{S}(X/S)\ar[r]^-{\pi^*\circ\mr{C}}&\rH(E,\omega_{X/S}),
  }
  \qquad
  \xymatrix{
  \ms{S}(X/S)_C\ar[d]_{\mr{C}^{\mu}}\ar[r]^-{\mc{E}'(-,f)}&\partial\ms{S}(U/S)\ar[r]^-{\mr{rw}\circ\partial\mr{C}}&\Hbm(U)\\
 \rH_C(E,\omega_{X/S})\ar[rr]^-{(\mr{d}_{\nu}f)^*}&&\Hbm(Z_f).
  \ar[u]
  }
\end{equation*}
Here, $U\subset X$ is an open subscheme, $f\in\nu(U/S)$, $Z_f$ is the closed subscheme $(\mr{d}_{\nu}f)^{-1}(C)$ of $U$,
and $\pi\colon E\rightarrow X$ is the structural morphism.
The construction is given in the next paragraph.

\subsection{}
Let us start the construction of $\mr{C}^{\mu}$.
By Lemma \ref{homcommdiagdefsys}, we have the following homotopy commutative diagram $D\colon\Delta^1\times\Delta^1\rightarrow\PShv_{\Mod_R}((X/S)_{\mr{zar}})$ on the left:
\begin{equation*}
 \xymatrix@C=40pt{
  \ms{S}\otimes\nu\ar[r]^-{\mc{E}'}\ar[d]_{\mr{C}\circ\Sigma}\ar@{}[rd]|{D}&
  \partial\ms{S}\ar[d]^{\mr{rw}\circ\partial\mr{C}}\\
 \mc{H}\ar@{-}[r]^-{\sim}&\mc{H},
  }\qquad
 \xymatrix{
  X\ar[d]&U\times_SA\ar[l]\ar[d]&U_{\mbf{f}}\ar[l]\ar[d]\\
 S&A\ar[l]&A.\ar@{=}[l]
  }
\end{equation*}
Let $\mr{Sect}\rightarrow\mr{Fun}(\Delta^2,(X/S)^{\mr{op}}_{\mr{zar}})$ be the functor sending
$(U,A,\mbf{f})$ to the above diagram (in $(X/S)_{\mr{zar}}$) on the right.
By composing these two functors and taking an appropriate simplicial subset, we obtain the functor $\mr{Sect}\rightarrow\mr{Fun}(\Delta^1\times\Delta^1,\Mod_R)$
which sends $(U,A,\mbf{f})\in\mr{Sect}$ to the homotopy commutative diagram on the left side below:
\begin{equation*}
 \xymatrix@C=30pt{
  \Gamma(X/S,\ms{S})\ar[r]^-{\mc{E}'(-,\mbf{f})}\ar[d]_{\mr{C}}\ar@{}[rd]|{\ccirc{$\star\star$}}&
  \Gamma(U_{\mbf{f}}/A,\partial\ms{S})\ar[d]^{\mr{rw}\circ\partial\mr{C}}\\
 \Gamma(U_A/A,\mc{H})\ar[r]^-{\mr{res}}&\Gamma(U_{\mbf{f}}/A,\mc{H}),
  }
  \quad
  \xymatrix@C=15pt{
  \Gamma(U_{\mbf{f}}/A,\partial\ms{S})\ar[r]\ar@{}[rd]|{\ccirc{$\star\star$}}&
  \Gamma(U_{\mbf{f}}/A,\mc{H})\ar@{-}[r]^-{\sim}&
  \rH(U_{\mbf{f}}/X,\omega_{X/S})\\
 \Gamma(X/S,\ms{S})\ar[r]\ar[u]^{\mc{E}'(-,\mbf{f})}&
  \Gamma(U_A/A,\mc{H})\ar[u]\ar@{-}[r]^-{\sim}&
  \rH(U_A/X,\omega_{X/S}).\ar[u]
  }
\end{equation*}
Note that for any object of $(U,A,\mbf{f})\in\mr{Sect}$, the morphism $A\rightarrow S$ is smooth.
Thus, by applying Lemma \ref{compbmcohcoh}, we get the diagram on the right above.
Picking up the outer square diagram, we get a diagram $\mr{Sect}\rightarrow\mr{Fun}(\Delta^1\times\Delta^1,\Mod_R)$.
Combining this with $\ccirc{$\star$}$ of \ref{framediagH}, we get the functor
\begin{equation*}
 \mbf{\Delta}_{\mr{Frm}}
  \xrightarrow{\Tot}
  \mr{Sect}
  \rightarrow
  \mr{Fun}(\Lambda^2_2\times\Delta^1,\Mod_R).
\end{equation*}
This functor can be depicted as
\begin{equation*}
 \xymatrix{
  \Gamma(\mc{U}_{\bullet},\partial\ms{S})\ar[r]&
  \rH(\mc{U}_{\bullet}/X,\omega_{X/S})&
  \rH((E\setminus C)_{\mc{U}}/X,\omega_{X/S})\ar[l]_-{\sim}\ar@{}[ld]|{\ccirc{$\star$}}\\
 \Gamma(X,\ms{S})\ar[r]\ar[u]^{\mc{E}'(-,-)}&
  \rH(\mb{A}^{\bullet}_{\mc{U}}/X,\omega_{X/S})\ar[u]&
  \rH(E_{\mc{U}}/X,\omega_{X/S}).\ar[l]_-{\sim}\ar[u]
  }
\end{equation*}
Now, consider the composite
\begin{equation*}
 \varphi_C\colon
 \Gamma(X,\ms{S})_C
  \hookrightarrow
  \Gamma(X,\ms{S})
  \xrightarrow{\mc{E}'(-,-)}
  \Gamma(\mc{U}_{\bullet},\partial\ms{S})
\end{equation*}
in $\mr{Fun}\bigl(\mbf{\Delta}_{\mr{Frm}},\Mod_R\bigr)$.
For each vertex $v$ of $\mbf{\Delta}_{\mr{Frm}}$, the morphism $\varphi_C(v)$ in $\Mod_R$ is homotopic to $0$ by definition.
This implies that $\pi_0\varphi_C$ is $0$.
Moreover, since $\Gamma(X,\ms{S})_C$ is in $\Mod_R^{\heartsuit}$ and $\Gamma(\mc{U}_{\bullet},\partial\ms{S})$ is in $(\Mod_R)_{\leq0}$,
$\varphi_C$ is homotopic to the composite
\begin{equation*}
 \Gamma(X,\ms{S})_C
  \xrightarrow{\pi_0\varphi_C}\tau_{\geq0}
  \Gamma(\mc{U}_{\bullet},\partial\ms{S})
  \rightarrow
  \Gamma(\mc{U}_{\bullet},\partial\ms{S}).
\end{equation*}
Thus, $\varphi_C$ is homotopic to $0$ in $\mr{Fun}\bigl(\mbf{\Delta}_{\mr{Frm}},\Mod_R\bigr)$.
This homotopy yields a homotopy between the composite
$\Gamma(X,\ms{S})_C\rightarrow\Gamma(X,\ms{S})\rightarrow\rH((E\setminus C)_{\mc{U}}/X,\omega_{X/S})$ and $0$.
Furthermore, by Lemma \ref{desclemeas}, this induces the diagram
\begin{equation*}
 \xymatrix{
  &\Gamma(X,\ms{S})_C\ar[d]^{\pi^*\circ\mr{C}}\ar[rd]^-{0}\ar@{-->}[dl]_-{\mr{C}^{\mu}}&\\
 \rH_C(E,\omega_{X/S})\ar[r]&\rH(E,\omega_{X/S})\ar[r]&\rH(E\setminus C,\omega_{X/S}),
  }
\end{equation*}
where the triangle is a diagram $\Delta^2\rightarrow\Mod_R$, and the row is a cofiber sequence.
Thus, we obtain the desired morphism $\mr{C}^{\mu}$ lifting $\pi^*\mr{C}$.

\subsection{Functorialities}\mbox{}\\
\label{functofmicccabs}
Finally, we discuss some functorialities of the microlocalized map.
We consider 4 types of functorialities: the functoriality when we change the base data, when we change the support systems,
when we change the base, and when we change the schemes.
\medskip

\noindent
{\bf \underline{Change of base data:}}\\
Assume we are given a morphism $m\colon(X,E',\nu',\mr{d}_{\nu'})\rightarrow(X,E,\nu,\mr{d}_{\nu})$
of base data with enough frames, and a presheaf with support structure $(\ms{S},\mc{E}')$ with repsect to $(X,E,\nu,\mr{d}_\nu)$.
Abusing the notation, we also denote by $m\colon E'\rightarrow E$ the associated morphism.
We define $\ms{S}\otimes\nu'\rightarrow\ms{S}$ by the composite
\begin{equation*}
 m^*\mc{E}'\colon
  \ms{S}\otimes\nu'
  \rightarrow
  \ms{S}\otimes\nu\xrightarrow{\mc{E}'}
  \partial\ms{S}.
\end{equation*}
This defines a support structure.
Let $C\subset E$ be a conic.
By construction, we have the homomorphism $\Gamma(X,\ms{S})_{\mc{E}',C}\rightarrow\Gamma(X,\ms{S})_{m^*\mc{E}',m^{-1}(C)}$.
Also by construction, we have the following diagram $\Delta^2\times\Delta^1\rightarrow\Mod_R$ depicted as
\begin{equation*}
 \xymatrix@R=10pt{
  \Gamma(X,\ms{S})_{\mc{E}',C}\ar[dd]\ar[rd]\ar@/^15pt/[rrd]^(.7){0}&&\\
 &\rH(E,\omega_{X/S})\ar[dd]\ar[r]&\rH(E\setminus C,\omega_{X/S})\ar[dd]\\
 \Gamma(X,\ms{S})_{m^*\mc{E}',m^{-1}(C)}\ar[rd]\ar@/^15pt/[rrd]|(.47){\hole}^(.7){0}&&\\
 &\rH(E',\omega_{X/S})\ar[r]&\rH(E'\setminus m^{-1}(C),\omega_{X/S})
  }
\end{equation*}
This diagram induces the homotopy commutative diagram
\begin{equation*}
 \xymatrix{
  \Gamma(X,\ms{S})_{\mc{E}',C}\ar[r]\ar[d]&\rH_{C}(E,\omega_{X/S})\ar[d]\\
 \Gamma(X,\ms{S})_{m^*\mc{E}',m^{-1}(C)}\ar[r]&\rH_{m^{-1}(C)}(E',\omega_{X/S}).
  }
\end{equation*}
\medskip

\noindent
{\bf \underline{Change of presheaves:}}\\
Let us fix base data $(X,E,\nu,\mr{d}_\nu)$, and assume we are given a homomorphism of presheaves
$\phi\colon\ms{S}_0\rightarrow\ms{S}_1$ endowed with support structures $\mc{E}'_0$, $\mc{E}'_1$
such that the following diagram on the left commutes.
Similarly to the previous functoriality, we may construct the following homotopy commutative diagram on the right below.
The detail is left to the reader.
\begin{equation*}
 \xymatrix{
  \ms{S}_0\otimes\nu\ar[r]^-{\mc{E}'_0}\ar[d]_{\phi\otimes\mr{id}}&\partial\ms{S}_0\ar[d]^{\partial\phi}\\
 \ms{S}_1\otimes\nu\ar[r]^-{\mc{E}'_1}&\partial\ms{S}_1,
  }\qquad
 \xymatrix{
  \Gamma(X,\ms{S}_0)_{C}\ar[r]\ar[d]&\rH_{C}(E,\omega_{X/S})\ar@{=}[d]\\
 \Gamma(X,\ms{S}_1)_{C}\ar[r]&\rH_{C}(E,\omega_{X/S}).
  }
\end{equation*}
\medskip

\noindent
{\bf \underline{Change of base:}}\\
Assume we are given a {\em smooth} morphism $g\colon S'\rightarrow S$.
Since any smooth scheme over $S'$ is smooth over $S$ as well, we have the homomorphism
$\Gamma(X,\ms{S})_{C}\rightarrow\Gamma(X_{S'},g^*\ms{S})_{C_{S'}}$ for any conic $C\subset E$,
and we have the following homotopy commutative diagram, whose detailed construction is left to the reader:
\begin{equation*}
 \xymatrix{
  \Gamma(X,\ms{S})_{C}\ar[r]\ar[d]&\rH_{C}(E,\omega_{X/S})\ar[d]\\
 \Gamma(X_{S'},g^*\ms{S})_{C_{S'}}\ar[r]&\rH_{C_{S'}}(E_{S'},\omega_{X_{S'}/S'}).
  } 
\end{equation*}
\medskip

\noindent
{\bf \underline{Change of schemes:}}\\
For the last functoriality, assume we are given a proper morphism $h\colon X\rightarrow Y$ over $S$.
We assume given base data $(Y,E,\nu,\mr{d}_\nu)$.
This defines a base data $(X,E\times_Y X,h^{-1}\nu,\mr{d}_{h^{-1}\nu})$, and let $(\ms{S},\mc{E}')$
be a presheaf with support structure with respect to this base data.
Let $g\colon E\times_Y X\rightarrow E$ be the induced proper morphism.
The morphism
\begin{equation*}
 h_*\mc{E}'\colon
 h_*\ms{S}\otimes\nu
 \cong
  h_*(\ms{S}\otimes h^{-1}\nu)\rightarrow h_*\partial\ms{S}
  \cong
  \partial h_*\ms{S}
\end{equation*}
defines a presheaf with support structure $(h_*\ms{S},h_*\mc{E}')$.
For $V\rightarrow T$ in $(Y/S)_{\mr{zar}}$, by definition, we have the following commutative diagram:
\begin{equation*}
 \xymatrix@C=40pt{
  \Gamma(V\times_Y X/T,\ms{S}\otimes h^{-1}\nu)\ar[r]^-{\mc{E}'}\ar@{-}[d]_{\sim}&
  \Gamma(V\times_Y X/T,\partial\ms{S})\ar@{-}[d]^-{\sim}\\
 \Gamma(V/T,h_*\ms{S}\otimes\nu)\ar[r]^-{h_*\mc{E}'}&\Gamma(V/T,\partial h_*\ms{S}).\\
  }
\end{equation*}
Let $C\subset E\times_Y X$ be a conic.
This commutative diagram defines a homomorphism $\Gamma(X,\ms{S})_C\rightarrow\Gamma(Y,h_*\ms{S})_{g(C)}$.
Using this, we have the following homotopy commutative diagram, whose detailed construction is left to the reader
\begin{equation*}
 \xymatrix{
  \Gamma(X,\ms{S})_{C}\ar[r]\ar[d]&\rH_{C}(E\times_YX,\omega_{X/S})\ar[d]\\
 \Gamma(Y,h_*\ms{S})_{g(C)}\ar[r]&\rH_{g(C)}(E,\omega_{Y/S}).
  }
\end{equation*}

\section{Preliminary --- Nearby cycles and flat \'{e}tale systems}
\label{sect4}
The rest of this paper is devoted mainly to constructing a deformation system.
Our deformation system is constructed from \'{e}tale sheaves.
In this section, all the schemes are over the field $k$ of positive characteristic,
and $\Lambda$ be a commutative finite local ring in which $\mr{char}(k)\neq0$ is invertible.
For a scheme $X$, the derived ($1$-)category of \'{e}tale $\Lambda$-modules is denoted by $D(X,\Lambda)$ or $D(X)$.
We denote by $D_{\mr{c}}(X)$ (resp.\ $D_{\mr{ctf}}(X)$) be the full subcategory of $D(X)$
consisting of constructible (resp.\ constructible and finite Tor-dimensional) complexes.
When $X$ is quasi-compact, note that $D_{\mr{ctf}}(X)=D^{\mr{b}}_{\mr{ctf}}(X)$.

\subsection{}
\label{setuptermncy}
A {\em geometric point} $s$ of a scheme $S$ is a morphism $p\colon s\rightarrow S$ where $s$ is a spectrum of a separably closed field $k(s)$.
We abbreviate the geometric point $p$ by $s$ if no confusion may arise.
If we are given another geometric point $p'\colon s'\rightarrow S$, a morphism of geometric points $p\rightarrow p'$ is nothing but an $S$-morphism $s\rightarrow s'$.
The geometric point $s$ is said to be {\em algebraic} if the extension of fields $k(s)/k(p(s))$ is algebraic.
We denote by $S_{(s)}$, or more precisely $S_{(p)}$, the strict henselization of $S$ at $s$.
A {\em specialization map} $\phi$ of geometric points, denoted by $\phi\colon t\rightsquigarrow s$, is a morphism of $S$-schemes $t\rightarrow S_{(s)}$.
Let $f\colon S\rightarrow T$ be a morphism.
We denote the geometric point $s\rightarrow S\xrightarrow{f}T$ of $T$ by $f(s)$.
A specialization map $t\rightsquigarrow s$ on $S$ induces the specialization map $f(t)\rightsquigarrow f(s)$ on $T$.

\subsection{}
Let $f\colon X\rightarrow S$ be a morphism of finite type between noetherian schemes.
For a geometric point $s\rightarrow S$, we put $X_{(s)}:=X\times_S S_{(s)}$.
Assume we are given a specialization map $\phi\colon t\rightsquigarrow s$ of geometric points on $S$.
Consider the Cartesian diagram
\begin{equation*}
 \xymatrix{
  X_s\ar[r]^-{i}\ar[d]\ar@{}[rd]|\square&
  X_{(s)}\ar[d]\ar@{}[rd]|\square&
  X_{t}\ar[l]_-{j}\ar[d]\\
 s\ar[r]&S_{(s)}&t\ar[l]_-{\phi}
  }
\end{equation*}
Then we define
\begin{equation*}
 \Psi_{f,\phi}:=i^*\circ\mr{R}j_*
  \colon
  D^+(X_{t})\rightarrow D^+(X_s).
\end{equation*}
We often denote $\Psi_{f,\phi}$ simply by $\Psi_\phi$ or $\Psi_{s\leftarrow t}$ if no confusion may arise.

\begin{rem*}
 \label{insensseprem}
 \begin{enumerate}
  \item\label{insensseprem-1}
        The functor $\Psi_{f,\phi}$ induces a functor $D^{\mr{b}}(X_{t},\Lambda)\rightarrow D^{\mr{b}}(X_s,\Lambda)$.
	Indeed, since $S$ is quasi-compact, the dimension of the fibers of $f$ is bounded by [EGA IV, 13.1.7].
	If $t$ is algebraic, then $\Psi_{\phi}$ preserves the boundedness by \cite[Prop 3.1]{O1} (or its evident analogue for our $\Psi_{\phi}$).
	Let $\alpha\colon s\rightarrow s'$ be a morphism of geometric points of $S$.
	To treat the general case, it suffices to show that for $\mc{F}\in D^{\mr{b}}(X_s)$, the pushforward $\mr{R}\alpha_*\mc{F}$ is bounded.
	We may assume that $\mc{F}$ is a sheaf.
	We have $\dim(X_s)\leq N$ since $X_s$ is of finite type over $s$.
	For an open affine subscheme $U\subset X_{s'}$, $U\times_{X_{s'}}X_s$ is affine as well, and [SGA 4, Exp.\ XIV, Cor 3.2] implies that
	$\mr{H}^i(U\times_{X_{s'}}X_s,\mc{F})=0$ for any $i>N$. Thus $\mr{R}^i\alpha_*\mc{F}=0$ for $i>N$.
	Furthermore, by the same argument as \cite[Rem 8.3]{O1} using [SGA 4, Exp.\ XVII, 5.2.11],
	$\Psi_{\phi}$ preserves finite Tor-dimensional complexes as well.

  \item Let $x$ be a geometric point of $X_s$.
	This yields a point $(\phi,x)$ of the topos $X\overleftarrow{\times}_SS$.
	For $\mc{F}\in D^+(X)$, we may consider the nearby cycle functor over general base
	$\mr{R}\Psi_f(\mc{F})\in D^+(X\overleftarrow{\times}_SS)$ (cf.\ \cite[1.2]{I}).
	If $t$ is a geometric generic point, then we have $\mr{R}\Psi_f(\mc{F})_{(\phi,x)}\xrightarrow{\sim}\Psi_{f,\phi}(\mc{F}_t)_x$.
	In general, these are different since $\mr{R}\Psi_f(\mc{F})_{(\phi,x)}$ is the cohomology of ``Milnor tube''
	whereas $\Psi_{f,\phi}(\mc{F})_x$ is that of ``Milnor fiber''.
	However, if $(f,\mc{F})$ is $\Psi$-good in the sense of Illusie \cite[1.5]{I}\footnote{
	The notion first appeared in \cite{O1} as far as the author knows, even though Orgogozo did not give a name.},
	these coincide for any geometric point $t$.
 \end{enumerate}
\end{rem*}

\begin{dfn}
 \label{examplegood}
 A complex $\mc{F}\in D^+(X_t)$ is said to be {\em $f$-good {\normalfont(}at $t${\normalfont)}}
 if for any morphism $g\colon S'\rightarrow S$, any specialization map
 $\phi'\colon t'\rightsquigarrow s'$ on $S'$, and a morphism of geometric points $h\colon g(t')\rightarrow t$, the canonical morphism
 \begin{equation*}
  \Psi_{f,g(\phi')}\bigl(h^*\mc{F}\bigr)
   \rightarrow
   \Psi_{f',\phi'}\bigl(h^*\mc{F}\bigr),
 \end{equation*}
 where $f'$ is the base change $f_{S'}=f\times_S S'$,
 is an isomorphism in $D^+(X_{s'})$.
 A complex $\mc{G}\in D^+(X)$ is said to be {\em $f$-good} if for any geometric point $t$ of $S$,
 the pullback $\mc{G}_t\in D^+(X_t)$ is $f$-good.
\end{dfn}

\begin{ex*}
 For $\mc{F}\in D^+(X)$, assume that $(f,\mc{F})$ is $\Psi$-good in the sense of Illusie.
 Then $\mc{F}$ is $f$-good in the sense above.
 Thus, we have the following examples of $f$-good complexes:
 \begin{enumerate}
  \item If $\mc{F}$ is universally locally acyclic over $S$, then $\mc{F}$ is $f$-good by \cite[1.7 (b)]{I}.

  \item\label{examplegood-2}
       Let $\mc{F}\in D^{\mr{b}}_{\mr{c}}(X,\Lambda)$, and assume that there exists a subscheme $Z\subset X$ such that
       the composite $Z\rightarrow S$ is quasi-finite and $\mc{F}|_{X\setminus Z}$ is universally locally acyclic relative to $f'$,
       where $f'\colon X\setminus Z\rightarrow S$.
       Then $\mc{F}$ is $f$-good by \cite[1.7 (c)]{I}.
 \end{enumerate}
 In fact, Gabber showed that $f$-good complexes are $\Psi$-good in the sense of Illusie,
 but we do not use this result in this paper.
\end{ex*}

\subsection{}
The following corollary of Orgogozo's theorem is fundamental to us:

\begin{cor*}[of Orgogozo's theorem]
 \label{vgofnbcycgen}
 Let $f\colon X\rightarrow S$ be a morphism of finite type between noetherian schemes,
 $u$ be a geometric point of $S$, $\phi\colon u\rightsquigarrow t$ be a specialization map on $S$,
 and $\mc{F}$ be an object of $D^{\mr{b}}_{\mr{c}}(X_u)$.
 \begin{enumerate}
  \item Assume that $u$ is {\em algebraic}.
	There exists a modification $g\colon S'\rightarrow S$ such that, for any geometric point $u'$ of $S'$ with $g(u')=u$,
	the complex $\mc{F}$ considered as an object of $D^+(X_{u'})$ is $f_{S'}$-good.

  \item If $\mc{F}$ is $f$-good and belongs to $D^+_{\mr{c}}(X_u)$ {\normalfont(}resp.\ to $D_{\mr{ctf}}(X_u)${\normalfont)},
	the complex $\Psi_{f,\phi}(\mc{F})$ belongs to $D^+_{\mr{c}}(X_t)$ {\normalfont(}resp.\ to $D_{\mr{ctf}}(X_t)${\normalfont)}.
 \end{enumerate}
\end{cor*}
\begin{proof}
 Let us show the first claim.
 By \cite[Lem 4.3]{O1}, we may assume that $u$ is a geometric {\em generic} point.
 Since $\mc{F}$ is bounded, we may assume that $\mc{F}$ is a $\Lambda$-module.
 Since $u$ is algebraic and $\mc{F}$ is bounded and constructible,
 there exists a finite surjective morphism $S_0\rightarrow S$ and a constructible $\Lambda$-module $\mc{G}$
 on $X_{S_0}$ whose pullback $\mc{G}_u$ coincides with $\mc{F}$.
 There exists a modification $S'_0\rightarrow S_0$ over which $\mc{G}$ becomes good by Orgogozo's theorem \cite[Thm 2.1]{O1}.
 By \cite[Lem 3.2]{O1}, we can take a following diagram
 \begin{equation*}
  \xymatrix@C=50pt{
   T\ar[r]\ar[d]_{\mr{fin\ surj}}&S'_0\ar[d]^{\mr{alt}}\\
  S'\ar[r]^-{\mr{modif}}&S.}
 \end{equation*}
 Then $S'$ is what we need.
 Indeed, any specialization map $u\rightsquigarrow s'$ on $S'$ can be lifted to that on $T$ since the morphism $T\rightarrow S'$ is finite surjective,
 and we have $\Psi_{f_T,s'\leftarrow u}\cong\Psi_{f_{S'},s'\leftarrow u}$ since the morphism $T\rightarrow S'$ is finite.

 Let us show the second claim.
 We may assume that $u$ is a geometric generic point.
 Even if $u$ is not algebraic, we can still take a proper surjection $S_0\rightarrow S$ and $\mc{G}$ as above.
 By \cite[Thm 8.1]{O1}, we can take a modification $T\rightarrow S_0$ such that $\mc{G}_T$ is $f_T$-good and
 $\mr{R}\Psi_{f_T}(\mc{G}_T)\in D^{\mr{b}}(X_T\overleftarrow{\times}_TT)$ is constructible.
 Since $T\rightarrow S$ is proper surjective, we can take a lifting of $\phi$ on $T$.
 Using this lifting, $\phi$ can be viewed as a point of the topos $T\overleftarrow{\times}_TT$,
 and yields a morphism $\sigma_{\phi}\colon X_t\rightarrow X_T\times_T(T\overleftarrow{\times}_T T)\cong X_T\overleftarrow{\times}_TT$,
 where the product is taken in the category of topoi.
 Since $\mc{F}$ is assumed $f$-good, $\sigma_{\phi}^*\mr{R}\Psi_{f_T}(\mc{G}_T)$ is isomorphic to $\Psi_{f,\phi}(\mc{F})$,
 and the constructibility follows.
 The claim for the finite Tor-dimensionality follows by Remark \ref{insensseprem}.\ref{insensseprem-1}.
\end{proof}

\subsection{}
\label{Gabbthm}
A certain transitivity of nearby cycle functor is a key to Saito's construction of characteristic cycle,
and this appears rather implicitly in \cite[Prop 2.8]{S}.
This is also the case for our construction, and we use the following properties of nearby cycle functor due to Gabber:

\begin{thm*}[Gabber]
 Let $f\colon X\rightarrow S$ be a morphism of finite type between noetherian schemes,
 $\eta$ be a geometric point of $S$, and $\mc{F}\in D^+(X_{\eta})$.
 Let $\phi\colon\eta\rightsquigarrow\xi$, $\psi\colon\xi\rightsquigarrow s$ be specialization maps.
 If $\mc{F}$ is $f$-good, the following hold.
 \begin{enumerate}
  \item The complex $\Psi_{f,\phi}(\mc{F})$ in $D^+(X_\xi)$ is $f$-good.

  \item The canonical homomorphism $\Psi_{f,\psi\circ\phi}(\mc{F})\rightarrow(\Psi_{f,\psi}\circ\Psi_{f,\phi})(\mc{F})$ is an isomorphism.
 \end{enumerate}
\end{thm*}

\subsection{}
\label{dfngooddesce}
Let $S$ be a scheme.
Let $s\in S$ be a point.
We often denote by $\overline{s}$ an algebraic geometric point over $s$.
Now, we put $\mr{Cons}(S):=D_{\mr{ctf}}(S,\Lambda)$.
Let $\mr{Cons}_d(S)$ be the full subcategory of $D_{\mr{ctf}}(S,\Lambda)$ consisting of complexes $\mc{F}$
such that $\mr{Supp}(\mc{H}^i\mc{F})$ is of dimension $\leq d$ for any $i$.
Since the pullback is exact and by considering the localization sequence, the morphism $\mr{K}_0\mr{Cons}_d(S)\rightarrow\mr{K}_0\mr{Cons}(S)$ is injective.

Now, let $u\rightsquigarrow s$ be a specialization map on $S$.
Since $\Psi_{s\leftarrow u}$ is a functor $\mr{Cons}(X_u)\rightarrow D^{\mr{b}}(X_s)$
and may not factor through $\mr{Cons}(X_s)$, it does not induce a homomorphism
$\mr{K}_0\mr{Cons}(X_u)\rightarrow\mr{K}_0\mr{Cons}(X_s)$.
However, we have the following:

\begin{lem*}
 \label{basicstratcons}
 Let $X\rightarrow S$ be a morphism of finite type between noetherian schemes, and let $u\rightsquigarrow s$ be a specialization map of geometric points of $S$.
 Let $\mc{F}_+$, $\mc{F}_-$, $\mc{G}_+$, $\mc{G}_-$ be good objects in $\mr{Cons}(X_u)$ over $S$.
 Assume $[\mc{F}_+]-[\mc{F}_-]=[\mc{G}_+]-[\mc{G}_-]$ in $\mr{K}_0\mr{Cons}(X_u)$.
 Then for any specialization map $u\rightsquigarrow s$, we have the following in $\mr{K}_0\mr{Cons}(X_s)$
 \begin{equation*}
  [\Psi_{s\leftarrow u}(\mc{F}_+)]-[\Psi_{s\leftarrow u}(\mc{F}_-)]
   =
   [\Psi_{s\leftarrow u}(\mc{G}_+)]-[\Psi_{s\leftarrow u}(\mc{G}_-)].
 \end{equation*}
\end{lem*}
\begin{proof}
 By replacing $S$ with the closure of the image of $u$ in $S$, we may assume that $u$ is a geometric generic point of $S$.
 By definition, there exist finite sets $K$, $K'$ and exact triangles
 $\{\mc{H}^{(k)}_0\rightarrow\mc{H}^{(k)}_1\rightarrow\mc{H}^{(k)}_2\}_{k\in K\sqcup K'}$ in $\mr{Cons}(X_u)$ such that
 \begin{equation*}
  [\mc{F}_+]-[\mc{G}_+]-[\mc{F}_-]+[\mc{G}_-]=
   \sum_{k\in K}\bigl([\mc{H}_1^{(k)}]-[\mc{H}_0^{(k)}]-[\mc{H}_2^{(k)}]\bigr)
   -
   \sum_{k\in K'}\bigl([\mc{H}_1^{(k)}]-[\mc{H}_0^{(k)}]-[\mc{H}_2^{(k)}]\bigr)
 \end{equation*}
 in the free group generated by objects of $\mr{Cons}(X_u)$.
 We can take a modification $g\colon S'\rightarrow S$ over which $\mc{H}^{(k)}_i$ are good for any $i=0,1,2$ and $k\in K\sqcup K'$.
 Because $\mc{H}^{(k)}_i$ are good, $\Psi_{s'\leftarrow u}(\mc{H}_i^{(k)})$
 are constructible for any specialization map $u\rightsquigarrow s'$ on $S'$.
 Thus, we can check the equality in $\mr{K}_0\mr{Cons}(X_{s'})$ for the specialization map $u\rightsquigarrow s'$ easily.
 We may find a specialization $u\rightsquigarrow s'$ in $S'$ such that $s'$ sits over $s$ and the extension $k(s')=k(s)$.
 Because $\mc{F}_{\pm}$ is good, $\Psi_{s\leftarrow u}(\mc{F}_{\pm})=\Psi_{s'\leftarrow u}(g^*\mc{F}_{\pm})$
 in $\mr{Cons}(X_{s'})=\mr{Cons}(X_s)$, and similarly for $\mc{G}_{\pm}$.
 Thus, the claim follows.
\end{proof}

\begin{dfn*}
 Let $h\colon X\rightarrow S$ be a morphism of finite type between noetherian schemes,
 and let $u$ be a geometric point of $S$ over $s\in S$.
 An element $F\in\mr{K}_0\mr{Cons}({X_u})$ is said to be {\em $h$-good}
 if there exists a presentation $F=\sum n_i[\mc{F}_i]$ such that $\mc{F}_i$ is $h$-good for any $i$.
 The subgroup of $h$-good elements of $\mr{K}_0\mr{Cons}(X_u)$ is denoted by $\mr{K}_0\mr{Cons}^{\mr{gd}}(X_u,h)$.
\end{dfn*}
Let $F$ be an $h$-good element.
The above lemma implies that if we are given a specialization $u\rightsquigarrow t$, 
the element $\sum n_i[\Psi_{t\leftarrow u}(\mc{F}_i)]$ does not depend on the choice of the presentations.
We denote this by $\Psi_{t\leftarrow u}(F)$, which belongs to $\mr{K}_0\mr{Cons}^{\mr{gd}}(X_t,h)$.
Namely, we have a homomorphism $\Psi_{t\leftarrow u}\colon\mr{K}_0\mr{Cons}^{\mr{gd}}(X_u,h)\rightarrow\mr{K}_0\mr{Cons}^{\mr{gd}}(X_t,h)$.

\subsection{}
We now introduce one of the central objects of this paper:

\begin{dfn*}
 Let $S^0$ be the set of generic points of $S$,
 and fix {\em algebraic} geometric points $\{\overline{\eta}\}_{\eta\in S^0}$.
 A collection $F:=\{F_\eta\}_{\eta\in S^0}$ in $\prod_{\eta\in S^0}\mr{K}_0\mr{Cons}^{\mr{gd}}(X_{\overline{\eta}},h)$
 is said to be a {\em flat \'{e}tale system} if for any algebraic geometric point $\overline{s}\rightarrow S$,
 and specialization maps $\overline{\eta}\rightsquigarrow\overline{s}$,
 $\overline{\eta}'\rightsquigarrow\overline{s}$ such that $\eta,\eta'\in S^0$, we have an equality
 $\Psi_{\overline{s}\leftarrow\overline{\eta}}(F_{\eta})=\Psi_{\overline{s}\leftarrow\overline{\eta}'}(F_{\eta'})$.
 The set of flat \'{e}tale systems forms a subgroup, and this group is denoted by $\Comp(X/S,\{\overline{\eta}\})$.
 We also define the subgroup of {\em flat \'{e}tale system of dimension $\leq d$} by
 \begin{equation*}
  \Comp_d(X/S,\{\overline{\eta}\})=
   \prod_{\eta\in S^0}\mr{K}_0\mr{Cons}_d(X_{\overline{\eta}})\cap\Comp(X/S,\{\overline{\eta}\}),
 \end{equation*}
 where the intersection is taken in $\prod_{\eta}\mr{K}_0\mr{Cons}(X_{\overline{\eta}})$.
\end{dfn*}

\begin{lem*}
 Let $\star\in\{\emptyset,d\}$, and assume we are given sets of algebraic geometric points
 $E=\{\overline{\eta}\}_{\eta\in S^0}$ and $E'=\{\overline{\eta}'\}_{\eta\in S^0}$.
 Then there exists a canonical isomorphism $p_{E',E}\colon\Comp_{\star}(X/S,E)\rightarrow\Comp_{\star}(X/S,E')$ such that if we are given another set $E''$, we have
 $\epsilon_{E'',E'}\circ\epsilon_{E',E}=\epsilon_{E'',E}$ and $\epsilon_{E,E}=\mr{id}$.
 Thus, $\Comp_{\star}(X/S,E)$ does not depend on the choice of $E$, and we denote this simply by $\Comp_{\star}(X/S)$.
\end{lem*}
\begin{proof}
 A morphism $E'\rightarrow E$ is a collection of $S$-isomorphisms $\overline{\eta}'\rightarrow\overline{\eta}$.
 For a morphism $\phi=\{\phi_\eta\}\colon E'\rightarrow E$, we may consider the pullback $\phi^*\colon\Comp_{\star}(X/S,E)\rightarrow\Comp_{\star}(X/S,E')$.
 Assume we are given another morphism $\phi'=\{\phi'_\eta\}$.
 Then we can find $\sigma_{\eta}\colon\overline{\eta}\rightarrow\overline{\eta}$ over $\eta$ such that $\phi'_{\eta}=\sigma_{\eta}\circ\phi_{\eta}$ for each $\eta\in S^0$.
 Let $F=\{F_\eta\}$ be an element of $\Comp_{\star}(X/S,E)$.
 We have $\phi^*_{\eta}(F_\eta)=\phi'^*_\eta(\sigma_{\eta*}F_{\eta})=\phi'^*_\eta(\Psi_{\sigma}(F_{\eta}))=\phi'^*_\eta(F_\eta)$,
 and we are allowed to define $\epsilon_{E',E}:=\phi^*$.
 It is easy to check the required properties of $\epsilon$.
\end{proof}

\begin{ex*}
 \label{tensorofflet}
 Let $K$ be a field and $X\rightarrow\mr{Spec}(K)$ be a morphism of finite type.
 For any object $\mc{L}$ of $\mr{Cons}(X)$, the pullback $[\mc{L}_{\overline{K}}]$ defines an element of $\Comp(X/K)$.
 Moreover, the tensor product on $\mr{K}_0\mr{Cons}(X_{\overline{K}})$ induces a monoidal structure on $\Comp(X/K)$.
 Thus, if we are given $\mc{F}\in\Comp(X/K)$ and $\mc{L}\in\mr{Cons}(X)$, we have $[\mc{F}\otimes\mc{L}_{\overline{K}}]$ in $\Comp(X/K)$.
 This element is denoted by $\mc{F}\otimes\mc{L}$.
\end{ex*}

Assume we are given a specialization map $\eta\rightsquigarrow s$ on $S$.
To ease the notation, from now on, we sometimes abbreviate $\Psi_{\overline{s}\leftarrow\overline{\eta}}$ by $\Psi_{s\leftarrow \eta}$.

\begin{lem}
 \label{neadimboun}
 Let $X\rightarrow S$ be a morphism of finite type between noetherian schemes.
 \begin{enumerate}
  \item  Let $g\colon S'\rightarrow S$ be a morphism between noetherian scheme,
	 and $F$ be a flat \'{e}tale system on $X$ over $S$ of dimension $\leq d$.
	 For each $\eta'\in S'^0$, we can take a specialization map $\eta\rightsquigarrow g(\eta')$ where $\eta$ is some generic point of $S^0$,
	 and we may consider
	 $\bigl\{\Psi_{g(\eta')\leftarrow\eta}(F)\bigr\}_{\eta'\in S'^0}\in\prod_{\eta'\in S'^0}\mr{K}_0(\mr{Cons}(X_{S',\overline{\eta}'}))$.
	 Then this does not depend on the choice of the specialization map and it belongs to $\Comp_d(X_{S'}/S')$.
	 This flat \'{e}tale system is called the {\em pullback} of $F$ by $g$, and denoted by $g^*(F)$.
	 The pullback is transitive.

  \item Assume that $S$ is {\em excellent}.
	For each generic point $\eta\in S^0$, assume we are given $F_\eta\in\Comp_d(X_\eta/\eta)$.
	Then there exists a modification $S'\rightarrow S$ and $\{G_\xi\}_{\xi\in S'^0}\in\Comp_d(X_{S'}/S')$ such that $F_\xi=G_\xi$
	in $\mr{K}_0\mr{Cons}(X_{\overline{\xi}})$ for any $\xi\in S'^0\xrightarrow{\sim}S^0$.
 \end{enumerate}
\end{lem}
\begin{proof}
 Let us show the first claim.
 First, $\Psi_{g(\eta')\leftarrow\eta}(F)$ is good by Theorem \ref{Gabbthm}.
 The compatibility condition can be checked by the definition of goodness,
 and the last transitivity follows by Theorem \ref{Gabbthm}.
 It remains to show that it is of dimension $\leq d$.
 For a specialization map $\eta\rightsquigarrow g(\eta')$ as in the statement,
 we may find a discrete valuation ring $R$ with residue field $k$ and a commutative diagram
 \begin{equation*}
  \xymatrix{
   v:=\mr{Spec}(k)\ar[r]\ar@{^{(}->}[d]&S'\ar[d]\\
  V:=\mr{Spec}(R)\ar[r]&S
   }
 \end{equation*}
 such that $v$ is sent to $\eta'$, and the generic point of $\mr{Spec}(R)$ is sent to $\eta$
 ({\it e.g.}\ blowup $S$ at the closure of $g(\eta')$,
 and take an extension of discrete valuation ring corresponding to the generic point of the exceptional divisor).
 Let $\eta_V$ be the generic point of $V$.
 Since $F$ is good, it is compatible with base change, and it suffices to show the claim for $S'=v$, $S=V$.
 Now, we may assume that $F=[\mc{F}]$ for $\mc{F}\in\mr{Cons}_d(X_{\overline{\eta}})$.
 Since $\mr{Supp}(\Psi_{v\leftarrow\eta_V}(\mc{F}))\subset(\overline{\mr{Supp}(\mc{F})})_v$,
 and there exists a scheme structure on $\overline{\mr{Supp}(\mc{F})}$ which is flat over $R$ by [EGA IV, 2.8.5],
 the claim follows.

 Let us show the second claim.
 By considering the modification $\coprod_{\eta\in S^0}\overline{S}_\eta\rightarrow S$,
 we may assume $S$ is irreducible.
 Write $F_\eta=\sum n_i[\mc{F}_i]$ where $\eta$ is the generic point of $S$.
 By modifying $S$, we may assume that $\mc{F}_i$ is good over $S$ for any $i$ by Corollary \ref{vgofnbcycgen}.
 Further modifying $S$, we may assume that $S$ is normal since $S$ is assumed to be excellent,
 in which case $F_\eta$ is a flat \'{e}tale system.
 Indeed, let $\overline{s}\rightarrow S$ be a geometric point over a point $s\in S$, and assume we are given specialization maps
 $\phi,\phi'\colon\overline{\eta}\rightsquigarrow\overline{s}$.
 We wish to show that the nearby cycles along $\phi$ and $\phi'$ are the same.
 Since \'{e}tale topos is insensitive to inseparable extensions, we may assume that $\overline{\eta}$ is a spectrum of an algebraically closed field.
 Now, giving $\phi$, $\phi'$ is equivalent to giving morphisms $\widetilde{\phi},\widetilde{\phi}'\colon\overline{\eta}\rightarrow S_{(\overline{s})}$.
 Since $S$ is geometrically unibranch, $S_{(\overline{s})}$ is irreducible.
 Let $\xi_s$ be its generic point.
 Then $\widetilde{\phi}$, $\widetilde{\phi}'$ induce $\psi,\psi'\colon\overline{\eta}\rightarrow\xi_s$.
 By uniqueness of algebraic closure, we may find $\sigma$ rendering the following diagram commutative:
 \begin{equation*}
  \xymatrix@C=50pt@R=5pt{
   \overline{\eta}\ar@/^8pt/[rrd]^(.7){\widetilde{\phi}}\ar[dd]_\sigma\ar[rd]_(.3){\psi}&&\\
  &\xi_s\ar[r]&S_{(\overline{s})}.\\
  \overline{\eta}\ar@/_8pt/[urr]_(.7){\widetilde{\phi}'}\ar[ur]^(.3){\psi'}&
   }
 \end{equation*}
 Since $F_\eta$ is a flat \'{e}tale system over $\eta$,
 this implies that $\Psi_{\phi}(F)=\Psi_{\phi'\circ\sigma}(F)=\Psi_{\phi'}(\sigma_*F)=\Psi_{\phi'}(F)$
 (where $\sigma_*:=(\sigma^*)^{-1}$), and $F_\eta$ is a flat \'{e}tale system over $S$ as required.
\end{proof}

\begin{dfn*}
 Let $X\rightarrow S$ be a morphism of finite type between noetherian schemes.
 By the lemma above, the assignment $\Comp_d(X_T/T)$ to $T\in\mr{Sch}_{/S}$ defines a presheaf on $\mr{Sch}_{/S}$.
 This presheaf is denoted by $\Comp_{d,X/S}$, and $\Comp_{X/S}=\bigcup_{d\geq0}\Comp_{d,X/S}$.
\end{dfn*}

\begin{rem*}
 Let $\mr{geom}(S)$ be the set of isomorphism classes of geometric points of $S$.
 Let $\Comp_{\mr{TS}}(X/S)$ be the subgroup of $\prod_{s\in\mr{geom}(S)}\mr{K}_0\mr{Cons}^{\mr{gd}}(X_s,h)$
 consisting of elements $\{F_s\}$ such that for any specialization map $t\rightsquigarrow s$,
 we have $\Psi_{s\leftarrow t}(F_t)=F_s$.
 In view of the lemmas above, the restriction homomorphism $\Comp_{\mr{TS}}(X/S)\rightarrow\Comp(X/S)$ is an isomorphism.
 The description of $\Comp_{\mr{TS}}$ is closer to the spirit of the flat function in \cite[Def 2.1]{S},
 and illustrates the terminology ``flat \'{e}tale system'' more clearly.
 On the other hand, $\Comp$ is closer to the spirit of the relative cycles in \cite{SV}, and has an advantage in its brevity for constructing elements,
 which is why we employed $\Comp$ rather than $\Comp_{\mr{TS}}$.
\end{rem*}

\subsection{}
\label{dfnofrk0assmap}
Let $K$ be a field of characteristic $p>0$.
Since $\Lambda$ is assumed to a local ring, we have the unique homomorphism
$\dim_{\Lambda}\colon\mr{K}_0\mr{Cons}\bigl(\mr{Spec}(\overline{K})\bigr)\xrightarrow{\sim}\mb{Z}$,
called the {\em dimension homomorphism} (or {\em rank homomorphism}), such that $[\Lambda]$ is sent to $1$.
We sometimes denote $\dim_\Lambda$ by $\dim$.
Let $Y\rightarrow\mr{Spec}(K)$ be a morphism of finite type.
We define a homomorphism
\begin{equation*}
 \rk_Y\colon
  \Comp_0(Y/K)
  \rightarrow Z_0(Y)[1/p];\quad
  \mc{G}\mapsto
  \sum_{s\in|Y|}
  \frac{\dim_{\Lambda}(\mc{G}_{\overline{s}})}{[k(s):K]_{\mr{insep}}}\cdot[s],
\end{equation*}
where $Z_0$ is the (free) abelian group of $0$-cycles, $|Y|$ is the set of closed points of $Y$, $[:]_{\mr{insep}}$ denotes the inseparable degree of the extension of fields,
and $\overline{s}$ is any geometric point of $Y_{\overline{K}}$ lying above $s$.
Since $\mc{G}$ is a flat \'{e}tale system, $\dim_{\Lambda}(\mc{G}_{\overline{s}})$ does not depend on the choice of $\overline{s}$,
and the homomorphism is well-defined.

Let $X\rightarrow S$ be a morphism of noetherian schemes of finite type, and $n\geq0$ is an integer.
Suslin-Voevodsky defined an abelian group $\Gamma(S,z(X/S,n))$ in \cite[\S3.3]{SV}.
In this paper, we denote $\Gamma(S,z(X/S,n))$ simply by $z(X/S,n)$,
and $z(X/S,n)$ does {\em not} mean the presheaf contrary to the notation of \cite{SV}.

\begin{lem*}
 \label{compdim0nea}
 Let $h\colon X\rightarrow S$ be a morphism of finite type between noetherian $k$-schemes, and $p:=\mr{char}(k)>0$.
 The composite
 \begin{equation*}
  \rk_h\colon\Comp_0(X/S)
   \rightarrow
   \prod_{\eta\in S^0}\Comp_0(X_{\eta}/\eta)
   \xrightarrow{\prod \rk_{X_\eta}}
 \prod_{\eta\in S^0} Z_0(X_\eta)[1/p]
 \end{equation*}
 factors through $z(X/S,0)[1/p]$, and defines a homomorphism $\Comp_0(X/S)\rightarrow z(X/S,0)[1/p]$.
 Moreover, if we are given a morphism $g\colon S'\rightarrow S$ between noetherian schemes,
 the following diagram of abelian groups commutes, where $\mr{cycl}(g)$ denotes the pullback homomorphism for the relative cycle group:
 \begin{equation*}
  \xymatrix@C=50pt{
   \Comp_0(X/S)
   \ar[r]^-{\rk_h}\ar[d]_{g^*}&
   z(X/S,0)[1/p]\ar[d]^{\mr{cycl}(g)}\\
  \Comp_0(X_{S'}/S')\ar[r]^-{\rk_{h'}}&
   z(X_{S'}/S',0)[1/p].
   }
 \end{equation*}
\end{lem*}

\begin{proof}
 Let $g\colon\mr{Spec}(L)\rightarrow\mr{Spec}(K)$ be an extension of fields, and $f\colon Y\rightarrow Y'$ be a proper morphism of $K$-schemes of finite type.
 Then the following diagrams commute:
 \begin{equation*}
  \xymatrix{
   \Comp_0(Y/K)\ar[r]^-{\rk_{Y}}\ar[d]_{g^*}\ar@{}[rd]|{\mbox{A}}&
   Z_0(Y)[1/p]\ar@{^{(}->}[d]^{\otimes_KL}\\
  \Comp_0(Y_{L}/L)\ar[r]^-{\rk_{Y_L}}&
   Z_0(Y_L)[1/p],
   }\qquad
   \xymatrix{
   \Comp_0(Y/K)\ar[r]^-{\rk_{Y}}\ar[d]_{f_*}\ar@{}[rd]|{\mbox{B}}&
   Z_0(Y)[1/p]\ar[d]^{f_*}\\
  \Comp_0(Y'/K)\ar[r]^-{\rk_{Y_L}}&
   Z_0(Y')[1/p].
   }
 \end{equation*}
 The verification is straightforward, and we only note that the denominator $[k(s):K]_{\mr{insep}}$ in the definition of $\rk_Y$ plays a crucial role.

 When $S$ is a trait ({\it i.e.}\ the spectrum of a discrete valuation ring), by \cite[3.2.6, 3.3.15]{SV},
 we have $Z_0(X_\eta)\cong z(X/S,0)$, and $d_h$ is defined automatically.
 Let us check the compatibility with pullback in the case where $g$ is the closed immersion from the closed point $s$ of $S$.
 By using the compatibility A above, we may assume $S$ to be a henselian trait.
 If we are given a proper morphism $f\colon X\rightarrow X'$ over $S$, recall that the pushforward $f_*\colon z(X/S,0)\rightarrow z(X'/S,0)$ is defined in \cite[3.6.3]{SV}
 by restricting $f_*\colon Z_0(X_\eta)\rightarrow Z_0(X'_\eta)$.
 Combining with the compatibility B above, we have the commutativity of $f_*$ and $\rk_{Y^{(\prime)}}$.
 Now, the group $\Comp_0(X/S)$ is generated by elements of the form $[\rho^*\Lambda_x]$ where $x\in X_\eta$ is a closed point and $\rho\colon X_{\overline{\eta}}\rightarrow X_\eta$,
 and it suffices to check the equality $\mr{cycl}(g)d_h([\rho^*\Lambda_x])=d_{h_s}g^*([\Lambda_x])$.
 By using the compatibility of $f_*$ for a closed immersion $f$ and $\rk_{-}$, we may replace $X$ with the closure of $x$ in $X$.
 Especially, we may assume that $X$ is irreducible and $h$ is quasi-finite.
 Since $S$ is assumed to be henselian, $h$ is either finite or $X_s=\emptyset$ by [EGA IV, 18.12.3].
 There is nothing to prove in the case $X_s=\emptyset$, so we may assume $h$ to be finite.
 Since $X_s$ consists of a single point, $h_{s*}\colon Z_0(X_s)\rightarrow Z_0(s)$ is injective.
 Thus, it suffices to show the compatibility after taking $h_*$, and we are reduced to the case where $h$ is the identity.
 This case is evident.

 Finally, let us treat the general case.
 For an element $\{z_\eta\}\in\prod Z_0(X_\eta)[1/p]$ to lies in $\mr{Cycl}(X/S,0)[1/p]$, we must check the following by definition \cite[3.1.3]{SV}:
 Let $\psi\colon\mr{Spec}(\Omega)\rightarrow S$ be a morphism from a field, and $\phi\colon\mr{Spec}(R)\rightarrow S$ be a morphism from a trait such that
 the generic point of $\mr{Spec}(R)$ lies over $\eta\in S^0$ and the morphism from the closed point to $S$ factors $\psi$.
 Then $i^*(z_\eta\otimes\mr{Frac}(R))\otimes\Omega$ depends only on $\psi$ and not on $\phi$,
 where $z_\eta\otimes\mr{Frac}(R)$ is viewed as an element of $z(X_R/R,0)$ and $i$ is the closed immersion from the closed point of $\mr{Spec}(R)$.
 Thus, by the case we have already treated, $\rk_h(F)$ belongs to $\mr{Cycl}(X/S,0)[1/p]$ for $F\in\Comp_0(X/S)$.
 To conclude, we claim that $\mr{Cycl}(X/S,0)[1/p]=z(X/S,0)[1/p]$.
 Indeed, the proof of \cite[3.3.14]{SV} in fact shows that the abelian presheaf $\mr{Cycl}(X/S,r)/z(X/S,r)$ is $p$-torsion.
\end{proof}

\subsection{}
\label{ASsheafrecal}
Let us recall the Artin-Schreier $\Lambda$-module.
Let $Y:=\mr{Spec}(\mb{F}_p[T])\rightarrow\mr{Spec}(\mb{F}_p[t])$ defined by sending $t$ to $T^p-T$.
This is a Galois covering with Galois group canonically isomorphic to $\mb{Z}/p\mb{Z}$.
Let $\psi\colon\mb{Z}/p\mb{Z}\rightarrow\Lambda^{\times}$ be a character.
Then, we get the representation
\begin{equation*}
 \pi_1(\mb{A}^1_{\mb{F}_p})\rightarrow
  \mr{Gal}(Y/\mb{A}^1)\cong
  \mb{Z}/p\mb{Z}\xrightarrow{\psi}\Lambda^{\times}.
\end{equation*}
The associated smooth sheaf of rank $1$ is denoted by $\mc{L}_{\psi}$.
Let $X$ be a scheme over $\mb{F}_p$, and assume we are given a morphism $f\colon X\rightarrow\mb{A}^1_{\mb{F}_p}$.
Then the pullback $f^*\mc{L}_\psi$ is denoted by $\mc{L}_{\psi}(f)$.

\begin{rem*}
 For $d>0$, we do not have the homomorphism $\Comp_d(Y/T)\rightarrow z(Y/T,d)$.
 For example, consider $T=\mb{A}_t^1$, $Y=\mb{A}^1_{T,x}$, and $F=[\mc{L}_{\psi}(x/t)]$.
 Then the generic rank is $1$ but the rank of $\Psi_{0\leftarrow\eta}(F)$ is $0$.
 In a similar vein, $\mr{Supp}(F)$ is not necessarily in $z(Y/T,d)$ when $d>0$.
 For this, consider $T=\mr{Spec}(k[t,s]/(ts))$, $Y=\mb{A}^1_{T,x}$.
 If we put $F=[\mc{L}_{\psi}(x/t)]$, this becomes a flat \'{e}tale system because $\Psi_{0\leftarrow\eta_t}(F)=0$,
 where $\eta_t$ is the generic point of $T$ corresponding to the irreducible component $\mr{Spec}(k[t])$.
\end{rem*}

\section{Construction of deformation system}
\label{sect5}
In this section, we construct a deformation system, and construct our characteristic cycle.
In the course of construction, we use one result whose proof will be postponed to Part II of this paper.
We keep the notations from \S\ref{sect4}.
Throughout this section, we fix an excellent separated $k$-scheme $S$ of finite Krull dimension.
We also fix a non-trivial character $\psi\colon\mb{Z}/p\mb{Z}\rightarrow\Lambda^{\times}$,
and the locally constant sheaf $\mc{L}_\psi(f)$ introduced in \ref{ASsheafrecal} is denoted simply by $\mc{L}(f)$.

\subsection{}
Assume we are given a morphism $X\rightarrow S$ of finite type.
We fix this morphism, and we will construct a deformation system associated to it.
We fix a finite open covering $\{X_i\}_{1\leq i\leq n}$ of $X$ such that $X_i\rightarrow S$ is affine.
Since $S$ is assumed to be separated, such covering always exists.

\medskip

\noindent
{\bf \underline{Construction of $\mu$}}:
Let $\mc{O}_{X_i/S}$ be the presheaf on $\mr{Sch}_{/S}$ associating $\Gamma(X_i\times_S T,\mc{O}_{X_i\times_S T})$ to $T$.
We set the basepoint to be $0$, and consider it as a presheaf of pointed sets.
We put $\mu_{X/S}:=\bigvee_i\mc{O}_{X_i/S}$.
Here, $\bigvee$ is the coproduct in the category of pointed sets.
In other words, $\mu_{X/S}(T)=\coprod_{i}\mc{O}_{X_i/S}(T)/\sim$, where $0_i\sim 0_{i'}$ ($0_i\in\mc{O}_{X_i/S}(T)$) for any $1\leq i,i'\leq n$.

Let us construct a contraction morphism.
For $T$, take $f\in\mc{O}_{X_i/S}(T)$.
Let $s$ be the coordinate of $\mb{A}^1$.
We have $s\cdot f\in\Gamma(T\times\mb{A}^1,\mc{O}_{X_i/S})$, which defines a homomorphism $\mc{O}_{X_i/S}\rightarrow p_*p^*\mc{O}_{X_i/S}$.
The construction is functorial, the homomorphism preserves the base points, and the composition with restriction to $1\in\mb{A}^1$ is the identity.
Moreover, the composition with restriction to $0\in\mb{A}^1$ is the map sending $f\in\mu$ to $0\in\mu$.
Thus, we have a morphism of presheaves $c_\mu\colon\mu_{X/S}\rightarrow p_*p^*\mu_{X/S}$, defining a contraction morphism.

\subsection{}
We define a presheaf $\ms{S}$ equipped with filtration and endow a structure of deformation system on the graded pieces.
The presheaf $\ms{S}$ is easy to define:
\medskip

\noindent
{\bf \underline{Construction of $\ms{S}$}}:
We put $\ms{S}_{X/S}:=\Comp_{X/S}$ (cf.\ Definition \ref{neadimboun}).
This presheaf is equipped with the filtration $\{\Comp_{d,X/S}\}_{d\geq0}$, but we define a finer filtration on $\ms{S}$.
For $d>0$, $n\geq i\geq 0$, we put $\ms{S}_{d,i}$ to be the following submodule:
\begin{equation*}
 \mr{Ker}\bigl(\Comp_{d,X/S}\rightarrow\mr{gr}_d\Comp_{U_i/S}\bigr),
\end{equation*}
where $U_i:=\bigcup_{j>i}X_j$.
In other words, $F=\{F_\eta\}_{\eta\in T^0}\in\Comp_{d,X/S}(T)$ belongs to
$\ms{S}_{d,i}(T)$ if $F_\eta|_{(U_i)_{\overline{\eta}}}\in\mr{K}_0\mr{Cons}_{d-1}((U_i)_{\overline{\eta}})$ for each $\eta\in T^0$.
We have $\ms{S}_{d,n}=\ms{S}_d$ and $\ms{S}_{d,0}=\ms{S}_{d-1}$.
The filtration can be depicted as
\begin{equation*}
 \ms{S}_0=
  \ms{S}_{1,0}
  \subset
  \ms{S}_{1,1}
  \subset\dots\subset
  \ms{S}_{1,n}=\ms{S}_1
  =\ms{S}_{2,0}
  \subset
  \ms{S}_{2,1}\subset
  \dots.
\end{equation*}
Even though the filtration appears to be double filtered, we can make it a filtration indexed by $\mb{N}$ by relabeling suitably.
It is exhaustive since if we take $d$ to be greater than the relative dimension of $X$ over $S$,
namely $\max\{\dim(X_s);s\in S\}$, then $\Comp_{d,X/S}=\Comp_{X/S}$.

\subsection{}
To define the homomorphism $\mc{E}$, we will use the following lemma:

\begin{lem*}
 \label{cdppartlem}
 Let $f\colon V\rightarrow T\times\Box$ be a modification of $S$-schemes of finite type.
 Then, there exists an open dense subscheme $U\subset T_{\mr{red}}$ such that $f_U\colon V\times_T U\rightarrow U\times\Box$ is an isomorphism.
\end{lem*}
\begin{proof}
 Shrinking $T$ and considering irreducible component-wise, we may assume that $T$ is integral.
 Since $S$ is excellent and of finite dimension, the set of normal points of $T$ is open, thus shrinking $T$ further,
 we may assume that $T$ is normal of dimension $d<\infty$.
 Let $W\subset T\times\Box$ be the set of points $x\in T\times\Box$ such that $\dim(f^{-1}(x))\geq 1$.
 Then $W$ is closed by [EGA IV, 13.1.3]. 
 By Zariski main theorem [EGA III, 4.4.9], $f$ is isomorphic outside of $W$.
 If $\dim(W)\geq d$, then by [EGA IV, 5.6.8], $\dim(f^{-1}(W))\geq d+1$, and contradict with the assumption that $V$ is irreducible.
 Thus, $W$ is of dimension $\leq d-1$.
 By [EGA IV, 5.6.8], $\overline{p(W)}$ in $T$, where $p\colon T\times\Box\rightarrow T$ is the projection, is of dimension $\leq d-1$ as well.
 Thus, by [EGA IV, 5.1.8], $U:=T\setminus\overline{p(W)}$ is an open dense subscheme which satisfies the condition.
\end{proof}

\subsection{}
\label{constEdel}
Let $T$ be an $S$-scheme, and take $F\in\ms{S}_{d,i}(T)$ for $i\geq1$ and $f\in\mc{O}_{X_i/S}(T)$.
Let $V\rightarrow T\times\Box$ be a morphism in $\mr{Sch}_{/S\times\Box}$, and denote by $t$ the coordinate of $\Box$.
Let $\pi\colon V\rightarrow T$ and $j_{\overline{\xi}}\colon (X_i)_{\overline{\xi}}\hookrightarrow X_{\overline{\xi}}$ for each $\xi\in V^0$.
Since $(j_{\overline{\xi}})_!$ is an exact functor, it induces a homomorphism between the Grothendieck groups, and we may consider
$j_{\overline{\xi}!}\,j_{\overline{\xi}}^*(\pi^*F\otimes\mc{L}(ft))_{\xi}\in\mr{K}_0\mr{Cons}(X_{\overline{\xi}})$ where $\xi\in V^0$.
In view of Example \ref{tensorofflet}, this element belongs to $\Comp(X_\xi/\xi)$.
Recall from Example \ref{basicsitediag} that $V_{\mr{fl}}$ is the union of irreducible components of $V$ whose generic point lies over $\eta\in\Box$.
We put
\begin{equation*}
 (F_V\otimes\mc{L}(ft))_{\mr{gen}}:=\Bigl\{j_{\overline{\xi}!}j_{\overline{\xi}}^*(\pi^*F\otimes\mc{L}(ft))_{\xi}\Bigr\}
  \in\prod_{\xi\in V_{\mr{fl}}^0}\Comp(X_\xi/\xi).
\end{equation*}
Of course, this element is not a flat \'{e}tale system in general.
We say that $V$ is {\em adapted to $(F,f)$} if there exists $C=\{C_\xi\}_{\xi\in V^0_{\mr{fl}}}\in\ms{S}_{d,i}(V_{\mr{fl}})$ such that
\begin{equation*}
 C_\xi|_{(X_i)_{\overline{\xi}}}
  =
  (F_V\otimes\mc{L}(ft))_\xi\,|_{(X_i)_{\overline{\xi}}}.
\end{equation*}
The flat \'{e}tale system $C$ is determined uniquely up to differences in $\ms{S}_{d,i-1}(V_{\mr{fl}})$.
Indeed, assume we are given another $C'$ satisfying the condition.
Then we have $\bigl(C_{\xi}-C'_{\xi}\bigr)|_{(X_i)_{\overline{\xi}}}=0$ by the condition.
On the other hand, since $C^{(\prime)}\in\ms{S}_{d,i}$,
the restriction $C^{(\prime)}_{\xi}|_{(U_i)_{\overline{\xi}}}$ belongs to $\mr{K}_0\mr{Cons}_{d-1}((U_i)_{\overline{\xi}})$.
Combining these, the restriction of $C_{\xi}-C'_{\xi}$ to $U_{i-1}:=\bigcup_{j\geq i}X_j=U_i\cup X_i$ belongs to
$\mr{K}_0\mr{Cons}_{d-1}((U_{i-1})_{\overline{\xi}})$, and $C-C'$ belongs to $\ms{S}_{d,i-1}$ as required.
We denote the flat \'{e}tale system $C$ in $\mr{gr}_{d,i}\ms{S}(V_{\mr{fl}})$
by $F[V]$, or more precisely, $(F;f)[V]$.
Let us denote this sieve generated by objects $V$ adapted to $(F;f)$ by $\mr{R}(F;f)$.

\begin{lem*}
 \begin{enumerate}
  \item The sieve $\mr{R}(F;f)$ is an $\eta$-cdp covering sieve of $T\times\Box$.

  \item\label{constEdel-2}
       Let $\phi\colon V'\rightarrow V$ be a morphism of $T\times\Box$-schemes.
       Assume that $V$ is adapted to $(F;f)$.
       Then $V'$ is adapted as well, and we have $\phi_{\mr{fl}}^*F[V]=F[V']$ in $\mr{gr}_{d,i}\ms{S}(V'_{\mr{fl}})$,
       where $\phi_{\mr{fl}}\colon V'_{\mr{fl}}\rightarrow V_{\mr{fl}}$ is the morphism induced by $\phi$.
 \end{enumerate}
\end{lem*}
\begin{proof}
 Let us show the first claim.
 We may assume that $T$ is reduced.
 It suffices to find an $\eta$-cdp covering $\{V_i\}\rightarrow T\times\Box$ such that $V_i\rightarrow T\times\Box$ is adapted to $(F;f)$.
 First, there exists a modification $V\rightarrow T\times\Box$ over which $(F_{T\times\Box}\otimes\mc{L}(ft))_{\mr{gen}}$
 becomes a flat \'{e}tale system by Lemma \ref{neadimboun}.
 By definition, this $V$ is adapted to $(F;f)$.
 By Lemma \ref{cdppartlem}, there exists an open dense subscheme $U_0\subset T$ such that $V\times_T U_0\rightarrow U_0\times\Box$ is an isomorphism.
 Put $T_1:= T\setminus U_0$, and we can take a modification $V_1$ of $T_1\times\Box$ over which
 $(F_{T_1\times\Box}\otimes\mc{L}(ft))_{\mr{gen}}$ becomes a flat \'{e}tale system.
 Using the lemma, we can take a closed subscheme over which $V_1\rightarrow T_1\times\Box$ is an isomorphism.
 We iteratively construct $T_i$, $V_i$, and this process terminates in finitely many steps since $T$ is noetherian.
 Then $\{V_i\rightarrow T\times\Box\}$ is an $\eta$-cdp covering.

 Let us show the second claim.
 We take a lifting of $F[V]$ in $\ms{S}_{d,i}$, and abusively denote by $F[V]$.
 By construction of $F[V']$, it suffices to show that
 $(F_{V'}\otimes\mc{L}(ft))_{\eta'}|_{(X_i)_{\overline{\eta}'}}=\phi_{\mr{fl}}^*(F[V]|_{(X_i)_{\overline{\eta}'}})$
 in $\mr{K}_0\mr{Cons}((X_i)_{\overline{\eta}'})$ for each $\eta'\in V'^0_{\mr{fl}}$.
 Thus, we may assume that $X=X_i$.
 In this situation, $f$ is a regular function.
 Take a generic point $\eta'$ of $V'_{\mr{fl}}$.
 By definition of $V'_{\mr{fl}}$, the image of $\eta'$ in $\Box$ is the generic point of $\Box$.
 Thus, we can find a specialization $\eta\rightsquigarrow\phi(\eta')$ from a generic point $\eta$ of $V_{\mr{fl}}$.
 The situation can be depicted as below.
 \begin{equation*}
  \xymatrix@R=15pt{
   V'_{\mr{fl}}\ar[d]_{\phi_{\mr{fl}}}&&\eta'\ar@{~>}[d]\\
  V_{\mr{fl}}&\eta\ar@{~>}[r]&\phi(\eta')
   }
 \end{equation*}
 By the choice, $\phi(\eta')$ sits over the generic point of $\Box$, and $f|_{V}t$ is regular around $\phi(\eta')$.
 Thus, $\mc{L}(f|_Vt)$ is smooth around $\phi(\eta')$, which implies
 \begin{equation*}
  \Psi_{\phi(\eta')\leftarrow\eta}\bigl(F_V\otimes\mc{L}(f|_{V}t)\bigr)
   =
   F_{\phi(\eta')}\otimes\mc{L}(f|_Vt).
 \end{equation*}
 By definition, this means that
 $(F_{V'}\otimes\mc{L}(ft))_{\eta'}|_{X_{\overline{\eta}'}}=\phi_{\mr{fl}}^*(F[V]|_{X_{\overline{\eta}'}})$.
\end{proof}

Now, we are ready to construct $\mc{E}$.
\medskip

\noindent
{\bf \underline{Construction of $\mc{E}$}}:
Assume we are given an element $F\in\ms{S}_{d,i}(T)$, and $f\in\mc{O}_{X_j/S}(T)$ for some $j$.
If $j\neq i$, we simply put $\mc{E}_{X/S}(F;f)=F$.
The map is non-trivial only in the case $i=j$.
In this case, the above lemma implies that $\{F[V]\}_{V\in\mr{R}(F;f)}$ restricted to $V_{\mr{fl}}\times_{\Box}\{\infty\}$ determines an element of
\begin{equation*}
 \invlim_{V\in\mr{R}(F;f)}\Gamma(V_{\mr{fl}}\times_{\Box}\{\infty\},\mr{gr}_{d,i}(\ms{S}))
  \cong
  \invlim_{V\in\mr{R}(F;f)}\Gamma(V_{\mr{fl}},i_{\infty*}\mr{gr}_{d,i}(\ms{S})).
\end{equation*}
Since $\mr{R}$ is a covering sieve of $T\times\Box$, this element determines an element of $\pi_0\partial\mr{gr}_{d,i}(\ms{S})(T)$.
The construction does not depend on the differences by $\ms{S}_{d,i-1}$, and it induces the morphism
$(\mr{gr}_{d,i}\ms{S}\otimes\mc{O}_{X_i/S})(T)\rightarrow\pi_0\partial\mr{gr}_{d,i}\ms{S}(T)$.
This element is defined to be $\mc{E}(F;f)$.
These morphisms define a morphism of presheaves again by the lemma.
Note that $\mc{E}(F;0)=F$ for $0\in\mc{O}_{X_i/S}(T)$, so the maps glue together to get
$\ms{S}_{d,i}\otimes\mu\rightarrow\pi_0\partial\mr{gr}_{d,i}\ms{S}$.
Composing with $\pi_0\partial\mr{gr}_{d,i}\ms{S}\rightarrow\partial\mr{gr}_{d,i}\ms{S}$, we define the morphism $\mc{E}$.
\medskip

Summing up, assume we are given a morphism $X\rightarrow S$ of finite type,
and a finite open covering $\{X_i\}_{1\leq i\leq n}$ of $X$ such that $X_i\rightarrow S$ is affine.
Then we have the deformation system $\mb{S}_{d,i}:=(\mr{gr}_{d,i}(\ms{S}_{X/S}),\mu_{X/S},\mc{E}_{X/S})$ over $S$ for $d>0$ and $n\geq i\geq 0$.
This system {\em does} depend on the choice of the open coverings,
but since merely the existence is important for us, we do not put this data in the notation.

\subsection{}
\label{microlocconstofE}
Let us give a variant of the previous construction.
Assume we are given a morphism $X\rightarrow S$ of finite type, but we do not need to take an open cover of $X$.
Take an $S$-scheme $T$, $F\in\Comp(X_T/T)$, $V$, and $f\in\mc{O}_{X/S}(T)$, as before.
We may consider $\{(\pi^*F\otimes\mc{L}(ft))_\eta\}_{\eta\in V^0_{\mr{fl}}}$ as well.
This time, since $f$ is defined everywhere, we have a homomorphism
\begin{equation*}
 \mc{E}'_{X/S}\colon\ms{S}_{X/S}\otimes\mc{O}_{X/S}\rightarrow\partial\ms{S}_{X/S}.
\end{equation*}
The triple $(\ms{S}_{X/S},\mc{O}_{X/S},\mc{E}'_{X/S})$ forms a deformation system on $S$.
Recall the Cartesian fibration $p\colon(X/S)_{\mr{zar}}\rightarrow\mr{Sch}_{/S}$ considered in \ref{suppstrdfn}.
Since the construction is functorial, the assignment $(\ms{S}_{U/T}(T),\mc{O}_{U/T}(T),\mc{E}'_{U/T}(T)))$ to $U/T\in(X/S)_{\mr{zar}}$ is promoted to a $p$-deformation system.

\subsection{}
We wish to show that the deformation system $\mb{S}_{d,i}$ is effaceable except for the case $(d,i)=(1,0)$.
For this, let us analyze the condition that $\mc{E}^{\infty}(F;\mbf{f})=0$ for $F\in\ms{S}_d(T)$ for an integer $d>0$.
First, we introduce some terminologies.
Fix $T\in\mr{Sch}_{/S}$.
For a scheme $T'$ over $T$, let $\Mdf_{T'}$ be the category of proper morphisms $\mbf{V}:=(V\rightarrow T'\times\Box)$.
For such $\mbf{V}$, we put $\mbf{V}_\infty:=V_{\mr{fl}}\times_{\Box}\{\infty\}$.
This is a $T'$-scheme.
An {\em $\Mdf$-sequence over $T$} ({\em of length $n\geq0$}) is a set $\{\mbf{V}_{i}\in\Mdf_{\mbf{V}_{i-1,\infty}}\}_{i=1,\dots,n}$
such that $\mbf{V}_{0}:=(T\times\Box\rightarrow T\times\Box)$.
The sequence is said to be {\em trivial} if $\mbf{V}_i=(T\times\Box\rightarrow T\times\Box)$ for any $i$.
A morphism of $\Mdf$-sequences $\{\mbf{V}_i\}\rightarrow\{\mbf{V}'_i\}$ is a collection of morphisms $\mbf{V}_i\rightarrow\mbf{V}'_i$
compatible with $\mbf{V}_{i-1,\infty}\rightarrow\mbf{V}'_{i-1,\infty}$ in an evident sense.
Similarly to the definition of $\mc{E}$, we define as follows:

\begin{dfn*}
 Assume that $X_T$ and $T$ are affine, and let $F\in\ms{S}_d(T)$, $\mbf{f}:=(f_1,\dots,f_n)\in\mc{O}_{X_T}(T)^{\infty}$.
 \begin{enumerate}
  \item  Let $\{\mbf{V}_i\}$ be an $\Mdf$-sequence of length $n$.
	 We define the notion of $\{\mbf{V}_i\}$ being {\em adapted to $(F;\mbf{f})$}, along with the element $F[\{\mbf{V}_i\}]_{\infty}\in\ms{S}(\mbf{V}_{n,\infty})$
	 (provided the condition holds), by induction on $n$ as follows:
	 The truncated sequence $\{\mbf{V}_i\}_{i<n}$ is adapted to $(F;\mbf{f})$
	 and $F[\{\mbf{V}_i\}_{i<n}]_{\infty}\otimes\mc{L}(f_nt)$ is a flat \'{e}tale system on $\mbf{V}_{n}$ (where we put $F[\emptyset]_{\infty}=F$ in the case $n=1$).
	 We then put $F[\{\mbf{V}_i\}]_{\infty}$ the restriction to $\mbf{V}_{n,\infty}$ of this flat \'{e}tale system.
	 
  \item  An $\Mdf$-sequence $\{\mbf{V}_i\}$ is said to be a {\em principal sequence}
	 if each morphism $\mbf{V}_i\colon V_i\rightarrow\mbf{V}_{i-1,\infty}\times\Box$ is surjective.
 \end{enumerate}
\end{dfn*}

We may detect the kernel of $\mc{E}^n$ as follows:

\begin{lem*}
 \label{baicpropmdfseq}
 Let $F\in\ms{S}_d(T)$, and $\mbf{f}:=(f_1,\dots,f_n)\in\mc{O}_{X_T}(T)^{\infty}$.
 \begin{enumerate}
  \item\label{baicpropmdfseq-1}
       Let $\phi=\{\phi_i\}\colon\{\mbf{V}'_i\}\rightarrow\{\mbf{V}_i\}$ be a morphism of $\Mdf$-sequences of length $n$ {\em adapted to $(F;\mbf{f})$}.
       Then $F[\{\mbf{V}'_i\}]_{\infty}=\phi^*_nF[\{\mbf{V}_i\}]_{\infty}$.

  \item\label{baicpropmdfseq-2}
       Let $\{\mbf{V}_i\}$ be a principal $\Mdf$-sequence of length $n$, and $\{\mbf{W}_i\}$ be any $\Mdf$-sequence of length $n$.
       Then, we can find a diagram $\{\mbf{V}_i\}\leftarrow\{\mbf{V}'_i\}\rightarrow\{\mbf{W}_i\}$ such that
       $\mbf{V}'_{i,\infty}\rightarrow\mbf{W}_{i,\infty}$ is surjective for any $i\leq n$.

  \item A principal $\Mdf$-sequence adapted to $(F;\mbf{f})$ exists.

  \item\label{baicpropmdfseq-4}
       If there exists a principal sequence $\{\mbf{V}_i\}$ adapted to $(F;\mbf{f})$
       such that $F[\{\mbf{V}_i\}]_{\infty}$ is in $\ms{S}_{d-1}$, then $\mc{E}^n(F;\mbf{f})=0$.
 \end{enumerate}
\end{lem*}
\begin{proof}
 The first claim follows by the functoriality of flat \'{e}tale systems, and can be checked as in Lemma \ref{constEdel}.\ref{constEdel-2}.
 Let us show the second claim. We use the induction on the length. For $n=0$, the claim is obvious.
 Assume that the claim is true for $n<l$. We show the claim for $n=l$.
 We have already constructed the following diagram of solid arrows:
 \begin{equation*}
  \xymatrix{
   \mbf{V}_n\ar@{->>}[d]&Z\ar@{.>}[l]\ar@{.>>}[r]\ar@{.>}[d]&\mbf{W}_n\ar[d]\\
   \mbf{V}_{n-1,\infty}\times\Box&\mbf{V}'_{n-1,\infty}\times\Box\ar[l]\ar@{->>}[r]&\mbf{W}_{n-1,\infty}\times\Box.
    }
 \end{equation*}
 We can find $Z$ which is of finite type rendering the diagram commutative with the vertical dotted arrow proper.
 By putting $\mbf{V}'_n:=(Z\rightarrow\mbf{V}'_{n-1,\infty}\times\Box)$, the sequence $\{\mbf{V}'_n\}$ has the desired properties.

 Let us show the third claim.
 Let $\mbf{V}_1:=(V_1\rightarrow T\times\Box)$ be a modification over which $F\otimes\mc{L}(f_1t)$ becomes a flat \'{e}tale system.
 Then put $F_1:=(F\otimes\mc{L}(f_1t))|_{\mbf{V}_{1,\infty}}$.
 Then take a modification $\mbf{V}_2:=(V_2\rightarrow\mbf{V}_{1,\infty}\times\Box)$ so that $F_1\otimes\mc{L}(f_2t)$ becomes a flat \'{e}tale system.
 Put $F_2:=(F_1\otimes\mc{L}(f_2t))|_{\mbf{V}_{2,\infty}}$, and repeat this process till we construct $\mbf{V}_n$, $F_n$.
 Then $\{\mbf{V}_i\}$ is a principal $\Mdf$-sequence adapted to $(F;\mbf{f})$.

 Let us show the last claim.
 Let $\mb{W}:=\{\mbf{W}_i\}_{i\leq n}$ be an $\Mdf$-sequence of $T$ adapted to $(F;\mbf{f})$.
 Let $\mb{W}_{\leq k}:=\{\mbf{W}_i\}_{i\leq k}$.
 Fix $\{\mbf{W}_i\}_{i<n}$ and vary $\mbf{W}_n$.
 Then the family $\{F[\mb{W}]_{\infty}\}_{\mbf{W}_n}$
 determines an element $F\left<\mb{W}_{<n}\right>_\infty$ of $\partial\ms{S}(\mbf{W}_{n-1,\infty})$.
 By \ref{baicpropmdfseq-1}, this element is functorial with respect to varying $\mbf{W}_{n-1}$ fixing $\mb{W}_{\leq n-2}$.
 Thus, the family $\{F\left<\mb{W}_{\leq n-1}\right>_{\infty}\}_{\mbf{W}_{n-1}}$ determines an element of $\partial^2\ms{S}(\mbf{W}_{n-2,\infty})$.
 Repeating this, $\{F[\mb{W}]_{\infty}\}_{\mb{W}}$ determines an element of $\partial^n\ms{S}(T)$,
 which is nothing but $\mc{E}^n(F;\mbf{f})$.
 If $F[\mb{W}]_\infty$ belongs to $\ms{S}_{d-1}(\mbf{W}_{n,\infty})$ for any $\mb{W}$, then $\mc{E}^n(F;\mbf{f})=0$.
 Thus, in order to show the claim, it suffices to check that $F[\mb{W}]_{\infty}\in\ms{S}_{d-1}$ for any $\mb{W}$.
 By \ref{baicpropmdfseq-2}, we can find a diagram $\{\mbf{V}_i\}\leftarrow\{\mbf{V}'_i\}\rightarrow\{\mbf{W}_i\}$ such that
 $\mbf{V}'_n\twoheadrightarrow\mbf{W}_n$.
 By \ref{baicpropmdfseq-1}, $F[\{\mbf{V}'_i\}]$ belongs to $\ms{S}_{d-1}(\mbf{V}'_{n,\infty})$.
 The surjectivity of $\mbf{V}'_n\twoheadrightarrow\mbf{W}_n$ and \ref{baicpropmdfseq-1} implies
 that $F[\mb{W}]_{\infty}\in\ms{S}_{d-1}(\mbf{W}_{n,\infty})$.
\end{proof}

\subsection{}
\label{restopart2}
The main result of the second part is the following existence result.
Since the proof is completely of different flavor and independent, we postpone the proof to Part II.

\begin{thm*}[Main result of {[Part II]}]
 Assume that $X\rightarrow S$ is a morphism of finite type between noetherian affine schemes over $k$ 
 {\normalfont(}not necessarily of finite type{\normalfont)}.
 For any $F\in\ms{S}_d(S)$, $m\geq0$, and $\mbf{f}\in\mc{O}_{X\times\Delta_{\mr{geom}}^m}(X\times\Delta_{\mr{geom}}^m)^{\times l}$,
 we can find $\mbf{g}\in\mc{O}_{X}(X)^{\times l'}$ such that for a principal $\Mdf$-sequence $\{\mbf{V}_i\}$ adapted to $(F;\mbf{f}\vee\mbf{g})$,
 $F[\{\mbf{V}_i\}]_{\infty}$ is in $\ms{S}_{d-1}$.
\end{thm*}

\begin{cor}
 \label{rampaiconssheaf}
 The deformation system $(\mr{gr}_{d,i}(\ms{S}_{X/S}),\mu_{X/S},\mc{E}_{X/S})$ is effaceable for any $(d,i)\neq(1,0)$.
\end{cor}
\begin{proof}
 We must check the condition of Definition \ref{dfnramsys}.
 Assume we are given $(F;\mbf{f})\in(\ms{S}_{d,i}\otimes\mu^{\infty})(T\times\Delta_{\mr{geom}}^m)$.
 It suffices to find a finite sequence $\mbf{g}\subset\mc{O}_{X_i/S}(T)\subset\mu_{X/S}(T)$ Zariski locally with respect to each point of $T$
 so that $(F;\mbf{f}\vee\mbf{g})$ belongs to $\Eff_{\mb{S}_{d,i}}$.
 Shrinking $T$ around the point, we may assume that $T$ is affine.
 Since we took $X_i\rightarrow S$ to be affine, $(X_i)_T$ is affine as well.
 Given Lemma \ref{baicpropmdfseq}.\ref{baicpropmdfseq-4}, checking the condition is precisely what is asserted in Theorem \ref{restopart2}.
\end{proof}

\begin{test}
 Previous proof:
 We must check the condition of Definition \ref{dfnramsys} for each piece of the filtration.
 Assume we are given $\{(F_k;\mbf{f}_k)\}_{k\in I}$ in $(\ms{S}_{d,i}\otimes\mu^{\infty})(T\times\Delta^m)$ where $I$ is a finite set.
 It suffices to find a finite sequence $\mbf{g}\subset\mc{O}_{X_i/S}(T)\subset\mu_{X/S}(T)$ Zariski locally with respect to each $t\in T$.
 Shrinking $T$ around $t$, we may assume that $T$ is affine.
 Since we took $X_i\rightarrow S$ to be affine, $(X_i)_T$ is affine as well.
 Thus, in the case $\#I=1$, the condition is exactly the contents of Corollary \ref{restopart2}.
 In the general case where $I=\{1,\dots,m\}$, we find a function $\mbf{g}_1$ so that $(F_1;\mbf{f}_1\vee\mbf{g}_1)$ is in $\ms{S}_{d,i-1}$.
 Then we find a function $\mbf{g}_2$ so that $(F_2;\mbf{f}_2\vee\mbf{g}_1\vee\mbf{g}_2)$ is in $\ms{S}_{d,i-1}$.
 We repeat this process to get $\mbf{g}_1,\dots,\mbf{g}_m$.
 The sequence of functions $\mbf{g}=\mbf{g}_1\vee\mbf{g}_2\vee\dots\vee\mbf{g}_m$ meets our need.
\end{test}

\subsection*{Construction of characteristic cycles}

\subsection{}
Now, let us construct the characteristic cycle.
First let us explain the construction of non-microlocalized characteristic cycles fixing a morphism $X\rightarrow S$.
Logically, this paragraph is unnecessary, but this should make it easier to understand the main idea.

Let $X\rightarrow S$ be a morphism, and assume we are given a six functor formalism as in \ref{recalfixsixthe},
and recall from Corollary \ref{coefsys} that we have a reversible coefficient $\Hcoe_{X/S}$.
We also assume that {\em $\mr{char}(k)$ is invertible in $R$} and the six functor formalism admits trace maps (cf.\ \ref{tracemapdfn}).
Let $\ul{z}_{X/S}\colon\mr{Sch}_{/S}^{\mr{op}}\rightarrow\Mod_R$ be the functor sending $T$ to the abelian group $z(X_T/T,0)\otimes R$.
We have the morphism of presheaves:
\begin{equation*}
 \ms{S}_{X/S,0}
  \rightarrow
  \ul{z}_{X/S}
  \rightarrow
  \Hcoe_{X/S},
\end{equation*}
where the first homomorphism is constructed in \ref{dfnofrk0assmap}, and the second by the trace map \ref{tracemapdfn}.
Invoking Corollary \ref{rampaiconssheaf}, the inclusion $\ms{S}_{X/S,0}\hookrightarrow\ms{S}_{X/S}$ is localizing in the sense of \ref{mainconscc}.
Thus, by Theorem \ref{mainconscc}, the above homomorphism extends (essentially uniquely) to a homomorphism
\begin{equation*}
 \mr{CC}_X\colon\ms{S}_{X/S}\rightarrow\Hcoe_{X/S}.
\end{equation*}
By taking the global section and taking $\pi_0$, we have the homomorphism of $R$-modules
\begin{equation*}
 \mr{CC}\colon
  \Comp(X/S)\rightarrow\pi_0\Gamma(S,\ms{S}_{X/S})\rightarrow
  \pi_0\Gamma(S,\Hcoe_{X/S})
  \simeq
  \pi_0\Hbm(X/S,R_S).
\end{equation*}
This homomorphism is called the ({\em non-microlocalized}) {\em characteristic cycle homomorphism}.
Note that, in the case $S=\mr{Spec}(k)$ where $k$ is perfect and if we take the six functor formalism to be the
one associated with the $\mb{Z}[1/p]$-motivic cohomology theory (cf.\ Example \ref{introsixfuncform}),
$\pi_0\Hbm(X/S,\mb{Z}[1/p])\cong\mr{CH}_0(X)[1/p]$ and $\Comp(X/S)\cong\mr{K}_0\mr{Cons}(X)$.
We wish to microlocalize the construction by means of \S\ref{sect3}.
For this, we need to upgrade this construction to a relative setting.

\subsection{}
We have an abelian presheaf $z_{X/S}\colon(X/S)_{\mr{zar}}^{\mr{op}}\rightarrow\Mod_R$ sending $U\rightarrow T$ to $z(U/T,0)\otimes R$.
In \ref{introopenbm}, we defined the $\infty$-presheaf $\Hbm_{X/S}:={}^{\mr{open}}\Hbm_{X/S}\colon(X/S)_{\mr{zar}}^{\mr{op}}\rightarrow\Mod_R$
sending $U\rightarrow T$ to $\Hbm(U/T,R_T)$.
In view of the exact sequence \cite[4.3.1]{SV} and the construction of $\Hbm_{X/S}$ as well as Lemma \ref{cdhdsvalshv}.\ref{cdhdsvalshv-2},
the trace map induces the morphism of functors $z_{X/S}\rightarrow\Hbm_{X/S}$.

On the other hand, we have the $\infty$-presheaf $\ms{S}:=\ms{S}_{X/S}$ on $(X/S)_{\mr{zar}}$.
By Lemma \ref{compdim0nea}, we have the morphism $\ms{S}_0\rightarrow z_{X/S}$ of abelian sheaves.
Combining together, we have the diagram $\ms{S}\leftarrow\ms{S}_0\xrightarrow{\mr{CC}}\Hbm$ in $\mr{Fun}((X/S)^{\mr{op}}_{\mr{zar}},\Mod_{\mb{Z}})$.
Now, the functor $(X/S)_{\mr{zar}}\rightarrow\mr{Sch}_{/S}$ sending $U\rightarrow T$ to $T\rightarrow S$ is a Cartesian fibration.
The inclusion $\ms{S}_0\circ a_{U/T}\hookrightarrow\ms{S}\circ a_{U/T}$ is localizing in the sense of \ref{mainconscc} by Corollary \ref{rampaiconssheaf}.
On the other hand, $\Hbm\circ a_{U/T}$ is reversible by Theorem \ref{revcons}.
Thus, invoking Corollary \ref{familyextthm}, the morphism $\ms{S}_0\rightarrow\Hbm$ extends essentially uniquely to a morphism
$\mr{CC}_{X/S}\colon\ms{S}_{X/S}\rightarrow\Hbm$ in $\mr{Fun}((X/S)^{\mr{op}}_{\mr{zar}},\Mod_{\mb{Z}})$.
This is the parameterized version of non-microlocalized characteristic cycle map.

Now, let us consider the functoriality.
Assume we are given a proper morphism $h\colon X\rightarrow Y$ over $S$.
Since the functor $(Y/S)_{\mr{zar}}\rightarrow(X/S)_{\mr{zar}}$ sending $(V/T)$ to $(V\times_YX/T)$
is continuous and commutes with finite limits fiberwise, it induces the geometric morphism $\Shv((X/S)_{\mr{zar}})\rightarrow\Shv((Y/S)_{\mr{zar}})$, also denoted by $h$.
Since the pushforward functor along proper morphism and the nearby cycle functor commute,
we have the following morphism (on the left) and homotopy commutative diagram (on the right):
\begin{equation*}
 \label{zerohdia}\tag{$\star$}
 \xymatrix{
  h_*\ms{S}_{X/S}\ar[d]_{h_*}\\
 \ms{S}_{Y/S},}
  \qquad
  \xymatrix@C=50pt{
  h_*\ms{S}_{X/S,0}\ar[r]^-{\mr{CC}_{X/S}}\ar[d]_{h_*}&h_*\Hbm_{X/S}\ar[d]^{h_*}\\
 \ms{S}_{Y/S,0}\ar[r]^-{\mr{CC}_{Y/S}}&\Hbm_{Y/S}.
  }
\end{equation*}
Now, we have the following diagram (which is possibly not commutative for the moment):
\begin{equation*}
 \xymatrix@C=50pt{
  h_*\ms{S}_{X/S}\ar[r]^-{\mr{CC}_{X/S}}\ar[d]_{h_*}&h_*\Hbm_{X/S}\ar[d]^{h_*}\\
 \ms{S}_{Y/S}\ar[r]^-{\mr{CC}_{Y/S}}&\Hbm_{Y/S}.
  }
\end{equation*}
In fact this diagram is homotopy commutative.
Indeed, by invoking Corollary \ref{familyextthm} again, the morphism
$\Map(h_*\ms{S}_{X/S},\Hbm_{Y/S})\rightarrow\Map(h_*\ms{S}_{X/S,0},\Hbm_{Y/S})$ is an equivalence.
Thus, we must show the homotopy commutativity of the diagram after restricting to $h_*\ms{S}_{X/S,0}$,
but this is nothing but the homotopy commutative diagram (\ref{zerohdia}).
It is worth noting that we did not make any smoothness assumption so far.
In this sense, our non-microlocalized characteristic cycle is a generalization of that of Saito that we will recall in the next section.

\subsection{}
Finally, we will microlocalize our characteristic cycle.
Let $X\rightarrow S$ be a scheme of finite type, and we have constructed the morphism $\ms{S}_{X/S}\rightarrow\Hbm_{X/S}$
of $\infty$-presheaves on $(X/S)_{\mr{zar}}$.
If $X$ is {\em smooth} over $S$, we consider the base data $(X,T^*(X/S),\mc{O}_X,\mr{d}_X)$ (cf.\ Example \ref{dfnbasedata}),
where $\mr{d}_X\colon\mc{O}_X\rightarrow\Omega^1_{X/S}$ is the usual differential.
We have constructed in \ref{microlocconstofE} the $p$-deformation system for this $\ms{S}_{X/S}$.
By the construction of \ref{consmicccabs}, we have the ({\em microlocalized}) {\em characteristic cycle map}
\begin{equation*}
 \mr{CC}^{\mu}_{X/S}\colon(\ms{S}_{X/S})_C\rightarrow\rH_C\bigl(T^*(X/S)/X,\omega_{X/S}\bigr).
\end{equation*}
Here, we recall that $\rH_C\bigl(T^*(X/S)/X,\mc{F})\cong\Mor_{\mc{D}(C)}(R_C,i^!\pi^*\mc{F})$ for $\mc{F}\in\mc{D}(X)$ where $C\xrightarrow{i}T^*X\xrightarrow{\pi}X$, and
$\omega_{X/S}:=f^!R_S$ where $f\colon X\rightarrow S$.
The formation is compatible with base change of $S$ by smooth morphism $S'\rightarrow S$ by \ref{functofmicccabs}.

\subsection{}
Now, let us establish the functoriality.
Assume we are given a proper morphism $h\colon X\rightarrow Y$ of smooth $S$-schemes of finite type.
Consider the canonical morphisms $T^*(X/S)\xleftarrow{g}T^*(Y/S)\times_YX\xrightarrow{h'}T^*(Y/S)$.
We will construct the following homotopy commutative diagrams:
\begin{gather*}
 \xymatrix{
  (\ms{S}_{X/S})_C\ar[r]\ar[d]\ar@{}[rd]|{\ccirc{1}}&\rH_C(T^*X)\ar[d]\\
 (\ms{S}_{X/S})_{g^{-1}(C)}\ar[r]&\rH_{g^{-1}(C)}(T^*Y\times_YX),
  }
  \qquad
  \xymatrix{
  (\ms{S}_{X/S})_{g^{-1}(C)}\ar[r]\ar[d]\ar@{}[rd]|{\ccirc{2}}&\rH_{g^{-1}(C)}(T^*Y\times_YX)\ar[d]\\
 (h_*\ms{S}_{X/S})_{h'g^{-1}(C)}\ar[r]&\rH_{h'g^{-1}(C)}(T^*Y),
  }\\
 \xymatrix{
 (h_*\ms{S}_{X/S})_{h'g^{-1}(C)}\ar[r]\ar[d]\ar@{}[rd]|{\ccirc{3}}&\rH_{h'g^{-1}(C)}(T^*Y)\ar[d]\\
 (\ms{S}_{Y/S})_{h'g^{-1}(C)}\ar[r]&\rH_{h'g^{-1}(C)}(T^*Y).
 }
\end{gather*}
These diagrams all follow from \ref{functofmicccabs}.
The diagram $\ccirc{1}$ follows from the functoriality of base data for the morphism
$(X,T^*(Y/S)\times_YX,h^{-1}\mc{O}_Y,\mr{d}_Y)\rightarrow(X,T^*(X/S),\mc{O}_X,\mr{d}_X)$.
The diagram $\ccirc{2}$ follows from the functoriality of the pushforward via the morphism of schemes $h$.
The diagram $\ccirc{3}$ follows from the functoriality with respect to the homomorphism of presheaves $h_*\ms{S}_{X/S}\rightarrow\ms{S}_{Y/S}$.
Combining these diagrams, we have the following theorem.

\begin{thm}
 \label{mainresul}
 Recall that $S$ is an excellent scheme of finite Krull dimension,
 and we are fixing a six functor formalism as in {\normalfont\ref{recalfixsixthe}} with trace maps such that the ring of coefficient $R$ contains $\mr{char}(k)^{-1}$.
 Let $h\colon X\rightarrow Y$ be a proper morphism between smooth $S$-schemes of finite type, and let $C$ be a conic in $T^*(X/S)$.
 Consider the canonical morphisms $T^*(X/S)\xleftarrow{g}T^*(Y/S)\times_YX\xrightarrow{h'}T^*(Y/S)$, and let $h_\circ(C):=h'(g^{-1}(C))$.
 We have the following homotopy commutative diagram:
 \begin{equation*}
  \xymatrix@C=50pt{
   (\ms{S}_{X/S})_C\ar[d]_{h_*}\ar[r]^-{\mr{CC}^{\mu}_{X/S}}&
   \rH_C(T^*(X/S),\omega_{X/S})\ar[d]^{h'_*g^*}\\
  (\ms{S}_{Y/S})_{h_\circ C}\ar[r]^-{\mr{CC}^{\mu}_{Y/S}}&\rH_{h_\circ C}(T^*(Y/S),\omega_{Y/S}).
   }
 \end{equation*}
\end{thm}

\begin{rem*}
 For a separated scheme $X$ of finite type over $S$ and a closed embedding $i\colon X\hookrightarrow P$ such that $P$ is a smooth separated $S$-scheme,
 we may consider the base data $(X,T^*(P/S)\times_P X,i^{-1}\mc{O}_P,i^{-1}\mr{d}_P)$.
 We endow $\ms{S}_{X/S}$ with the support structure
 $\ms{S}_{X/S}\otimes i^{-1}\mc{O}_P\rightarrow\ms{S}_{X/S}\otimes\mc{O}_X\xrightarrow{\mc{E}'_{X/S}}\partial\ms{S}_{X/S}$.
 Then we have the characteristic cycle map $(\ms{S}_{X/S})_C\rightarrow\rH_C(T^*(P/S)\times_PX,\omega_{X/S})$.
 We also have a variant of the pushforward formula if we are given a morphism $P\rightarrow Y$ between smooth $S$-schemes
 such that the composite $X\rightarrow Y$ is proper.
 This is exactly the form conjectured by Saito in \cite[Conjecture 1]{S2}.
 Since the verification is the same, we leave the details to the reader.
\end{rem*}

\subsection{Summary}\hspace{1.7ex}
\label{summarymain}
For the convenience of the reader, let us summarize what we have proven in specific cases of interest.
Let $k$ be a perfect field of characteristic $p>0$, and $S$ be a $k$-scheme of finite type.
We choose the six functor formalism associated to the $\mb{Z}[1/p]$-motivic cohomology theory (cf.\ Example \ref{introsixfuncform}).
Then we may apply the theorem above.
If we take $S=\mr{Spec}(k)$, we have $\pi_0\rH_C(T^*(X/S),\omega_{X/S})\cong\mr{CH}_{\dim(X)}(C)[1/p]$.
Thus, by taking $\pi_0$ in the homotopy commutative diagram of Theorem \ref{mainresul},
we get a similar equality as the one in Theorem of Introduction for our characteristic cycle.
At this moment, we have not proven the coincidence between our characteristic cycle and Saito's.
This will be proven in the next section.

\begin{rem*}
 \begin{enumerate}
  \item In the case where $S$ is a $k$-scheme of finite type, we have another interesting consequence,
	at least for non-microlocalized characteristic cycles.
	Since the non-microlocalized characteristic cycle homomorphism
	$\mr{CC}_{X/S}\colon\ms{S}_{X/S}\rightarrow\Hbm_{X/S}$ is a homomorphism of {\em presheaves},
	it contains interesting functorial information by itself.
	For example, assume $S$ is a smooth integral curve with generic point $\eta$ and $k$ is separably closed.
	Then we have a map $\mr{K}_0\mr{Cons}(X_{\eta})\rightarrow\Gamma(X/S,\ms{S}_{X/S})$
	since any sheaf on $X$ is already good over $S$.
	Let $i\colon s\hookrightarrow S$ be a closed immersion from a closed point $s$.
	Then the functoriality of $\mr{CC}_{X/S}$ implies that the following diagram commutes
	\begin{equation*}
	 \xymatrix@C=50pt{
	  \mr{K}_0\mr{Cons}(X_\eta)\ar[r]^-{\mr{CC}_{X_\eta/\eta}}\ar[d]_{\Psi_{s\leftarrow\overline{\eta}}}&\Hbm(X/S)\ar[d]^{i^*}\\
	 \mr{K}_0\mr{Cons}(X_s)\ar[r]^-{\mr{CC}_{X_s/s}}&\Hbm(X_s/s).
	  }
	\end{equation*}
	This can be regarded as a compatibility of the characteristic cycle homomorphism and the nearby cycle functor.
	This can even be microlocalized, but for this, we need to modify the definition of microsupport in Definition \ref{suppstrdfn}.
	More precisely, in the definition, we considered only $T\rightarrow S$ which is smooth,
	but we need to consider morphisms of the form $T\rightarrow Z\hookrightarrow S$
	where the first morphism is smooth and the second one is an immersion.
	The detail is left to the interested reader.

  \item The assumption that $S$ is of finite type over a field $k$
	is made merely because the trace map for motivic cohomology has been constructed only under the assumption.
	We did not try particularly hard to lift these assumptions.
	In view of the previous remark, constructing the trace map in the case where $S$ is a trait may be of interest.

  \item \label{integrem}
	In the theorem above, we need to invert $\mr{char}(k)$.
	We have two places where we need to make this assumption.
	First is to construct the trace map in the motivic setting in \cite{Atr}.
	If we assume the existence of the resolution of singularities, this assumption can be lifted.
	More seriously, we need the inversion in the construction of the homomorphism $r$ in \ref{dfnofrk0assmap}.
	For this construction, the inversion is indispensable.
 \end{enumerate}
\end{rem*}

\section{Comparison with Saito's characteristic cycles}
\label{sect6}
In this final section, we compare the microlocalized characteristic cycle we constructed with that of T. Saito.
We keep using the ring $\Lambda$ from the previous section.

\subsection{}
Let us recall the definition of singular support:

\begin{dfn*}[Beilinson \cite{BSS}]
 Let $X$ be a smooth scheme over a field $k$, and $\mc{F}$ be a constructible $\Lambda$-module on $X$.
 \begin{enumerate}
  \item Let $C\subset T^*X$ be a closed conic subset.
	A pair of morphisms $(X\xleftarrow{f} U\xrightarrow{g}Y)$ is said to be {\em $C$-transversal at $u\in U$}
	if $\mr{Ker}(\mr{d}f)\cap C=0$ in $T^*X\times_X\{u\}$ and $(\mr{d}g)^{-1}(f^{\circ}C)=0$ in $T^*Y\times_Y\{u\}$,
	where $f^{\circ}C:=\mr{Im}(C\times_X U\hookrightarrow T^*X\times_XU\xrightarrow{\mr{d}f} T^*U)$.

  \item We define the {\em singular support of $\mc{F}$} denoted by $\mr{SS}(\mc{F})$ to be the smallest conic in $T^*X$ having the following property:
	For any pair $(X\xleftarrow{f}U\xrightarrow{g}Y)$ and $u\in U$ which is $\mr{SS}(\mc{F})$-transversal at $u\in U$,
	$g$ is universally locally acyclic at $u$ with respect to $f^*\mc{F}$.
 \end{enumerate}
\end{dfn*}
The main result of Beilinson states that $\mr{SS}(\mc{F})$ exists and, surprisingly, it is equidimensional of dimension equal to $\dim(X)$.

\subsection{}
\label{dfnSaitocc}
Now, let us recall the definition of characteristic cycle:

\begin{dfn*}[Saito {\cite[Def 5.10]{S}}]
 Let $X$ be a smooth scheme of finite type over a {\em perfect field} $k$ equidimensional of dimension $d$,
 and $\mc{F}\in D_{\mr{ctf}}(X,\Lambda)$.
 \begin{enumerate}
  \item Let $C$ be a conic closed subset of $T^*X$.
	A point $x\in X$ is said to be an {\em isolated $C$-characteristic point} of a function $f\colon X\rightarrow\mb{A}^1$
	if $x$ is an isolated point in $\iota_X^{-1}(\mr{d}f)^{-1}(C)$. Here $\iota_X\colon X\hookrightarrow T^*\mb{A}^1\times_{\mb{A}^1}X$
	is the section defined by $\mr{d}s$ using the coordinate $s$ of $\mb{A}^1$.

  \item Write $\mr{SS}(\mc{F})=\bigcup_{i\in I}W_i$ by irreducible components where $W_i$ is integral.
	There exists a unique number in $n_i\in\mb{Z}[1/p]$ for each $i\in I$ such that,
	for any finite extension $l/k$ of fields,
	for any local function $f\colon U\rightarrow\mb{A}^1_k$ where $U\subset X\otimes_kl$ is an open subscheme,
	and for any isolated $\mr{SS}(\mc{F})$-characteristic point $x\in U$ of $f$,
	we have the equality
	\begin{equation}
	 \label{totdimccdef}
	 -\mr{totdim}\bigl(\Phi_f(\mc{F}|_U)_{\overline{x}}\bigr)=\bigl(\mr{d}f,\sum n_i[W_i]\bigr)_{T^*U,x},
	\end{equation}
	where $\mr{totdim}$ denotes the total dimension ({\it i.e.}\ $\mr{totdim}(V)=\dim(V)+\mr{Sw}(V)$),
	and the right hand side denotes the intersection number in $T^*U$ at $x$.
	We define $\mr{CC}(\mc{F}):=\sum n_i[W_i]$ in $\mr{Z}_d(T^*X)[1/p]$.
 \end{enumerate}
\end{dfn*}

\begin{rem*}
 In fact, Saito uses \'{e}tale local functions with target scheme arbitrary smooth curve,
 contrary to using only Zariski local functions on $X\otimes_k l$ with target scheme $\mb{A}^1$, to define characteristic cycles.
 Thus, {\em a priori}, it is not clear if the numbers $n_i$ are uniquely determined in our definition.
 However, they are in fact unique.
 To see this, note that it suffices to show the existence in the case where $k$ is algebraically closed.
 In this situation, Saito checks conditions of \cite[Prop 5.8.1]{S} for the $\mb{Z}[1/p]$-valued function on
 isolated $\mr{SS}(\mc{F})$-characteristic points defined by the left hand side of (\ref{totdimccdef}) above,
 and \cite[Prop 5.8.1]{S} uses only Zariski-local functions as in our definition.
 Moreover, in view of \cite[Prop 5.8.2]{S}, we may assume that the target of the function is $\mb{A}^1$.
\end{rem*}

\subsection{}
\label{Laumonthm}
First, let us reformulate results of Laumon on local Fourier transform, which is a crucial observation for us:

\begin{thm*}[Laumon]
 Let $k$ be a perfect field, let $T:=(\mb{A}^1_k)_{(0)}$ {\normalfont(}resp.\ $T':=(\mb{A}^1_k)_{(0)}${\normalfont)}
 be the henselization with coordinate $t$ {\normalfont(}resp.\ $t'${\normalfont)},
 and let $T\xleftarrow{\mr{pr_1}}T\times_k T'\xrightarrow{\mr{pr}_2}T'$ be the projections.
 For $\mc{F}\in D_{\mr{ctf}}(T)$, we have
 \begin{equation*}
  -\mr{totdim}(\Phi_{\mr{id}}(\mc{F}))=\dim\bigl(\Psi_{\mr{pr}_2}(\mr{pr}_1^*\mc{F}\otimes\mc{L}_!(t/t'))_{(\overline{0},\overline{0}')}\bigr).
 \end{equation*}
 Here, $\mc{L}_!(t/t')$ is the zero-extension of $\mc{L}(t/t')$ on $T\times T'\setminus T\times\{0'\}$.
\end{thm*}
\begin{proof}
 Since the invariants appearing in the equality are stable under taking $\otimes^{\mb{L}}_\Lambda(\Lambda/\mf{m})$ where $\mf{m}$ is the maximal ideal,
 we may assume that $\Lambda$ is a field.
 Thus, we may assume that $\mc{F}$ is a constructible sheaf.
 Let $i\colon\{0\}\hookrightarrow T$ be the closed point and $j\colon\eta\hookrightarrow T$ be the generic point.
 By Brauer's theory, we may take a lift $\widetilde{\mc{F}}_{\eta}\in D^{\mr{b}}_{\mr{c}}(\eta,E)$,
 where $E$ is a finite extension of $\mb{Q}_\ell$ whose residue field is $\Lambda$,
 such that the reduction of $\widetilde{\mc{F}}_{\eta}$ is equal to $\mc{F}_\eta:=j^*\mc{F}$
 in the Grothendieck group $\mr{K}_0D^{\mr{b}}_{\mr{c}}(\eta,\Lambda)$.

 Now, we have
 $\mr{totdim}(\Phi_{\mr{id}}(\mc{F}))=\mr{totdim}(\mc{F}_{\eta})-\dim(\mc{F}_{\overline{0}})=\mr{totdim}(\widetilde{\mc{F}}_{\eta})-\dim(\mc{F}_{\overline{0}})$.
 On the other hand, we have $\mr{pr}_1^*i_*i^*\mc{F}\otimes\mc{L}_!(t/t')\cong\mc{F}_{0}\otimes j'_!j'^*\Lambda_{0\times T'}$.
 This implies that $\Psi_{\mr{pr}_2}(\mr{pr}_1^*i_*i^*\mc{F}\otimes\mc{L}_!(t/t'))_{(\overline{0},\overline{0}')}\cong\mc{F}_{\overline{0}}$.
 Since $\mc{L}_!(t/t')$ is $0$ on $\mr{pr}_2^{-1}(0')$, we have $\Psi_{\mr{pr}_2}(\dots)\cong\Phi_{\mr{pr}_2}(\dots)$.
 Recall the definition of local Fourier transform from \cite[2.4.2.3]{Lau}.
 Using this, we can compute as
 \begin{align*}
  \dim&\bigl(\Psi_{\mr{pr}_2}(\mr{pr}_1^*\mc{F}\otimes\mc{L}(t/t'))_{(\overline{0},\overline{0}')}\bigr)\\
  &=
  \dim\bigl(\Psi_{\mr{pr}_2}(\mr{pr}_1^*j_!j^*\mc{F}\otimes\mc{L}(t/t'))_{(\overline{0},\overline{0}')}\bigr)
  +
  \dim\bigl(\Psi_{\mr{pr}_2}(\mr{pr}_1^*i_*i^*\mc{F}\otimes\mc{L}(t/t'))_{(\overline{0},\overline{0}')}\bigr)\\
  &=
  -\dim_{E}\ms{F}^{(0,\infty')}(\widetilde{\mc{F}}_\eta)+\dim_{\Lambda}(\mc{F}_{\overline{0}}),
 \end{align*}
 where we used the vanishing result \cite[2.4.2.2]{Lau} at the last equality.
 Invoking \cite[2.4.3 (i)-(b)]{Lau}, the theorem follows.
\end{proof}

\subsection{}
\label{simplprojform}
We need the following basic result:
\begin{lem*}
 Let $g\colon X\rightarrow Y$ be a morphism of schemes {\normalfont(}not necessary of finite type{\normalfont)}.
 Let $\mc{L}$ be a locally constant constructible flat $\Lambda$-module on $Y$.
 Then for any $\mc{F}\in D^+(X,\Lambda)$ on $X$, we have a canonical isomorphism
 \begin{equation*}
  \mr{R}g_*(\mc{F})\otimes\mc{L}\xrightarrow{\sim}
   \mr{R}g_*(\mc{F}\otimes g^*\mc{L}).
 \end{equation*}
\end{lem*}
\begin{proof}
 The construction of the homomorphism is standard using adjoints.
 Since, now, the claim is local, we may assume that $Y$ is the spectrum
 of a strictly henselian ring. In this situation, we may assume that
 $\mc{L}$ is the constant sheaf $\Lambda^{\oplus n}$, in which case, the
 claim is obvious.
\end{proof}

\begin{prop}
 \label{mainpropcomp}
 Let $X$, $S$ be schemes of finite type over a perfect field $k$ such that $S$ is smooth and integral.
 Let $\mc{F}$ be an object of $D_{\mr{ctf}}(X)$,
 and $f\colon X\times S\rightarrow\mb{A}^1_S$ be an $S$-morphism, in other words, a ``family of functions parameterized by $S$''.
 Let $x\in X$, $s\in S$ be points, and we denote by $f_s:=f|_{X\times\{s\}}$.
 We put $\Psi_{(x,s)}:=\Psi_{(s,\infty)\leftarrow\eta}\bigl(\mc{F}\otimes\mc{L}_!(ft)\bigr)_{\overline{x}}$,
 where the nearby cycle is taken with respect to the projection
 $\mr{pr}\colon X\times S\times\Box\rightarrow S\times\Box$, $\eta$ is the generic point of $S\times\Box$, and $\mc{L}_!$ is the zero extension of $\mc{L}$.
 \begin{enumerate}
  \item\label{mainpropcomp-1}
       Assume $f_s$ is $\mr{SS}(\mc{F})$-transversal
       {\normalfont(}namely, $f$ is smooth and $df_s(X)\cap\mr{SS}(\mc{F})=\{0\}$ in $T^*X${\normalfont)}.
       Then $\mc{F}\otimes\mc{L}_!(ft)$ is $\mr{pr}$-universally locally acyclic, and in particular, $\Psi_{(x,s)}=0$.

  \item\label{mainpropcomp-2}
       Assume $S=\mr{Spec}(k)$ and $x$ is an isolated characteristic point of $f$.
       Then $\dim(\Psi_x)=-\mr{totdim}(\Phi_f(\mc{F})_{\overline{x}})$.
 \end{enumerate}
\end{prop}
\begin{proof}
 Let us initialize the notation for a moment.
 For a morphism $f\colon X\rightarrow S$ of schemes and a geometric point $x$ of $X$,
 we denote by $f_{(x)}\colon X_{(x)}\rightarrow S_{(f(x))}$ the induced morphism between henselian local schemes.
 If we are given a morphism $z\colon Z\rightarrow X$, we often denote $z^*\mc{F}$ simply by $\mc{F}$ if no confusion may arise.
 If we are given a morphism $S'\rightarrow S$ and a geometric point $x'$ of $X':=X_{S'}$, we consider the following induced commutative diagram:
 \begin{equation*}
  \xymatrix{
   X'_{(x')}\ar[r]^-{g'}\ar[d]_{f'_{(x')}}&X_{(x)}\ar[d]^{f_{(x)}}\\
  S'_{(s')}\ar[r]^-{g}&S_{(s).}
   }
 \end{equation*}
 Here, $x$, $s$, $s'$ are geometric points of $X$, $S$, $S'$ induced by $x'$.
 Let $\mc{F}\in D_{\mr{ctf}}(X)$.
 The pair $(f,\mc{F})$ is said to {\em satisfy property P} if $\mc{F}$ is $f$-good,
 and for any morphism $S'\rightarrow S$ and any geometric point $x'$ of $X'$, $f'_{(x')}$ is cohomologically proper with respect to $g'^*\mc{F}$.
 In such a case, the adjunction morphism $g^*\mr{R}f_{(x)*}\mc{F}\rightarrow\mr{R}f'_{(x')*}g'^*\mc{F}$ is an isomorphism.
 Indeed, by cohomological properness, the stalks of the source and target at each geometric point of $S'_{(s')}$ coincide with their respective nearby cycles,
 and using this identification, the adjunction morphism can be identified with the canonical morphism in Definition \ref{examplegood}.
 Property P is satisfied for $(f,\mc{F})$ when $S$ is a {\it trait} by \cite[Thm 2.1, Thm 7.1]{O1},
 or when $\mc{F}$ is $f$-universally locally acyclic by [SGA $4\frac{1}{2}$, Th.\ Finitude, A.2.8].

 Now, let the notation be as in the statement of the proposition.
 Let $W\rightarrow(S\times\Box)_{(\overline{s},\overline{\infty})}$ be a morphism of strictly henselian local schemes such that the closed point $w$ of $W$ is sent to $(s,\infty)$.
 Let $\mathring{W}$ be the preimage of $S\times(\Box\setminus\{\infty\})$ in $W$.
 Fix a geometric point $x$ of $X$, and for a geometric point $\xi$ of $W$, we put
 $  \Psi_{\xi}:=\Psi_{\mr{pr}_W,\overline{w}\leftarrow\xi}\bigl(\mc{F}\otimes\mc{L}_!(ft)\bigr)_{(\overline{x})}$.
 To compute this, we consider the following commutative diagram:
 \begin{equation*}
  \xymatrix@C=55pt{
  (X\times W)_{(\overline{x},\overline{w})}
  \ar[r]_-{f'=(f\times\mr{id})_{(\overline{x},\overline{w})}}
  \ar@/^15pt/[rr]^-{\mr{pr_2}}&
   (\mb{A}^1_{y}\times W)_{(f(\overline{x}),\overline{w})}
   \ar[r]_-{\pi}&
   W\ar[r]&(S\times\Box)_{(\overline{s},\overline{\infty})}
   }
 \end{equation*}
 where $\pi$, $\mr{pr}_2$ denote the projections.
 We {\em assume that $(f_{(x,s)},\mc{F})$ satisfies property P}.
 By the observation at the beginning, we have $\mr{R}f'_{\xi,*}(\mc{F})\cong h^*\mr{R}f_{(x,s)*}(\mc{F})$,
 where $h\colon(\mb{A}^1\times W)_{(f(\overline{x}),\overline{w})}\times_W\xi\rightarrow(\mb{A}^1_S)_{(f(\overline{x}),\overline{s})}$.
 For a morphism $z\colon Z\rightarrow W$, we denote $Z\times_W\xi$ by $z^{-1}(\xi)$.
 If, furthermore, $\xi$ is a geometric point of $\mathring{W}$, using Lemma \ref{simplprojform}, we have
 \begin{align}
  \label{nbcyislocom}
  \tag{$\star$}
  \Psi_{\xi}
  &\cong
  \mr{R}\Gamma\bigl(\mr{pr}^{-1}_2(\xi),(\mc{F}\otimes\mc{L}(yt))_{\xi}\bigr)
  \cong
  \mr{R}\Gamma\bigl(\pi^{-1}(\xi),\mr{R}f'_{\xi,*}(\mc{F}\otimes f'^*_\xi\mc{L}(yt))\bigr)\\
  \notag
  &\cong
  \mr{R}\Gamma\bigl(\pi^{-1}(\xi),\mr{R}f'_{\xi,*}(\mc{F})\otimes\mc{L}(yt)\bigr)
  \cong
  \Psi_{\pi,\overline{w}\leftarrow\xi}(h^*\mr{R}f_{(x,s)*}(\mc{F})\otimes\mc{L}(yt)).
 \end{align}

 Let us show the second statement first.
 In this case, take $W=\Box_{(\infty)}$ and $\xi$ to be a geometric generic point.
 Since $(\mb{A}^1_S)_{(f(\overline{x}),\overline{s})}$ is a {\it trait}, (\ref{nbcyislocom}) holds.
 Invoking the result of Laumon, Theorem \ref{Laumonthm}, the result follows because $\Phi_{\mr{id}}\bigl(\mr{R}f_{(x,s)*}\mc{F}\bigr)\cong\Phi_{f}(\mc{F})_{\overline{x}}$.

 Let us show the first statement.
 Under the assumption, we claim that $f_{(x,s)}$ is universally locally acyclic with respect to $\mc{F}$.
 Admit this for a moment, let us show $\Psi_\xi=0$ for any $\xi$.
 If $\xi$ lies over $W\setminus\mathring{W}$, $\mc{L}_!(ft)_\xi=0$, and there is nothing to show, so we may assume that $\xi$ is a geometric point of $\mathring{W}$.
 By the local acyclicity, the isomorphism (\ref{nbcyislocom}) holds.
 Furthermore, $\mr{R}f_{(x,s)*}(\mc{F})$ is constant on $(\mb{A}^1_S)_{(f(x),s)}$ with value $\mc{F}_{\overline{x}}$
 by [SGA $4\frac{1}{2}$, Th.\ Finitude, A.2.6 and A.2.4].
 By the theorem of Katz-Laumon (cf.\ \cite[1.3.1.2]{Lau}), we know that
 $\Psi_{\pi,\overline{w}\leftarrow\xi}(\mc{L}(yt))=0$, and the result follows.
 To conclude the proof, we are left to show the local acyclicity of $f_{(x,s)}$.

 In fact, the claim follows easily from the definition of singular support as follows:
 Consider the test pair $(X\xleftarrow{p_1} X\times S\xrightarrow{f}\mb{A}^1_S)$.
 By definition of singular support, it suffices to show that the test pair is $\mr{SS}(\mc{F})$-transversal at
 $(x,s)\in X\times S$.
 Since $p_1$ is smooth, $\mr{SS}(\mc{F})$ is $p_1$-transversal,
 and $p_1^{\circ}\mr{SS}(\mc{F})=\mr{SS}(\mc{F})\times S\cong\mr{SS}(\mc{F})\times T^*_SS\subset T^*X\times T^*S$.
 Let $(x_1,\dots,x_d)$ and $(s_1,\dots,s_n)$ be local coordinates of $X$ and $S$ at $x$ and $s$.
 Let $y$ be the coordinate of $\mb{A}^1_S$.
 Then $\mr{d}f(\mr{d}y)=\omega$ such that $\omega(x,s)=\mr{d}f_s(x)+g_1\mr{d}s_1+\dots+g_n\mr{d}s_n$,
 and $\mr{d}f(\mr{d}s_i)=\mr{d}s_i$.
 This presentation implies that
 \begin{equation*}
  (\mr{d}f)^{-1}(\mr{SS}(\mc{F})\times T^*_SS)\times_S\{s\}
   \xrightarrow{\sim}
   (\mr{d}f_s)^{-1}(\mr{SS}(\mc{F}))\times\{s\},
 \end{equation*}
 where the morphism is induced by $T^*\mb{A}^1\times T^*S\rightarrow T^*\mb{A}^1\times S$.
 Thus, the test pair is $\mr{SS}(\mc{F})$-transversal at $(x,s)$ if and only if $f_s$ is $\mr{SS}(\mc{F})$-transversal at $x$.
\end{proof}

\begin{cor}
 \label{compwithsaitocor}
 Assume we are in the situation of {\normalfont\ref{summarymain}}.
 Assume further that the trace map of the motivic cohomology theory is chosen so that the induced map
 \begin{equation*}
  z(\mr{Spec}(k)/\mr{Spec}(k),0)
   \rightarrow
   \pi_0\Hbm(\mr{Spec}(k)/\mr{Spec}(k))
   \cong
   \mr{CH}_0(\mr{Spec}(k))[1/p]
 \end{equation*}
 sends $[\mr{Spec}(k)]$ to $[\mr{Spec}(k)]$, which exists by {\normalfont\cite[3.5, 6.4]{Atr}}.
 Let $X$ be a smooth scheme over a perfect field $k$, and let $\mc{F}\in D_{\mr{ctf}}(X,\Lambda)$.
 \begin{enumerate}
  \item The class $[\mc{F}]\in\mr{K}_0\mr{Cons}(X)$ is microsupported on $\mr{SS}(\mc{F})$
	in the sense of Definition {\normalfont\ref{suppstrdfn}}.

  \item We have $\mr{CC}^\mu_{X/\mr{Spec}(k)}([\mc{F}])=\mr{CC}(\mc{F})$.
 \end{enumerate}
\end{cor}
\begin{proof}
 The first claim follows immediately from Proposition \ref{mainpropcomp}.\ref{mainpropcomp-1}.
 For the second claim, we only need to check the Milnor formula (\ref{totdimccdef}) for $\mr{CC}^{\mu}_{X/\mr{Spec}(k)}([\mc{F}])$.
 Let $f\colon U\rightarrow\mb{A}^1$ be a morphism such that $U\subset X\otimes_k l$ for some finite extension $l$ of $k$ as in Definition \ref{dfnSaitocc}.
 Since $\mr{CC}^{\mu}_{X/\mr{Spec}(k)}$ is stable under base change by \'{e}tale extension of $k$, we may assume that $l=k$.
 We wish to show that the left side of (\ref{totdimccdef}) is equal to $\bigl(\mr{d}f,\mr{CC}^\mu_{X/\mr{Spec}(k)}([\mc{F}])\bigr)$.
 Let $Z:=\iota_U^{-1}(\mr{d}f)^{-1}(C)$, using the notation of Definition \ref{dfnSaitocc}, where $C=\mr{SS}(\mc{F})$.
 By definition, $Z$ is of dimension $0$.
 This implies that $\mc{E}'(-,f)$ (cf.\ \ref{microlocconstofE}) induces the homomorphism $\ms{S}(U)_C\rightarrow\ms{S}_0(U)$
 by Proposition \ref{mainpropcomp}.\ref{mainpropcomp-1} and Example \ref{examplegood}.\ref{examplegood-2}.
 Put $V:=U\setminus Z$.
 Consider the following homotopy commutative diagram:
 \begin{equation*}
  \xymatrix@R=10pt@C=40pt{
   &&&z(Z,0)\ar[r]\ar[d]&\Hbm(Z,0)\ar[d]\\
   \phi&\ms{S}(U)_{C}\ar[r]^-{\mc{E}'_{\infty}(-,f)}\ar[d]&
    \ms{S}_0(U)\ar[r]^-{\rk_U}\ar[d]&z(U,0)\ar[r]\ar[d]&
    \Hbm(U,0)\ar[d]^{\mr{res}}\\
  &\ms{S}(V)_{C}\ar[r]^-{\mc{E}'_{\infty}(-,f)}&
    \ms{S}_0(V)\ar[r]^-{\rk_V}&z(V,0)\ar[r]&\Hbm(V,0).\\
  }
 \end{equation*}
 Let $\phi$ be the composite of the 2nd row.
 Since $\mc{E}'_{\infty}([\mc{F}],f)|_V=0$, this yields a homotopy between $\mr{res}\circ\phi([\mc{F}])$ and $0$.
 Since the last column is a cofiber sequence, this homotopy determines an element in
 $\pi_0\Hbm(U\setminus V,0)\cong\mr{CH}_0(U\setminus V)[1/p]\cong\bigoplus_{x\in|U\setminus V|}\mb{Z}[1/p]\cdot[x]$.
 By construction, this element coincides with $\bigl(\mr{d}f,\mr{CC}^\mu_{X/\mr{Spec}(k)}([\mc{F}])\bigr)$.
 Now, by definition, $\mc{E}'([\mc{F}],f)=\left[\Psi_{\infty\leftarrow\eta}(\mc{F}\otimes\mc{L}(ft))\right]$,
 where $\eta$ is the generic point of $\Box$.
 By Proposition \ref{mainpropcomp}.\ref{mainpropcomp-2}, $\rk_U\circ\mc{E}'([\mc{F}],f)$ coincides with the cycle
 $-\sum_{x}\mr{totdim}\bigl(\Phi_f(\mc{F}|_U)_{\overline{x}}\bigr)\cdot[x]$ in $z(U\setminus V)$.
 Thus, the Milnor formula holds for $\mr{CC}^\mu_{X/\mr{Spec}(k)}([\mc{F}])$, which terminates the proof.
\end{proof}

\begin{rem*}
 In view of Theorem \ref{mainresul}, we have the pushforward formula for $\mr{CC}^{\mu}_{X/\mr{Spec}(k)}$.
 Together with Remark \ref{mainresul}, this corollary implies that \cite[Conjecture 1]{S2} is true
 if we invert $\mr{char}(k)$ in $\mr{CH}_d(T^*Y)$, as we promised in the introduction.
\end{rem*}

\appendix
\section{Functoriality of $\infty$-topoi}
\label{functinftop}

\subsection{}
\label{contfunctopos}
Let $\mc{C}$ be an $\infty$-site.
Instead of saying that a subcategory $\mr{R}\subset\mc{C}_{/C}$ is a sieve of $C$ (cf.\ \cite[6.2.2.1]{HTT}),
we often say that $\mr{R}\rightarrow C$ is a sieve for short.
Let us recall a well-known construction of geometric morphisms of $\infty$-topoi.

\begin{dfn*}
 Let $\mc{C}$, $\mc{D}$ be $\infty$-sites.
 Let $f\colon\mc{C}\rightarrow\mc{D}$ be a functor.
 The morphism $f$ is said to be {\em continuous} if any covering sieve $\mr{R}\rightarrow C$ of $C\in\mc{C}$,
 the sieve of $f(C)$ generated by $\{f(r)\rightarrow f(C)\}_{r\in\mr{R}}$ is a covering sieve in $\mc{D}$.
\end{dfn*}

\begin{lem*}[$\infty$-analogue of {[SGA 4, Exp.\ III, 1.6 and 1.3 (5)]}]
 Let $f\colon\mc{C}\rightarrow\mc{D}$ be a continuous morphism between $\infty$-sites.
 \begin{enumerate}
  \item Assume that $\mc{C}$ admits pullbacks, and $f$ commutes with pullbacks.
	The functor $\Shv(\mc{D})\rightarrow\PShv(\mc{C})$
	defined by composing with $f^{\mr{op}}$ factors through $\Shv(\mc{C})$,
	and induces the functor $f^{\mr{s}}\colon\Shv(\mc{D})\rightarrow\Shv(\mc{C})$.

  \item Assume further that $\mc{C}$ admits finite limits, and $f$ commutes with finite limits.
	Then $f^{\mr{s}}$ admits an exact left adjoint, and defines a geometric morphism $\Shv(\mc{D})\rightarrow\Shv(\mc{C})$ of $\infty$-topoi.
 \end{enumerate}
\end{lem*}
\begin{proof}
 The first statement follows by \cite[20.6.1.3]{SAG}, and the second claim follows by \cite[6.2.3.20]{HTT}
 since the composite $\mc{C}\xrightarrow{f}\mc{D}\rightarrow\Shv(\mc{D})$ commutes with finite limits by \cite[5.1.3.2]{HTT}.
\end{proof}

\subsection{}
\label{cocontmorlem}
We introduce the notion of cocontinuous functors between $\infty$-categories, which has already been appeared in some literature.
For an $\infty$-category $\mc{C}$ and $c\in\mc{C}$, we denote by $\widehat{c}$ the image of $c$ by the Yoneda embedding $\mc{C}\rightarrow\PShv(\mc{C})$.

\begin{dfn*}
 Let $\mc{C}$, $\mc{D}$ be $\infty$-sites.
 A functor $f\colon\mc{C}\rightarrow\mc{D}$ is said to be {\em cocontinuous} if for any $c\in\mc{C}$
 and a covering sieve $\mr{R}\hookrightarrow\widehat{f(c)}\simeq f_!\widehat{c}$ of $\mc{D}$,
 the pullback sieve $f^*\mr{R}\times_{f^*f_!\widehat{c}}\widehat{c}\hookrightarrow\widehat{c}$ is a covering sieve of $\mc{C}$.
\end{dfn*}

\begin{rem*}
 The pullback sieve is the sieve corresponding to the full subcategory of $\mc{C}_{/c}$ spanned by arrows $d\rightarrow c$ such that $f(d)\rightarrow f(c)$ belongs to $\mr{R}$.
 This description implies that the functor $f$ is cocontinuous if and only if $\mr{h}f$ is cocontinuous in the sense of [SGA 4, Exp.\ III, 2.1].
\end{rem*}

\begin{lem*}[$\infty$-analogue of {[SGA 4, Exp.\ III, Prop 2.3]}]
 Let $f\colon\mc{C}\rightarrow\mc{D}$ be a cocontinuous functor of $\infty$-sites.
 \begin{enumerate}
  \item The pushforward $f_*$ of presheaves preserves sheaves, and induces a geometric morphism of $\infty$-topoi $F\colon\Shv(\mc{C})\rightarrow\Shv(\mc{D})$
	such that $\iota_{\mc{D}}\circ F_*\simeq f_*\circ\iota_{\mc{C}}$.
	We also have $F^*\simeq L_{\mc{D}}\circ f^*\circ\iota_{\mc{C}}$.
	This $F$ is often denoted simply by $f$, abusing notations.

  \item If, moreover, $\mc{C}$ admits pullbacks and $f$ is continuous and preserves pullbacks, $F^*$ is equivalent to the functor $\circ f$.
 \end{enumerate}
\end{lem*}
\begin{proof}
 For $\mc{G}\in\PShv(\mc{C})$, let $\mc{C}_{/\mc{G}}:=\mc{C}\times_{\PShv(\mc{C})}\PShv(\mc{C})_{/\mc{G}}$.
 For a morphism $f\colon\mc{G}'\rightarrow\mc{G}$ in $\PShv(\mc{C})$,
 put $\mc{C}_{/f}:=\mr{Fun}(\Delta^1,\mc{C})\times_{\mr{Fun}(\Delta^1,\PShv(\mc{C}))}\mr{Fun}(\Delta^1,\PShv(\mc{C}))_{/f}$.
 The functor $\alpha\colon\mc{C}_{/f}\rightarrow\mc{C}_{/\mc{G}}$ sending $c'\rightarrow c$ over $f$ to $c$ is a coCartesian fibration.
 Indeed, $\mr{Fun}(\Delta^1,\PShv(\mc{C}))_{/f}\rightarrow\mc{C}_{/\mc{G}}$ is a coCartesian fibration by \cite[2.4.3.2]{HTT}.
 Furthermore, the description of coCartesian edges in \cite[2.4.3.2]{HTT}, we have the claim.
 Likewise, the functor $\beta\colon\mc{C}_{/f}\rightarrow\mc{C}_{/\mc{G}'}$ sending $c'\rightarrow c$ to $c'$ is Cartesian.
 This functor is (left) cofinal.
 Indeed, $\beta$ admits a left adjoint given by sending $c'$ to $c'\xrightarrow{\mr{id}}c'$.

 We show that $f_*$ preserves sheaves.
 Let $\mc{F}\in\Shv(\mc{C})$, $d\in\mc{D}$, and let $\mr{R}\hookrightarrow\widehat{d}$ be a covering sieve.
 We must show that $\Map_{\PShv(\mc{D})}(\widehat{d},f_*\mc{F})\rightarrow\Map_{\PShv(\mc{D})}(\mr{R},f_*\mc{F})$ is an equivalence.
 Put $j\colon f^*\mr{R}\rightarrow f^*\widehat{d}$, and let $M_0\colon(\mc{C}_{/j})^{\mr{op}}\rightarrow\Spc$ be the functor sending
 $c'\rightarrow c$ to $\Map(\widehat{c}',\mc{F})\simeq\mc{F}(c')$.
 Then the right Kan extension $\alpha_*M_0$ of $M_0$ along $\alpha$ is equivalent to the functor $M_1$ sending $c$ to $\Map(\widehat{c},\mc{F})$.
 Indeed, by universality, we have the morphism $M_1\rightarrow\alpha_*M_0$.
 To show that this is an equivalence, it suffices to check pointwise.
 Since $(\mc{C}_{/j})^{\mr{op}}\rightarrow(\mc{C}_{/f^*\widehat{d}})^{\mr{op}}$ is a Cartesian fibration, we have
 $\alpha_*M_0(c)\simeq\invlim_{\alpha^{-1}(c)}M_0$.
 Since $\alpha^{-1}(c)$ is equivalent to the diagram $(\mc{C}_{/(f^*\mr{R}\times_{f^*\widehat{d}}c)})^{\mr{op}}$ and $\mc{F}$ is assumed to be a sheaf,
 it suffices to show that the sieve $f^*\mr{R}\times_{f^*\widehat{d}}c\hookrightarrow\widehat{c}$ is a covering sieve.
 This follows by the proof of [SGA 4, Exp III, 2.2].
 Now, we may compute
 \begin{align*}
  \Map&(\widehat{d},f_*\mc{F})
  \simeq
  \Map(f^*\widehat{d},\mc{F})
  \simeq
  \invlim_{c\in\mc{C}_{/f^*\widehat{d}}}\Map(\widehat{c},\mc{F})
  \xleftarrow{\sim}
  \invlim_{c\in\mc{C}_{/f^*\widehat{d}}}\alpha_*M_0\\
  &\simeq
  \invlim_{(c'\rightarrow c)\in\mc{C}_{/j}}\Map(\widehat{c}',\mc{F})
  \xleftarrow{\sim}
  \invlim_{c'\in\mc{C}_{/f^*\mr{R}}}\Map(\widehat{c}',\mc{F})
  \simeq
  \Map(f^*\mr{R},\mc{F})
  \simeq
  \Map(\mr{R},f_*\mc{F}).
 \end{align*}
 Here, we used \cite[5.1.5.3]{HTT}.
 Thus, $f_*\mc{F}$ is a sheaf.
 Since $\iota_{\mc{D}}$ is fully faithful, it is easy to check that $L_{\mc{D}}\circ f^*\circ \iota_{\mc{C}}$ is left adjoint to $F_*$.
 Since it is a composite of left exact functors, $F_*$ is left exact as well, and yields a geometric morphism of $\infty$-topoi.
 The second claim follows from Lemma \ref{contfunctopos}.
\end{proof}

Let $1\mbox{-}\site^{\mr{cocont}}$ be the $(2,1)$-category of $1$-sites whose morphism consists of cocontinuous functors.
We let $\infty\mbox{-}\site^{\mr{cocont}}:=1\mbox{-}\site^{\mr{cocont}}\times_{\Cat_{(2,1)},\mr{h}}\Cat_\infty$,
where $\mr{h}\colon\Cat_\infty\rightarrow\Cat_{(2,1)}$ is the functor sending to the homotopy category.
Objects of $\infty\mbox{-}\site^{\mr{cocont}}$ are $\infty$-sites and the morphisms are cocontinuous functors between them.
Now, we have the functor $F\colon(\infty\mbox{-}\site^{\mr{cocont}})^{\mr{op}}\rightarrow\Cat_{\infty}\xrightarrow{\PShv}\LTop\subset\widehat{\Cat}_{\infty}$.
Here, a cocontinuous functor $f$ is sent to $f^*$.
For each $(\mc{C},\tau)\in\infty\mbox{-}\site^{\mr{cocont}}$, the object $F(\mc{C})\simeq\PShv(\mc{C})$ is equipped with a localization functor
$\PShv(\mc{C})\rightarrow\Shv(\mc{C},\tau)$.
By taking unstraightening, consider the coCartesian fibration $\mr{Un}(F)\rightarrow(\infty\mbox{-}\site^{\mr{cocont}})^{\mr{op}}$.
Note that $\infty$-category of $\infty$-categories can be fully faithfully embedded into the $\infty$-category of $\infty$-operads by \cite[2.1.4.11]{HA}.
Using this identification, the fiberwise localization is compatible with $(\infty\mbox{-}\site^{\mr{cocont}})^{\mr{op}}$-monoidal structure by the above lemma,
and we may apply \cite[2.2.1.9]{HA}.
Thus, if we let $\mc{X}\subset\mr{Un}(F)$ to be the full subcategory spanned by objects $\mc{F}$ over $(\mc{C},\tau)\in\infty\mbox{-}\site^{\mr{cocont}}$
such that $\mc{F}\in\Shv(\mc{C},\tau)$, the evident functor $\mc{X}\rightarrow(\infty\mbox{-}\site^{\mr{cocont}})^{\mr{op}}$ is a coCartesian fibration.
By taking straightening, we have a functor $(\infty\mbox{-}\site^{\mr{cocont}})^{\mr{op}}\rightarrow\Cat_{\infty}$.
By construction, this functor factors through $\LTop$, and using \cite[6.3.1.8]{HTT}, we have the functor $\infty\mbox{-}\site^{\mr{cocont}}\rightarrow\RTop$,
which informally the functor sending $(\mc{C},\tau)$ to $\Shv(\mc{C},\tau)$ and $f$ to the geometric morphism $f$ defined by the previous lemma.

\subsection{}
 \label{upggeomor}
Let $\mc{X}$ be an $\infty$-topos (cf.\ \cite[6.1.0.4]{HTT}), and $\mc{D}$ be an $\infty$-category.
Recall from \cite[1.3.1.4]{SAG} that the $\infty$-category of $\mc{D}$-valued sheaves denoted by $\Shv_{\mc{D}}(\mc{X})$
is the full subcategory of $\mr{Fun}(\mc{X}^{\mr{op}},\mc{D})$ spanned by limit preserving functors.
For an $\infty$-site $\mc{C}$ ({\em i.e.}\ small $\infty$-categories endowed with Grothendieck topology \cite[6.2.2.1]{HTT}),
we may define the notion of $\mc{D}$-valued sheaves on $\mc{C}$ by descent property (cf.\ \cite[1.3.1.1]{SAG})
and this forms an $\infty$-category $\Shv_{\mc{D}}(\mc{C})$.
The $\infty$-categories $\Shv_{\mc{D}}(\Shv(\mc{C}))$ and $\Shv_{\mc{D}}(\mc{C})$ do not coincide in general,
but if $\mc{D}$ admits small limits, these two $\infty$-categories coincide by \cite[1.3.1.7]{SAG}.
In this paper, we talk about $\mc{D}$-valued sheaves exclusively in the case where $\mc{D}$ is presentable.
In particular, in that case, $\mc{D}$ admits small limits by \cite[5.5.2.4]{HTT}.
For this reason, we identify the above two $\infty$-categories without further mentioning.

\begin{lem*}
 Let $f\colon\mc{X}\rightarrow\mc{Y}$ be a geometric morphism of $\infty$-topoi, and let $\mc{D}$ be a presentable $\infty$-category.
 The composition with $(f^*)^{\mr{op}}$ induces the functor $f^{\mc{D}}_*\colon\Shv_{\mc{D}}(\mc{X})\rightarrow\Shv_{\mc{D}}(\mc{Y})$,
 and it admits a left adjoint $f^*_{\mc{D}}$. The functors $f^{\mc{D}}_*$, $f^*_{\mc{D}}$ are often denoted by $f_*$, $f^*$ for simplicity.
\end{lem*}
\begin{proof}
 Consider the following diagram of $\infty$-categories
 \begin{equation*}
  \xymatrix{
   \Shv_{\mc{D}}(\mc{Y})\ar@{^{(}->}[r]^-{\iota_{\mc{Y}}}&\mr{Fun}(\mc{Y}^{\mr{op}},\mc{D})\ar@<-.5ex>[d]_{G}\\
  \Shv_{\mc{D}}(\mc{X})\ar@{^{(}->}[r]^-{\iota_{\mc{X}}}&\mr{Fun}(\mc{X}^{\mr{op}},\mc{D})\ar@<-.5ex>[u]_{F}
   }
 \end{equation*}
 where $F:=\circ(f^*)^{\mr{op}}$, and $G$ is a left Kan extension whose existence is assured by \cite[4.3.2.15]{HTT}
 and the assumption that $\mc{D}$ is presentable.
 Thus, by \cite[4.3.3.7]{HTT}, $G$ is left adjoint to $F$.
 Now, since $f^*$ preserves small colimits, its opposite $(f^*)^{\mr{op}}$ preserves small limits.
 This implies that the essential image of $F\circ\iota_{\mc{X}}$ is contained in $\Shv_{\mc{D}}(\mc{Y})$, and yields the desired functor $f^{\mc{D}}_*$.
 The functor $\iota_{\mc{X}}$ admits a left adjoint $L_{\mc{X}}$ by \cite[1.3.4.3]{SAG}.
 Since $\iota_{\mc{Y}}$ is fully faithful, the composite $L_{\mc{X}}\circ G\circ \iota_{\mc{Y}}$ is a left adjoint to $f^{\mc{D}}_*$.
\end{proof}

\begin{test}
  Indeed, given $\mc{F}\in\Shv_{\mc{D}}(\mc{X})$, $\mc{G}\in\Shv_{\mc{D}}(\mc{Y})$, we have
 \begin{equation*}
  \mr{Map}\bigl(L_{\mc{X}}\circ G\circ\iota_{\mc{Y}}(\mc{G}),\mc{F}\bigr)
   \simeq
   \mr{Map}\bigl(\iota_{\mc{Y}}(\mc{G}),\iota_{\mc{Y}}f^{\mc{D}}_*(\mc{F})\bigr)
 \end{equation*}
 by the adjointness, and since $\iota_{\mc{Y}}$ is fully faithful, we get the claim.
\end{test}

\subsection{}
\label{shvcotetc}
We consider the following condition:
\begin{quote}
 (*)\quad
 Let $\mc{A}$ be a closed symmetric monoidal $\infty$-category (cf.\ {\normalfont\cite[4.2.1.32]{HA}})
 and $\mc{D}$ be an $\infty$-category enriched, left-tensored, and cotensored over $\mc{A}$ (cf.\ {\normalfont\cite[4.2.1.28]{HA}}).
 Furthermore, we assume that both $\mc{A}$ and $\mc{D}$ are compactly generated (cf.\ \cite[5.5.7.1]{HTT}), in particular presentable.
\end{quote}
Very informally, these mean that we are endowed with left-tensor functor $\otimes\colon\mc{A}\times\mc{D}\rightarrow\mc{D}$
(as well as great amount of homotopy coherence data),
and for $a\in\mc{A}$, $d,d'\in\mc{D}$, the morphism object $\Mor(d,d')\in\mc{A}$ and cotensor object ${}^ad\in\mc{D}$ are defined.
Let $\mc{X}$ be an $\infty$-topos.
By \cite[2.1.3.4]{HA}, $\mr{Fun}(\mc{X}^{\mr{op}},\mc{D})$ is left-tensored over $\mr{Fun}(\mc{X}^{\mr{op}},\mc{A})$.
This left-tensor structure descends to a left-tensor structure for sheaves as follows:

\begin{lem*}
 For an $\infty$-topos $\mc{X}$, the $\infty$-category $\Shv_{\mc{A}}(\mc{X})$ is closed symmetric monoidal,
 and $\Shv_{\mc{D}}(\mc{X})$ is an $\infty$-category enriched, left-tensored, and cotensored over $\Shv_{\mc{A}}(\mc{X})$.
\end{lem*}
\begin{proof}
 We borrow the argument from \cite[1.3.4.6]{SAG}.
 Let $\otimes\colon\mc{A}\times\mc{D}\rightarrow\mc{D}$ denote the left-tensor functor (cf.\ \cite[4.2.1.19]{HA}).
 The existence of morphism objects and exponential objects implies that $\otimes$ commutes with colimits separately in each variable.
 By \cite[5.5.8.6, 5.5.8.3]{HTT}, the functor $\otimes$ commutes with small {\em filtered} colimits.
 Now, the left-tensor on $\mr{Fun}(\mc{X}^{\mr{op}},\mc{D})$ is pointwise, and can be described as the composite
 \begin{equation*}
   \mr{Fun}(\mc{X}^{\mr{op}},\mc{A})\times\mr{Fun}(\mc{X}^{\mr{op}},\mc{D})
   \cong
   \mr{Fun}(\mc{X}^{\mr{op}},\mc{A}\times\mc{D})
   \xrightarrow{\otimes}
   \mr{Fun}(\mc{X}^{\mr{op}},\mc{D}),
 \end{equation*}
 Note that the product of compactly generated $\infty$-categories remains to be compactly generated by \cite[5.5.7.6]{HTT}.
 By \cite[1.3.4.4]{SAG}, since we are assuming $\mc{A}$, $\mc{D}$ are compactly generated, this functor preserves $L$-equivalences.
 Similarly, the monoidal structure of $\mr{Fun}(\mc{X}^{\mr{op}},\mc{A})$ preserves $L$-equivalence as well.
 Thus, as \cite[2.2.1.7]{HA}, the left-tensor and monoidal structure are compatible with $L$, and
 invoking \cite[2.2.1.9]{HA}, we conclude that $\Shv_{\mc{A}}(\mc{X})$ is a symmetric monoidal $\infty$-category and
 $\Shv_{\mc{D}}(\mc{X})$ is left-tensored over it.
 Moreover, \cite[2.2.1.9]{HA} implies that for $A\in\Shv_{\mc{A}}(\mc{X})$ and $D\in\Shv_{\mc{D}}(\mc{X})$,
 the left-tensor $A\otimes D$ is equivalent to $L(\iota(A)\otimes\iota(D))$,
 where $\iota\colon\Shv_{-}(\mc{X})\rightarrow\mr{Fun}(\mc{X}^{\mr{op}},-)$ is the inclusion functor.
 This presentation implies that the left-tensor $\Shv_{\mc{A}}(\mc{X})\times\Shv_{\mc{D}}(\mc{X})\rightarrow\Shv_{\mc{D}}(\mc{X})$
 commutes with small colimits separately in each variable.
 Since $\Shv_{\mc{D}}(\mc{X})$ and $\Shv_{\mc{A}}(\mc{X})$ are presentable,
 by \cite[4.2.1.33]{HA}, $\Shv_{\mc{D}}(\mc{X})$ is cotensored and enriched.
\end{proof}

\begin{ex*}
 Let $\mc{C}$ be a compactly generated symmetric monoidal $\infty$-category
 and the tensor product preserves small colimits separately in each variable.
 Then for an algebra object $A$ of $\mc{C}$, the $\infty$-category $\mr{RMod}_A(\mc{C})$ is compactly generated.
 In fact, the proof of \cite[5.3.2.12 (3)]{HA} shows this.
\end{ex*}

\begin{test}
 See also \cite[21.1.2.19]{SAG} and \cite[5.3.2.11]{HA}.
\end{test}

\begin{lem}
 \label{pullbackmonoidal}
 We assume {\normalfont(*)} of {\normalfont\ref{shvcotetc}}.
 Assume we are given a geometric morphism $f\colon\mc{X}\rightarrow\mc{Y}$ of $\infty$-topoi.
 Then $f^*_{\mc{A}}\colon\Shv_{\mc{A}}(\mc{Y})\rightarrow\Shv_{\mc{A}}(\mc{X})$
 can be promoted to a monoidal functor and
 $f^*_{\mc{D}}\colon\Shv_{\mc{D}}(\mc{Y})\rightarrow\Shv_{\mc{D}}(\mc{X})$ is monoidal as well
 with respect to the left-tensor structure
 {\normalfont(}more precisely $\mc{LM}^{\otimes}$-monoidal using the notation of {\normalfont\cite[4.2.1.13]{HA}}{\normalfont)}.
\end{lem}
\begin{proof}
 Let $f^\mr{p}_{\mc{D}}\colon\mr{Fun}(\mc{Y}^{\mr{op}},\mc{D})\rightarrow\mr{Fun}(\mc{X}^{\mr{op}},\mc{D})$,
 or $f^{\mr{p}}$ for short, be the left adjoint of the functor $(\circ f^*)$.
 Recall from the proof of Lemma \ref{upggeomor} that $f^*_{\mc{D}}\simeq L_{\mc{X}}\circ f^{\mr{p}}_{\mc{D}}$.
 Assume that $D\rightarrow D'$ is an $L_{\mc{Y}}$-equivalence in $\mr{Fun}(\mc{Y}^{\mr{op}},\mc{D})$.
 For any $E$ in $\Shv_{\mc{D}}(\mc{X})$, we have
 \begin{equation*}
  \Map(f^{\mr{p}}D',E)
   \simeq
   \Map(D',E\circ f^*)
   \xrightarrow{\sim}
   \Map(D,E\circ f^*)
   \simeq
   \Map(f^{\mr{p}}D,E),
 \end{equation*}
 where the second equivalence follows since $E\circ f^*$ remains to be a sheaf.
 By definition (cf.\ \cite[5.5.4.1]{HTT}), the localization functor $L$ is compatible with $f^\mr{p}$,
 namely $f^\mr{p}(D)\rightarrow f^{\mr{p}}(D')$ is $L_{\mc{X}}$-equivalent.

 For $Z\in\{\mc{X},\mc{Y}\}$, let $\mc{C}_{Z}^{\otimes}\rightarrow\mc{LM}^{\otimes}$ be the coCartesian fibration of
 $\infty$-operads which exhibits the pointwise left-tensor structure of $\mr{Fun}(Z^{\mr{op}},\mc{D})$.
 Then we have the localization map $L^{\otimes}_Z\colon\mc{C}_Z^{\otimes}\rightarrow\mc{D}_Z^{\otimes}$ so that $\mc{D}_Z^{\otimes}$
 exhibits the left-tensor structure of $\Shv_{\mc{D}}(Z)$ over $\Shv_{\mc{A}}(Z)$.
 Since $L_Z^{\otimes}$ and the inclusion are maps of $\infty$-operads by \cite[2.2.1.9]{HA},
 we define a map of $\infty$-operads as the composite
 \begin{equation*}
  F\colon
   \mc{D}^{\otimes}_{\mc{Y}}
   \hookrightarrow
   \mc{C}^{\otimes}_{\mc{Y}}
   \xrightarrow{f^{\mr{p},\otimes}}
   \mc{C}^{\otimes}_{\mc{X}}
   \xrightarrow{L^{\otimes}_{\mc{X}}}
   \mc{D}^{\otimes}_{\mc{X}}.
 \end{equation*}
 By construction of $f^*$ in Lemma \ref{upggeomor}, the underlying functor coincides with $f^*_{\mc{A}}\times f^*_{\mc{D}}$.
 We must show that this map preserves coCartesian edges (over $\mc{LM}^{\otimes}$) to conclude the proof.
 It suffices to show that for $A,A'\in\Shv_{\mc{A}}(\mc{Y})$, $D\in\Shv_{\mc{D}}(\mc{Y})$, the morphisms
 \begin{equation*}
  F(A\otimes A')\rightarrow F(A)\otimes F(A'),
   \qquad
  F(A\otimes D)\rightarrow F(A)\otimes F(D)
 \end{equation*}
 are equivalences.
 These follow by the compatibility of $L$ and $f^{\mr{p}}_{\mc{A}}$, $f^{\mr{p}}_{\mc{D}}$.
\end{proof}

\subsection{}
\label{carttopo}
Let $p\colon\mc{X}\rightarrow\mc{C}$ be a Cartesian fibration between $\infty$-categories.
For $x\in\mc{X}$, by taking the $p$-right Kan extension, there exists a functor $a_x\colon\mc{C}_{/p(x)}\rightarrow\mc{X}$ unique up to contractible space of choices
such that $a_x(p(x))=x$ and $a_x$ preserves Cartesian edges (over $\mc{C}$).

\begin{lem*}
 Let $p\colon\mc{X}\rightarrow\mc{C}$ be a Cartesian fibration of $\infty$-categories, and assume $\mc{C}$ admits pullbacks.
 If $\tau$ is a Grothendieck topology on $\mc{C}$, there exists a unique Grothendieck topology, called the {\em $p$-Cartesian topology},
 on $\mc{X}$ characterized by the following property:
 For $x\in\mc{X}$, a family $\{x_i\rightarrow x\}$ is a covering if there exists a refinement $\{y_j\rightarrow x\}$
 such that each edge $y_j\rightarrow x$ is $p$-Cartesian and the induced family $\{p(y_j)\rightarrow p(x)\}$ is a covering family of $\mc{C}$.
 The functor $a_x$ is cocontinuous, continuous, and preserves pullbacks.
 We will denote the associated geometric morphism of topoi by $r_x$.
 A presheaf $\mc{F}\in\PShv(\mc{X})$ is a sheaf if and only if $r_x^*(\mc{F})$ is a sheaf on $\mc{C}_{/p(x)}$ for any $x\in\mc{X}$.
\end{lem*}
\begin{proof}
 The $\infty$-category $\mc{X}$ may not admit pullbacks in general, but any Cartesian edges in $\mc{X}$ are {\em quarrable}.
 In other words, if $f\colon y\rightarrow x$ is a $p$-Cartesian edge in $\mc{X}$, and if we are given a morphism $z\rightarrow x$, the fiber product $y\times_x z$ exists.
 Indeed, in view of \cite[2.4.4.3]{HTT}, $y\times_xz$ is a Cartesian pullback of $z$ along the morphism $p(y)\times_{p(x)}p(z)\rightarrow p(z)$.
 This implies that if $\{x_i\rightarrow x\}$ is a family consisting of quarrable morphisms, $\mr{R}\rightarrow x$ is the sieve generated by the family,
 and $y\rightarrow x$ is a morphism, then the pullback $\mr{R}\times_{\widehat{x}}\widehat{y}$ is the sieve generated by $\{x_i\times_xy\rightarrow y\}$,
 and the verification of the claims are straightforward.
 For the last claim, assume we are given $\mc{F}\in\PShv(\mc{X})$ such that $r_x^*(\mc{F})\in\Shv(\mc{C}_{/p(x)})$ for any $x$.
 We wish to show that the adjunction $\mc{F}\rightarrow L_*L^*\mc{F}$ is an equivalence, where $L\colon\Shv(\mc{X})\rightarrow\PShv(\mc{X})$ is the geometric morphism of topoi.
 Since checking the equivalence is pointwise, it suffices to show this after precomposing with $a_x$.
 By Lemma \ref{cocontmorlem}, $r_x^*$ can be computed by composing with $a_x$, we get the claim.
\end{proof}

\begin{lem}
 \label{commurestpush}
 Let $f\colon\mc{C}'\rightarrow\mc{C}$ be a cocontinuous functor between $\infty$-sites which admits pullbacks
 (but we do not require that $f$ preserves pullbacks).
 Let $p'\colon\mc{X}':=\mc{X}\times_{\mc{C}}\mc{C'}\rightarrow\mc{C}'$ be the Cartesian fibration, with which we endow $\mc{X}'$ with $p'$-Cartesian topology.
 Then the projection $f_{\mc{X}}\colon\mc{X}'\rightarrow\mc{X}$ is cocontinuous.
 If, moreover, $f$ admits a right adjoint $g$, for any $x\in\mc{X}$, we have the homotopy commutative diagram of geometric morphisms:
  \begin{equation*}
  \xymatrix@C=50pt{
   \Shv(\mc{C}'_{/gp(x)})\ar[r]^-{r_{x'}}\ar[d]_{F_{x'}}&\Shv(\mc{X}')\ar[d]^{F_{\mc{X}}}\\
  \Shv(\mc{C}_{/p(x)})\ar[r]^-{r_x}&\Shv(\mc{X}),
   }
 \end{equation*}
 where $x'\in\mc{X}'$ over $gp(x)$ is a point such that $f_{\mc{X}}(x')$ is a Cartesian pullback of $x$ along the morphism $fgp(x)\rightarrow p(x)$.
 Moreover, the morphism $r_x^*\circ F_{\mc{X},*}\rightarrow F_{x',*}\circ r^*_{x'}$ is an equivalence.
\end{lem}
\begin{proof}
 By the description of coverings in Lemma \ref{carttopo}, $f_{\mc{X}}$ is cocontinuous, and the diagram follows by \ref{cocontmorlem} with the aid of \cite[2.4.2.8]{HTT}.
 Let us show the last claim.
 In view of Lemma \ref{cocontmorlem},
 the 4 functors appearing in the base change morphism are compatible with $\iota\colon\Shv(\mc{C})\hookrightarrow\PShv(\mc{C})$.
 Thus, it suffices to show that the corresponding morphism for presheaves is an equivalence.
 Before proving this, we wish to show the following claim.
 Let $a\colon\mc{X}\rightarrow\mc{Y}$ be a functor which admits a right adjoint $b$.
 Then the functor $a_*\colon\PShv(\mc{X})\rightarrow\PShv(\mc{Y})$ is the composition functor $\circ b$.
 Indeed, we have the morphism $a^*\circ(\circ b)\rightarrow\mr{id}$ in $\mr{Fun}(\mc{X}^{\mr{op}},\Spc)$ since $a^*\simeq(\circ a)$.
 Since $a_*$ is defined by Kan extension, we have the morphism of functors $(\circ b)\rightarrow a_*$, which we wish to show to be an equivalence.
 Since the verification is pointwise, we only need to show that for $\mc{F}\in\PShv(\mc{X})$, $y\in\mc{Y}$,
 the induced morphism $\mc{F}(b(y))\rightarrow(a_*\mc{F})(y)$ is an equivalence.
 The latter space is equivalent to $\invlim\bigl((\mc{X}\times_{\mc{Y}}\mc{Y}_{/y})^{\mr{op}}\rightarrow\mc{X}^{\mr{op}}\xrightarrow{\mc{F}}\Spc\bigr)$.
 In view of \cite[5.2.4.2]{HTT}, $b(y)$ is a final object of $\mc{X}\times_{\mc{Y}}\mc{Y}_{/y}$, and the claim follows.

 Now, let us go back to the proof.
 In view of \cite[5.2.5.2]{HTT}, $g$ is a right adjoint to the functor $f_{/x'}$.
 On the other hand, $f_{\mc{X}}$ also admits a right adjoint given by $(\mr{id},g\circ p)\colon\mc{X}\rightarrow\mc{X}\times_{\mc{C}}\mc{C}'\cong\mc{X}'$.
 Indeed, by taking right Kan extension (relative to $\mc{C}$), we may take a functor $h\colon\mc{X}\rightarrow\mc{X}$
 which sends $x$ to the Cartesian pullback to $x$ by the morphism $fgp(x)\rightarrow p(x)$.
 Let $g_{\mc{X}}:=(h,g)\colon\mc{X}\rightarrow\mc{X}\times_{\mc{C}}\mc{C'}=:\mc{X}'$.
 By construction $f_{\mc{X}}\circ g_{\mc{X}}=h$, and we have the morphism of functors $h\rightarrow\mr{id}$, which we wish to show to be a unit transformation.
 By using \cite[2.4.4.3]{HTT}, the verification is straightforward.
 By the description of the adjoints, we have $g_{\mc{X}}\circ a_x\simeq a_{(h(x),gp(x))}\circ g$, and the lemma follows.
\end{proof}

\begin{test}
 For this, we have
 \begin{align*}
   \Map_{\mc{X}'}&\bigl((x',c'),g_{\mc{X}}(x)\bigr)
   \cong
   \Map_{\mc{X}'}\bigl((x',c'),(h(x),gp(x))\bigr)\\
   &\simeq
   \Map_{\mc{X}}(x',h(x))\times_{\Map_{\mc{C}}(p(x'),fgp(x))}\Map_{\mc{C}'}(c',gp(x))\\
   &\simeq
   \bigl(\Map_{\mc{X}}(x',x)\times_{\Map(p(x'),p(x))}\Map(p(x'),fgp(x))\bigr)\times_{\Map_{\mc{C}}(p(x'),fgp(x))}\Map_{\mc{C}'}(c',gp(x))\\
   &\simeq
   \Map_{\mc{X}}(x',x)\times_{\Map(p(x'),p(x))}\Map_{\mc{C}'}(c',gp(x))\\
   &\simeq
   \Map_{\mc{X}}(x',x)\times_{\Map(p(x'),p(x))}\Map_{\mc{C}}(f(c'),p(x))
   \simeq
   \Map_{\mc{X}}(x',x).
 \end{align*}
  Here, we used \cite[2.4.4.3]{HTT} at the 3rd equivalence, $p(x')=f(c')$ at the last equivalence.
\end{test}

\section{Family of morphism objects}
\label{sec-mor}
Assume we are given a  six functor formalism for $S$-schemes as in \ref{recalfixsixthe}.
Fix an $S$-scheme $X$.
Then for each $S$-scheme $T$ and an open subscheme $U\subset X_T$,
we may associate the object $\cH(U/T):=f_{U/T,!}\mbf{1}_U\in\mc{D}(T)$, where $f_{U/T}\colon U\subset X_T\rightarrow T$.
If we are given a morphism $g\colon T'\rightarrow T$, we have an equivalence $g^*\cH(U/T)\simeq\cH(U_{T'}/T')$, and if we are given an immersion $U\subset U'$,
we have a morphism $\cH(U/T)\rightarrow\cH(U'/T)$.
Now, $\mc{D}(T)$ admits {\em morphism objects}: For $\mc{F},\mc{G}\in\mc{D}(T)$, we have $\Mor_{\mc{D}(T)}(\mc{F},\mc{G})$ in $\Mod_R$.
We wish to construct a functor assigning $\Hbm(U/T):=\Mor_{\mc{D}(T)}(\cH(U/T),\mbf{1}_T)$ to $U/T$ as above.
If we are given a morphism $(U/T)\rightarrow (U'/T')$ of such objects,
the functoriality for $\cH$ yields the morphism $\Hbm(U'/T')\rightarrow\Hbm(U/T)$.
In this section, we establish a categorical formalism to get an $\infty$-enhancement of the functor $\Hbm$ from $\cH$, and prove Proposition \ref{concdescM}.

\subsection{}
\label{LKEbclem}
For the reader's convenience, we recall the following property of left Kan extension.
Consider the following homotopy commutative diagram $D$ of $\infty$-categories where $p$ is an inner fibration:
\begin{equation*}
 \xymatrix{
  \mc{D}'\ar[r]^-{\alpha}\ar[d]_{q}&\mc{M}\ar[d]^{p}\\
 \mc{D}\ar[r]\ar@{-->}[ru]&\mc{B}.
  }
\end{equation*}
A {\em $p$-left Kan extension of $\alpha$ along $q$} is a functor $\mr{LKE}_D\colon\mc{D}\rightarrow\mc{M}$ over $\mc{B}$ together with a morphism
$\phi\colon\alpha\rightarrow\mr{LKE}_D\circ q$ in $\mr{Fun}_{\mc{B}}(\mc{D}',\mc{M})$ such that the condition of \cite[4.3.3.2]{HTT} holds.
In view of (the proof of) \cite[5.2.4.1]{HTT}, this condition is equivalent to the induced diagram $(\mc{D}'\times_{\mc{D}}\mc{D}_{/D})^{\triangleright}\rightarrow\mc{M}$
defined by $\alpha$ and the cone point is sent to $\mr{LKE}_D(D)$ being a $p$-colimit diagram for any $D\in\mc{D}$.

\begin{lem*}
 Assume we are given a diagram $D$ as above, and assume that $q$ is a coCartesian fibration.
 Let $h\colon\mc{B}'\rightarrow\mc{B}$ be a functor of $\infty$-categories, and put $D':=D\times_{\mc{B}}\mc{B}'$.
 If a $p$-left Kan extension $\mr{LKE}_D$ exists, then $\mr{LKE}_{D'}$ also exists and $\mr{LKE}_{D'}\simeq\mr{LKE}_D\times_{\mc{B}}\mc{B}'$.
\end{lem*}
\begin{proof}
 Take $D\in\mc{D}$, and let $F\colon\mc{D}'\times_{\mc{D}}\mc{D}_{/D}\rightarrow\mc{D}'\xrightarrow{\alpha}\mc{M}$ be the composite.
 Let $i\colon\mc{D}'_D\hookrightarrow\mc{D}'\times_{\mc{D}}\mc{D}_{/D}$ be the inclusion.
 We claim that the existence of a $p$-colimit of $F$ is equivalent to that of a $p$-colimit of $F\circ i$, and the $p$-colimits are equivalent if they exist.
 To show this, by \cite[4.3.1.7]{HTT}, it suffices to show that $i$ is a cofinal map.
 Since the inclusion $\{D\}\hookrightarrow\mc{D}_{/D}$ is cofinal, the map $i$ is cofinal by \cite[4.1.2.10, 4.1.2.15]{HTT}, and the claim follows.
 In view of \cite[4.3.1.5 (4)]{HTT}, the lemma follows.
\end{proof}

\subsection{}
\label{dualfibint}
Let $p\colon\mc{D}\rightarrow\mc{B}$ be a coCartesian fibration.
Let $\tau\colon\Cat_\infty\rightarrow\Cat_\infty$ be the functor sending $\mc{B}$ to $\mc{B}^{\mr{op}}$.
We have the coCartesian fibration $\widetilde{\mc{D}}:=\mr{Un}(\tau\circ\mr{St}(p))\rightarrow\mc{B}$.
This is a coCartesian fibration whose fiber over $X\in\mc{B}$ is $\mc{D}_X^{\mr{op}}$.
Since $\widetilde{\mc{D}}$ depends on $p$, the notation is ambiguous, and we use the notation unless any confusion may arise.
The construction is functorial for functors which preserve coCartesian edges.
Namely, if we are given a functor $F\colon\mc{D}\rightarrow\mc{D}'$ which preserves coCartesian edges,
we have the induced functor $\widetilde{F}\colon\widetilde{\mc{D}}\rightarrow\widetilde{\mc{D}}'$.
Caution that if $F$ does not preserve coCartesian edges, we generally do not have the induced functor.
We also have a similar construction for Cartesian fibrations.
If $q\colon\mc{E}\rightarrow\mc{B}$ is a Cartesian fibration, we denote by $\widehat{\mc{E}}\rightarrow\mc{B}$
the Cartesian fibration whose fiber over any $b\in\mc{B}$ is $q^{-1}(b)^{\mr{op}}$ constructed similarly.
By construction, we have $(\widetilde{\mc{D}})^{\mr{op}}\simeq\widehat{\mc{D}^{\mr{op}}}$ for a coCartesian fibration $\mc{D}$ over $\mc{B}$.

\subsection{}
\label{relaticotten}
For an $\infty$-category $\mc{B}$, the $\infty$-category of twisted arrows $\mr{Tw}\mc{B}$ is defined in \cite[5.2.1.1]{HA}.
If we are given a categorical equivalence $F\colon\mc{B}\rightarrow\mc{D}$ of $\infty$-categories,
the induced functor $\mr{Tw}(F)$ is also a categorical equivalence.
Indeed, it suffices to show that the induced functor
$\mr{Tw}(\mc{B})\rightarrow\mr{Tw}(\mc{D})\times_{(\mc{D}\times\mc{D}^{\mr{op}})}(\mc{B}\times\mc{B}^{\mr{op}})$
is a categorical equivalence, but since both are right fibrations over $\mc{B}\times\mc{B}^{\mr{op}}$ by \cite[5.2.1.3]{HA},
the verification is fiberwise.
On the other hand, since each fiber is the $\infty$-groupo\"{i}d of mapping spaces, the claim follows.
Thus, since $\Cat_\infty$ is the localization of $\sSet^{\circ}$
({\em i.e.}\ the full subcategory of $\sSet$ spanned by fibrant and cofibrant objects) with respect to categorical equivalences,
$\mr{Tw}$ induces the functor $\mr{Tw}\colon\Cat_\infty\rightarrow\Cat_\infty$.
Let $p\colon\mc{D}\rightarrow\mc{B}$ be a coCartesian fibration.
We put $\mr{Tw}_{\mc{B}}(\mc{D}):=\mr{Un}(\mr{Tw}\circ\mr{St}(p))$, which is the coCartesian fibration over $\mc{B}$
whose fiber over $b\in\mc{B}$ is $\mr{Tw}(\mc{D}_b)$.
This is equipped with a functor $\mr{Tw}_{\mc{B}}(\mc{D})\rightarrow\mc{D}\times_{\mc{B}}\widetilde{\mc{D}}\cong\widetilde{\mc{D}}\times_{\mc{B}}\mc{D}$
of coCartesian fibrations over $\mc{B}$ which preserves coCartesian edges.

\subsection{}
\label{setuptodefmor}
Before we define the object $M_G(m)$ we introduced in \ref{morphobjstat}, we recall the setting.
Let $\mc{B}$ be an $\infty$-category, and let $\mc{A}\rightarrow\mc{B}$, $p\colon\mc{M}\rightarrow\mc{B}$ be coCartesian fibrations.
Assume we are given a left-tensor functor relative to $\mc{B}$, namely a functor $t_{\mc{M}}\colon\mc{A}\times_{\mc{B}}\mc{M}\rightarrow\mc{M}$
over $\mc{B}$ which preserves coCartesian edges.
For $a\in\mc{A}$ and $m\in\mc{M}$ which lie over the same object of $\mc{B}$, we often denote $t_{\mc{M}}(a,m)$ by $a\otimes m$.
We further assume that:
\begin{quote}
 (*) The coCartesian fibrations $\mc{A},\mc{M}\rightarrow\mc{B}$ are presentable fibrations,
 and the left-tensor commutes with colimits in each variable fiberwise.
\end{quote}
Now, let $\mc{E}\rightarrow\mc{B}$ be a coCartesian fibration, and assume we are given functors
$F\colon\widetilde{\mc{E}}\rightarrow\mc{A}$, $G\colon\mc{E}\rightarrow\mc{M}$ over $\mc{B}$, and a section $m\colon\mc{B}\rightarrow\mc{M}$.
Then this induces a functor
\begin{equation*}
 F\boxtimes G\colon
  \mr{Tw}_{\mc{B}}\mc{E}\rightarrow
  \widetilde{\mc{E}}\times_{\mc{B}}\mc{E}\xrightarrow{F\times G}\mc{A}\times_{\mc{B}}\mc{M}\xrightarrow{t_{\mc{M}}}\mc{M}.
\end{equation*}
Let $\mr{tw}_{\mc{E}}\colon\mr{Tw}_{\mc{B}}\mc{E}\rightarrow\mc{B}$ be the canonical functor.
The $p$-left Kan extension of $F\boxtimes G$
along the morphism $\mr{Tw}_{\mc{B}}\mc{E}\rightarrow\mc{B}$, and denoted by $\left<F,G\right>_{\mc{B}}\colon\mc{B}\rightarrow\mc{M}$.

\begin{lem*}
 \label{pairbasicprop}
 We keep the notation.
 \begin{enumerate}
  \item The pairing $\left<-,-\right>_{\mc{B}}$ is compatible with base change.
	More precisely, assume we are given a functor of $\infty$-categories $h\colon\mc{B}'\rightarrow\mc{B}$.
	Then there is a canonical equivalence $\left<-,-\right>_{\mc{B}}\times_{\mc{B}}\mc{B}'\simeq\left<-,-\right>_{\mc{B}'}$.

  \item Let $\mc{M}'$ be a coCartesian over $\mc{B}$ which is left-tensored by $\mc{A}$ relative to $\mc{B}$ and satisfies the assumption of {\normalfont(*)} above.
	Let $f\colon\mc{M}\rightarrow\mc{M}'$ be a functor which preserves the left-tensor, namely equipped with an equivalence
	$f\circ t_{\mc{M}}\simeq t_{\mc{M}'}\circ(\mr{id}_{\mc{A}}\times f)$.
	If, furthermore, the fiber $f_c$ commutes with colimits for each $c\in\mc{B}$, then we have an equivalence $f\circ\left<F,G\right>\simeq\left<F,f\circ G\right>$.
 \end{enumerate}
\end{lem*}
\begin{proof}
 The first claim follows readily by Lemma \ref{LKEbclem}, and the second by the definition of left Kan extension.
\end{proof}

\subsection{}
\label{consadjmorcot}
Let us fix a functor $G\colon\mc{E}\rightarrow\mc{M}$ over $\mc{B}$.
Then we have the functor
\begin{equation*}
 \mr{RT}_G\colon\mr{Fun}_{\mc{B}}(\widetilde{\mc{E}},\mc{A})\rightarrow\mr{Fun}_{\mc{B}}(\mc{B},\mc{M}),
\end{equation*}
sending $F$ to $\left<F,G\right>_{\mc{B}}$.
Let us show that these functors admit right adjoints.
First, $\mr{RT}_G$ commutes with arbitrary colimits.
Indeed, since $\mc{A}\rightarrow\mc{B}$ is assumed to be presentable and in particular a Cartesian fibration, colimits in
$\mr{Fun}_{\mc{B}}(\widetilde{\mc{E}},\mc{A})\cong\mr{Fun}_{\widetilde{\mc{E}}}(\widetilde{\mc{E}},\mc{A}\times_{\mc{B}}\widetilde{\mc{E}})$
can be computed pointwise by \cite[5.1.2.2]{HTT}.
Similarly, colimits in $\mr{Fun}_{\mc{B}}(\mr{Tw}_{\mc{B}}\mc{E},\mc{M})$ can also be computed pointwise.
Thus, the functor $\mr{Fun}_{\mc{B}}(\widetilde{\mc{E}},\mc{A})\rightarrow\mr{Fun}_{\mc{B}}(\mr{Tw}_{\mc{B}}\mc{E},\mc{M})$,
sending $F\colon\widetilde{\mc{E}}\rightarrow\mc{A}$ to the composite
$\mr{Tw}_{\mc{B}}\mc{E}\rightarrow\widetilde{\mc{E}}\times_{\mc{B}}\mc{E}\xrightarrow{F\times G}\mc{A}\times_{\mc{B}}\mc{M}\xrightarrow{t}\mc{M}$,
commutes with colimits since the left-tensor functor commutes with colimits separately in each variable by assumption.
Since the left Kan extension functor commutes with colimits by (obvious relative version of) \cite[4.3.3.7]{HTT}, $\mr{RT}_G$ commutes with colimits.
By \cite[5.5.3.17]{HTT}, both $\mr{Fun}_{\mc{B}}(\widetilde{\mc{E}},\mc{A})$ and $\mr{Fun}_{\mc{B}}(\mc{B},\mc{M})$ are presentable.
Thus, by adjoint functor theorem \cite[5.5.2.9]{HTT}, $\mr{RT}_G$ admits a right adjoint.
We denote this adjoint by $M_G$.
The proof of Proposition \ref{concdescM} occupies the rest of this section.

\begin{rem*}
 For $F\colon\widetilde{\mc{E}}\rightarrow\mc{A}$, we may consider the functor $\mr{LT}_F\colon\mr{Fun}_{\mc{B}}(\mc{E},\mc{M})\rightarrow\mr{Fun}_{\mc{B}}(\mc{B},\mc{M})$
 sending $G$ to $\left<F,G\right>_{\mc{B}}$.
 By a similar, argument, $\mr{LT}_F$ admits a right adjoint ${}^F(\cdot)$.
 A similar results as Proposition \ref{concdescM} hold also for ${^F(\cdot)}$.
\end{rem*}

\begin{lem}
 \label{adjoilemM}
 Let $\alpha\colon\mc{M}\rightarrow\mc{M}'$ be a map of coCartesian fibrations over $\mc{B}$
 which admits a right adjoint $\beta$ relative to $\mc{B}$ in the sense of {\normalfont\cite[7.3.2.2]{HA}}.
 Assume that $\mc{M}$, $\mc{M}'$ are left-tensored over $\mc{A}$ with the left-tensor functors $t_{\mc{M}}$, $t_{\mc{M}'}$,
 and satisfy the conditions in {\normalfont\ref{setuptodefmor}}.
 Assume further that $\alpha$ is compatible with the left-tensor structure, namely
 $\alpha\circ t_{\mc{M}}\simeq t_{\mc{M}'}\circ(\mr{id}_{\mc{A}}\times \alpha)$.
 Then we have the equivalence $M_{G}\bigl(\beta(-)\bigr)\simeq M_{\alpha\circ G}(-)$.
\end{lem}
\begin{proof}
 We note that $\alpha$ preserves coCartesian edges by \cite[7.3.2.6]{HA}, and commutes with colimits relative to $\mc{B}$.
 Since we only need to check the universal property, we leave the details to the reader.
\end{proof}

\subsection{}
Let us start our analysis in the case where $\mc{B}=\Delta^0$.
For objects $m,m'\in\mc{M}$, we have the morphism object $\Mor_{\mc{M}}(m,m')$ in the sense of \cite[4.2.1.28]{HA}.
Namely, we define $\Mor_{\mc{M}}(m,-)$ to be the right adjoint functor of $-\otimes m\colon\mc{A}\rightarrow\mc{M}$.
Arguing as in \cite[4.2.1.31]{HA}, the morphism object is organized into a functor $\Mor_{\mc{M}}\colon\mc{M}^{\mr{op}}\times\mc{M}\rightarrow\mc{A}$
sending $(m,m')$ to $\Mor_{\mc{M}}(m,m')$.

\begin{lem*}
 \label{endcoendadj}
 Let $\mc{B}=\Delta^0$.
 Then we have canonical equivalences $M_G(m)\simeq\Mor_{\mc{M}}(G^{\mr{op}},m)$.
\end{lem*}
\begin{proof}
 Let $F$ be an object in $\mr{Fun}(\mc{E}^{\mr{op}},\mc{A})$.
 The lemma follows by the following sequence of equivalences:
 \begin{align*}
  \Map_{\mr{Fun}(\mc{E}^{\mr{op}},\mc{A})}\bigl(F,\Mor_{\mc{M}}(G,m)\bigr)&\simeq
  \invlim_{(x\rightarrow y)\in\mr{Tw}^{\mr{op}}(\mc{E}^{\mr{op}})}\Map_{\mc{M}}\bigl(F(x),\Mor_{\mc{M}}(G,m)(y)\bigr)\\
  &\simeq
  \invlim_{(x\rightarrow y)\in\mr{Tw}^{\mr{op}}(\mc{E}^{\mr{op}})}\Map_{\mc{M}}\bigl(F(x)\otimes G(y),m\bigr)\\
  &\simeq
  \Map_{\mc{M}}\bigl(\indlim_{(x\rightarrow y)\in\mr{Tw}(\mc{E}^{\mr{op}})}F(x)\otimes G(y),m\bigr)\\
  &\simeq
  \Map_{\mc{M}}\bigl(F\otimes G,m\bigr),
 \end{align*}
 where we used \cite[5.1]{GHN} for the first equivalence.
\end{proof}

\subsection{}
It is a nice occasion to record the following straightforwards corollary of \cite[5.1]{GHN}:

\begin{lem*}
 \label{easycorofGHN}
 Let $\mc{E}$, $\mc{D}$ be $\infty$-categories, let $F,G,H\colon\mc{E}\rightarrow\mc{D}$ be functors,
 and let $\alpha\colon F\rightarrow G$ be a morphism in $\mr{Fun}(\mc{E},\mc{D})$.
 Assume that for any objects $x,y$ in $\mc{E}$, the induced morphism
 $\Map_{\mc{D}}\bigl(G(x),H(y)\bigr)\rightarrow\Map_{\mc{D}}\bigl(F(x),H(y)\bigr)$ is an equivalence.
 Then the induced morphism $\Map_{\mr{Fun}(\mc{E},\mc{D})}(G,H)\rightarrow\Map_{\mr{Fun}(\mc{E},\mc{D})}(F,H)$
 is an equivalence.
\end{lem*}

\subsection{}
\label{colimofrepobj}
We need a dual version of variants of the Yoneda lemma.
Let $\mc{C}$ be an $\infty$-category.
Let $c\in\mc{C}$.
Consider the constant functor $*\colon\Delta^0\rightarrow\Spc$ whose value is $\Delta^0$,
and the constant functor $f_c\colon \Delta^0\rightarrow\mc{C}^{\mr{op}}$ whose value is $c$.
We denote by $\ul{c}\in\PShv(\mc{C})$ the left Kan extension of $*$ along $f_c$.
In other words, $\ul{c}$ is the functor represented by $c$.
Let $\mc{D}$ be a presentable $\infty$-category.
Let $d$ be an object of $\mc{D}$.
Since any presentable $\infty$-category is tensored over $\Spc$ (cf.\ \cite[4.4.4]{HTT}), we may consider the composite
\begin{equation*}
 \mc{C}^{\mr{op}}\cong\mc{C}^{\mr{op}}\times\Delta^0\xrightarrow{\ul{c}\times\{d\}}\Spc\times\mc{D}\xrightarrow{\otimes}\mc{D}.
\end{equation*}
This functor is denoted by $\ul{c}\cdot d$.
A functor in $\mr{Fun}(\mc{C}^{\mr{op}},\mc{D})$ equivalent to a functor of the form $\ul{c}\cdot d$ is called a {\em $\mc{D}$-representable functor}.

\begin{lem*}
 Any object of $\mr{Fun}(\mc{C}^{\mr{op}},\mc{D})$ can be written as a colimit of $\mc{D}$-representable functors.
\end{lem*}
\begin{proof}
 Let $j_{\mc{C}}\colon\mc{C}\hookrightarrow\PShv(\mc{C})$ be the Yoneda embedding (cf.\ \cite[5.1.3.1]{HTT}).
 The presheaf $\ul{c}$ is equivalent to $j_{\mc{C}}(c)$.
 Indeed, in view of \cite[5.1.5.2]{HTT}, we have an equivalence $\Map_{\PShv(\mc{C})}(j_{\mc{C}}(c),\mc{F})\simeq\mc{F}(c)$ for any $\mc{F}\in\PShv(\mc{C})$.
 By \cite[4.3.3.7]{HTT}, $j_{\mc{C}}(c)$ is equivalent to $\ul{c}$.
 Now, consider the composite functor
 \begin{equation*}
  i\colon\mr{Fun}(\mc{C}^{\mr{op}},\mc{D})\xrightarrow{j_{\mc{D}}\circ}
  \mr{Fun}(\mc{C}^{\mr{op}},\PShv(\mc{D}))\cong\mr{Fun}\bigl((\mc{C}\times\mc{D})^{\mr{op}},\Spc\bigr).
 \end{equation*}
 Unwinding the construction, for $d\in\mc{D}$, $c\in\mc{C}$, we have $i(\ul{c}\cdot d)\simeq\ul{(c,d)}$.
 By \cite[5.1.5.3]{HTT}, any object of $\mr{Fun}\bigl((\mc{C}\times\mc{D})^{\mr{op}},\Spc\bigr)$ can be written as a colimit of objects of the form $\ul{(c,d)}$.
 By \cite[1.2.13.7]{HTT} applied to the fully faithful functor $i$, the lemma is proven.
\end{proof}

\begin{lem}[co-Yoneda lemma]
 \label{coYonedalem}
 Let $G\colon\mc{C}\rightarrow\mc{M}$ be a functor between $\infty$-categories,
 and we assume $\mc{M}$ to be presentable and left-tensored over a presentable $\infty$-category $\mc{A}$.
 We assume further that the left-tensor commutes with colimits in each variable.
 Then for any $c\in\mc{C}$ and $a\in\mc{A}$, the composite morphism
 \begin{equation*}
  a\otimes G(c)
   \xrightarrow{\mr{id}_c}
   \bigl(\mr{Map}_{\mc{C}}(c,c)\cdot a\bigr)\otimes G(c)
   \simeq
   \left<\ul{c}(c)\cdot a,G(c)\right>_{\Delta^0}
   \rightarrow
   \left<\ul{c}\cdot a,G\right>_{\Delta^0}
 \end{equation*}
 is an equivalence.
\end{lem}
\begin{proof}
 Let $F\colon\mc{X}\rightarrow\mc{Y}$, $a\colon\mc{X}\rightarrow\mc{X}'$ be functors.
 We denote by $\mr{LKE}_a(F)$ the left Kan extension of $F$ along $a$, and by $\mr{LKE}_{\mc{X}}(F)$ if $\mc{X}'=\Delta^0$.
 Let us consider the following diagram:
 \begin{equation*}
  \xymatrix@C=30pt{
   \mc{T}\ar@{}[rd]|\square\ar[r]^-{t}\ar[d]_{i'}&\Delta^{\{c\}}\times\mc{C}\ar[rr]^-{a\otimes G}\ar[d]^{i}&&\mc{M}.\\
  \mr{Tw}(\mc{C})\ar[r]^-{\mr{tw}}&\mc{C}^{\mr{op}}\times\mc{C}.\ar@/_10pt/@{.>}[rru]&&
   }
 \end{equation*}
 The left Kan extension of $a\otimes G$ along $i$ is equivalent to $(\ul{c}\cdot a)\boxtimes G$.
 Indeed, $\ul{c}\cdot a$ is equivalent to the left Kan extension of the constant functor $\Delta^0\rightarrow\mc{A}$
 values at $a$ along $\Delta^0\rightarrow\mc{C}^{\mr{op}}$ determined by the vertex $c$.
 The claim follows since the left-tensor commutes with colimits in each variable.
 Now, we claim that for any $(f\colon \beta\rightarrow\alpha)\in\mr{Tw}(\mc{C})$,
 the functor $\mr{tw}_{/f}\colon\colon\mr{Tw}(\mc{C})_{/f}\rightarrow(\mc{C}^{\mr{op}}\times\mc{C})_{/(\alpha,\beta)}$ is a trivial fibration.
 Indeed, we need to check the right lifting property with respect to the inclusion $\partial\Delta^n\hookrightarrow\Delta^n$.
 Unwinding the definition, this is equivalent to checking the right lifting property of the unique functor $\mc{C}\rightarrow\Delta^0$
 with respect to the inclusion $K\hookrightarrow\Delta^{2n+1}$,
 where $K$ is the simplicial subset appearing in the proof of \cite[5.2.1.3]{HA} for the inclusion $\Lambda^n_n\hookrightarrow\Delta^n$ (namely $i=n$).
 It is proven there that this inclusion is an inner anodyne, thus the required right lifting property holds.
 Since the functor $\mr{tw}_{/f}$ is categorically equivalent, we have $\mr{tw}\circ\mr{LKE}_i(a\otimes G)\simeq\mr{LKE}_{i'}((a\otimes G)\circ t)$.
 Using these equivalences and the transitivity of Kan extension, we may compute as
 \begin{align*}
  \left<\ul{c}\cdot a,G\right>
  &:=
  \mr{LKE}_{\mr{Tw}(\mc{C})}\bigl(\mr{tw}\circ((\ul{c}\cdot a)\boxtimes G)\bigr)
  \simeq
  \mr{LKE}_{\mr{Tw}(\mc{C})}\bigl(\mr{tw}\circ\mr{LKE}_i(a\otimes G)\bigr)\\
  &\simeq
  \mr{LKE}_{\mr{Tw}(\mc{C})}(\mr{LKE}_{i'}((a\otimes G)\circ t))
  \simeq
  \mr{LKE}_{\mc{T}}((a\otimes G)\circ t).
 \end{align*}
 The object $(\mr{id}_c\colon c\rightarrow c)\in\mc{T}$ is a final object of $\mc{T}$.
 Thus, $\mr{LKE}_{\mc{T}}\bigl((a\otimes G)\circ t\bigr)\simeq\indlim_{\mc{T}}\bigl((a\otimes G)\circ t\bigr)\simeq\bigl((a\otimes G)\circ t\bigr)(\mr{id}_c)$,
 and the lemma follows.
\end{proof}

\subsection{Proof of Proposition \ref{concdescM}}\label{proofofprop}\mbox{}\\
Proposition \ref{concdescM}.\ref{concdescM-1} follows by combining Lemma \ref{endcoendadj} and \ref{concdescM}.\ref{concdescM-3}.
Thus, we concentrate on showing \ref{concdescM-2} and \ref{concdescM-3}.
We first reduce \ref{concdescM-3} to the case where $\mc{B}=\Delta^1$ and $\mc{B}'=\{0\}$.
Let us recall the notation.
Let the notation be as in {\normalfont\ref{consadjmorcot}}, and assume $G$ preserves coCartesian edges.
Let $h\colon\mc{B}'\rightarrow\mc{B}$ be a functor, and let $\mc{M}'$, $\mc{A}'$, $\mc{E}'$ be the base changes to $\mc{B}'$.
We also denote by $G'$, $m'$ the base changes of $G$, $m$.
By the base change property of $\left<-,-\right>$, we have a canonical morphism $M_G(m)\times_{\mc{B}}\mc{B}'\rightarrow M_{G'}(m')$,
which we must show to be an equivalence.
This morphism is functorial in $G$ and $m$ in an evident sense.
Let $p_{\mc{A}}\colon\mc{A}\rightarrow\mc{B}$.
We have the $p_{\mc{A}}$-left Kan extension functor
\begin{equation*}
 h_!\colon
  \mr{Fun}_{\mc{B}'}(\widetilde{\mc{E}}',\mc{A}')\cong
  \mr{Fun}_{\mc{B}}(\widetilde{\mc{E}}',\mc{A})\rightarrow
  \mr{Fun}_{\mc{B}}(\widetilde{\mc{E}},\mc{A}).
\end{equation*}
By abusing the notation, we denote by $h_!$ also for the left Kan extension
$\mr{Fun}_{\mc{B}'}(\mc{B}',\mc{M}')\rightarrow\mr{Fun}_{\mc{B}}(\mc{B},\mc{M})$ defined similarly.
The right adjoint, which is merely the restriction to $\widetilde{\mc{E}}'$ or to $\mc{B}'$, is denoted by $h^*$.
Let $\mr{RT}_{G'}\colon\mr{Fun}_{\mc{B}'}(\widetilde{\mc{E}}',\mc{A}')\rightarrow\mr{Fun}_{\mc{B}'}(\mc{B}',\mc{M}')$
be the functor sending $F'$ to $\left<F',G'\right>$.
By universal property of left Kan extension, we have a morphism $\mr{RT}_{G'}\rightarrow h^*\circ\mr{RT}_G\circ h_!$, and by taking the adjoint,
we have the morphism $\pi_h\colon h_!\circ\mr{RT}_{G'}\rightarrow\mr{RT}_G\circ h_!$.
By adjunction, the proposition is equivalent to showing the equivalence of $\pi_h$, which we will prove.

Assume first that $h$ is a coCartesian fibration.
In this case, for $F\in\mr{Fun}_{\mc{B}'}(\widetilde{\mc{E}}',\mc{A}')$ and $e\in\widetilde{\mc{E}}$,
we have $h_!F(e)\simeq\indlim\bigl(F|_{\widetilde{\mc{E}}'_{e}}\colon(\mc{E}'_e)^{\mr{op}}\rightarrow\mc{A}_b\bigr)$
where $b$ is the image of $e$ in $\mc{B}$ by \cite[4.3.1.9, 4.3.1.10]{HTT}.
By construction, we have the equivalence $\mr{Tw}_{\mc{B}'}\mc{E}'\simeq\mr{Tw}_{\mc{B}}\mc{E}\times_{\mc{B}}\mc{B}'$.
Let $h^{\mr{tw}}\colon\mr{tw}_{\mc{B}'}\mc{E}'\rightarrow\mr{Tw}_{\mc{B}}\mc{E}$ be the canonical map.
We have the canonical morphism $h^{\mr{tw}}_!(F\boxtimes h^*G)\rightarrow h_!F\boxtimes G$.
This morphism is an equivalence.
Indeed, by replacing $\mr{Tw}_{\mc{B}'}\mc{E}'$ with $\mr{Tw}_{\mc{B}}\mc{E}\times_{\mc{B}}\mc{B}'$, we may assume that $h^{\mr{tw}}$ is a coCartesian fibration.
In this case, we may compute the both sides fiberwise, and the equivalence follows by the commutativity of left-tensor with colimits.

Since the verification is pointwise, and the construction of the morphism is functorial with respect to the composition of $h$,
it suffices to show the claim for $\mc{B}'=\Delta^0$.
Let $\phi\in\mc{B}$ be the vertex determined by $h\colon\Delta^0\rightarrow\mc{B}$.
Then $h$ factors through the left fibration $\alpha\colon\mc{B}_{\phi/}\rightarrow\mc{B}$.
We have already showed that $\pi_\alpha$ is an equivalence, so it suffice to show the claim for the inclusion $\{\phi\}\hookrightarrow\mc{B}_{\phi/}$.
In particular, we may assume that $\phi$ is an initial object.
Now, assume $h$ factors through a full subcategory $\mc{B}_0\subset\mc{B}$, and let $h_0\colon\mc{B}'\rightarrow\mc{B}_0$.
Then for $F\in\mr{Fun}_{\mc{B}'}(\widetilde{\mc{E}}',\mc{A}')$, $(h_!F)|_{\mc{B}_0}\simeq h_{0!}F$, and similarly for an object in $\mr{Fun}_{\mc{B}'}(\mc{B}',\mc{M})$.
We are left to show that $\pi_{\alpha}(b)$ is an equivalence for each $b\in\mc{B}$.
Let $e\colon \Delta^1\rightarrow\mc{B}$ be the functor determined by an edge connecting $\phi$ and $b$.
Since $\phi$ is an initial object, $e$ is fully faithful.
By applying the observation for $\mc{B}_0=\Delta^1$, it suffices to show the claim for the case where $h$ is the inclusion $\{0\}\hookrightarrow\Delta^1$.

To finish the proof, we show \ref{concdescM-3} in the case $\{0\}\hookrightarrow\Delta^1$ and \ref{concdescM-2}.
Let $\rho\colon\mc{M}_0\rightarrow\mc{M}_1$ be the functor associated to the coCartesian fibration $\mc{M}\rightarrow\mc{B}=\Delta^1$.
By Lemma \ref{colimofrepobj}, we may write $\mr{Fun}(\mc{E}^{\mr{op}}_0,\mc{A}_0)$ as a colimit of $\mc{A}_0$-representable sheaves.
Since $\mr{RT}$ and $h_!$ are left adjoints, they commute with arbitrary colimit.
Thus, it suffices to show that the morphism
$\beta\colon h_!\circ\mr{RT}_{G_0}(a\cdot\ul{e})\rightarrow\mr{RT}_G\circ h_!(a\cdot\ul{e})$ is an equivalence for $a\in\mc{A}_0$, $e\in\mc{E}_0$,
and we only need to show the equivalence for the value at $1\in\Delta^1$.
By the co-Yoneda lemma \ref{coYonedalem},
we have $\mr{RT}_{G_0}(a\cdot\ul{e})\simeq a\otimes G_0(e)$, thus $(h_!\circ\mr{RT}_{G_0}(a\cdot\ul{e}))(1)\simeq \rho_{\mc{M}}(a\otimes G_0(e))$.
We also have $(\mr{RT}_G\circ h_!(a\cdot\ul{e}))(1)\simeq\mr{RT}_{G_1}(a\cdot\rho_{\mc{E}}(e))\simeq a\otimes G_1(\rho_{\mc{E}}(e))$.
The morphism $\beta(1)\in\mr{Fun}(\Delta^1,\mc{M})$ is equivalent to the morphism $\rho_{\mc{M}}(a\otimes G_0(e))\rightarrow a\otimes G_1(\rho_{\mc{E}}(e))$
induced by the morphism $\rho_{\mc{M}}\circ G_0\rightarrow G_1\circ\rho_{\mc{E}}$ determined by the morphism $G$ of coCartesian fibrations over $\Delta^1$.
The assumption that $G$ preserves coCartesian edges readily implies that this morphism is an equivalence, as required.
The claim \ref{concdescM-2} follows by unwinding the construction.
\qed

\section{Enhanced cohomological operations and motivic cohomology}
\label{EnCoOp}

\subsection{}
\label{introsixfuncform}
Let $R$ be a commutative ring, and $S$ be a noetherian scheme over $\mb{F}_p$.
Let $\LinCat_R$ be the $\infty$-category of $R$-linear $\infty$-categories, or more precisely the underlying $\infty$-category of $\mbf{LinCat}_R$
(which is equivalent to $\mr{RMod}_{\Mod_R}^{\mb{A}_\infty}(\PrL)$ by \cite[3.11]{ABiv} using the notation therein).
We fix a functor $\mc{D}^*\colon\mr{SCH}_{/S}^{\mr{op}}\rightarrow\LinCat_R$ between $\infty$-categories such that for each edge $e$ of
$\mr{SCH}_{/S}^{\mr{op}}$, $\mc{D}(e)$ admits a right adjoint.
We assume that the restriction of this functor to $\mr{Sch}_{/S}$ admits an extension to $(\infty,2)$-functor
$\mbf{D}^*_!\colon\mbf{Corr}(S)^{\mr{prop}}_{\mr{sep};\mr{all}}\rightarrow\mbf{LinCat}^{\mr{2-op}}_R$
which enhances a motivic category of coefficients (cf.\ \cite[6.3]{ABiv} for the details),
and $\mbf{D}^*_!(T)$ is a {\em compactly generated} symmetric monoidal $\infty$-category for each $T\in\mr{Sch}_{/S}$.
Furthermore, we assume that the theory is weakly $\tau$-continuous for some twist $\tau$ in the sense of \cite[4.3.2]{CD}
(see also Remark \ref{recalfixsixthe}).

We do not try to recall the definition of these terminologies, but informally, this means that we have an $\mb{A}^1$-invariant 6-functor formalism.
In particular, for any $T\in\mr{SCH}_{/S}$, an $R$-linear $\infty$-category $\mc{D}_T$ is associated, and given a morphism $f\colon T\rightarrow T'$ in $\mr{Sch}_{/S}$,
we have functors $f_*,f_!\colon\mc{D}_T\rightarrow\mc{D}_{T'}$ and $f^*,f^!\colon\mc{D}_{T'}\rightarrow\mc{D}_T$ with usual adjointness properties.
The functors $f_*$ and $f^*$ are defined also for $f$ in $\mr{SCH}_{/S}$.
We have a morphism of functors $f_!\rightarrow f_*$ and this is an equivalence if $f$ is proper.
If we are given a closed immersion $i\colon Z\hookrightarrow T$, let $j\colon T\setminus Z\hookrightarrow T$ be the complement.
Then we have the localization cofiber sequence $j_!j^*\rightarrow\mr{id}\rightarrow i_*i^*$ in $\mc{D}_T$.
Now, the functor $\mc{D}^*$ sends $T\in\mr{SCH}_{/S}$ to $\mc{D}_T$ and $f\colon T\rightarrow T'$ in $\mr{Sch}_{/S}$ to $f^*$.
By unstraightening this functor, we have a coCartesian fibration $\mc{D}\rightarrow\mr{SCH}_{/S}^{\mr{op}}$.
Since $f^*$ admits the right adjoint $f_*$, this coCartesian fibration is a Cartesian fibration as well by \cite[5.2.2.5]{HTT}.
Since the fibers are presentable, the fibration is in fact a presentable fibration.
Since the fibration is deduced from a functor $\mr{SCH}_{/S}^{\mr{op}}\rightarrow\LinCat_R$, $\mc{D}$ is left-tensored over $\Mod_R$.
By definition, morphisms of $\mbf{LinCat}_R$ commute with the left-tensor.
Thus, the left-tensor $\Mod_R\times\mc{D}\rightarrow\mc{D}$ preserves coCartesian edges.
Since any functor which belongs to $\PrL$ commutes with small colimits by definition,
the left-tensor commutes with small colimits separately in each variable.
In particular, the presentable fibrations
$\mc{D}\rightarrow\mr{SCH}^{\mr{op}}_{/S}$ and $\Mod_R\times\mr{SCH}_{/S}^{\mr{op}}\rightarrow\mr{SCH}_{/S}^{\mr{op}}$
satisfy the requirements of \ref{setuptodefmor}, and we may apply results of the previous section.

\begin{ex*}
 Either of the following cohomology theories satisfy the requirements of this paragraph.
 \begin{enumerate}
  \item  Let $R$ be a torsion ring in which $p$ is invertible. 
	 Then the derived categories of (\'{e}tale) constructible $R$-modules form a motivic category of coefficients and it admits an enhancement.
	 See \cite[6.6]{ABiv} for a detail.

  \item  For a separated noetherian scheme of finite dimension $S$ over $k$, Voevodsky constructed the stable homotopy $\infty$-category $\mc{SH}(S)$.
	 For a commutative ring $R$, he also constructed the ``motivic Eilenberg-MacLane spectrum'' $\mb{H}R_S$, which is an algebra object in $\mc{SH}(S)$.
	 The categories $\mc{SH}(S)$ admit a six functor formalism, and in particular, if we are given a morphism $f\colon T\rightarrow S$,
	 we may consider the pullback $\mb{H}R_{T/S}:=f^*\mb{H}R_S$.
	 We put $\mc{D}_T:=\Mod_{\mb{H}R_{T/S}}(\mc{SH}(T))$.
	 Since $\mc{SH}(T)$ is compactly generated (cf.\ \cite[1.4.4.3]{HA}), so is $\mc{D}_T$ by Example \ref{shvcotetc}.
	 These categories are organized into a six functor formalism.
	 The construction is due essentially to Voevodsky, Ayoub, Cisinski-D\'{e}glise, and Gaitsgory-Rozenblyum,
	 and the details as well as references can be found in \cite[6.8]{ABiv}.
	 In particular, if $S=\mr{Spec}(k)$ where $k$ is a perfect field and $\mr{char}(k)$ is invertible in $R$,
	 then the associated Borel-Moore homology (of certain degree)\footnote{
	 In general, Borel-Moore homology should coincide with Bloch's higher Chow group,
	 but we could not find a reference.}
	 coincides with Chow group (see, for example, \cite[4.1]{Atr}).
	 The continuity follows by \cite[2.6, 2.11]{CD2}.
 \end{enumerate}

\end{ex*}

\subsection{}
\label{compsuppcohrec}
We may extract a fundamental functor from the six functor formalism.
Let $\mr{Ar}^{\mr{prop}}_S$ be the subcategory of $\mr{Fun}(\Delta^1,\mr{Sch}_{/S})$ such that the objects are the same and a morphism
$(X\rightarrow T)\rightarrow (X'\rightarrow T')$ belongs to $\mr{Ar}^{\mr{prop}}_S$ if the induced morphism $X\rightarrow X'\times_{T'}T$ is {\em proper}.
In this paragraph and the next, we let $\mr{Ar}:=\mr{Ar}^{\mr{prop}}_S$.
In \cite[5.19]{ABiv}, we constructed a map $\cH\colon\mr{Ar}^{\mr{op}}\rightarrow\mc{D}$ of coCartesian fibrations over
$\mr{Sch}_{/S}^{\mr{op}}$ which preserves coCartesian edges.
Informally, this functor associates to a morphism $f\colon X\rightarrow Y$ in $\mr{Ar}$ the object $f_!f^*R_Y$.
The construction of $\cH$ is involved since we need the full strength of Gaitsgory-Rozenblyum's formalism,
and is one of the main themes of \cite{ABiv}.

\subsection{}
\label{tracemapdfn}
We put $\Hbm:=M_{\cH}(R)\colon\widehat{\mr{Ar}}\vphantom{Ar}^{\mr{op}}\rightarrow\Mod_R$ using the notation of \ref{consadjmorcot}.
This is the {\em bivariant homology functor}.
Since $\cH$ preserves coCartesian edges, this functor sends $X\rightarrow Y$ to
$\Hbm(X/Y):=\Mor_{\mc{D}_Y}\bigl(\cH(X/Y),\mbf{1}_Y\bigr)$ by Proposition \ref{concdescM}.\ref{concdescM-3}, Lemma \ref{endcoendadj}, and moreover,
Proposition \ref{concdescM}.\ref{concdescM-2} shows that it serves as an $\infty$-enhancement of the ordinary Borel-Moore homology.
Recall the convention from \ref{dfnofrk0assmap}.
By the results of Suslin-Voevodsky, the association $z(X/T,0)\otimes R$ to $(X\rightarrow T)\in\mr{Ar}$ can be promoted to a functor
$\widehat{\mr{Ar}}\vphantom{Ar}^{\mr{op}}\rightarrow\Mod_R$ (cf.\ \cite[6.2]{Atr}).
This functor is denoted by $z(0)$.
The cohomology theory is said to {\em admit trace maps}\footnote{More precisely, trace maps {\em of degree $0$}.}
if there exists a morphism of functors $z(0)\rightarrow\Hbm$ in $\mr{Fun}(\widehat{\mr{Ar}}\vphantom{Ar}^{\mr{op}},\Mod_R)$.

\begin{ex*}
 When $p$ is invertible in $R$, the cohomology theories in Example \ref{introsixfuncform} admit trace maps.
 For the motivic case, this follows from the main result of \cite[6.4]{Atr}
 because the assumption of \cite[6.3]{Atr} is satisfied for our $\Hbm$.
 For the \'{e}tale case, this follows by considering \cite[3.6]{Atr}.
\end{ex*}

\subsection{}
\label{introopenbm}
From now on, we prove some basic functorial property of (co)homology arise from the six functor formalism.
A goal is to show Lemma \ref{compbmcohcoh}.
The claim is straightforward if we consider in the homotopy category, but to make it homotopy coherent,
it requires some intricate argument as is usual the case in the higher category theory.

Let $\Corr$ be the category whose objects are the same as those of $\mr{Sch}_{/S}$ and a morphism $Y'\rightarrow Y$ is a correspondence,
namely a diagram in $\mr{Sch}_{/S}$ of the form $Y'\leftarrow Y''\rightarrow Y$.
We have the canonical functor $\Corr\rightarrow\mbf{Corr}(S)^{\mr{prop}}_{\mr{sep};\mr{all}}$.
Now, put $\mr{All}:=\mr{Fun}(\Delta^1,\mr{Sch}_{/S})^{\wedge}$.
Let $\mr{All}^{\mr{op}}\rightarrow\Corr$ be the functor sending $Y\rightarrow T$ to $Y$
and the morphism $F\colon(Y\rightarrow T)\rightarrow (Y'\rightarrow T')$ in $\mr{All}$
to the correspondence $Y'\xleftarrow{g}Y'_{T'}\xrightarrow{f}Y$ where $g$ is the projection.
Composing with the underlying $(\infty,1)$-functor of $\mbf{D}^*_!$, we get a functor $\mr{All}^{\mr{op}}\rightarrow\PrL$.
By unstraightening, we have the coCartesian fibration $D\colon\mc{D}^*_!\rightarrow\mr{All}^{\mr{op}}$
of coCartesian fibrations over $\mr{Sch}_{/S}^{\mr{op}}$.
Informally, the fiber over $(Y\rightarrow T)$ is $\mc{D}_Y$, and given a morphism $F$ as above, the associated functor is
$\mc{D}_{Y'}\xrightarrow{g^*}\mc{D}_{Y'_T}\xrightarrow{f_!}\mc{D}_Y$.

Now, we use the notations of \ref{suppstrdfn}.
Let $h\colon X\rightarrow S$ be a morphism of finite type between noetherian separated schemes.
Let $\mc{U}\colon\widehat{(X/S)}\vphantom{)}_{\mr{zar}}^{\mr{op}}\rightarrow\mr{All}^{\mr{op}}$ be the functor sending $U/T$ to $U/T$,
and let $b_X\colon\widehat{(X/S)}\vphantom{)}^{\mr{op}}_{\mr{zar}}\rightarrow\mr{Sch}_{/S}^{\mr{op}}$ be the functor sending $(U/T)$ to $X_T$.
Unwinding the definition, we have the following commutative diagram on the left
\begin{equation*}
 \xymatrix{
  \mc{U}^*\mc{D}^*_!\times\{0\}\ar[rr]\ar[d]&&\mc{D}^*_!\ar[d]^{D}\\
 \mc{U}^*\mc{D}^*_!\times\Delta^1\ar[r]\ar@{-->}[urr]&
  \mr{All}^{\mr{op}}\times\Delta^1\ar[r]^-{b}&
  \mr{All}^{\mr{op}},
  }\qquad
  \xymatrix{
  \widehat{(X/S)}\vphantom{)}^{\mr{op}}_{\mr{zar}}\times\{0\}\ar[r]^-{F}\ar[d]&
  \mc{D}^*_!\ar[d]^{D}\\
 \widehat{(X/S)}\vphantom{)}^{\mr{op}}_{\mr{zar}}\times\Delta^1\ar[r]^-{c}&\mr{All}^{\mr{op}},
  }
\end{equation*}
where $b$ sends $(U/T,0)$ to $(U/T)$ and $(U/T,1)$ to $(X_T/T)$.
By taking a $D$-left Kan extension and restricting to $\mc{U}^*\mc{D}^*_!\times\{1\}$,
we have the functor $J_!\colon\mc{U}^*\mc{D}^*_!\rightarrow b_X^*\mc{D}$, which preserves coCartesian edges.
The fiber of this map over $U/T\in\widehat{(X/S)}\vphantom{)}^{\mr{op}}_{\mr{zar}}$ is equivalent to the functor $j_!\colon\mc{D}_U\rightarrow\mc{D}_{X_T}$
where $j\colon U\hookrightarrow X_T$ is the open immersion.
Thus, invoking \cite[7.3.2.6]{HA}, $J_!$ admits a right adjoint $J^*$ relative to $\widehat{(X/S)}\vphantom{)}^{\mr{op}}_{\mr{zar}}$.
The section $R\colon\widehat{(X/S)}\vphantom{)}^{\mr{op}}_{\mr{zar}}\rightarrow b_X^*\mc{D}$, sending $(U/T)$ to $R_{X_T}$, induces the functor
$F:=J^*\circ R\colon\widehat{(X/S)}\vphantom{)}^{\mr{op}}_{\mr{zar}}\rightarrow\mc{U}^*\mc{D}^*_!$.
Now, consider the right diagram above, where $c$ is the functor sending $(U/T,0)$ to $U/T$ and $(U/T,1)$ to $T/T$.
We take a $D$-left Kan extension of the right diagram above.
We restrict this left Kan extension to $\widehat{(X/S)}\vphantom{)}^{\mr{op}}_{\mr{zar}}\times\{1\}$, which we denote by
$\cH\colon\widehat{(X/S)}\vphantom{)}^{\mr{op}}_{\mr{zar}}\rightarrow\mc{D}^*_!$.
Informally, this functor sends $U/T$ to $\cH(U/T)\in\mc{D}_T$.
By construction, this functor factors through $\mc{D}^*_!\times_{\mr{All}^{\mr{op}}}\mr{Sch}_{/S}^{\mr{op}}$,
where $\mr{Sch}_{/S}\rightarrow\mr{All}$ sends $T$ to $T\rightarrow T$.
Thus, we have
\begin{equation*}
 {}^{\mr{open}}\cH\colon\widehat{(X/S)}\vphantom{)}^{\mr{op}}_{\mr{zar}}\xrightarrow{\cH}
 \mc{D}^*_!\times_{\mr{All}^{\mr{op}}}\mr{Sch}_{/S}^{\mr{op}}
  \simeq
  \mc{D}
\end{equation*}
This functor preserves coCartesian edges over $\mr{Sch}^{\mr{op}}_{/S}$.
We finally put
\begin{equation*}
 {}^{\mr{open}}\Hbm:=M_{{}^{\mr{open}}\cH}(R)
  \colon
  (X/S)_{\mr{zar}}^{\mr{op}}\rightarrow\Mod_R.
\end{equation*}

\subsection{}
\label{compbmcohcoh}
Let $\mc{F}$ be an object of $\mc{D}_S$.
Consider the Cartesian fibration $q\colon\widetilde{\mc{D}}^{\mr{op}}\rightarrow\mr{Sch}_{/S}$.
Assume we are given a functor $F\colon\mc{C}\rightarrow\mr{Sch}_{/S}$ between $\infty$-categories.
Recall we have the section $R\colon\mr{Sch}_{/S}\rightarrow\widetilde{\mc{D}}^{\mr{op}}$ sending $T/S$ to $R_T$.
Consider the following diagram
\begin{equation*}
 \xymatrix{
  \mc{C}\times\{0\}\ar[d]\ar[r]^-{R\circ F}&\widetilde{\mc{D}}^{\mr{op}}\ar[d]^{q}\\
 \mc{C}\times\Delta^1\ar[r]&\mr{Sch}_{/S},
  }
\end{equation*}
where the lower horizontal map sends $\mc{C}\times\{1\}$ to $S$.
We {\em assume} that a $q$-left Kan extension of $R\circ F$ exists, which we denote by $\mr{LKE}(R\circ F)$.
Then we denote by $\rH(F,\mc{F})$ the composite
\begin{equation*}
 \mc{C}^{\mr{op}}
  \xrightarrow{\mr{LKE}(R\circ F)|_{\mc{C}\times\{1\}}}
  (\widetilde{\mc{D}}^{\mr{op}}_S)^{\mr{op}}
  \simeq
  \mc{D}^{\mr{op}}_S
  \xrightarrow{\mr{id}\times\{\mc{F}\}}
  \mc{D}^{\mr{op}}_S\times\mc{D}_S
  \xrightarrow{\Mor}
  \Mod_R.
\end{equation*}
Let $\mr{Sm}_S$ be the full subcategory of $\mr{Sch}_{/S}$ spanned by smooth morphisms $T\rightarrow S$,
and let $\iota\colon\mr{Sm}_S\rightarrow\mr{Sch}_{/S}$ be the inclusion.
In this situation, $\rH(\iota,\mc{F})$ is defined.
Indeed, for $f\colon T\rightarrow S\in\mr{Sm}_S$ and $\mc{F}\in\mc{D}_T$, it suffices to check the existence of a $q$-coCartesian edge
$\mc{F}\rightarrow f_{\sharp}\mc{F}$ over $f$.
This follows by \cite[1.5]{ABiv}.

\begin{lem*}
 Let $h\colon X\rightarrow S$ be as before, and let $\omega_{X/S}:=h^!R_S$ in $\mc{D}_X$.
 Let $(X/S)^{\mr{open}}_{\mr{sm}}$ be the full subcategory of $(X/S)_{\mr{zar}}$ spanned by objects $U/T$ such that $T\rightarrow S$ is smooth,
 and let $\tau\colon(X/S)^{\mr{open}}_{\mr{sm}}\rightarrow\mr{Sch}_{/X}$ be the functor sending $U/T$ to $U\subset X_T\rightarrow X$.
 Then we have the equivalence $\rH(\tau,\omega_{X/S})\simeq{}^{\mr{open}}\Hbm|_{(X/S)^{\mr{open}}_{\mr{sm}}}$.
\end{lem*}

\begin{rem*}
 The equivalence $\mr{h}\rH(\tau,\omega_{X/S})\simeq\mr{h}({}^{\mr{open}}\Hbm|_{(X/S)^{\mr{open}}_{\mr{sm}}})$
 of functors between the homotopy categories is an elementary exercise.
 Indeed, assume we are given $U/T\in(X/S)^{\mr{open}}_{\mr{sm}}$,
 which yields the following diagram
 \begin{equation*}
  \xymatrix{
   U\ar[d]\ar@{^{(}->}[r]^-{j}&X_T\ar[r]^{g'}\ar[d]_{f_T}\ar@{}[rd]|\square&X\ar[d]^{f}\\
  T\ar[r]^-{=}&T\ar[r]^-{g}&S.
   }
 \end{equation*}
 Then $\rH(\tau,\omega_{X/S})(U/T)$ is equivalent to $\mr{Hom}_{\mc{D}_X}(R_X,g'_*j_*j^*g'^*(\omega_{X/S}))$.
 This is equivalent to
 $\mr{Hom}_{\mc{D}_U}(R_U,j^*g'^*f^!R_S)\simeq\mr{Hom}(j_!R_U,f^!_Tg^*R_S)
 \simeq\mr{Hom}(\cH(U/T),R_T)$, where the first equivalence follows since $g$ is smooth.
 The last one is equivalent to ${}^{\mr{open}}\Hbm(U/T)$.
\end{rem*}

\begin{proof}
 We denote $(X/S)^{\mr{open}}_{\mr{sm}}$ by $(X/S)_{\mr{sm}}$ for short.
 Let $\mc{C}:=(\mr{Sm}_S)^{\mr{op}}$, and $\mc{D}_{\mc{C}}:=\mc{D}\times_{\mr{Sch}_{/S}^{\mr{op}}}\mc{C}$.
 Let $\pi\colon\mc{C}\rightarrow\Delta^0$.
 The pullback functor induces the functor $f^*\colon\pi^*\mc{D}_S\rightarrow\mc{D}_{\mc{C}}$ by left Kan extension.
 The functor $f^*$ preserves coCartesian edges and $f^*_T\colon\mc{D}_S\rightarrow\mc{D}_T$ admits a left adjoint, which we denote by $f_{T\sharp}$,
 for each $T\in\mc{C}$ since $T\rightarrow S$ is smooth.
 Thus, we have a left adjoint $f_{\sharp}^{\mr{op}}\colon\widetilde{\mc{D}}_{\mc{C}}^{\mr{op}}\rightarrow\pi^*_{\mc{C}^{\mr{op}}}\mc{D}_S$ relative to $\mc{C}^{\mr{op}}$
 of the functor between {\em Cartesian} fibrations
 $(\widetilde{f^*})^{\mr{op}}\colon\pi^*_{\mc{C}^{\mr{op}}}\mc{D}_S\rightarrow\widetilde{\mc{D}}_{\mc{C}}^{\mr{op}}$ by \cite[7.3.2.6]{HA}.
 The functor $f_{\sharp}$ commutes with the left-tensor structure $\Mod_R\times\mc{D}\rightarrow\mc{D}$
 since $f_{\sharp}$ is a certain twist of $f_!$.
 Consider the composite
 \begin{equation*}
  G\colon
   (X/S)^{\mr{op}}_{\mr{sm}}
   \xrightarrow{\widetilde{{}^{\mr{open}}\cH}}
   \widetilde{\mc{D}}_{\mc{C}}
   \xrightarrow{f_{\sharp}}
   \pi^*\mc{D}_S^{\mr{op}}
   \xrightarrow{\mr{pr}}
   \mc{D}_S^{\mr{op}}.
 \end{equation*}
 By Lemma \ref{adjoilemM} and Proposition \ref{concdescM}, ${}^{\mr{open}}\Hbm|_{(X/S)_{\mr{sm}}}$ is equivalent to the composite
 $(X/S)^{\mr{op}}_{\mr{sm}}\xrightarrow{G\times\{R\}}\mc{D}_S^{\mr{op}}\times\mc{D}_S\xrightarrow{\Mor}\Mod_R$.
 It will be sufficient to show that $G$ is equivalent to $h_!\circ\mr{LKE}(R\circ\tau)$, where $h_!\colon\mc{D}_X\rightarrow\mc{D}_S$.
 We will analyze $G$ more closely.

 By taking a suitbale model for $(-)^{\sim}$ and taking $(-)^{\mr{op}}$,
 we have a categorical fibration $(\widetilde{\mc{D}}^*_!)^{\mr{op}}\rightarrow\mr{Ar}:=\mr{Fun}(\Delta^1,\mr{Sch}_{/S})$
 of Cartesian fibrations over $\mr{Sch}_{/S}$ which preserves Cartesian edges.
 Let $\mr{Ar}':=\mr{Ar}\times_{\mr{Sch}_{/S}}\mr{Sm}_S$.
 We consider the base change $r\colon(\widetilde{\mc{D}}^*_!)^{\mr{op}}\times_{\mr{Ar}}\mr{Ar}'\rightarrow\mr{Ar}'$
 which is a map of Cartesian fibrations over $\mr{Sm}_S$.
 This map is neither Cartesian nor coCartesian (and only categorical fibration by our choice).
 However, we have:
 \begin{quote}
  (*)\, Let $g\colon T\rightarrow T'$ be a {\em smooth} morphism in $\mr{Sm}_S$.
  For any edge $e\colon\Delta^1\rightarrow\mr{Ar}'$ over $g$ and any object
  $\mc{F}\in(\widetilde{\mc{D}}^*_!)^{\mr{op}}\times_{\mr{Ar}}\mr{Ar}'$ such that $r(\mc{F})=e(0)$,
  there exists an $r$-coCartesian edge $f\colon\mc{F}\rightarrow\mc{G}$ over $e$.
 \end{quote} 
 Indeed, by \cite[1.5]{ABiv}, it suffices to show that the fiber $r_g:=r\times_{\mr{Sm}_S,g}\Delta^1$ is a coCartesian fibration.
 To show this, we first show that $r_g$ is a Cartesian fibration.
 Indeed, we must check the conditions of \cite[1.4.14]{LGood}.
 The condition \cite[1.4.14 (a)]{LGood} is easy.
 For $Z\in\mr{Sch}_{/S}$, the fiber $D_Z\colon(\mc{D}^*_!)_Z\rightarrow\mr{All}^{\mr{op}}_Z\simeq\mr{Sch}_{/Z}$
 is a coCartesian fibration associated to the functor $\mr{Sch}_{/Z}\rightarrow\Cat_\infty$ sending $Y/Z$ to $\mc{D}_Y$ and $f$ to $f_!$.
 Since $f_!$ admits a right adjoint, $f^!$, $D_Z$ is in fact a Cartesian fibration as well by \cite[5.2.2.5]{HTT}.
 Since $(r_g)_0=r_T$, $(r_g)_1=r_{T'}$ are base changes of $D_T$, $D_{T'}$, they are Cartesian fibrations.
 Thus, \cite[1.4.14 (b)]{LGood} is satisfied.
 Since $f^!$ for any morphism $f$ and $g^*$ commutes because $g$ is assumed to be smooth, and \cite[1.4.14 (c)]{LGood} follows.
 Thus $r_g$ is a Cartesian fibration as desired.
 Note that up till now, we have not used that $g$ is smooth.
 To show that $r_g$ is coCartesian, it suffices to show this over a fixed edge of $(\mr{Ar}')_g$.
 Assume we are given an edge of $(\mr{Ar}')_g$ depicted as
 \begin{equation*}
  \xymatrix{
   U\ar[r]^-{f}\ar[d]&U'_T\ar[r]^-{\widetilde{g}}\ar@{}[rd]|\square\ar[d]&U'\ar[d]\\
  T\ar[r]^-{=}&T\ar[r]^-{g}&T'.
   }
 \end{equation*}
 Unwinding the definition, the Cartesian pullback associated to this edge is $f^!\widetilde{g}^*$.
 This admits a left adjoint $\widetilde{g}_{\sharp}f_!$, thus $r_g$ is a coCartesian fibration as desired.

 Let $(X/S)^{\unrhd}$ be the (ordinary) category whose set of objects is $\mr{Obj}\bigl((X/S)_{\mr{sm}}\bigr)\sqcup\mr{Obj}(\mr{Sm}_S)$.
 The object corresponding to $T\in\mr{Sm}_S$ is denoted by $T/T$.
 The morphism is defined so that $\mr{Hom}(Y/T,Y'/T')$ is equal;
 1.\ to $\mr{Hom}(Y/T,Y'/T')$ if both $Y/T$ and $Y'/T'$ are in $(X/S)_{\mr{sm}}$;
 2.\ to $\mr{Hom}(T,T')$ if $Y'/T'\in\mr{Sm}_S$;
 and 3.\ to $\emptyset$ if $Y/T\in\mr{Sm}_S$ and $Y'/T'\in(X/S)_{\mr{sm}}$.
 In other words, it is a Grothendieck construction of the functor $\mr{Sm}_S^{\mr{op}}\rightarrow\Cat_{\infty}$ sending $T$ to
 $\mr{Open}(X_T)^{\triangleright}$, where $\mr{Open}$ is the category of open subsets.
 The map $(X/S)^{\unrhd}\rightarrow\mr{Sm}_S$ is a Cartesian fibration, and we have the inclusions
 $(X/S)_{\mr{sm}}\hookrightarrow(X/S)^{\unrhd}\hookleftarrow\mr{Sm}_S$.
 We have the functor $\alpha\colon(X/S)^{\unrhd}\rightarrow\mr{Ar}'$
 extending the functor $u\colon(X/S)_{\mr{sm}}\rightarrow\mr{Ar}'$ sending $U/T$ to $U/T$ such that
 a map $U/T\rightarrow T'/T'$ in $(X/S)^{\unrhd}$ where $U/T\in(X/S)_{\mr{sm}}$ is sent to the diagram in $\mr{Ar}'$
 \begin{equation*}
  \xymatrix{
   U\ar[r]\ar[d]&T\ar[r]\ar[d]\ar@{}[rd]|\square&T'\ar[d]\\
  T\ar[r]^-{=}&T\ar[r]&T'
   }
 \end{equation*}
 Consider the base change $\beta=\alpha^*r\colon\mc{X}:=\alpha^*(\widetilde{\mc{D}}^*_!)^{\mr{op}}\rightarrow(X/S)^{\unrhd}$.
 Since $r$ and $\alpha$ preserves Cartesian edges over $\mr{Sm}_S$, $\mc{X}$ is a Cartesian fibration over $\mr{Sm}_S$ as well.
 Thus, we have the property (*) also for $\beta$ using \cite[1.5]{ABiv}.
 Moreover, unwinding the definition, we have the following homotopy Cartesian diagram
 \begin{equation}
  \label{bcrelYandX}
   \tag{$\star$}
   \xymatrix{
   \mc{X}\ar[d]_{\beta=\alpha^*r}&\mc{Y}\ar[d]^{\beta'}\ar[l]\ar[r]\ar@{}[rd]|\square\ar@{}[ld]|\square&\widetilde{\mc{D}}^{\mr{op}}\ar[d]\\
  (X/S)_{\mr{sm}}^{\unrhd}&(X/S)_{\mr{sm}}\ar[l]\ar[r]^-{u}&\mr{Sch}_{/S}
   }
 \end{equation}
 of categorical fibrations, where $u$ is the functor sending $U/T$ to $U$.

 Consider the functor $(X/S)_{\mr{sm}}\rightarrow\mr{Sch}_{/S}$ sending $U/T$ to $U$.
 By using the section $R\colon\mr{Sch}_{/S}\rightarrow\widetilde{\mc{D}}^{\mr{op}}$ sending $T$ to $R_T$,
 we have the composite $(X/S)_{\mr{sm}}\xrightarrow{\tau}\mr{Sch}_{/S}\xrightarrow{R}\widetilde{\mc{D}}^{\mr{op}}$,
 and (\ref{bcrelYandX}) induces the functor $G'\colon(X/S)_{\mr{sm}}\rightarrow\mc{X}$.
 Unwinding the definition, this is equivalent to $\widetilde{F}^{\mr{op}}$ of \ref{introopenbm}.
 We consider the following commutative diagram on the left
 \begin{equation*}
  \xymatrix@C=50pt{
   (X/S)_{\mr{sm}}\times\{0\}\ar[d]_{i}\ar[r]^-{G'}&\mc{X}\ar[d]^{\beta}\\
  (X/S)_{\mr{sm}}\times\Delta^2\ar[r]^-{\gamma}&(X/S)^{\unrhd},
   }\qquad
  \xymatrix{
   U\ar[r]\ar[d]&T\ar[r]\ar[d]\ar@{}[rd]|\square&S\ar[d]\\
  T\ar[r]^-{=}&T\ar[r]&S,}
 \end{equation*}
 where $\gamma$ sends $U/T$ to the diagram on the right above.
 We wish to show that a $\beta$-left Kan extension of $G'$ along $i$ exists.
 In view of \cite[4.3.2.15]{HTT}, it suffices to show the existence of certain $\beta$-colimits.
 This existence follows by using the property (*), which is possible since the morphism $T\rightarrow S$ in the diagram above is smooth,
 and \cite[4.3.1.9]{HTT} to reduce the computation to colimits on the fiber over $S\in\mr{Sm}_S$,
 and invoke \cite[4.3.1.15]{HTT}.

 Let us denote the $\beta$-left Kan extension by $\mr{LKE}_i(G')$.
 The functor $\mr{LKE}_i(G')|_{(X/S)_{\mr{sm}}\times\{2\}}$ is equivalent to $G^{\mr{op}}$ by using \cite[4.3.2.8]{HTT}.
 Let $i'\colon(X/S)_{\mr{sm}}\times\{0\}\hookrightarrow(X/S)_{\mr{sm}}\times\Delta^{\{0,2\}}$,
 where $(X/S)_{\mr{sm}}\times\Delta^{\{0,2\}}$ is considered over $(X/S)^{\unrhd}$ via $\gamma|_{\Delta^{\{0,2\}}}$.
 We also consider $\gamma'\colon(X/S)_{\mr{sm}}\times\Delta^2\rightarrow(X/S)^{\unrhd}$
 sending $U/T$ to the diagram
 \begin{equation*}
  \xymatrix{
   U\ar[r]\ar[d]&X\ar[r]\ar[d]&S\ar[d]\\
  T\ar[r]&S\ar[r]^-{=}&S,}
 \end{equation*}
 and $k\colon(X/S)_{\mr{sm}}=(X/S)_{\mr{sm}}\times\{0\}\hookrightarrow(X/S)_{\mr{sm}}\times\Delta^2$ over $(X/S)^{\unrhd}$ via $\gamma'$.
 Then
 \begin{equation*}
  G^{\mr{op}}
   \simeq
  \mr{LKE}_i(G')|_{(X/S)_{\mr{sm}}\times\{2\}}\simeq\mr{LKE}_{i'}(G')|_{(X/S)_{\mr{sm}}\times\{2\}}
   \simeq
   \mr{LKE}_k(G')|_{(X/S)_{\mr{sm}}\times\{2\}},
 \end{equation*}
 where the existence of the left Kan extensions follows similarly.
 Let $k'\colon(X/S)_{\mr{sm}}\hookrightarrow(X/S)_{\mr{sm}}\times\Delta^{\{0,1\}}$,
 where $\Delta^{\{0,1\}}$ is considered to be over $(X/S)^{\unrhd}$ via $\gamma'$,
 and recall that $h\colon X\rightarrow S$ is the morphism we are fixing.
 Then $G^{\mr{op}}$ is equivalent to the composite
 \begin{equation*}
  (X/S)_{\mr{sm}}
   \xrightarrow{\mr{LKE}_{k'}(G')|_{(X/S)_{\mr{sm}}\times\{1\}}}
   \mc{X}_{(X/S)}
   \simeq
   \mc{D}_X
   \xrightarrow{h_!}
   \mc{D}_S.
 \end{equation*}
 Now, the diagram for $\mr{LKE}_{k'}(G')$ factors as
 \begin{equation*}
  \xymatrix@C=60pt{
   (X/S)_{\mr{sm}}\times\{0\}\ar[r]_-{G'_{\mc{Y}}}\ar@/^15pt/[rr]^-{G'}\ar[d]_{k'}&\mc{Y}\ar[r]\ar[d]^-{\beta'}\ar@{}[rd]|\square&\mc{X}\ar[d]^{\beta}\\
  (X/S)_{\mr{sm}}\times\Delta^{\{0,1\}}\ar[r]^-{\gamma'|_{\Delta^{\{0,1\}}}}&(X/S)_{\mr{sm}}\ar[r]&(X/S)^{\unrhd}.
   }
 \end{equation*}
 Since $\beta$ is presentable, $\mr{LKE}_{k'}(G'_{\mc{Y}})$, the $\beta'$-left Kan extension of $G'_{\mc{Y}}$ along $k'$,
 coincides with $\mr{LKE}_{k'}(G')$.
 Finally, in view of the diagram (\ref{bcrelYandX}) as well as invoking \cite[4.3.1.15]{HTT},
 we have $\mr{LKE}_{k'}(G'_{\mc{Y}})\simeq\mr{LKE}(R\circ\tau)$.
 Thus, we get the claim because $h^!R_S\simeq\omega_{X/S}$.
\end{proof}

\subsection{}
\label{bgprop}
Let us recall Brown-Gersten property (B.G.-property for short) of cdh-topology.
For the detail, see \cite{V}.
Let $S$ be a noetherian separated scheme of finite dimension,
and let $P$ be a collection of commutative diagrams $Q$ in $\mr{Sch}_{/S}$, stable under isomorphisms, of the form:
\begin{equation}
 \label{bgpropQ}
 \xymatrix{
  B\ar[r]\ar[d]&Y\ar[d]^{p}\\
 A\ar[r]^-{e}&X.
  }
\end{equation}
Let $S_P$ be the site whose underlying category is $\mr{Sch}_{/S}$, and whose topology is the
coarsest topology containing the sieves generated by the morphisms $\{A\rightarrow X,Y\rightarrow X\}$ for some $Q\in P$.
On the other hand, $P$ induces a collection of maps $G_P^0:=\{Y\sqcup_B A\rightarrow X\}_{Q\in P}$ in $\PShv(\mr{Sch}_{/S})$.
Let $G_P$ be the union of $G^0_P$ and $\emptyset'\rightarrow\emptyset$, where $\emptyset'$ is an initial object of
$\PShv(\mr{Sch}_{/S})$ and $\emptyset$ is the presheaf represented by the initial object of $\mr{Sch}_{/S}$.
One of the favorable features of the theory is that we can detect sheaves on $S_P$ easily, namely it possesses a B.G.-property:

\begin{lem*}
 Let $\mc{D}$ be a presentable $\infty$-category.
 Assume we are given a collection of squares $P$ which is regular complete and bounded
 {\normalfont(}cf.\ {\normalfont\cite{V}} for the terminologies{\normalfont)}.
 Let $\mc{F}\colon\mr{Sch}_{/S}^{\mr{op}}\rightarrow\mc{D}$ be a functor such that the canonical morphism
 $\mc{F}(X)\rightarrow\mc{F}(Y)\times_{\mc{F}(B)}\mc{F}(A)$ is an equivalence for any $Q$ and $\mc{F}(\emptyset)$ is an initial object.
 Then $\mc{F}$ is a $\mc{D}$-valued sheaf on $S_P$.
\end{lem*}
\begin{proof}
 By \cite[5.1.5.6]{HTT}, $\mc{F}$ extends to a colimit preserving functor
 $\widetilde{\mc{F}}\colon\PShv(S_P)\rightarrow\mc{D}^{\mr{op}}$.
 By \cite[Proposition 3.8]{V}, $G_P^{-1}\PShv(S_P)\simeq\Shv(S_P)$.
 The assumption of the lemma implies that $\widetilde{\mc{F}}$ sends any morphism in $G_P$ to an equivalence in $\mc{D}^{\mr{op}}$.
 By \cite[5.5.4.20]{HTT}, we get that $\widetilde{\mc{F}}$ factors through a colimit preserving functor $\Shv(S_P)\rightarrow\mc{D}^{\mr{op}}$,
 and the lemma follows.
\end{proof}

\subsection{}
\label{bgpropofetacdp}
We consider the following types of squares.
A square $Q$ is said to be a {\em blowup square} if the square is Cartesian, $e$ is a closed immersion, $p$ is proper,
and $p\times_{X}(X\setminus A)$ is an isomorphism.
A square $Q$ is said to be a {\em Nisnevich square} if the square is Cartesian, $e$ is an open immersion, $p$ is \'{e}tale,
and $p\times_{X}(X\setminus A)_{\mr{red}}$ is an isomorphism.
If we put $P_{\mr{cdh}}$ to be the collection of blowup squares and Nisnevich squares, then $S_{P_{\mr{cdh}}}$ is nothing but the cdh-site on $S$.
The main theorems of \cite[Theorem 2.2]{V2} states that $P_{\mr{cdh}}$ is regular complete and bounded, and we may apply the previous lemma.
Recall that we denote by $\eta$ the generic point of $\Box$.
A square $Q$ in $\mr{Sch}_{/S\times\Box}$ is called a {\em $\eta$-blowup square}
if the square is Cartesian, $e$ is a closed immersion and $p$ is proper,
the base change $Q\times_{\Box}\eta$ is a blowup square.
The collection of $\eta$-blowup squares is denoted by $P_{\eta\mbox{-}\mr{cdp}}$.
It is straightforward to check that $(S\times\Box)_{P_{\eta\mbox{-}\mr{cdp}}}$ coincides with the Grothendieck site $\mr{Sch}_{/S\times\Box,\eta\mbox{-}\mr{cdp}}$
introduced in \S\ref{sect1}.

\begin{lem*}
 The collection $P_{\eta\mbox{-}\mr{cdp}}$ is regular complete and bounded.
\end{lem*}
\begin{proof}
 It is complete by \cite[Lemma 2.5]{V}.
 To show that it is regular, we wish to check the conditions of \cite[Lemma 2.11]{V}.
 The first two conditions are obvious, and we must check that $d(Q)$, using the notation therein, belongs to $P_{W\mbox{-}\mr{cdp}}$.
 Let $U$ be an open neighborhood of $W$ such that $Q\times_{\Box}U$ is a cdh-square.
 Since $d(Q)$ is a Cartesian square and the morphisms are closed immersions,
 it suffices to check that $d(Q)\times_{\Box}U\cong d(Q\times_{\Box}U)$ is a blowup square.
 This follows by the proof of \cite[Lemma 2.14]{V2}
 in which Voevodsky checks the conditions of \cite[Lemma 2.11]{V}
 for cdh-squares.

 Finally, we define the density function for $X\in\mr{Sch}_{/S\times\Box}$ as follows.
 Let $D_i$ be the standard density function defined in \cite[right after Lemma 2.3]{V2}.
 First, we assume $W=U$ is an open subset of $\Box$.
 For $0\leq i\leq\dim(X\times_{\Box}U)+1$, we define $D^{U}_i(X)$ to be the family of open immersions
 $j\colon V\hookrightarrow X$ such that $j\times_\Box U\in D_i(X\times_{\Box}U)$.
 We then define $D^{U}_{\dim(X\times_{\Box}U)+2}(X)$ to be the family of isomorphisms.
 For general $W$, we define $D^W_i:=\lim_{W\subset U}D_i^U$, which is well-defined since the function $U\mapsto D^U_i(X)$ is stationary for each $X$.
 By construction the dimension of $X$ with respect to $D^{W}$ is less than or equal to $\dim(X)+1$ unless $X$ is empty.
 Thus, $D_*^{W}$ is locally finite dimension.
 We are left to showing that $D^{W}$ is reducing.
 By limit argument, it suffices to check this when $W=U$ is an open subscheme.
 Assume we are given a square $Q$ of the form (\ref{bgpropQ}).
 We may assume that $p^{-1}(X\setminus A)\cong X\setminus A$ is dense in $Y$ because replacing $Y$ with the closure of $p^{-1}(X\setminus A)$ is a refinement.
 Under this assumption, we will check the condition of \cite[Definition 2.21]{V} for such $Q$.
 Let $i<\dim(X\times_{\Box}U)+1$.
 By applying \cite[Lemma 2.12]{V2} to $Q\times_{\Box}U$, we have an open immersion $X'\rightarrow X\times_{\Box}U$
 and a distinguished square $Q'$ satisfying the condition for $Q\times_{\Box}U$.
 Since the composite $X'\subset X\times_{\Box}U\subset X$ belongs to $D^{U}_{i}(X)$, we may take $Q'$ to be the desired distinguished square also for $Q$.
 Assume $i\geq\dim(X\times_{\Box}U)+1$.
 Then all the morphisms of $D^{U}_{i+1}(X)$ are isomorphisms.
 Since $A$ is closed in $X$, we have $\dim(A\times_{\Box}U)\leq\dim(X\times_{\Box}U)$.
 Since $p^{-1}(X\setminus A)\cong X\setminus A$ is assumed to be dense in $Y$,
 $\dim(Y\times_{\Box}U)=\dim\bigl((X\setminus A)\times_{\Box}U\bigr)\leq\dim(X\times_{\Box}U)$.
 Moreover, this density implies that $\dim\bigl(Y\times_{\Box}U\setminus p^{-1}(X\setminus A)\bigr)\leq\dim(Y\times_{\Box}U)-1$.
 These show that $D^{U}_{i+1}(A)$, $D^{U}_{i+1}(Y)$, and $D^{U}_{i}(B)$ consist of isomorphisms, and we may take $X'=X$ in this situation.
\end{proof}

\begin{lem}
 \label{cdhdsvalshv}
 \begin{enumerate}
  \item Let $H\colon\mr{Sch}_{/S}^{\mr{op}}\rightarrow\mc{D}_S$ be a functor such that;
	1.\ For each object $f\colon T\rightarrow S$ of $\mr{Sch}_{/S}^{\mr{op}}$, we have $H(T)\simeq f_*f^*H(S)$;
	2.\ For each morphism $T\rightarrow T'$ in $\mr{Sch}_{/S}^{\mr{op}}$,
	the morphism $H(T)\rightarrow H(T')$ is defined by the adjunction via the equivalence of the first condition.
	Then $H$ is a $\mc{D}_S$-valued cdh-sheaf.

  \item\label{cdhdsvalshv-2}
       Let $G\colon\mr{Sch}_{/S}^{\mr{op}}\rightarrow\mc{D}$ be a functor over $\mr{Sch}_{/S}^{\mr{op}}$ which preserves coCartesian edges.
	Then $M_F(R)\colon\mr{Sch}_{/S}^{\mr{op}}\rightarrow\Mod_R$ is a cdh-sheaf.

  \item For each $Y\rightarrow T$ in $(X/S)_{\mr{zar}}$, the functor
	$Y_{\mr{zar}}^{\mr{op}}\rightarrow(X/S)_{\mr{zar}}^{\mr{op}}\xrightarrow{{}^{\mr{open}}\Hbm}\Mod_R$
	is a $\Mod_R$-valued Zariski-sheaf on $Y$.
 \end{enumerate}
\end{lem}
\begin{proof}
 Let us show the first claim.
 By Lemma \ref{bgprop}, it suffices to check the following conditions:
 \begin{enumerate}
  \item $H(\emptyset)$ is an initial object;

  \item For a blowup square (resp.\ Nisvevich square) $Q$ of the form (\ref{bgpropQ}), the morphism
	$H(X)\rightarrow H(Y)\times_{H(B)}H(A)$ is an equivalence;
 \end{enumerate}
 The first condition is obvious, and the verification of the second condition is standard.
 For the convenience of the reader, we recall an argument for the second condition.
 We only treat the blowup square.
 For an object $f\colon T\rightarrow S$ of $\mr{Sch}_{/S}^{\mr{op}}$, let $\mc{F}_T:=f^*H(S)$.
 Let $i=e\colon A\rightarrow X$ be the closed immersion,
 $j\colon X\setminus A\rightarrow X$ be the complement, $f=p\colon Y\rightarrow X$, and $h\colon B\rightarrow X$.
 We must show that the map $\mc{F}_X\rightarrow f_*\mc{F}_Y\times_{h_*\mc{F}_B}i_*\mc{F}_A$ is an equivalence.
 Since the pair $(j_!j^*,i_*i^*)$ is conservative (cf.\ \cite[6.1]{ABiv}),
 it suffices to check the equivalence after applying $j_!j^*$ and $i_*i^*$ separately.
 This is easy by using localization sequence.

 Let us show the second claim.
 Let $\H\colon\mr{Sch}_{/S}^{\mr{op}}\rightarrow\mc{D}_S$ be the functor sending $f\colon T\rightarrow S$ to $f_*R$
 defined by taking a  $p$-right Kan extension of $R$.
 By Lemma \ref{adjoilemM} and Proposition \ref{concdescM}, $M_G(R)$ is equivalent to the composite
 $\mr{Sch}_{/S}^{\mr{op}}\xrightarrow{\{G(S)\}\times\H}\mc{D}_S^{\mr{op}}\times\mc{D}_S\xrightarrow{\Mor}\Mod_R$.
 Since $\H$ is a $\mc{D}_S$-valued cdh-sheaf by the first claim, this composite is a $\Mod_R$-valued cdh-sheaf as required.
 For the last claim, since Zariski topology also satisfy the B.G.-property with respect to Mayer-Vietoris squares,
 and we may prove similarly.
\end{proof}

\noindent
Tomoyuki Abe:\\
Kavli Institute for the Physics and Mathematics of the Universe (WPI)\\
The University of Tokyo\\
5-1-5 Kashiwanoha,  
Kashiwa, Chiba, 277-8583, Japan\\
e-mail: {\tt tomoyuki.abe@ipmu.jp}

\end{document}